\documentclass[a4paper,reqno]{amsart}

\usepackage[french,british]{babel}
\usepackage{amssymb,amsmath,amsthm,amscd,mathrsfs,graphicx,color,multicol,booktabs,bm,manfnt} 
\usepackage[cmtip,all,matrix,arrow,tips,curve]{xy}
\usepackage[T1]{fontenc} 
\usepackage{booktabs}

\usepackage{caption}     
\usepackage{blindtext}        
\usepackage{tikz} 

\usepackage{mathrsfs}

\usepackage[colorlinks=true,citecolor=blue, urlcolor=blue, linkcolor=blue, pagebackref]{hyperref}

\usepackage{mathabx}

\usepackage{manfnt} 
\usepackage{centernot}    

\newcommand{\bD}{{\mathbf D}}
\newcommand{\bE}{{\mathbf E}}
\newcommand{\bepsilon}{{\bm\epsilon}}

\newcommand{\biota}{{\bm\iota}}
\newcommand{\bP}{{\mathbf P}}

\newcommand{\brho}{{\bm\rho}}
\newcommand{\btau}{{\bm\tau}}
\newcommand{\bu}{\bullet}

\newcommand{\bvarphi}{{\bm\varphi}}

\newcommand{\CC}{{\mathbb C}}
\newcommand{\cA}{{\mathscr A}}

\newcommand{\cB}{{\mathscr B}}
\newcommand{\cC}{{\mathscr C}}
\newcommand{\cCH}{\mathscr{C\mkern-4mu H}}
\newcommand{\cD}{{\mathscr D}}
\newcommand{\cE}{{\mathscr E}}
\newcommand{\cF}{{\mathscr F}}
\newcommand{\cG}{{\mathscr G}} 

\newcommand{\cH}{{\mathscr H}}
\newcommand{\cI}{{\mathscr I}}
  
\newcommand{\cK}{{\mathscr K}}
\newcommand{\cL}{{\mathscr L}}
\newcommand{\cM}{{\mathscr M}}
\newcommand{\cN}{{\mathscr N}}

\newcommand{\cO}{{\mathscr O}}
\newcommand{\cP}{{\mathscr P}}  
\newcommand{\cQ}{{\mathscr Q}}
\newcommand{\cR}{{\mathscr R}}
\newcommand{\cS}{{\mathscr S}}

\newcommand{\cU}{{\mathscr U}}
\newcommand{\cV}{{\mathscr V}} 
\newcommand{\cW}{{\mathscr W}}

\newcommand{\cX}{{\mathscr X}} 
\newcommand{\cY}{{\mathscr Y}}
\newcommand{\cZ}{{\mathscr Z}}

\newcommand{\dra}{\dashrightarrow}
\newcommand{\EE}{{\mathbb E}}
\newcommand{\es}{\emptyset}

\newcommand{\GG}{{\mathbb G}}

\newcommand{\gotg}{\mathfrak{g}}

\newcommand{\gm}{\mathfrak{m}}

\newcommand{\hra}{\hookrightarrow}

\newcommand{\la}{\langle}

\newcommand{\lra}{\longrightarrow}

\newcommand{\NN}{{\mathbb N}}
\newcommand{\ov}{\overline}
\newcommand{\PP}{{\mathbb P}}

\newcommand{\QQ}{{\mathbb Q}}
\newcommand{\ra}{\rangle}
\newcommand{\RR}{{\mathbb R}}

\newcommand{\ul}{\underline}

\newcommand{\wh}{\widehat}
\newcommand{\wt}{\widetilde}
\newcommand{\ww}{{\bf w}}
\newcommand{\ZZ}{{\mathbb Z}}
\newcommand{\Gr}{\mathrm{Gr}}

\theoremstyle{plain} 

\newtheorem{thm}{Theorem}[section]
\newtheorem*{thm*}{Theorem}

\newtheorem{clm}[thm]{Claim}

\newtheorem{crl}[thm]{Corollary}

\newtheorem*{hyp*}{Hypothesis}
\newtheorem{lmm}[thm]{Lemma}
\newtheorem{prp}[thm]{Proposition}
\newtheorem{prp-dfn}[thm]{Proposition-Definition}

\theoremstyle{definition}

\newtheorem{ass}[thm]{Assumption}
\newtheorem{ass-dfn}[thm]{Assumption-Definition}
\newtheorem{hyp-dfn}[thm]{Hypothesis-Definition}
\newtheorem{dfn}[thm]{Definition}

\newtheorem{ntn}[thm]{Notation}

\theoremstyle{remark}

\newtheorem*{qst*}{Main Question}
\newtheorem{rmk}[thm]{Remark}
\newtheorem{rmk-dfn}[thm]{Remark-Definition}
\newtheorem{key-rmk}[thm]{Key Remark}

\DeclareMathOperator{\Amp}{Amp}
\DeclareMathOperator{\Ann}{Ann}

\DeclareMathOperator{\BKR}{BKR}

\DeclareMathOperator{\ch}{ch}
\DeclareMathOperator{\CH}{CH}

\DeclareMathOperator{\cl}{cl}

\DeclareMathOperator{\coker}{coker}

\DeclareMathOperator{\divisore}{div}

\DeclareMathOperator{\ext}{ext}
\DeclareMathOperator{\Ext}{Ext}

\DeclareMathOperator{\GL}{GL}

\DeclareMathOperator{\Hom}{Hom}
\DeclareMathOperator{\Id}{Id}
\DeclareMathOperator{\im}{Im}

\DeclareMathOperator{\Kum}{Kum}

\DeclareMathOperator{\Mon}{Mon}

\DeclareMathOperator{\Nef}{Nef}

\DeclareMathOperator{\NS}{NS}

\DeclareMathOperator{\PGL}{PGL}

\DeclareMathOperator{\pr}{pr}
\DeclareMathOperator{\Princ}{Princ}

\DeclareMathOperator{\SL}{SL}
\DeclareMathOperator{\Span}{span}
\DeclareMathOperator{\Spec}{Spec}

\DeclareMathOperator{\supp}{supp}
\DeclareMathOperator{\Sym}{Sym}
\DeclareMathOperator{\Td}{Td}

\DeclareMathOperator{\Tor}{Tor}
\DeclareMathOperator{\Tors}{Tors}

\usepackage{multirow}

 \makeindex
 
\begin{document}  
 \title[HK manifolds of Type $K3^{[a^2+1]}$ as moduli spaces of projective bundles]
 {HK manifolds of Type $K3^{[a^2+1]}$ as moduli spaces of projective bundles on  HK manifolds of Type $K3^{[2]}$}
 \author{Kieran G. O'Grady}
 \address{Dipartimento di Matematica, 
 Sapienza Universit\`a di Roma,
 P.le A.~Moro 5,
 00185 Roma - ITALIA}
 \email{ogrady@mat.uniroma1.it}
\date{\today}
\thanks{Partially supported by PRIN 2017YRA3LK}
\begin{abstract}
We prove results on moduli spaces of $\omega$ slope stable bundles of projective spaces on a hyperk\"ahler (HK) manifold $X$ of Type $K3^{[2]}$ (a HK manifold
is of \emph{Type $K3^{[m]}$} if it  deforms to  the Hilbert scheme $S^{[m]}$ parametrizing length-$m$ subschemes of a $K3$ surface $S$) with K\"ahler class  $\omega$. 
Let $X$ be projective and  $h$ be a (generic) ample class of $X$. We prove that the moduli space  $M_{\ov{\ww}_a}(X,h)$  parametrizing $h$ slope stable bundles with   a suitable mock Mukai vector $\ov{\ww}_a$ contains an irreducible component 
$M_{\ov{\ww}_a}(X,h)^{\bullet}$ whose normalization $\wt{M}_{\ov{\ww}_a}(X,h)^{\bullet}$ is a (projective)
HK manifold of Type $K3^{[a^2+1]}$, and that conversely every projective
HK manifold $W$ of Type $K3^{[a^2+1]}$ is isomorphic to $\wt{M}_{\ov{\ww}_a}(X,h)^{\bullet}$ for a suitable $(X,h)$ as above.  

Moreover the universal bundle of projective spaces on $X\times \wt{M}_{\ov{\ww}_a}(X,h)^{\bullet}$ defines a vector bundle  whose
$2nd$ Chern class defines a rational Hodge isometry $H^2(X)\to H^2(\wt{M}_{\ov{\ww}_a}(X,h)^{\bullet})$. From this and a result of Markman one gets that the analogue of the Shafarevich  conjecture (a special case of the Hodge conjecture) holds for rational Hodge isometries $H^2(W_1) \lra H^2(W_2)$ between projective HK manifolds $W_1,W_2$ of Types $K3^{[a_1^2+1]}$ and 
$K3^{[a_2^2+1]}$ respectively.

We prove results also for $(X,\omega)$  a general HK manifold
 of Type $K3^{[2]}$. In fact one ingredient in our proof is Verbistsky's theory of projectively hyperhomolorphic vector bundles.
\end{abstract}

  \maketitle
\bibliographystyle{amsalpha}
\tableofcontents
\section{Introduction}\label{sec:intro}
\subsection{Main results}\label{subsec:risulprinc}
\setcounter{equation}{0}
In  order to present our results we briefly discuss  $\PP^{r-1}$-bundles over a complex variety $X$. 
Such a bundle is   a holomorphic map  $\rho\colon\cP\to X$ of complex spaces with fiber $\PP^{r-1}$ which is locally trivial in the \emph{classical} topology. One associates (see~\cite[Sect.~2]{huybrechts-schroer:brauergrp}) to $\cP$ a class $\eta_r(\cP)\in H^2(X,\mu_r) $ as follows.
The exact sequence of sheaves (for the classical topology)
\begin{equation*}
1\lra \mu_r\lra \SL_r(\cO_X)\lra\PGL_r(\cO_X)\lra 1
\end{equation*}
defines an exact sequence of cohomology groups
\begin{equation}\label{esseproj}
H^1(X, \mu_r)\lra H^1(X,\SL_r(\cO_X))\lra H^1(X, \PGL_r(\cO_X))\overset{\partial}{\lra} H^2(X,\mu_r).
\end{equation}
A  local trivialization of $\cP$ gives a $1$-cocycle with values in $\PGL_r(\cO_X)$, whose cohomology class $\gamma(\cP)\in H^1(X, \PGL_r(\cO_X))$ is independent of the trivialization. Set
\begin{equation}\label{etadipi}
\eta_r(\cP)\coloneq \partial(\gamma(\cP)).
\end{equation}
Next we define a class $\Delta(\cP)\in H^4(X,\ZZ) $. 
Let
\begin{equation}
\gotg(\cP)\coloneq\rho_{*}\Theta_{\cP/Y}
\end{equation}
be the pushforward of the  bundle of vertical tangent vectors. If $U\subset X$ is an open subset   such that 
$\rho^{-1}(U)\xrightarrow{\sim} U\times\PP(F)$, where $F\to U$ is a (holomorphic) vector bundle  of rank $r$ then
\begin{equation}\label{senzatraccia}
\gotg(\cP)_{|U}\cong End^0(F),
\end{equation}
where $End^0(F)$ is the vector bundle of traceless endomorphisms of $F$. Thus $\gotg(\cP)$ is locally-free of rank $r^2-1$. 
Set
\begin{equation}
\Delta(\cP)\coloneq c_2(\gotg(\cP)).
\end{equation}
 If $(X,\omega)$ is a compact K\"ahler manifold there is the notion of $\omega$-slope stability for projective bundles on $X$, see Subsection~\ref{subsec:twistandstab}. Fix a \emph{mock Mukai vector}
\begin{equation}\label{doppiavubarra}
\ov{\ww}=(r,\eta,s)\in \NN_{+}\times H^2(X,\mu_r)\times H^{2,2}_{\ZZ}(X).
\end{equation}
 There is a moduli space $M_{\ov{\ww}}(X,\omega)$ (an analytic space) of $\omega$-slope stable projective $\PP^{r-1}$-bundles $\cP\to X$ with 
 $\eta_r(\cP)=\eta$ and $\Delta(\cP)=s$. 
 
Next we  define, for $a\in\NN_{+}$, a mock Mukai vector $\ov{\ww}_a$ for a HK manifold $X$ of Type $K3^{[2]}$ as follows. Let 
 $q_X$ be the Beauville-Bogomolov-Fujiki (BBF) quadratic  form of $X$ ($q_X$ denotes both the BBF quadratic form and its polarization), and
\begin{equation}
\ov{\gamma}_a=[4a^2\wh{\gamma}_a]\in H^2(X,\mu_{8a^3})\cong H^2(X;\ZZ)/8a^3 H^2(X;\ZZ)
\end{equation}
 where $\wh{\gamma}_a=[\gamma_a]\in H^2(X,\ZZ)/2a H^2(X;\ZZ)$ with  $\gamma_a\in H^2(X,\ZZ)$ such that
\begin{equation}\label{divequad}
\divisore (\gamma_a)=1,
\qquad
q_X(\gamma_a)\equiv 2a^2-2  \pmod{4a}. 
\end{equation}
Then
\begin{equation}\label{eccowbarra}
\ov{\ww}_a\coloneq(8a^3,\ov{\gamma}_a,  \frac{ 4a^6}{3} c_2(X) ).
\end{equation}
\begin{thm}\label{thm:isogmodvb}
Let $X$ be hyperk\"ahler (HK) manifold of Type $K3^{[2]}$. Let $\omega$ be a K\"ahler class on $X$ such that
\begin{equation}\label{nonsumuro}
\{\xi\in H^{1,1}_{\ZZ}(X)\mid q_X(\omega,\xi)=0,\quad -2^{16}\cdot 5\cdot a^{18}\le q_X(\xi)<0\}=\es,
\end{equation}
The following hold:
\begin{enumerate}
\item
  $M_{\ov{\ww}_a}(X,\omega)$ has an irreducible component $M_{\ov{\ww}_a}(X,\omega)^{\bullet}$ whose normalization  is a holomorphic symplectic manifold  $\wt{M}_{\ov{\ww}_a}(X,\omega)^{\bullet}$ bimeromorphic to a HK manifold of Type $K3^{[a^2+1]}$ (and deformation equivalent to such a manifold).
\item
If $H^{1,1}_{\ZZ}(X)=\{0\}$, or $X$ is projective and $\omega=c_1(L)$ for an ample line bundle $L$, then  
$\wt{M}_{\ov{\ww}_a}(X,\omega)^{\bullet}$ is a HK manifold of Type $K3^{[a^2+1]}$.
\item
Let $\wt{\cP}_{\ov{\ww}_a}$ be the pull-back to $X\times \wt{M}_{\ov{\ww}_a}(X,\omega)^{\bullet}$ of the universal $\PP^{8a^3-1}$-bundle on $X\times M_{\ov{\ww}_a}(X,\omega)^{\bullet}$, and let $\gotg(\wt{\cP}^{\bullet}_{\ov{\ww}_a})$ be the corresponding holomorphic vector bundle on 
$X\times \wt{M}_{\ov{\ww}_a}(X,\omega)^{\bullet}$. Let
\begin{equation}\label{ciduedipi}
c_2(\gotg(\wt{\cP}^{\bullet}_{\ov{\ww}_a}))^{2,2}\in H^2(X)\otimes H^2(\wt{M}_{\ov{\ww}_a}(X,\omega)^{\bullet})
\end{equation}
be the K\"unneth $(2,2)$-component of $c_2(\gotg(\wt{\cP}^{\bullet}_{\ov{\ww}_a}))$, which we view as an element 
of $H^2(X)^{\vee}\otimes H^2(\wt{M}_{\ov{\ww}_a}(X,\omega)^{\bullet})$ via the 
 isomorphism $H^2(X)\xrightarrow{\sim}H^2(X)^{\vee}$ defined by  the (non degenerate)   quadratic form $q_X$.
Then 
\begin{equation}\label{stessisuqu}
32^{-1} \cdot a^{-4}\cdot c_2(\gotg(\wt{\cP}^{\bullet}_{\ov{\ww}_a}))^{2,2}\colon H^2(X)\lra H^2(\wt{M}_{\ov{\ww}_a}(X,\omega)^{\bullet})
\end{equation}
 is a rational Hodge isometry, i.e.~an isomorphism of rational Hodge structures  matching the BBF quadratic forms (this make sense  even if  $\wt{M}_{\ov{\ww}_a}(X,\omega)^{\bullet}$ is not K\"ahler   because of Item~(1)).  
\end{enumerate}
\end{thm}
\begin{rmk}\label{rmk:qualecomp}
We do not know whether $M_{\ov{\ww}_a}(X,\omega)$ is irreducible. Independently of this open problem,
 $M_{\ov{\ww}_a}(X,\omega)^{\bullet}$ can be described as follows: it is obtained by deformation of a well-defined irreducible component of a moduli space of vector bundles on a suitably polarized $S^{[2]}$ where  $S$ is  a $K3$ surface (see Subsection~\ref{subsec:panorama}). In theory monodromy could \lq\lq generate\rq\rq\ more than one such irreducible component.
\end{rmk}
\begin{thm}\label{thm:ognik3a2+1}
Let $W$ be HK manifold of Type $K3^{[a^2+1]}$, where $a$ is a positive integer. There exist a HK manifold $X$ of Type $K3^{[2]}$ and a K\"ahler class $\omega$  on $X$ for which~\eqref{nonsumuro}  holds, such that  
$W$ is bimeromorphic to  
$\wt{M}_{\ov{\ww}_a}(X,\omega)^{\bullet}$, where the latter is as in Item~(1) of Theorem~\ref{thm:isogmodvb}. If $H^{1,1}_{\ZZ}(W)=\{0\}$, 
then there exist $(X,\omega)$ as above such that $W$ is isomorphic to  
$\wt{M}_{\ov{\ww}_a}(X,\omega)^{\bullet}$. If $W$ is projective then there exists a polarized HK manifold $(X,h)$ of Type $K3^{[2]}$
(for which~\eqref{nonsumuro}  holds with $\omega$ replaced by $h$) such that $W$ is isomorphic to  
$\wt{M}_{\ov{\ww}_a}(X,h)^{\bullet}$. 
\end{thm}
\begin{rmk}\label{rmk:unopiuaquadro}
Let $m\ge 2$. If $X$ is a  HK manifold of Type $K3^{[m]}$ then  $H^2(X;\ZZ)$ equipped with the BBF quadratic form is isometric to 
\begin{equation}\label{lambdaemme}
\Lambda_m\coloneq U^{\oplus 3}\oplus (-E_8)^{\oplus 2}\oplus (-2(m-1)),
\end{equation}
 where direct sums are orthogonal, $U$ is the hyperbolic lattice, $-E_8$ is the unique rank-$8$ unimodular even negative definite lattice, and  
 $(-2(m-1))$ is $\ZZ$ with generator of square $-2(m-1)$.  If $m,n\ge 2$ and there is an isometry 
 $\Lambda_m\otimes\QQ\xrightarrow{\sim}\Lambda_n\otimes\QQ$, then $(m-1)/(n-1)$ is the square of a rational number because the discriminants of  $\Lambda_m,\Lambda_n$ are equal to $2(m-1),2(n-1)$. 
This shows the necessity of the \lq\lq Type $K3^{[a^2+1]}$\rq\rq\ in  the statements of 
 Theorems~\ref{thm:isogmodvb} and~\ref{thm:ognik3a2+1}.  It suggests that it should be possible to generalize the statements above 
 with \lq\lq Type $K3^{[2]}$\rq\rq\ and  \lq\lq Type $K3^{[a^2+1]}$\rq\rq\ replaced by  \lq\lq Type $K3^{[m]}$\rq\rq\ and  \lq\lq Type $K3^{[a^2(m-1)+1]}$\rq\rq\ respectively.
\end{rmk}
\begin{rmk}\label{rmk:biswas}
In~\cite{biswschum:modprinckaehler}  the authors define a K\"ahler form on the moduli space of $\omega$ slope stable principal bundles on a (compact) K\"ahler manifold $(X,\omega)$ (see also~\cite[Thm.~6.3]{verb:iperolom}). If this gives a K\"ahler structure also on a singular moduli space (a statement which sounds plausible) then Item~(2) of Theorem~\ref{thm:isogmodvb} holds for any $X$, and the isomorphism 
$W\cong\wt{M}_{\ov{\ww}_a}(X,\omega)^{\bullet}$ in Theorem~\ref{thm:ognik3a2+1}  holds for any 
 HK manifold $W$ of Type $K3^{[a^2+1]}$.
\end{rmk}
\begin{rmk}\label{rmk:quantevolte}
A natural question  is the following. Given a HK manifold $W$ of Type $K3^{[a^2+1]}$, how many are the isomorphism   (or bimeromorphic equivalence) classes of 
HK manifolds $X$ of Type $K3^{[2]}$ such that $W$ is isomorphic to  
$\wt{M}_{\ov{\ww}_a}(X,\omega)^{\bullet}$? This brings to mind  the degree of the maps $\cF_{2n^2 k}\to\cF_{2k}$ between  moduli spaces of polarized $K3$ surfaces of degrees $2n^2 k$ and $2k$, see~\cite[Ch.~2]{og:tesiphd}. 
\end{rmk}
\begin{thm}\label{thm:isogenie}
Let $W_i$, for $i\in\{1,2\}$,  be projective HK manifolds of Types $K3^{[a_i^2+1]}$  ($a_i$ are integers), and suppose that  there exists a  rational Hodge isometry
\begin{equation*}
f\colon H^2(W_1) \lra H^2(W_2).
\end{equation*}
 There exists an algebraic cycle class $Z\in\CH^{2a_1^2+2}_{\QQ}(W_1\times W_2)$ whose cohomology class  belongs to $H^{4a_1^2+2}(W_1;\QQ) \otimes H^2(W_2;\QQ)\cong H^2(W_1;\QQ)^{\vee}\otimes H^2(W_2;\QQ)$  and  equals $f$.
\end{thm}
\subsection{Background and motivation}\label{subsec:retroterra}
\setcounter{equation}{0}
Every known HK manifold is, up to deformation, a moduli space (possibly desingularized) of semistable sheaves on a $K3$ or abelian surface $S$
 (in the latter case we factor out  $S$ and $S^{\vee}$ from the Beauville-Bogomolov decomposition of the moduli space). The  key proviso in the last sentence is \lq\lq up to deformation\rq\rq. A smooth moduli space of semistable sheaves on a $K3$  surface $S$ is of Type $K3^{[m]}$. If $m\ge 2$ such a HK manifold  has $21$ moduli ($20$ moduli as projective manifold), while $S$ has $20$ moduli (respectively $19$ moduli as projective manifold). Similar considerations apply to HK manifolds of Type OG10  (the difference between the number of deformations of OG10 and a $K3$ surface is $2$), or to HK manifolds of Type $\Kum_m$ ($m\ge 2$) or of Type OG6 (the difference between the number of deformations of $\Kum_m$ or OG6 and those of a two-dimensional complex torus is $1$ and $2$ respectively).

The strength of our Theorem~\ref{thm:ognik3a2+1}  is that it realizes an arbitrary projective HK manifold of Type  $K3^{[a^2+1]}$ as a moduli space of projective bundles on a (projective) HK manifold of Type  $K3^{[2]}$. One may compare our result to~\cite[Cor.~29.5]{blmnps:hermeinstprinc} stating that, given   coprime integers $a,b$, the
moduli spaces of Bridgeland stable objects in the Kuznetsov components of  general cubic fourfolds with a suitable Mukai vector $v(a,b)$ 
give a locally complete family of polarized HK manifolds $(Y,h)$ of Type  $K3^{[a^2-ab+b^2+1]}$. Since objects  in the Kuznetsov component of a cubic fourfold $W$ are objects in the derived category of $W$, we see a similarity with our results
 on (twisted) vector bundles on a HK fourfold. On the other hand  Kuznetsov components are  $2$-dimensional CY categories: this accounts  for the fact that their moduli spaces are smooth, and the irreducibility statement in their result.
 To our advantage there is the possibility of considering moduli space of slope stable  bundles of projective spaces on arbitrary HK manifolds, and hence using Verbitsky's powerful theory of projectively hyperholomorphic (twisted) vector bundles.

The main effort in the proof of our results goes into studying moduli spaces of slope stable vector bundles with certain characteristic classes  on the Hilbert scheme $S^{[2]}$, where $S$ is a projective $K3$ surface, see Subsection~\ref{subsec:panorama}. Recently there has been a surge of interest in moduli spaces of semistable sheaves on $HK$ manifolds. The author realized~\cite{ogfascimod} that if one deals with so-called modular (torsion-free) sheaves then variation of slope-stability  behaves as if the HK manifold were  a surface, and that stable modular sheaves on Lagrangian-fibered HK manifolds  behave similarly to stable sheaves on elliptic $K3$ surfaces (up to a point). This was used to prove results on slope stable rigid vector bundles on higher dimensional $HK$ manifolds which are reminiscent of  known results valid for $K3$ surfaces, see loc.~cit.~and~\cite{og:fascik3n,og:fascikum}. Markman~\cite{markman:rank1obs} and Beckmann~\cite{beckmann:atomic} introduced the important notion of atomic sheaf, or complex of sheaves, on a  $HK$ manifold. Atomic sheaves are key ingredients in Markman's proof of   the Shafarevich 
conjecture  for couples of HK manifolds (both) of Type $K3^{[m]}$~\cite{markman:isogenies}. Bottini~\cite{bottini:towardsog10,bottini:og10asmodspace}   proved that certain moduli spaces of (twisted) atomic sheaves on HK manifolds of Type $K3^{[2]}$ have an irreducible component which is a HK manifold of Type OG10. All of the sheaves (which are actually locally free) that appear in the works mentioned above are atomic.  In the present work the sheaves are modular but \emph{not} atomic. On the other hand, they are (again) locally free\footnote{\emph{ Where have all the singular sheaves gone?}} and projectively hyperholomorphic. 

We feel that the results of the present paper (and those in~\cite{og:highdim}) give a glimpse of the theory of sheaves on HK manifolds of Type $K3^{[2]}$, and that it reinforces the impression that it resembles the  theory of sheaves on $K3$ surfaces. In fact we get moduli spaces which are of arbitrarily  large dimension and which contain a HK component (up to normalization). On the other hand it is  not clear (at least to me) what the 
 actual picture might be. We do not know the answer to the most basic questions: when does there exist a slope stable (twisted) vector  bundle  with  assigned  characteristic classes? Is there a reasonable expected dimension of the moduli space of stable sheaves with given mock Mukai vector (but see~\cite[Subsect.~3.4]{bottini:towardsog10})? Are HK (components of)  moduli spaces of semistable sheaves on HK manifolds necessarily   deformations of the known HK manifolds?
\subsection{Outline of the proofs}\label{subsec:panorama}
\setcounter{equation}{0}
The starting point is  the following. Let $S$ be a $K3$ surface, and suppose that $\cE_1,\cE_2$ are vector bundles on $S$ such that
\begin{equation}\label{condiforti}
r(\cE_2) \cdot c_1(\cE_1)=r(\cE_1) \cdot c_1(\cE_2),\qquad \la v(\cE_1),v(\cE_2)\ra=0,
\end{equation}
where $r(\cE_i),v(\cE_i)$ are the rank and Mukai vector of $\cE_i$ respectively. 
The involution of $S^2$ exchanging the factors lifts to an involution of    $\cE_1\boxtimes\cE_2\oplus \cE_2\boxtimes\cE_1$, and hence the Bridgeland-King-Reid (BKR) equivalence associates to $\cE_1\boxtimes\cE_2\oplus \cE_2\boxtimes\cE_1$ a vector bundle 
$\cG(\cE_1,\cE_2)$ on $S^{[2]}$.
In~\cite[Prop.~2.7]{og:highdim} we proved that the discriminant of $\cG(\cE_1,\cE_2)$ is a multiple of $c_2(S^{[2]})$. Thus $\cG(\cE_1,\cE_2)$ is modular, and if it is slope-stable then it is projectively hyperholomorphic. 

Suppose that $\cE_1$ is spherical with $r(\cE_1)=2a$.  If 
$v(\cE_2)= av_1-(0,0,1)$ then the equalities in~\eqref{condiforti} hold. Let $h_S$ be a polarization of $S$. Assume that $\cE_1,\cE_2$ are (Giesker-Maruyama) stable. 
Because of the second equality in~\eqref{condiforti}  we expect 
\begin{equation}\label{flowersgone}
\Ext^1(\cE_1,\cE_2)=\{0\}
\end{equation}
 if $\cE_2$ is  general. If that is the case then  every small deformation of $\cG(\cE_1,\cE_2)$ is isomorphic to
$\cG(\cE_1,\cE'_2)$ for a   small deformation $\cE'_2$ of $\cE_2$, and  one identifies the deformation space of $\cE_2$ with that of 
$\cG(\cE_1,\cE_2)$, see~\cite[Prop.~2.6]{og:highdim}. 
The upshot is that if for some $h_S$-stable $\cE_2$ the equality in~\eqref{flowersgone} holds and the vector bundle $\cG(\cE_1,\cE_2)$ is stable for a polarization $h_{S^{[2]}}$  of $S^{[2]}$, then we have  a birational map
\begin{equation}
\begin{matrix}
\cM_{v(a)}\coloneq\cM_{v(a)}(S,h_S) & \dra & M_{\ww_a}(S^{[2]},h_{S^{[2]}})^{\bu}\eqcolon M^{\bu}_{\ww_a} \\
\qquad\qquad [\cE_2] & \mapsto & [\cG(\cE_1,\cE_2)]\qquad\quad
\end{matrix}
\end{equation}
where $v(a)\coloneq av_1-(0,0,1)$, $\ww_a$ is the mock Mukai vector (see~\cite[Subsec.~1.2]{og:highdim}) associated to the vector bundles 
$\cG(\cE_1,\cE_2)$, and $M_{\ww_a}(S^{[2]},h_{S^{[2]}})^{\bu}$ is an irreducible component of the moduli space $M_{\ww_a}(S^{[2]},h_{S^{[2]}})$ of $h_{S^{[2]}}$ slope stable vector bundles on $S^{[2]}$ with  mock Mukai vector $\ww_a$.
 Note that $\cM_{v(a)}$ is a (projective) HK manifold of Type $K3^{[a^2+1]}$. 
In loc.~cit.~we proved that if $a>1$ and
$[\cE_2]\in\cM_{v(a)}$  is general, then $\cG(\cE_1,\cE_2)$ is $h_{S^{[2]}}$ (slope) stable for suitable $h_{S^{[2]}}$. In the present paper we determine for which  $[\cE_2]\in\cM_{v(a)}$ the sheaf $\cG(\cE_1,\cE_2)$ is  semistable, and what are the semistable replacements 
of  $\cG(\cE_1,\cE_2)$ for the remaining
 $[\cE_2]\in\cM_{v(a)}$. A key r\^ole is played by the (determinantal) locus $\cD_{v(a)}\subset \cM_{v(a)}$ parametrizing sheaves $\cE_2$ such that~\eqref{flowersgone} is violated, or equivalently such that $\Hom(\cE_2,\cE_1)\not=0$. This is an irreducible divisor  containing all points parametrizing non locally free sheaves. 
 If $[\cE_2]\in(\cM_{v(a)}\setminus\cD_{v(a)})$  then $\cE_2$  is slope stable  and locally free, and this  implies that $\cG(\cE_1,\cE_2)$ is 
 ${\bm\mu}(h_S)$ slope stable, where ${\bm\mu}(h_S)$ is the pull-back of the symmetrization of $h_S$ on $S^{(2)}$ via the Hilbert-to-Chow map $S^{[2]}\to S^{(2)}$ (here $h_S$ is an ${\mathsf a}(v(a))$-generic polarization of $S$).
Note that ${\bm\mu}(h_S)$ is big and nef, but not ample. One of the virtues of modular sheaves such as $\cG(\cE_1,\cE_2)$ is that variation of slope stability behaves as on surfaces. It follows that $\cG(\cE_1,\cE_2)$ is slope stable for any polarization $h_{S^{[2]}}$  close enough to ${\bm\mu}(h_S)$, where \lq\lq close enough\rq\rq\ is determined by $a$. 
This proves that we have a regular map
\begin{equation}\label{mappazza}
\begin{matrix}
\cM_{v(a)}\setminus\cD_{v(a)} & \xrightarrow{\psi} & M^{\bu}_{\ww_a} \\
[\cE_2] & \mapsto & [\cG(\cE_1,\cE_2)]
\end{matrix}
\end{equation}
defining an isomorphism between $\cM_{v(a)}\setminus\cD_{v(a)}$ and an open 
 dense subset of  $M_{\ww_a}^{\bullet}$. 

In the end we prove that $\psi$ extends to a  regular map  
\begin{equation}\label{psibarra}
\ov{\psi}\colon \cM_{v(a)}\to M_{\ww_a}^{\bullet}
\end{equation}
although $\cG(\cE_1,\cE_2)$ is slope unstable
for $[\cE_2]\in \cD_{v(a)}$. Thus   $\ov{\psi}([\cE_2])$  represents the isomorphism class of a vector bundle which is \emph{not} isomorphic to  
$\cG(\cE_1,\cE_2)$. In order to explain what goes on we need to give more details on the divisor $\cD_{v(a)}\subset \cM_{v(a)}$. Let 
$\cD^k_{v(a)}\subset \cD_{v(a)}$ be the (closed) subset of points $[\cE_2]$ such that $\dim\Hom(\cE_2,\cE_1)\ge k$. We have the (Brill-Noether) decreasing chain of closed subsets of $\cM_{v(a)}$.
\begin{equation*}
\cD_{v(a)}=\cD_{v(a)}^{1}\supset \ldots \supset \cD_{v(a)}^{k}\supset \cD_{v(a)}^{k+1}\supset  \ldots \supset\cD_{v(a)}^a.
\end{equation*}
 Let $1\le k<a$ and let
$[\cE_2]\in(\cD_{v(a)}^{k}\setminus \cD_{v(a)}^{k+1})$. The  map $\cE_2\to \cE_1\otimes\Hom(\cE_2,\cE_1)^{\vee}$ is surjective, and it induces a surjection $\Psi_{\cE_2}^{+}\colon\cG(\cE_1,\cE_2)\twoheadrightarrow \cE_1[2]^{-}\otimes V_k$, where 
$\cE_1[2]^{-}$ is the rigid slope stable vector bundle studied in~\cite{ogfascimod}, and $V_k\coloneq \Hom(\cE_2,\cE_1)^{\vee}$. 
Let $\cA(\cE_2)^{+}\coloneq\ker \Psi_{\cE_2}^{+}$. The exact sequence 
\begin{equation}\label{lottatralinci}
0\lra \cA(\cE_2)^{+}\lra \cG(\cE_1,\cE_2)\overset{\Psi_{\cE_2}^{-}}{\lra} \cE_1[2]^{-}\otimes V_k\lra 0,
\end{equation}
is  slope desemistabilizing for any polarization of $S^{[2]}$.
The  map $\cE_2\to \cE_1\otimes\Hom(\cE_2,\cE_1)^{\vee}$ is no longer surjective if $[\cE_2]\in\cD_{v(a)}^a$, but there is an exact sequence similar to~\eqref{lottatralinci} which is slope desemistabiling for any polarization of $S^{[2]}$. Since $\cE_1[2]^{-}$ is slope stable, the first step of semistable reduction for a one-parameter family of semistable sheaves specializing to $\cG(\cE_1,\cE_2)$ (given by a suitable elementary modification) replaces 
$\cG(\cE_1,\cE_2)$ by a sheaf $\wt{\cG}$ sitting into an exact sequence obtained by switching the two sides of~\eqref{lottatralinci}:
\begin{equation}\label{volterra}
0\lra \cE_1[2]^{-}\otimes V_k\lra \wt{\cG}\overset{\Psi_{\cE_2}^{-}}{\lra} \cA(\cE_2)^{+}\lra 0.
\end{equation}
We prove that those extensions~\eqref{volterra} which actually occur after  elementary modification of a one-parameter family in a general direction normal to $\cD_{v(a)}^{k}\setminus \cD_{v(a)}^{k+1}$ (the latter space is smooth) are slope stable, provided that $h_S$ is an ${\mathsf a}(v(a))$-generic polarization of $S$ and the polarization of $S^{[2]}$ is close enough to $\bm{\mu}(h_S)$. Moreover the isomorphism class of $\wt{\cG}$ depends on $\cE_2$ but \emph{not} on the direction normal to $\cD_{v(a)}^{k}\setminus \cD_{v(a)}^{k+1}$.
Note that if $[\cE_2]\in\cD_{v(a)}^a$ then  $\cG(\cE_1,\cE_2)$ is not locally free, but the sheaves that we get after the elementary modifications  are locally free (pleasant surprise!).

The actual proof that $\psi$ extends   regularly to all of  $\cM_{v(a)}$ goes as follows. Let $\wh{f}\colon\wh{\cM}_{v(a)}\to\cM_{v(a)}$ be the birational map obtained by blowing up $\cD_{v(a)}^a$ (which is smooth), blowing up the strict transform of $\cD_{v(a)}^{a-1}$ (which is smooth) etc. Then $\wh{\cM}_{v(a)}$ contains smooth prime divisors $\wh{E}^a,\ldots,\wh{E}^1$ such that 
$\wh{f}(\wh{E}^k)=\cD_{v(a)}^k$, the divisor 
$\wh{E}^a+\ldots+\wh{E}^1$ has simple normal crossings, and 
 $\wh{f}$ identifies  $\wh{\cM}_{v(a)}\setminus(\wh{E}^a\cup\ldots\cup\wh{E}^1)$ with $\cM_{v(a)}\setminus\cD_{v(a)}$. For simplicity we assume that there is a universal sheaf on $S\times\cM_{v(a)}$ and hence also a  sheaf $\GG$ on 
$S^{[2]}\times \cM_{v(a)}$ which is $\cM_{v(a)}$-flat and such that $\GG_{|S^{[2]}\times \{[\cE_2]\}}\cong\cG(\cE_1,\cE_2)$ for every 
$[\cE_2]\in \cM_{v(a)}$. Let $\wh{\GG}$ be the pull-back of $\GG$ to $S^{[2]}\times \wh{\cM}_{v(a)}$. One performs  elementary modifications starting from $\wh{\GG}$, and the final result is  a vector bundle $\wt{\GG}$ on $S^{[2]}\times \wh{\cM}_{v(a)}$
which restricts to a slope stable vector bundle on $S^{[2]}\times\{x\}$ for every $x\in \wh{\cM}_{v(a)}$, and 
 which is isomorphic to  $\wh{\GG}$ on $S^{[2]}\times (\cM_{v(a)}\setminus\cD_{v(a)})$.
 This proves that the map $\psi$ in~\eqref{mappazza} extends to a regular map  $\wh{\psi}\colon\wh{\cM}_{v(a)}\to M_{\ww_a}^{\bullet}$. 
If $x\in(\wh{E}^k\setminus \wh{E}^{k+1})$ then the isomorphism class of $\wt{\GG}_{S^{[2]}\times\{x\}}$ depends only on $\wh{f}(x)\in(\cD^k_{v(a)}\setminus\cD^{k+1}_{v(a)})$, and is represented by a vector bundle fitting into \eqref{volterra}. This proves that $\wh{\psi}$ descends to 
a regular map  $\ov{\psi}\colon\cM_{v(a)}\to M_{\ww_a}^{\bullet}$ extending $\psi$. The explicit description of the vector bundles parametrized by $\ov{\psi}(\cD_{v(a)})$ gives that $\ov{\psi}$ is a bijection. Thus $\ov{\psi}$  identifies $\cM_{v(a)}$  with the normalization 
$\wt{M}_{\ww_a}^{\bullet}$ of 
$M_{\ww_a}^{\bullet}$. 
 
Items~(1) and~(2) of Theorem~\ref{thm:isogmodvb} follows from the existence of the extension $\ov{\psi}$ in~\eqref{psibarra} together with Verbitsky's fundamental results on projectively hyperholomorphic vector bundles. In fact, since 
\begin{equation*}
r(\cG(\cE_1,\cE_2))=8a^3,\quad \eta_{8a^3}(\PP\cG(\cE_1,\cE_2))=[4a^2({\bm\mu}(D)-a\delta)]\pmod{8a^3 H^2(S^{[2]};\ZZ)},
\end{equation*}
(see~\eqref{sanmichele} and~\eqref{duedelta} for the notation in the second equation), 
\begin{equation*}
c_2(\gotg(\PP\cG(\cE_1,\cE_2))=\frac{4a^6}{3}c_2(S^{[2]}),
\end{equation*}
and
the projectivization of a slope stable vector bundle is a slope stable (with respect to the same polarization) bundle of projective spaces, we may view  $M_{\ww_a}(S^{[2]},h_{S^{[2]}})^{\bullet}$ as a closed subset $M_{\ov{\ww}_a}(S^{[2]},h_{S^{[2]}})^{\bullet}$ of the moduli space $M_{\ov{\ww}_a}(S^{[2]},h_{S^{[2]}})$. By Verbitsky's results every $\PP^{8a^3-1}$-bundle $\cP_0$ parametrized by a point of
$M_{\ov{\ww}_a}(S^{[2]},h_{S^{[2]}})^{\bullet}$ extends to an $\omega_t$ slope stable $\PP^{8a^3-1}$-bundle $\cP_t$ on a(n arbitrary) twistor deformation $(X_t,\omega_t)$ of $(S^{[2]},\omega_0)$, where $\omega_0$ is any K\"ahler class in an open $2^{16}\cdot 5\cdot a^{18}$-chamber containing $\bm{\mu}(h_S)$ in its closure. Moreover the local (anaytic) structure of the moduli space 
$M_{\ov{\ww}_a(t)}(X_t,\omega_t)$ at $\cP_t$ is isomorphic to that of $M_{\ov{\ww}_a}(S^{[2]},\omega_0)$ at $\cP_0$.  This, together with some extra arguments, gives Item~(1) and~(2) of Theorem~\ref{thm:isogmodvb} because every deformation of $S^{[2]}$ can be reached by a chain of twistor families.

In order to prove Item~(3) of Theorem~\ref{thm:isogmodvb} it suffices to analyze the case of the pull-back of the universal $\PP^{8a^3-1}$-bundle to $S^{[2]}\times  \cM_{v(a)}$. We may assume that there is a universal sheaf $\EE$ on $S\times\cM_{v(a)}$. The relevant computation is done via the Mukai map $v(a)^{\bot}\to H^2(\cM_{v(a)})$ associated to $\EE$.

 Theorem~\ref{thm:ognik3a2+1} is proved as follows. Let $(X_0,\omega_0)$ be a HK manifold of Type $K3^{[2]}$ with K\"ahler class 
 $\omega_0$ such that $\wt{M}_{\ov{\ww}_a}(X_0,\omega_0)^{\bullet}$ is a HK of Type $K3^{[a^2+1]}$ (e.eg.~$X_0$ projective and $\omega_0$ an ample class).
 There exists a chain of generic twistor families deforming 
 $\wt{M}_{\ov{\ww}_a}(X_0,\omega_0)^{\bullet}$ to $W$.
Applying Item~(3) of  Theorem~\ref{thm:isogmodvb} one gets that there are corresponding (generic) twistor families deforming $(X_0,\omega_0)$ to 
 $(X,\omega)$ in such a way that $\wt{M}_{\ov{\ww}_a}(X,\omega)^{\bullet}$ 
 is bimeromorphic to $W$. If $H^{1,1}_{\ZZ}(W)=\{0\}$ one gets that the bimeromorphic map is an isomorphism. 
 If $W$ is projective a more refined argument is needed to prove isomorphism.
 
 Theorem~\ref{thm:isogenie} is proved by combining  Item~(3) of Theorem~\ref{thm:isogmodvb} and (the case $m=2$ of) Markman's Theorem~\cite[Thm.~1.1]{markman:isogenies} to the effect that   the Shafarevich 
conjecture  for couples of HK manifolds (both) of Type $K3^{[m]}$ holds. 
 
\subsection{Organization of the paper}
\setcounter{equation}{0}
In Section~\ref{sec:saperperdere} we discuss variation of slope stability for modular sheaves. This subject has been treated before for ample classes. Here we extend the discussion to arbitrary nef classes.
 We also discuss twisted (locally free) sheaves and variation of slope stability for such sheaves. 
 
 Section~\ref{sec:fascisuk3} is devoted to the moduli spaces $\cM_{v(a)}$ (and also $\cM_{v(t)}$ for $1\le t<a$) and the Brill-Noether stratification of the determinantal divisor $\cD_{v(a)}\subset \cM_{v(a)}$. The results are  analogous to those in~\cite{markman-bn-duality}, but we cannot refer to Markman's paper because the Mukai vector $v(a)$ has  divisible second entry (unless $a=1$).

  Section~\ref{sec:fibratistab} contains all the slope stability results for vector bundles on $S^{[2]}$ that we will be using.  Subsection~\ref{subsec:guidasezquattro} is a detailed road map for the section.
  
  Section~\ref{sec:estendo} contains the key result on moduli spaces of slope stable sheaves on $S^{[2]}$ that are needed to prove Item~(1) of Theorem~\ref{thm:isogmodvb}. A detailed road map is in Subsection~\ref{subsec:stazfori}.
  
  In Section~\ref{sec:conticonti} we carry out the computations needed to prove Item~(2) of Theorem~\ref{thm:isogmodvb} (once Item~(1) has been proved).
  
  Section~\ref{sec:bottofinale} contains the proofs of the main results. It consists of applying to the results of Sections~\ref{sec:estendo} and~\ref{sec:conticonti} Verbitsky's theory of projectively hyperholomorphic (twisted) vector bundles, coupled with the properties of twistor families.
\subsection{Notation}\label{subsec:notazione}
Schemes, varieties etc.~are over $\CC$. Sheaves are coherent sheaves unless we state differently.  
If $\cA,\cB$ are  sheaves on a variety $X$ we let  
\begin{equation*}
\ext^p_X(\cA,\cB)\coloneq \dim\Ext^p_X(\cA,\cB),\qquad \hom_X(\cA,\cB)\coloneq \dim\Hom_X(\cA,\cB).
\end{equation*}
A HK variety (in this paper) is a projective HK manifold.
We denote   complex vector spaces of dimension $d$ by
$V_d$, $W_d$ etc.
\subsection{Acknowledgements}
Many thanks go to Francesco Meazzini for several conversations on the subject of (projectively) hyperholomorphic vector bundles. Alessio Bottini and Emanuele Macr\`\i\  have made helpful comments to a preliminary version.
\section{HK-slope stability of (twisted) sheaves}\label{sec:saperperdere}
\subsection{Slope stability of sheaves}\label{subsubsec:pendenza}
\setcounter{equation}{0}
Let $X$ be a smooth projective irreducible variety of  dimension $m$. Let $h\in\NS(X)_{\QQ}\coloneq\NS(X)\otimes\QQ$. The \emph{$h$ slope} of  a sheaf $\cF$ on $X$ 
of positive rank $r(\cF)$  is 
\begin{equation}
\mu_h(\cF)\coloneq \int_X\frac{c_1(\cF)\cdot h^{m-1}}{r(\cF)}.
\end{equation}
One may define $h$ slope (semi)stability as in the case of  ample $h$. 
\begin{dfn}\label{dfn:pendstabgen}
Let $\cF$  be a torsion-free sheaf on $X$. Then $\cF$ is \emph{$h$ slope semistable} if for all  subsheaves $\cE\subset\cF$ with $0<r(\cE)<r(\cF)$ we have 
$\mu_h(\cE)\le \mu_h(\cF)$.
If strict inequality holds for all such $\cE$ then $\cF$ is \emph{$h$ slope stable}. 
\end{dfn}
Now assume that $X$ is a HK manifold  of dimension $2n$.  Let $h\in\NS(X)_{\QQ}$.    Let $\cF$ be a sheaf on $X$ of positive rank $r(\cF)$. The \emph{$h$ HK-slope} of $\cF$ is given by
\begin{equation}
\mu_h^{HK}(\cF)\coloneq q_X\left(\frac{c_1(\cF)}{r(\cF)}, h\right).
\end{equation}
\begin{dfn}\label{dfn:stabnef}
A torsion-free (non-zero) sheaf $\cF$ on  $X$ is \emph{$h$ HK-slope semistable} if for all  subsheaves $\cE\subset\cF$ with $0<r(\cE)<r(\cF)$ we have $\mu_h^{HK}(\cE)\le \mu_h^{HK}(\cF)$.
If strict inequality holds for all such $\cE$ then $\cF$ is \emph{$h$  HK-slope stable}. 
\end{dfn}
Let  $q_X$ be the Beauville-Bogomolov-Fujiki quadratic form of $X$, and let $c_X\in\QQ_{+}$ be the 
(small) Fujiki constant  of $X$. Then for  all $\alpha\in H^2(X)$
\begin{equation}\label{fujikiformula}
\int_X\alpha^{2n}=(2n-1)!! c_X q_X(\alpha)^{n}.
\end{equation}
\begin{rmk}\label{rmk:hkstabpos}
Let $\cF$ be a torsion-free (non-zero) sheaf on $X$. Let $h\in\NS(X)_{\QQ}$, and suppose that $q_X(h)>0$, e.g.~if $h$ is ample. By Fujiki's formula~\eqref{fujikiformula}  
\begin{equation*}
\int_X c_1(\cF)\cdot h^{2n-1}=(2n-1)!! c_X  q_X(c_1(\cF),h)\cdot q_X(h)^{n-1}.
\end{equation*}
 Thus $\cF$ is  $h$ HK-slope (semi)stable   if and only if it is $h$ slope (semi)stable. 
\end{rmk}
\begin{rmk}\label{rmk:hkstablag}
Suppose that  $X$ has a Lagrangian fibration $\pi\colon X\to\PP^n$. See~\cite[Prop.~5.9]{og:highdim}  for the meaning of $h$ HK-slope (semi)stability
if $h\coloneq\pi^{*}c_1(\cO_{\PP^n}(1))$.
\end{rmk}

\subsection{Variation of HK-slope (semi)stability for modular sheaves}\label{subsubsec:eccofascimod}
\setcounter{equation}{0}
The \emph{discriminant} of a  sheaf  $\cF$ on a manifold $X$ is the characteristic class given by
\begin{equation}\label{eccodisc}
 \Delta(\cF) \coloneq 2r(\cF) c_2(\cF)-(r(\cF)-1) c_1(\cF)^2=\ch_1(\cF)^2-2r(\cF)\ch_2(\cF).
\end{equation}
Let  $X$ be a HK manifold  of dimension $2n$.   A  torsion-free sheaf $\cF$  on $X$
  is \emph{modular}  if 
 there exists $d(\cF)\in\QQ$ such that  for all $\alpha\in H^2(X)$
\begin{equation}
\int_X \Delta(\cF) \cdot \alpha^{2n-2}=d(\cF) (2n-3)!! q_X(\alpha)^{n-1}.
\end{equation} 
 If $\cF$ is  modular  we let
\begin{equation}\label{esmeralda}
{\mathsf a}(\cF)\coloneq r(\cF)^2 \cdot d(\cF)/4c_X
\end{equation}  
\begin{rmk}
A  torsion-free sheaf   on  a HK manifold $X$ whose discriminant is a multiple of $c_2(X)$ is   modular. If $X$ is of type 
$K3^{[n]}$ and $\Delta(\cF)=kc_2(X)$ then 
\begin{equation}\label{formulaperaperdi}
d(\cF)=6(n+1)k,\qquad {\mathsf a}(\cF)=3 k(n+1)r(\cF)^2/ 2.
\end{equation} 
\end{rmk}
Let $S$ be a $K3$ surface and let $v\in \wt{H}^{1,1}(S)$ be a Mukai vector with positive first entry (here $\wt{H}(S)$ is the Mukai lattice of $S$). We let 
${\mathsf a}(v)\coloneq {\mathsf a}(\cF)$ 
where $\cF$ is any sheaf  such that $v(\cF)=v$. If $v=(r,l,s)$ then
\begin{equation}\label{adimukai}
{\mathsf a}(v)=r^2(v^2+2 r^2)/4.
\end{equation}
If $X$ is a HK manifold of dimension at least $4$  and  $\ww\in \NN_{+}\times \NS(X)\times H^{2,2}_{\ZZ}(X)$ is a mock Mukai vector (see~\cite[Sect.~1.2]{og:highdim})  we let ${\mathsf a}(\ww)\coloneq {\mathsf a}(\cF)$ 
where $\cF$ is any sheaf  such that $\ww(\cF)=\ww$.    
\begin{dfn}\label{dfn:poladatta}
Let $X$ be a  HK  variety, and let  $h_0\in\Nef(X)_{\QQ}$. 
Let  ${\mathsf a}> 0$. An ample class $h\in\Amp(X)_{\QQ}$ is \emph{${\mathsf a}$-suitable for $h_0$}  if  for all
 $\lambda\in \NS(X)$  such that $ -{\mathsf a}\le q_X(\lambda)<0$ one of the following holds:
\begin{enumerate}
\item
$q_X(\lambda,h)>0$ and $q_X(\lambda,h_0)\ge 0$.
\item
 $q_X(\lambda,h)=0$  and $q_X(\lambda,h_0)= 0$.
\item
 $q_X(\lambda,h)<0$  and $q_X(\lambda,h_0)\le 0$.
\end{enumerate}
\end{dfn}
\begin{rmk}
Let ${\mathsf a}$ be a positive natural number. A codimension-$1$ real subspace $W\subset\NS(X)_{\RR}\coloneq\NS(X)\otimes\RR$ is an ${\mathsf a}$-wall if $W=\xi^{\bot}$ (orthogonality with respect to the BBF quadratic form $q_X$) where $\xi\in\NS(X)$ is such that
$-{\mathsf a}\le q_X(\xi)<0$. The intersections of the  ${\mathsf a}$-walls with the  cone $\cC(X)_{\RR}\subset \Nef(X)_{\RR}$ of classes with positive square form a locally finite collection of codimension-$1$ submanifolds, and hence their union is closed in $\cC(X)_{\RR}$. A connected component of 
\begin{equation}
\cC(X)_{\RR}\setminus\bigcup_{\stackrel{\xi\in\NS(X)}{ -{\mathsf a}\le q_X(\xi)<0}}\xi^{\bot}
\end{equation}
 is an 
\emph{open ${\mathsf a}$-chamber}.  An ample class $h\in\Amp(X)_{\QQ}$ is \emph{${\mathsf a}$-generic} if it does not belong to any ${\mathsf a}$-wall, i.e.~it belongs to a (unique) open ${\mathsf a}$-chamber. Let $h_0,h_1$ be ${\mathsf a}$-generic ample classes which belong to the same 
open ${\mathsf a}$-chamber: then $h_0$ is ${\mathsf a}$-suitable for $h_1$ and viceversa.
\end{rmk} 
For sheaves $\cE,\cF$ on $X$  let
\begin{equation}
\lambda_{\cE,\cF}\coloneq r(\cF)\cdot c_1(\cE)-r(\cE)\cdot c_1(\cF).
\end{equation}
\begin{prp}\label{prp:paragonestab}
Let $X$ be a  HK  variety, and let  $h_0\in\Nef(X)_{\QQ}$. 
Let  ${\mathsf a}> 0$. Suppose that $h\in\Amp(X)_{\QQ}$  is ${\mathsf a}$-suitable for $h_0$. 
 Let  $\cF$ be  a torsion free  modular sheaf  on $X$ which  is
  $h_0$ HK-slope semistable  but not $h$ HK-slope stable, and
   that ${\mathsf a}(\cF)\le {\mathsf a}$.  
    Then there exists a subsheaf  $\cH\subset \cF$, with $0<r(\cH)<r(\cF)$, such that $q_X(\lambda_{\cH,\cF},h)\ge 0$ and $q_X(\lambda_{\cH,\cF},h_0)= 0$.
\end{prp}
\begin{proof}
This has been proved for $X$ with a Lagrangian fibration $\pi\colon X\to\PP^n$ and  
$h_0=\pi^{*}c_1(\cO_{\PP^n}(1))$, see~\cite[Prop.~5.9, Item~(a)]{og:highdim}. A similar argument  gives a proof of 
Proposition~\ref{prp:paragonestab}. We spell out the proof for the reader's convenience.

First assume  that $\cF$ is $h$ HK-slope semistable.  Since $\cF$ is not $h$ HK-slope stable there exists a subsheaf $\cH\subset\cF$  such that  $0<r(\cH)<r(\cF)$ and 
$q_X(\lambda_{\cH,\cF},h)= 0$. By~\cite[Prop.~3.10]{ogfascimod} we get that $-{\mathsf a}\le -{\mathsf a}(\cF)\le q_X(\lambda_{\cH,\cF})\le 0$. 
If $\lambda_{\cH,\cF}= 0$ then $q_X(\lambda_{\cH,\cF},h_0)=0$ trivially. If $\lambda_{\cH,\cF}\not= 0$ then 
$q_X(\lambda_{\cH,\cF})< 0$ (the restriction of $q_X$ to $h^{\bot}\cap \NS(X)$ is negative definite). Thus
\begin{equation}\label{barlume}
-{\mathsf a}\le  q_X(\lambda_{\cH,\cF})< 0.
\end{equation}
Since $h$ is ${\mathsf a}$-suitable for $h_0$ it follows that
$q_X(\lambda_{\cH,\cF},h_0)=0$. 

Next assume  that $\cF$ is not $h$ HK-slope semistable. Thus there exists $\cG\subset\cF$  such that  $0<r(\cG)<r(\cF)$ and 
$q_X(\lambda_{\cG,\cF},h)> 0$. If $q_X(\lambda_{\cG,\cF},h_0)=0$ we are done, hence we may assume that 
$q_X(\lambda_{\cG,\cF},h_0)\not=0$. Since $\cF$ is $h_0$ HK-slope semistable we have $q_X(\lambda_{\cG,\cF},f)<0$. 
Let $\mathsf S$ be the set  of rational numbers $s\in(0,1)$ for which there exists a subsheaf $\cH\subset\cF$,  with $0<r(\cH)<r(\cF)$,  such that 
\begin{equation}\label{comenovem}
q_X(\lambda_{\cH,\cF},(1-s)h+s h_0)=0.
\end{equation}
Then $\mathsf S$ is non  empty because there exists a  rational number $s\in(0,1)$ for which~\eqref{comenovem} holds with $\cH=\cG$.
 We claim that $\mathsf S$  is finite. In fact if~\eqref{comenovem} holds with $s\in(0,1)$ then $q_X(\lambda_{\cH,\cF},h)\ge 0$, because 
 $q_X(\lambda_{\cH,\cF},h_0)\le 0$ ($\cF$ is $h_0$ HK-slope semistable). The inequality $q_X(\lambda_{\cH,\cF},h)\ge 0$ holds if and only if 
 $\mu_h(\cH)\ge\mu_h(\cF)$ because $h$ is ample (see Remark~\ref{rmk:hkstabpos}). 
  Since the set of subsheaves  $\cH\subset\cF$ such that 
 $\mu_h(\cH)\ge\mu_h(\cF)$  is bounded it follows that $\mathsf S$  is finite. Hence $\mathsf S$ has a maximum 
 $s_{*}$. Let $h_{*}\coloneq (1-s_{*})h+s_{*} f$. 
 
 Suppose that $\cF$ is not $h_{*}$ HK-slope semistable. Then there exists a subsheaf 
  $\cH\subset\cF$    with $0<r(\cH)<r(\cF)$  such that $q_X(\lambda_{\cH,\cF},h_{*})>0$. If $q_X(\lambda_{\cH,\cF},h_0)<0$ then there exists $s\in (s_{*},1)$ such that~\eqref{comenovem} holds, and this is a contradiction because   $s_{*}$ is the maximum of    $\mathsf S$. Thus  $q_X(\lambda_{\cH,\cF},h_0)= 0$ ($\cF$ is $h_0$ HK-slope semistable).
   Since $q_X(\lambda_{\cH,\cF},h_{*})>0$ and $q_X(\lambda_{\cH,\cF},h_0)= 0$, we get that $q_X(\lambda_{\cH,\cF},h)>0$ and hence we are done. 

 Lastly suppose that $\cF$ is  $h_{*}$ HK-slope semistable.  Then there exists a subsheaf 
  $\cH\subset\cF$    with $0<r(\cH)<r(\cF)$  such that $q_X(\lambda_{\cH,\cF},h_{*})= 0$. We claim that 
  $q_X(\lambda_{\cH,\cF},h_0)=0$. Granting this for the moment being, we get that $q_X(\lambda_{\cH,\cF},h)= 0$ because
  $q_X(\lambda_{\cH,\cF},h_{*})= 0$, and hence we are done. We finish by proving that  $q_X(\lambda_{\cH,\cF},h_0)=0$. Suppose that
   $q_X(\lambda_{\cH,\cF},h_0)\not=0$. 
   By~\cite[Prop.~3.10]{ogfascimod} we have $-{\mathsf a}\le -{\mathsf a}(\cF)\le q_X(\lambda_{\cH,\cF})\le 0$, and since  $\lambda_{\cH,\cF}\not= 0$
   (because $q_X(\lambda_{\cH,\cF},h_0)\not=0$) we get that~\eqref{barlume} holds.
   Since   $q_X(\lambda_{\cH,\cF},h_0)\not=0$ we have
 $q_X(\lambda_{\cH,\cF},h_0)<0$ ($\cF$ is $h_0$ HK-slope semistable). Since~\eqref{barlume} holds and $h$ is ${\mathsf a}$-suitable for $h_0$, it follows that
$q_X(\lambda_{\cH,\cF},h)< 0$. The inequalities $q_X(\lambda_{\cH,\cF},h_0)<0$ 
and $q_X(\lambda_{\cH,\cF},h)< 0$ contradict the equality 
$q_X(\lambda_{\cH,\cF},h_{*})= 0$. Thus the hypothesis $q_X(\lambda_{\cH,\cF},h_0)\not=0$ leads to a contradiction. 
\end{proof}
\begin{crl}\label{crl:paragonestab}
Let $X$ be a  HK  variety, and let  $h_0\in\Nef(X)_{\QQ}$. 
Let  ${\mathsf a}> 0$. Suppose that $h\in\Amp(X)_{\QQ}$  is ${\mathsf a}$-suitable for $h_0$. 
 Let  $\cF$ be  a torsion free  modular sheaf  on $X$ such  that ${\mathsf a}(\cF)\le {\mathsf a}$. If $\cF$  is $h_0$ HK-slope stable
 then it is 
  $h$ HK-slope stable.
\end{crl}
\begin{proof}
Suppose that $\cF$ is not  $h$ HK-slope stable. By Proposition~\ref{prp:paragonestab} there exists a subsheaf  $\cH\subset \cF$,
 with $0<r(\cH)<r(\cF)$, such that $q_X(\lambda_{\cH,\cF},h)\ge 0$ and $q_X(\lambda_{\cH,\cF},h_0)= 0$. The last equality contradicts 
 $h_0$ HK-slope stability of $\cF$.
\end{proof}
\subsection{Twisted sheaves and slope stability}\label{subsec:twistandstab}
\setcounter{equation}{0} 
\subsubsection{Twisted  sheaves}
Here $X$ is a complex manifold with the classical topology, and sheaves are (analytic) coherent sheaves. 
Let $\alpha\in H^2(X,\cO_X^{*})$ be a class represented by the $2$-cocycle $\{\alpha_{ijk}\}$ for an open cover $X=\bigcup_{i\in I}U_i$.  
An \emph{$\alpha$-twisted coherent sheaf} $\cE$ on $X$ consists of a collection $\{\cE_i\}_{i\in I}$ of coherent sheaves $\cE_i$ on $U_i$ and gluing (holomorphic) isomorphisms $\varphi_{ji}\colon {\cE_i}_{|U_{ij}}\xrightarrow{\sim}{\cE_j}_{|U_{ij}}$ such that $\varphi_{ii}=$, 
$\varphi_{ij}=\varphi_{ji}^{-1}$, and 
on the triple intersection $U_{ijk}$ we have 
${\varphi_{kj}}_{|U_{ijk}}\circ {\varphi_{ji}}_{|U_{ijk}}=\alpha_{ijk}\cdot{\varphi_{ki}}_{|U_{ijk}}$. Morphisms of 
$\alpha$-twisted coherent sheaves are given by morphisms of sheaves on the $U_i$'s which are compatible with the gluings. Up to isomorphism, the abelian category of $\alpha$-twisted coherent sheaves depends only on the class $\alpha\in H^2(X,\cO_X^{*})$, see~\cite{caldararu:thesis}. 
\begin{rmk}\label{rmk:tortolibero}
An $\alpha$-twisted sheaf $\cE$ is \emph{locally free sheaf of rank $r$}
if all the $\cE_i$ are locally free  of rank $r$ - we also say that it is an \emph{$\alpha$-twisted vector bundle}. If this is the case, the projectivized gluings
$\PP(\varphi_{ji})\colon {\PP(\cE_i)}_{|U_{ij}}\xrightarrow{\sim}{\PP(\cE_j)}_{|U_{ij}}$ give a projective $\PP^{r-1}$-bundle $\PP(\cE)\to X$. Note that  
$\alpha$ is the image of $\eta_r(\PP(\cE))$ (see~\eqref{etadipi}) for the map $H^2(X,\mu_r)\to H^2(X,\cO_X^{*})$.
If $\cP\to X$ is a $\PP^{r-1}$-bundle there exists an $\alpha$-twisted coherent sheaf $\cE$ on $X$ (unique up to isomorphism) such that $\cP\cong\PP(\cE)$. 
\end{rmk}
\begin{rmk}\label{rmk:ciunotorto}
Suppose that $\cE$ is an $\alpha$-twisted  vector bundle of rank $r$. Then  $\{\varphi_{ji}\}_{i,j\in I}$  is a $1$-cocycle $c$ with values in the sheaf $\PGL_r(\cO_X)$. Since the coboundary class $\delta(c)\in H^2(X,\mu_r)$ (see~\eqref{esseproj}) is equal to $\alpha$, we get 
that $\alpha^r$ is the trivial class in 
$H^2(X,\cO_X^{*})$, in particular $\alpha$ is a torsion class. 
We also get that we may replace $\varphi_{ij}$ by $\varphi'_{ij}=\lambda_{ij}\varphi_{ij}$, where $\lambda_{ij}$ is a suitable  1-cochain with values in $\cO_X^{*}$, so that 
$\{\det\varphi'_{ji}\}_{i,j\in I}$  is an honest $1$-cocycle with values in $\cO_X^{*}$. We denote by $\det\cE$ the corresponding line bundle. We let $c_1(\cE)\coloneq c_1  (\det\cE)$.  Now suppose that $\cE$ is an $\alpha$-twisted  vector bundle of rank $r$. Then there exists  an open subset $X^0\subset X$ 
with complement of codimension at least $2$ such that 
$\cE$  is locally free on $X^0$. Since  the restriction map $H^2(X;R)\to H^2(X^0;R)$ is an isomorphism for any coefficient group $R$, it follows that 
 $\alpha^r$ is the trivial class (as above), and we may  define $c_1(\cE)$ as the unique $c_1(\cE)\in H^2(X;\ZZ)$ such that $c_1(\cE)_{|X^0}=c_1(\cE_{|X^0})$. 
\end{rmk}
\begin{rmk}
Suppose that $\cE$ is an $\alpha$-twisted  vector bundle of rank $r$. Then the gluings of the 
vector bundles $\{\cE_i^{\vee}\otimes\cE_i\}_{i\in I}$ induced by the $\varphi_{ji}$'s match on triple intersections (not only up to homothety), and hence $\cE^{\vee}\otimes \cE$ is an honest vector bundle, isomorphic to $\cO_X\oplus\gotg(\PP(\cE)$.
If $\dim X\ge 2$ we let
\begin{equation}\label{hommefemme}
\Delta(\cE)\coloneq c_2(\cE^{\vee}\otimes \cE)=c_2(\gotg(\PP(\cE)).
\end{equation}
\end{rmk}
\begin{dfn}
Let $X$ be a HK manifold of dimension $2n$. An $\alpha$-twisted  vector bundle $\cE$ on $X$   is \emph{modular}  if 
 there exists $d(\cE)\in\QQ$ such that  for all $\alpha\in H^2(X)$ 
\begin{equation}\label{fernand}
\int_X  \Delta(\cE) \cdot\alpha^{2n-2}=d(\cE) (2n-3)!! q_X(\alpha)^{n-1}.
\end{equation} 
\end{dfn}
\begin{rmk}
Let $\cE$ be an $\alpha$-twisted modular vector bundle on
 a HK manifold $X$ of dimension $2n$. Then $\cE^{\vee}\otimes\cE$ is a(n honest) modular vector bundle because 
\begin{equation}\label{danglars}
\int_X  \Delta(\cE^{\vee}\otimes\cE) \cdot\alpha^{2n-2}=2\cdot r(\cE)^2\cdot d(\cE) (2n-3)!! q_X(\alpha)^{n-1}.
\end{equation} 
\end{rmk}
\subsubsection{Slope stability of twisted  sheaves}
In the present subsubsection $X$ is a compact K\"ahler manifold (of dimension $m$) with  a K\"ahler class $\omega\in H^{1,1}_{\RR}(X)$. 
Let  $\alpha\in H^2(X,\cO_X^{*})$. If $\cE$ is an 
$\alpha$-twisted torsion-free sheaf on $X$   the 
$\omega$ \emph{slope} of $\cE$ is defined as usual (recall that by Remark~\ref{rmk:ciunotorto} $c_1(\cE)$ is defined):
\begin{equation}
\mu_{\omega}(\cE)\coloneq \int_X\frac{c_1(\cE)\cdot \omega^{m-1}}{r(\cE)}.
\end{equation}
\begin{dfn}
 An $\alpha$-twisted torsion-free sheaf  $\cE$ on $X$ is $\omega$ \emph{slope} semistable if for all  $\alpha$-twisted 
   subsheaves $\cG\subset\cE$ with $0<r(\cG)<r(\cE)$ we have $\mu_{\omega}(\cG)\le \mu_{\omega}(\cE)$.
The sheaf $\cE$ is \emph{${\omega}$ slope stable} if strict inequality holds for all such $\cG$, and it is $\omega$ \emph{slope} polystable if it is a direct sum of $\omega$ slope stable 
torsion-free sheaves with equal $\omega$-slopes. 
\end{dfn}
\begin{rmk}\label{rmk:stabequiv}
Let $\cE$ be an $\alpha$-twisted locally-free sheaf of rank $r$  on $X$. There is a notion of  $\omega$ \emph{slope} (semi)stability for the  associated principal $\PGL_r$-bundle $\Princ(\cE)\to X$, see for example~\cite{anchbisw:hermeinstprinc}. These notions are equivalent, i.e.~$\cE$ is $\omega$ \emph{slope} semistable if and only if 
$\Princ(\cE)$ is. 
\end{rmk}
\begin{thm}[Anchouche - Biswas]\label{thm:princadj}
 An $\alpha$-twisted vector bundle $\cE$   on $X$ is $\omega$ slope polystable (respectively semistable) if and only if $\cE^{\vee}\otimes\cE$ is 
 $\omega$ slope polystable (respectively semistable).
\end{thm}
\begin{proof}
The statement about polystability is~\cite[Cor.~3.8]{anchbisw:hermeinstprinc} (for $\PGL$-bundles), the statement about semistability is Prop.~2.10 in loc.~cit.
\end{proof}
Let $\ov{\ww}$ be as in~\eqref{doppiavubarra}. 
 The moduli space $M_{\ov{\ww}}(X,\omega)$ of $\omega$-slope stable projective $\PP^{r-1}$-bundles $\cP\to X$ with 
 $\eta_r(\cP)=\eta$ and $\Delta(\cP)=s$ is an  analytic space, see for example~\cite{anchbisw:hermeinstprinc}. By 
 Remarks~\ref{rmk:tortolibero} and~\ref{rmk:stabequiv} 
 we may, and will, view it as the moduli space of $\omega$-slope stable twisted locally-free sheaves $\cE$ on $X$ with $r(\cE)=r$, 
 $\eta_r(\PP(\cE))=\eta$, $\Delta(\cE)=s$.  
\subsubsection{Variation of slope stability for modular twisted  sheaves}
In the present subsubsection $X$ is a HK manifold. We let $c_X$ be the (small) Fujiki constant of $X$, see~\cite[(2.6.3)]{og:highdim}. 
Let $\cK(X)\subset H^{1,1}_{\RR}(X)$ be the K\"ahler cone 
(whose elements are the cohomology classes of K\"ahler metrics).  
Let ${\mathsf a}$ be a positive real number. An \emph{${\mathsf a}$-wall} of $\cK(X)$ is the intersection 
$\lambda^{\bot}\cap \cK(X)$, where 
  $\lambda\in H^{1,1}_{\ZZ}(X)$ is a class such that  
$ -{\mathsf a} \le q_X(\lambda)< 0$
(orthogonality is with respect to the BBF quadratic form $q_X$).
The set of ${\mathsf a}$-walls  is  locally finite, in particular the union of all the ${\mathsf a}$-walls  is closed in $\cK(X)$. An 
\emph{open ${\mathsf a}$-chamber} of $\cK(X)$  is a connected component of the complement of   the union of all the 
${\mathsf a}$-walls of $\cK(X)$. A K\"ahler class is ${\mathsf a}$-generic if it belongs to an open ${\mathsf a}$-chamber.

Let $\cE$ be an $\alpha$-twisted modular vector bundle on $X$, where  $\alpha\in H^2(X,\cO_X^{*})$. Set
\begin{equation}\label{atorto}
{\mathsf a}_{\rm tw}(\cE)\coloneq {\mathsf a}(\cE^{\vee}\otimes\cE)=\frac{r(\cE)^4\cdot d(\cE)}{2c_X},
\end{equation}
where $d(\cE)$ is as in~\eqref{fernand}. 
\begin{prp}\label{prp:campol}
Let $X$ be a  HK manifold,  $\omega_0\in\cK(X)$  an ${\mathsf a}$-generic  K\"ahler class on $X$, and  $\cCH\subset\cK(X)$ be the open 
${\mathsf a}$-chamber containing $\omega_0$. 
Let  $\cE$ be an $\alpha$-twisted modular vector bundle on $X$, where  $\alpha\in H^2(X,\cO_X^{*})$, such that 
${\mathsf a}_{\rm tw}(\cE)\le{\mathsf a}$. If $\cE$ is 
 $\omega_0$ slope stable then it is $\omega$ slope stable for every $\omega\in\cCH$. 
\end{prp}
\begin{proof}
First we note that if   $\omega\in\cC$ then $\cE$ is 
 $\omega$ slope semistable. In fact   suppose that $\cE$ is 
 not $\omega$ slope semistable.
 By Theorem~\ref{thm:princadj}
the (honest) vector bundle $\cE^{\vee}\otimes\cE$    is $\omega_0$ slope polystable, but  not $\omega$ slope semistable. Since 
$\cE^{\vee}\otimes\cE$    is modular, and ${\mathsf a}(\cE^{\vee}\otimes\cE)\le{\mathsf a}$, this contradicts the hypothesis that 
 $\omega_0,\omega$ belong to the same open ${\mathsf a}$-chamber. In fact  the proof is the same as the proof of~\cite[Prop.~5.4]{og:highdim}. 
  
  Now suppose that there exists $\omega\in\cC$ such that $\cE$ is 
 $\omega$ slope-semistable but not stable, i.e.~properly $\omega$ slope-semistable. Then there exists an $\alpha$-twisted subsheaf 
 $\cG\subset\cE$ with $0<r(\cG)<r(\cE)$ and (see~\cite[Lemma~5.1]{og:highdim})
\begin{equation}
q_X(r(\cE) c_1(\cG)-r(\cG) c_1(\cE),\omega)=0.
\end{equation}
Let $\omega_t\coloneq (1-t)\omega_0+t\omega$. 
Let $\epsilon>0$ such that $\omega_t\in\cK(X)$ for all $t\in (1,1+\epsilon)$. For such $t$ we have 
\begin{multline*}
q_X(r(\cE) c_1(\cG)-r(\cG) c_1(\cE),\omega_t)= \\
=(1-t)q_X(r(\cE) c_1(\cG)-r(\cG) c_1(\cE),\omega_0)+tq_X(r(\cE) c_1(\cG)-r(\cG) c_1(\cE),\omega)>0
\end{multline*}
because $\cE$ is 
 $\omega_0$ slope stable (see~\cite[Lemma~5.1]{og:highdim}). Thus  for all $t\in (1,1+\epsilon)$ the $\alpha$-twisted vector bundle $\cE$ is not $\omega_t$ slope-semistable. By  what we have just proved it follows that $\{\omega_t \mid t\in (1,1+\epsilon)\}\cap\cC=\es$:   contradiction because 
  $\cC$ is open. 
\end{proof}

\section{A series of moduli spaces of sheaves on K3 surfaces }\label{sec:fascisuk3}
\subsection{Set up}\label{subsec:pechino}
\setcounter{equation}{0}
Throughout the present section $S$ is a projective $K3$ surface and $h_S$ is an ample class on $S$. If $v=(r,l,s) \in 
H^0(S;\ZZ)\oplus\NS(S)\oplus H^4(S;\ZZ)$ is a vector with $r>0$ we let  $\cM_{v}=\cM_v(S,h_S)$ be the moduli space of $h_S$  semistable (meaning \emph{Gieseker-Maruyama} semistable) torsion-free sheaves on $S$ with Mukai vector $v(\cF)\coloneq\ch(\cF)\cdot\sqrt{\Td_S}$ equal to $v$. 
If $v$ is primitive and $h_S$ is ${\mathsf a}(v)$-generic then $\cM_{v}$ is a HK variety of Type $K3^{[n]}$ where $2n=2+\la v,v\ra$ and $\la ,\ra$ is the Mukai pairing. 
Whenever  there is no ambiguity we drop reference to $h_S$ when dealing with slope (semi)stability, or (semi)stability.
Let
\begin{equation}\label{eccovuuno}
v_1\coloneq (r_1,l_1,s_1)\in H^0(S;\ZZ)\oplus\NS(S)\oplus H^4(S;\ZZ)
\end{equation}
of square $-2$, with  $r_1>0$. Throughout the section we assume that $h_S$ is  ${\mathsf a}(v_1)$-generic. 
Since $v_1$ is primitive $\cM_{v_1}(S,h_S)$ is a singleton. We let $\cE_1$ be the unque (up to isomorphism) $h_S$ slope stable spherical vector bundle with 
\begin{equation}
v(\cE_1)=v_1= (r_1,l_1,s_1).
\end{equation}
Our notation ignores  the dependency of $\cE_1$  on $h_S$. 
For $t$ an  integer  we let
\begin{equation}\label{centocelle}
v(t)\coloneq tv_1-(0,0,1).
\end{equation}
\begin{rmk}\label{rmk:vutiprim}
The vector $v(t)$ is primitive. In fact suppose that a prime $p$ divides $v(t)$. By~\eqref{centocelle}  $p$ does not divide $t$ and hence it divides both $r_1$ and $l_1$. On 
the other hand  $p$ does not divide both $r_1$ and $l_1$ because 
 $v_1^2=-2$.  Contradiction.
\end{rmk}
\subsection{The boundary of  $\cM_{v(t)}$}
\setcounter{equation}{0}
\begin{dfn}
Let  $\cB_{v(t)}=\cB_{v(t)}(S,h_S)\subset \cM_{v(t)}$ be the closed subset  parametrizing  sheaves which are not locally free.
\end{dfn}
\begin{prp}\label{prp:bordodiemme}
Suppose that $h_S$ is ${\mathsf a}(v_1)$-generic.  Let $0<t\le r_1$. 
\begin{enumerate}
\item
If $[\cF]\in \cB_{v(t)}(S,h_S)$ then $\cF^{\vee\vee}\cong \cE_1^{\oplus t}$ and $\cF^{\vee\vee}/\cF\cong\CC_p$ for some $p\in S$. 
\item
By associating to $[\cF]\in \cB_{v(t)}(S,h_S)$ the support of $\cF^{\vee\vee}/\cF$ one gets a  fibration $\pi^a_{v(t)}\colon \cB_{v(t)}(S,h_S)\to S$ with fibers isomorphic to the Grassmannian $\Gr(r_1-t,r_1)$, locally trivial in the classical topology.
\end{enumerate}
\end{prp}
Before proving Proposition~\ref{prp:bordodiemme} we go through some preliminary results.  

\begin{lmm}\label{lmm:pocosing}
Suppose that $h_S$ is ${\mathsf a}(v_1)$-generic. Let  $\cF$ be  an $h_S$ slope semistable torsion free sheaf on $(S,h_S)$ such that 
\begin{equation}
v(\cF)=kv_1+q(0,0,1)
\end{equation}
for some $q\ge 0$. Then $\cF\cong \cE_1^{\oplus k}$. In particular $\cF$ is locally free and $q=0$.
\end{lmm}
\begin{proof}
We have
\begin{equation}
\chi(S,\cE_1^{\vee}\otimes\cF)=-\la v_1,v(\cF)\ra=-\la v_1,kv_1+q(0,0,1)\ra=2k+q\cdot r_1.
\end{equation}
By Serre duality it follows that either there exists a non zero map
\begin{equation}\label{allblacks}
\cE_1\overset{\phi\not=0}{\lra} \cF
\end{equation}
or 
a non zero map
\begin{equation}\label{pumas}
\cF\overset{\psi\not=0}{\lra} \cE_1. 
\end{equation}
We argue by induction on $k$. Let $k=1$. If a map $\phi$ as in~\eqref{allblacks} exists, then it has maximal rank in codimension $1$ because 
$\mu_{h_S}(\cE_1)=\mu_{h_S}(\cF)$. Composing $\phi$ with the inclusion 
$\cF\to \cF^{\vee\vee}$ we get 
a  map of vector bundles $\cE_1\to \cF^{\vee\vee}$ which is an isomorphism in codimension $1$, hence an isomorphism. It follows that 
$ \cF^{\vee\vee}\cong\cE_1$, and hence 
$\cF= \cF^{\vee\vee}\cong\cE_1$.

If a map $\psi$ as in~\eqref{pumas} exists,  arguing as above we get that 
$\cF^{\vee\vee}\cong\cE_1$.  The exact sequence
\begin{equation}\label{hardasnails}
0\lra \cF\lra \cF^{\vee\vee}\lra \cA\lra 0,
\end{equation}
with $\cA$  an Artinian sheaf gives that $v(\cF)=v( \cF^{\vee\vee})-l( \cA)(0,0,1)=v_1-l( \cA)(0,0,1)$. It follows that 
$\cA=0$ and hence $\cF=\cF^{\vee\vee}\cong  \cE_1$.

Next we prove the inductive step. Thus $k\ge 2$. If a map $\phi$ as in~\eqref{allblacks} exists, then it has maximal rank in codimension $1$, and hence we get an exact sequence
\begin{equation}\label{smalto}
0\lra \cE_1\lra \cF\lra \cQ\lra 0,
\end{equation}
where $v(\cQ)=(k-1)v_1+q(0,0,1)$. 

We claim that $\cQ$ is  torsion free. In fact the two step locally free resolution of $\cQ$ in~\eqref{smalto} gives that 
$\Ext^2(\cQ,\CC_P)=0$  for any $P\in S$. But if   $\cQ$ has torsion, then $\Tors(\cQ)$ has zero dimensional support 
because $\cQ$ is locally free in codimension $1$ ($\phi$  has maximal rank in codimension $1$), and hence 
$\Ext^2(\cQ,\CC_P)\not=0$  for any $P$ in the support of $\Tors(\cQ)$. 

We have $\mu_{h_S}(\cQ)=\mu_{h_S}(\cF)$  because $ \mu_{h_S}(\cE_1)=\mu_{h_S}(\cF)$. Since every quotient  of $\cQ$ is a quotient  of $\cF$ it follows that $\cQ$ is slope semistable.  

Since $\cQ$ is torsion free, slope semistable, and $v(\cQ)=(k-1)v_1+q(0,0,1)$ it follows from the inductive hypothesis that 
 $\cQ\cong \cE_1^{\oplus (k-1)}$.
Since $\cE_1$ is spherical we have $\Ext^1(\cE_1^{\oplus (k-1)},\cE_1)=0$, and hence 
$ \cF\cong \cE_1^{\oplus k}$. This finishes the proof if a map $\phi$ as in~\eqref{allblacks} exists.

If a map $\psi$ as in~\eqref{pumas} exists one argues similarly.

\end{proof}
Let  $V_t$ be a complex vector space of dimension $t>0$, and let $f\colon\cE_1\otimes V_t\to\CC_p$ be a surjection, where $\CC_p$ is the skyscraper sheaf at $p\in S$.  Of course $f$ is determined by the non zero linear  map
$f(p)\colon\cE_1(p)\otimes V_t\to\CC$.  We view $f(p)$ as a map 
\begin{equation}
f(p)\colon \cE_1(p)\lra V^{\vee}_t.
\end{equation}
Let  $\cF\coloneq \ker(f)$. Thus $\cF$ fits into the exact sequence
\begin{equation}\label{arthurash}
0\lra\cF\lra \cE_1\otimes V_t\overset{f}{\lra}\CC_p \lra 0.
\end{equation}
\begin{lmm}\label{lmm:stabse}
Suppose that $h_S$ is ${\mathsf a}(v_1)$-generic.  The sheaf $\cF$ appearing in~\eqref{arthurash} is  semistable if and only if 
$f(p)\colon\cE_1(p)\to V^{\vee}_t$ is surjective, and then it is actually stable.
\end{lmm}
\begin{proof}
Assume that $f(p)\colon\cE_1(p)\to V^{\vee}_t$ is not surjective. Then there exists $0\not=v\in V_t$ such that $\langle f(p)(s),v\rangle=0$ for every  
$s\in\cE_1(p)$.  Let $\Lambda\colon\cE_1\to \cE_1\otimes V_t$ be given by 
$\Lambda(s)\coloneq s\otimes v$. Then $\Lambda$ is an injection, and since $f\circ\Lambda=0$, it factors through an injection 
$\ov{\Lambda}\colon\cE_1\to\cF$. Since $r(\cF)\cdot \chi(S,\cE_1(n))>r(\cE_1)\cdot \chi(S,\cF(n))$ for 
$n\gg 0$ it follows that $\cF$ is not semistable. 

Assume that $f(p)$ is surjective. Suppose that $\cH\subset\cF$ is a destabilizing subsheaf, i.e.~that $0<r(\cH)<r(\cF)$ and
\begin{equation}\label{destabilizzo}
 r(\cF)\cdot \chi(S,\cH(n))\ge r(\cH)\cdot \chi(S,\cF(n)),\quad n\gg 0.
\end{equation}
Then we get an inclusion 
$\cH^{\vee\vee}\hra \cE_1\otimes V_t$, and $\mu(\cH)\ge  \mu(\cE_1\otimes V_t)$. Since $\cE_1$ is slope stable, it follows that $\cH^{\vee\vee}=\cE_1\otimes U$ where $U\subset V_t$ is a  subspace such that $0<\dim U<\dim V_t$, 
and the 
 inclusion $\cH^{\vee\vee}\hra \cE_1\otimes V_t$ is induced by $U\subset V_t$. Since $f(p)$ has maximal rank, it follows that we have an exact sequence
\begin{equation}
0\lra\cH\lra \cE_1\otimes U\lra\CC_p \lra 0.
\end{equation}
This contradicts the inequality in~\eqref{destabilizzo}.
\end{proof}
\begin{proof}[Proof of Proposition~\ref{prp:bordodiemme}]
(1): Let $[\cF]\in \cB_{v(t)}$. Then $\cF^{\vee\vee}$ is slope semistable and  $v(\cF^{\vee\vee})=tv_1+q(0,0,1)$ for some $q\ge 0$.  By 
Lemma~\ref{lmm:pocosing} it follows that $\cF^{\vee\vee}\cong\cE_1^{\oplus t}$. Since $v(\cF)=tv(\cE_1)-(0,0,1)$ it follows that $\cF$ fits into an exact sequence
\begin{equation}\label{autunno}
0\lra\cF\lra \cE_1\otimes V_t\overset{f}{\lra}\CC_p \lra 0
\end{equation}
where $V_t$ is a complex vector space of dimension $t$, and $p\in S$. 

(2): Suppose that the sheaf $\cF$ fits into the exact sequence in~\eqref{autunno}. By Lemma~\ref{lmm:stabse} $\cF$ is semistable if and only if  $f(p)\colon \cE_1(p)\to V^{\vee}_t$ is surjective, and then 
$\cF$ is stable. The group $\GL(V_t)$ acts on the set of maps $\cE_1(p)\to V^{\vee}_t$ of maximal rank, and 
$\ker(f)$ is isomorphic to $\ker(g)$ (for $f,g\colon \cE_1(p)\to V^{\vee}_t$ of maximal rank) if and only if the maps $f$ and $g$ are in the same $\GL(V_t)$-orbit.  By associating to a map $\cE_1(p)\to V^{\vee}_t$ 
of maximal rank its kernel (an element of $\Gr(r_1-t,\cE_1(p))$)   we identify $\Gr(r_1-t,\cE_1(p))$ with 
the quotient   of the set of maps $\cE_1(p)\to V^{\vee}_t$ of maximal rank modulo the action of $\GL(V_t)$.
\end{proof}
\subsection{Brill-Noether loci}\label{subsec:divisoredet}
\setcounter{equation}{0}
From now on we assume that $r_1=2a$ in~\eqref{eccovuuno}. Thus
\begin{equation}\label{ipotesisubi}
v_1=(2a,l_1,s_1).
\end{equation}
We are mainly interested in the moduli space $\cM_{v_2}$, where
\begin{equation}\label{ricotta}
v_2=av_1-(0,0,1)=v(a).
\end{equation}
More precisely we are interested in the determinantal loci $\cD^k_{v_2}\subset\cM_{v_2}$ defined below, see Definition~\ref{dfn:divukappa}.  This leads to consider also the  moduli spaces $\cM_{v(t)}$ 
for $1\le t< a$ and   the analogous determinantal loci $\cD^k_{v(t)}\subset\cM_{v(t)}$. The collection of the closed subsets $\cD^k_{v(t)}$ form a dualizable collection in the sense of~\cite[Def.~2.3]{markman-bn-duality} (the maps $f_{\bu,\bu}$ are essentially given by the maps $\pi_{v(t)}^k$ appearing in~\eqref{fibgrass}).
\begin{lmm}\label{lmm:stabmuku}
Let $1\le t\le a$. If the polarization $h_S$  is ${\mathsf a}(v_2)$-generic then
\begin{enumerate}
\item[(A)]
the polarization $h_S$  is ${\mathsf a}(v(t))$-generic for every $1\le t\le a$, 
\item[(B)]
there is no strictly semistable sheaf $\cF$ on $S$ with $v(\cF)=v(t)$, and
\item[(C)]
all sheaves parametrized by $\cM_{v(t)}(S,h_S)$ are stable.
\item[(D)]
$\cM_{v(t)}(S,h_S)$ is smooth (and projective).
\end{enumerate}
\end{lmm}
\begin{proof}
(A): Equation~\eqref{adimukai} and the equality
\begin{equation}\label{cicciomalefico}
\la v(t),v(t)\ra=2t(2a-t)
\end{equation}
give that ${\mathsf a}(v(t))\le {\mathsf a}(v(a))={\mathsf a}(v_2)$.

(B): Since $l_1^2-4a s_1=-2$, we have 
$\gcd\{\divisore(l_1),2a\}=1$. It follows 
that $v(t)$ is primitive.  Item~(B) follows because $h_S$  is $v(t)$-generic. 

(C): follows from (B).

(D): follows from (C) and the Mukai-Artamkin criterion for smoothness. 
\end{proof}
\begin{ntn}\label{ntn:genpol}
Throughout the present subsection  $h_S$  is an  ${\mathsf a}(v_2)$-generic polarization of $S$, and 
$\cM_{v(t)}=\cM_{v(t)}(S,h_S)$. 
\end{ntn}
\begin{dfn}\label{dfn:divukappa}
 Let $0<t\le a$. For $k$ a natural number let 
\begin{equation*}
\cD_{v(t)}^k\coloneq\{[\cF]\in\cM_{v(t)} \mid \hom_S(\cF,\cE_1)\ge k\}.
\end{equation*}
We let $\cD_{v(t)}=\cD_{v(t)}^1$. 
\end{dfn}
We have the chain of closed subsets
\begin{equation}\label{catenadidet}
\cM_{v(t)}=\cD_{v(t)}^0\supset\cD_{v(t)}^{1}\supset \ldots \supset \cD_{v(t)}^{k}\supset  \cD_{v(t)}^{ k+1}\ldots
\end{equation}
The next remark gives each $\cD_{v(t)}^k$ a structure of closed subscheme of $\cM_{v(t)}$. 
\begin{rmk}\label{rmk:stratchiusa}
Let $[\cF]\in \cM_{v(t)}$. Since $\Hom_S(\cE_1,\cF)$   vanishes by stability of $\cF$, we have $\Ext^2_S(\cF,\cE_1)=\{0\}$ by Serre duality.
Since
\begin{equation}\label{chieffe}
\chi_S(\cF,\cE_1)=-\la v(t),v_1\ra=-2(a-t),
\end{equation}
it follows that 
\begin{equation}\label{extuno}
\ext^1_S(\cF,\cE_1)=\hom_S(\cF,\cE_1)+2(a-t).
\end{equation}
The vanishing of $\Ext^2_S(\cF,\cE_1)$ gives also that
 $\cD_{v(t)}^{k}$   can be locally  (in the analytic or \'etale topology) described as the degeneracy locus of a map of vector bundles, i.e.~as a determinantal variety. More precisely, locally we have a map  $\varphi\colon V^m\lra V^{m+2(a-t)}$ of vector bundles of ranks $m$ and $m+2(a-t)$ respectively such that $\cD_{v(t)}^k$ is the set of points $x$ such that $\ker\varphi(x)$ has dimension at least $k$. Form this we get that 
 $\cD_{v(t)}^k$ has a structure of 
closed  subscheme of $\cM_{v(t)}$, of  expected codimension $k(2a-2t+k)$. The last assertion means that if  $\cD_{v(t)}^k$ is non empty then every one of its irreducible components has codimension at most $k(2a-2t+k)$. 
\end{rmk}
From now on we view $\cD_{v(t)}^k$ as a
closed  subscheme of $\cM_{v(t)}$. 
Note that~\eqref{catenadidet} is a chain of inclusions of closed subchemes.
\begin{prp}\label{prp:pendbrill}
Suppose that  $h_S$ is ${\mathsf a}(v_2)$-generic. Let $2\le t\le a$ and let 
$[\cF]\in \cM_{v(t)}(S,h_S)$.  If $\cF$ is locally free then the following are equivalent:
\begin{enumerate}
\item
$\cF$ is not slope stable.  
\item
$[\cF]\in \cD_{v(t)}$, i.e.~$\hom_S(\cF,\cE_1)>0$, and there is an exact sequence
\begin{equation}\label{italiasvizzera}
0\lra \cH\lra \cF\overset{\psi}{\lra} \cE_1\otimes\Hom_S(\cF,\cE_1)^{\vee}\lra 0 
\end{equation}
where  $\cH$ is   slope stable (in particular $r(\cH)>0$),   and $\psi$  corresponds to $\Id\in \Hom_S(\cF,\cE_1)^{\vee}\otimes\Hom_S(\cF,\cE_1)$.
\end{enumerate}
\end{prp}
\begin{proof}
We prove that (1) implies (2). 
Since $\cF$ is not slope stable there exists an exact sequence of torsion free sheaves
\begin{equation*}
0\lra \cH\lra \cF\lra \cQ\lra 0
\end{equation*}
where $0<r(\cQ)<r(\cF)$, $\mu(\cF)=\mu(\cQ)$ and $\cH$ is slope stable. Since  $h_S$ is ${\mathsf a}(v_2)$-generic it is also 
 ${\mathsf a}(v(t))$-generic (see Lemma~\ref{lmm:stabmuku}).  
 It follows (see~\cite[Pro.~3.10]{ogfascimod}) that $r(\cQ)c_1(\cF)-r(\cF)c_1(\cQ)=0$. 
This gives  that   $r(\cQ)=2ka$ and $c_1(\cQ)=kl_1$ where $k$ is an integer  such that $1\le k\le t-1$, and hence we may write $v(\cQ)=(2ka,kl_1,m)$ for some integer $m$. 
By stability of $\cF$ (and the equality $\mu(\cF)=\mu(\cQ)$) it follows that
\begin{equation*}
1+\frac{s_1}{2a}-\frac{1}{2at}=\frac{\chi(S,\cF)}{r(\cF)}<\frac{\chi(S,\cQ)}{r(\cQ)}=1+\frac{m}{2ka}.
\end{equation*}
This gives that $m\ge ks_1$ (recall that $1\le k\le t-1$). Hence $v(\cQ)=kv_1+q(0,0,1)$ where $q\ge 0$. 
Since $\cF$ is slope semistable and $\mu(\cF)=\mu(\cQ)$ the sheaf $\cQ$ is slope semistable. By
 Lemma~\ref{lmm:pocosing}
 we get that  $\cQ\cong\cE_1^{\oplus k}$. Since $\cH$ is slope stable we have 
 $\Hom_S(\cH,\cE_1)=0$ and hence $\Hom_S(\cF,\cE_1)\cong \Hom_S(\cE_1^{\oplus k},\cE_1)$. Since $\cE_1$ is stable  it follows that $\dim\Hom_S(\cF,\cE_1)=k$. This proves that (1) implies (2).
 
(2) implies  (1)  because $c_1(\cF)/r(\cF)=l_1/2a=c_1(\cE_1)/r(\cE_1)$. 
\end{proof}
\begin{rmk}\label{rmk:anatra}
Let $\cH$ be the sheaf appearing in~\eqref{italiasvizzera}. Then $\Hom_S(\cH,\cE_1)=0$. In fact this follows by applying the functor 
$\Hom_S(-,\cE_1)$ to the exact sequence in~\eqref{italiasvizzera} and noting that $\Ext^1_S(\cE_1,\cE_1)=0$ because $\cE_1$ is spherical.
\end{rmk}
\begin{rmk}\label{rmk:arancia}
Applying the functor 
$\Hom_S(\cE_1,-)$ to the exact sequence in~\eqref{italiasvizzera} we get the coboundary map
\begin{equation}
\CC \Id_{\cE_1}\otimes\Hom_S(\cF,\cE_1)^{\vee}\overset{\partial}{\lra} \Ext^1_S(\cE_1,\cH)
\end{equation}
The map is injective because  by stability of $\cF$ we have
 $\Hom_S(\cE_1,\cF)=0$.
\end{rmk}
The result below gives a converse of Proposition~\ref{prp:pendbrill}. 
\begin{prp}\label{prp:interbarca}
Suppose that  $h_S$ is ${\mathsf a}(v_2)$-generic. 
  Let $0<k<t\le a$. Let
 $[\cH]\in (\cM_{v(t-k)}\setminus \cD_{v(t-k)})$, and let $V_k\subset\Ext^1_S(\cE_1,\cH)$ be a $k$-dimensional vector subspace (by~\eqref{extuno} we have $2k\le \ext^1(\cE_1,\cH)$). Let
\begin{equation}\label{estensione}
0\lra \cH\lra \cF\overset{\varphi}{\lra}\cE_1 \otimes V_k\lra 0 
\end{equation}
be the  extension with extension class in $\Ext^1_S(\cE_1\otimes V_k,\cH)=V_k^{\vee}\otimes \Ext^1_S(\cE_1,\cH)$ corresponding to the inclusion map $V_k\hra \Ext^1_S(\cE_1,\cH)$. Then $\cF$ is  a stable  sheaf,  $[\cF]\in \cD_{v(t)}$, and 
$\hom_S(\cF,\cE_1)=k$.
\end{prp}
\begin{proof}
  Suppose that $\cF$ is  not stable. Then  $\cF$ is   unstable by Lemma~\ref{lmm:stabmuku}, i.e.~there exists a non zero subsheaf 
$\cG\subset\cF$ such that 
\begin{equation}\label{limite}
\frac{\chi(S,\cG(n))}{r(\cG)}>\frac{\chi(S,\cF(n))}{r(\cF)},\quad n\gg 0.
\end{equation}
 Consider the exact sequence
\begin{equation}
0\lra \ker(\varphi_{|\cG})\lra \cG\overset{\varphi_{|\cG}}{\lra}\varphi(\cG)\lra 0 
\end{equation}
Since $\mu(\cH)=\mu(\cF)=\mu(V_k\otimes_{\CC}\cE_1)$, it follows from the inequality in~\eqref{limite} that 
\begin{equation}\label{lalaland}
\mu(\cG)=\mu(\cF)=\mu(\ker(\varphi_{|\cG}))=\mu(\varphi(\cG)).
\end{equation}
(Of course $\mu(\ker(\varphi_{|\cG}))$ make sense  only if $\ker(\varphi_{|\cG})\not=0$ and similarly for $\mu(\varphi(\cG))$.)
Hence the inequality  in~\eqref{limite} implies that 
\begin{equation}
\frac{\chi(S,\cG)}{r(\cG)}>\frac{\chi(S,\cF)}{r(\cF)}.
\end{equation}
Suppose that $\ker(\varphi_{|\cG})$ is non zero. Then since it is a subsheaf of the stable sheaf $\cH$ we have 
\begin{equation*}
\frac{\chi(S,\ker(\varphi_{|\cG})}{r(\ker(\varphi_{|\cG}))}<\frac{\chi(S,\cH)}{r(\cH)}=1+\frac{s_1}{2a}-\frac{1}{2a(t-k)}
<1+\frac{s_1}{2a}-\frac{1}{2at}=\frac{\chi(S,\cF)}{r(\cF)}.
\end{equation*}
Similarly, if $\varphi(\cG)$ is non zero then $\chi(S,\varphi(\cG))/r(\varphi(\cG))\le \chi(S,V_k\otimes\cE_1)/r(V_k\otimes\cE_1)$, with equality if and only if $\varphi(\cG)=U_j\otimes \cE_1$ for some vector subspace $U_j\subset V_k$ of non zero dimension $j$ (recall the equalities in~\eqref{lalaland}).
The equality
\begin{equation*}
\frac{\chi(S,\cG)}{r(\cG)}=\frac{r(\ker(\varphi_{|\cG}))}{r(\cG)}\frac{\chi(S,\ker(\varphi_{|\cG})}{r(\ker(\varphi_{|\cG}))}+
\frac{r(\varphi(\cG))}{r(\cG)}\frac{\chi(S,\varphi(\cG))}{r(\varphi(\cG))}
\end{equation*}
gives that $\varphi_{|\cG}=0$ and $\varphi(\cG)=U_j\otimes \cE_1$. This is absurd: by our choice of extension class the inclusion 
$U_j\otimes \cE_1\subset V_k\otimes\cE_1$ does not lift to $\cF$.
\end{proof}
\begin{lmm}\label{lmm:codtandetvar}
Let $0<t\le a$. Suppose that $[\cF]\in (\cD^k_{v(t)}\setminus \cD^{k+1}_{v(t)})$. Then 
\begin{equation}\label{codtan}
\dim\cM_{v(t)}-\dim\Theta_{[\cF]}\cD^k_{v(t)}=k(2a-2t+k).
\end{equation}
\end{lmm}
\begin{proof}
By hypothesis $\hom_S(\cF,\cE_1)=k$. Let 
\begin{equation}
V_k\coloneq \Hom_S(\cF,\cE_1)^{\vee},
\end{equation}
 and let $f\colon \cF\lra \cE_1\otimes V_k$ be the tautological map. We have the linear map
\begin{equation}\label{pippobaudo}
\begin{matrix}
\Ext^1_S(\cF,\cF) & \overset{\Phi_\cF}{\lra} & \Ext^1_S(\cF,\cE_1\otimes V_k) \\
e & \mapsto &  e\cup f
\end{matrix}
\end{equation}
The Zariski tangent space $\Theta_{[\cF]}\cD^k_{v(t)}\subset \Theta_{[\cF]}\cM_{v(t)}=\Ext^1_S(\cF,\cF)$ is given by
\begin{equation}
\Theta_{[\cF]}\cD^k_{v(t)}=\ker\Phi_\cF.
\end{equation}
Assume that $\cF$ is locally free.  By Proposition~\ref{prp:pendbrill} we have the exact sequence
\begin{equation}\label{beeone}
0\lra \cH\lra \cF\overset{\psi}{\lra} \cE_1\otimes V_k\lra 0, 
\end{equation}
where $\cH$ is slope stable with $v(\cH)=v(t-k)$. Applying the functor $\Hom_S(\cF,-)$ to the exact sequence in~\eqref{beeone} we get the exact sequence
\begin{equation*}
 \Ext^1_S(\cF,\cF)  \overset{\Phi_\cF}{\lra}  \Ext^1_S(\cF,\cE_1)\otimes V_k \lra \Ext^2_S(\cF,\cH)\overset{\alpha}{\lra} \Ext^2_S(\cF,\cF).
\end{equation*}
We claim that $\alpha$ is an isomorphism, i.e.~that the transpose  $\alpha^t$  is an isomorphism. In fact $\alpha^t$ is equal (via Serre duality) to the map $\Hom_S(\cF,\cF)\to 
\Hom_S(\cH,\cF)$ obtained by composing with the inclusion $\cH\hra\cF$. Since $\cF$ and $\cH$ are stable (and hence simple), in order to show that $\alpha^t$ is an isomorphism it suffices to prove that there are no non-zero maps $\cH\to \cE_1\otimes V_k$. This follows from slope stability of $\cH$.
Since $\alpha$ is an isomorphism the map $\Phi_\cF$ is surjective, and hence
\begin{equation*}
\dim\cM_{v(t)}-\dim\Theta_{[\cF]}\cD^k_{v(t)}=\ext^1_S(\cF,\cE_1\otimes V_k).
\end{equation*}
Applying the functor $\Hom_S(-,\cE_1)$ to the exact sequence  in~\eqref{beeone} we get the exact sequence
\begin{equation}\label{baroque}
0\to \Ext^1_S(\cF,\cE_1) \to \Ext^1_S(\cH,\cE_1) \to \Ext^2_S(\cE_1,\cE_1)\otimes V_k^{\vee}\to 0.
\end{equation}
We have $\ext^1_S(\cH,\cE_1)=2(a-t+k)$ by Remark~\ref{rmk:anatra} and Equation~\eqref{extuno} (for $\cF=\cH$). Hence
$\ext^1_S(\cF,\cE_1) = 2a-2t+k$
and  the equality in~\eqref{codtan} follows (for $\cF$  locally free). 

Next assume that $\cF$ is not locally free.  By Proposition~\ref{prp:bordodiemme} we have an exact sequence
\begin{equation}\label{doppioduale}
0\lra\cF\lra \cE_1\otimes V_t\overset{f}{\lra}\CC_p \lra 0.
\end{equation}
In particular $\hom_S(\cF,\cE_1)=t$, i.e.~$k=t$. 
Applying the functor $\Hom_S(\cF,-)$ to the above exact sequence one gets the exact sequence
\begin{equation*}
 \Ext^1_S(\cF,\cF)  \overset{\Phi_\cF}{\lra}  \Ext^1_S(\cF,\cE_1)\otimes V_t \lra \Ext^1_S(\cF,\CC_p)\overset{\partial}{\lra} \Ext^2_S(\cF,\cF)\lra 0.
\end{equation*}
We have $\chi_S(\cF,\CC_p)=-\la v(\cF),v(\CC_p)\ra=2at$. Since $\hom_S(\cF,\CC_p)=2at+1$ (the stalk  $\cF_p$ is isomorphic to $\cO_{S,p}^{\oplus(2at-1)}\oplus\gm_p$) and $\ext^2_S(\cF,\CC_p)=\hom_S(\CC_p,\cF)=0$, it follows that $\ext^1_S(\cF,\CC_p)=1$. Hence the exact sequence above gives that $\Phi_\cF$ is surjective. Thus
\begin{equation*}
\dim\cM_{v(t)}-\dim\Theta_{[\cF]}\cD^t_{v(t)}=\ext^1_S(\cF,\cE_1\otimes V_t).
\end{equation*}
Applying the functor $\Hom_S(-,\cE_1)$ to the exact sequence  in~\eqref{doppioduale} we get the exact sequence
\begin{equation}
0\to \Ext^1_S(\cF,\cE_1) \to \Ext^2_S(\CC_p,\cE_1) \to \Ext^2_S(\cE_1,\cE_1)\otimes V_t^{\vee}\to 0.
\end{equation}
Hence $\ext^1_S(\cF,\cE_1)=2a-t$ and the equality in~\eqref{codtan} follows.
\end{proof}
\begin{prp}\label{prp:divuti}
  Let $0<t\le a$. Let $k$ be a natural number.
\begin{enumerate}
\item
If $k>t$ then $\cD_{v(t)}^{k}$ is empty.
\item
Let $0\le k\le t$. Then $\cD_{v(t)}^{k}\setminus \cD_{v(t)}^{k+1}$ is non empty and smooth of the expected dimension.
\item
We have $\cD_{v(t)}^{t}=\cB_{v(t)}$. 
\item
Let  $0\le k< t$. There is a fibration (locally trivial in the classical topology)
\begin{equation}\label{fibgrass}
\xymatrix{ \Gr(k,2(a-t+k))\ar[rr]    &  &  \cD_{v(t)}^{ k}\setminus \cD_{v(t)}^{ k+1} \ar[d]^{\pi_{v(t)}^k}\\ 
   & & \cM_{v(t-k)}\setminus\cD_{v(t-k)} }
\end{equation}
defined by $\pi_{v(t)}^k([\cF])\coloneq[\cH]$ where $\cH$  is the sheaf appearing in~\eqref{italiasvizzera} ($\cF$ is locally free by Item~(3) and $[\cH]\in \cM_{v(t-k)}\setminus\cD_{v(t-k)}$ by Remark~\ref{rmk:anatra}), and $(\pi_{v(t)}^k)^{-1}([\cH])=\Gr(k,\Ext^1_S(\cE_1,\cH))$.  
\end{enumerate}
\end{prp}
\begin{proof}
(1): Let $[\cF]\in\cD_{v(t)}^{k}$. Suppose that $\cF$ is locally free. Then $\cF$ fits into the exact sequence in~\eqref{italiasvizzera} with 
$\hom_S(\cF,\cE_1)\ge k$. Since $r(\cH)>0$ we have  $2at=r(\cF)>\hom_S(\cF,\cE_1) r(\cE_1)\ge kr(\cE_1)=2ak$. Thus $k<t$ if  $\cF$ is locally free. If $\cF$ is not locally free 
i.e.~$[\cF]\in \cB_{v(t)}$, then $k\le t$ by Item~(1) of Proposition~\ref{prp:bordodiemme}. 

(2): We have $\cB_{v(t)}\subset \cD_{v(t)}^{t}$ by Item~(1) of Proposition~\ref{prp:bordodiemme}. Thus $\cD_{v(t)}^{t}$ is non empty. Now let $0<k<t$. Then $\cM_{v(t-k)}$ is non empty because $\la v(t-k),v(t-k)\ra=2(t-k)(2a-t+k)>0$. Moreover $\cD_{v(t)}\not=\cM_{v(t)}$ by Lemma~\ref{lmm:codtandetvar}. Let $[\cH]\in (\cM_{v(t-k)}\setminus \cD_{v(t)})$, let $V_k\subset\Ext^1_S(\cE_1,\cH)$ be a $k$-dimensional vector subspace, and let $\cF$ be the sheaf fitting into the exact sequence in~\eqref{estensione}. By Proposition~\ref{prp:interbarca} the sheaf $\cF$ is 
stable and $[\cF]\in(\cD_{v(t)}^{k}\setminus \cD_{v(t)}^{k+1})$. This proves that $\cD_{v(t)}^{k}\setminus \cD_{v(t)}^{k+1}$ is non empty for $0<k\le t$. 
It is smooth of the expected dimension by Lemma~\ref{lmm:codtandetvar}, because the right-hand side of~\eqref{codtan} is equal to the expected codimension of $\cD_{v(t)}^{k}$. 

(3): In the proof of Item~(1) we have shown that we have the inclusion of sets
$\cD_{v(t)}^{t}\subset\cB_{v(t)}$. The reverse inclusion follows from Item~(1) of Proposition~\ref{prp:bordodiemme}.  

(4): This follows from Propositions~\ref{prp:pendbrill} and~\ref{prp:interbarca}.
\end{proof}
\begin{rmk}\label{rmk:eccoilnormale}
Let $[\cF]\in(\cD_{v(t)}^k\setminus \cD_{v(t)}^{k+1})$, and let $\cN_{\cD_{v(t)}^k/\cM_{v(t)}}([\cF])$ be the fiber at $[\cF]$ of the normal bundle of 
$\cD_{v(t)}^k$ in $\cM_{v(t)}$ (recall that $\cD_{v(t)}^k$ is smooth away from $\cD_{v(t)}^{k+1}$). The proof of Lemma~\ref{lmm:codtandetvar} gives an identification
\begin{equation}\label{normaleadi}
\cN_{\cD_{v(t)}^k/\cM_{v(t)}}([\cF])=\Hom_S(\cF,\cE_1)^{\vee}\otimes\Ext^1_S(\cF,\cE_1).
\end{equation}
\end{rmk}
\begin{rmk}\label{rmk:applin-normale}
Let $[\cF]\in(\cD_{v(t)}^k\setminus \cD_{v(t)}^{k+1})$. Without going through Lemma~\ref{lmm:codtandetvar} one defines  an (a priori) inclusion 
\begin{equation}\label{testarapata}
\cN_{\cD_{v(t)}^k/\cM_{v(t)}}([\cF])\subset \Hom_S(\cF,\cE_1)^{\vee}\otimes\Ext^1_S(\cF,\cE_1)
\end{equation}
as follows. Locally on $\cM_{v(t)}$ (in the Zariski topology if there exists a universal sheaf on 
$S\times \cM_{v(t)}$, in  the \'etale topology if  a universal sheaf does not exist) we have a morphism of vector bundles 
\begin{equation*}
V_d\otimes \cO_U \xrightarrow{\phi} W_{d+2(a-t)}\otimes \cO_U, 
\end{equation*}
where $V_d,W_{d+2(a-t)}$ are complex vector spaces of dimensions $d$ and $d+2(a-t)$ respectively such that 
$\cD^l\cap U$ (we work in the Zariski topology) for all $l$ is identified with the determinantal scheme 
$D^l(\phi)$.  Put differently, $\phi$ defines a morphism $\alpha\colon U\to V_d^{\vee}\otimes W_{d+2(a-t)}$
and $\cD^l\cap U=\alpha^{*}D^l(V_d^{\vee}\otimes W_{d+2(a-t)})$ where 
\begin{equation}\label{diellealto}
D^l(V_d^{\vee}\otimes W_{d+2(a-t)})\subset V_d^{\vee}\otimes W_{d+2(a-t)}
\end{equation}
is the determinantal subscheme parametrizes linear maps 
$V_d\to W_{d+2(a-t)}$ whose kernel has dimension at least $l$. Let $0=[\cF]\in U$, 
and let $K_0\coloneq \ker\phi(0)$,  $J_0\coloneq \coker\phi(0)$. Thus we have  identifications
\begin{equation}\label{kappajei}
K_0=\Hom_S(\cF,\cE_1),\qquad J_0=\Ext^1_S(\cF,\cE_1).
\end{equation}
Of course $K_0=\ker\alpha(0)$ and $J_0=\coker\alpha(0)$. Let $\mu_0\colon V_d^{\vee}\otimes W_{d+2(a-t)}\lra K_0^{\vee}\otimes J_0$ be the obvious map. 
An elementary  computation shows that 
\begin{equation*}
\Theta_{D^k(V_d^{\vee}\otimes W_{d+2(a-t)})}(\alpha(0))=
\ker\mu_0\subset V_d^{\vee}\otimes W_{d+2(a-t)}=\Theta_{V_d^{\vee}\otimes W_{d+2(a-t)}}
\end{equation*}
The above equality, together with the identifications in~\eqref{kappajei}, defines the inclusion in~\eqref{testarapata}.
\end{rmk}
\subsection{An iterated blow up}\label{subsec:scoppiascoppia}
\setcounter{equation}{0}
Throughout the present subsection we adopt the notation of Subsection~\ref{subsec:divisoredet}, in particular Notation~\ref{ntn:genpol} is in force. In addition  we let $\cM\coloneq \cM_{v_2}$ and $\cD^k\coloneq \cD^k_{v_2}$. 
\begin{dfn}\label{dfn:emmeconjei}
Let 
varieties $\cM(a),\ldots,\cM(0)$,  birational maps
\begin{equation*}
\cM(0)\xrightarrow{f_0}\cM(1)\xrightarrow{f_1}\ldots \cM(a)\xrightarrow{f_a}\cM,
\end{equation*}
  and for $j\in\{a,a-1,\ldots,0\}$  closed  subsets $E(j)^s,\cD(j)^t\subset \cM(j)$ for $s> j\ge  t\ge 0$
be as follows.
\begin{enumerate}
\item
$\cM(a)=\cM$,  $f_a\colon\cM(a)\to\cM$ is  the identity, $E(a)^s=\es$ and $\cD(a)^t=\cD^t$.
\item
 $f_{j-1}\colon \cM(j-1)\to \cM(j)$    and $E(j-1)^s,\cD(j-1)^t$ are defined inductively:
\begin{enumerate}
\item[(2a)]
$f_{j-1}$ is the blow up of $\cD(j)^{j}$
\item[(2b)]
$E(j-1)^{s}$ is the  strict transform of $E(j)^{s}$ if $ s> j$, and
 $E(j-1)^{j}$ is the exceptional divisor of $f_{j-1}$. 
\item[(2c)]
If $j-1\ge t\ge 0$ then 
$\cD(j-1)^t$ is the  strict transform of $\cD(j)^t$. 
\end{enumerate}
\end{enumerate}
\end{dfn}
Note that  $E(j)^s$ is non empty if and only if $a\ge s> j$. 
We let 
\begin{equation}\label{notacappuccio}
\cD(j)=\cD(j)^1,\quad\wh{\cM}=\cM(0),\quad \wh{E}^s=E(0)^{s}, \ a\ge s\ge 1.
\end{equation}
Note that $\cD(0)=\es$. Let
\begin{equation}\label{iteroscoppio}
\wh{f}\coloneq f_a\circ\ldots\circ f_0\colon \wh{\cM}\lra \cM.
\end{equation}
  We have $\wh{f}^{-1}(\cD)=\wh{E}^a\cup \ldots\cup\wh{E}^1$. 
\begin{rmk}
 Formally $f_a$ is the blow up of $\cD^{a+1}=\es$. The last blow up is an isomorphism because $\cD(1)$ (as any other $\cD(j)$) is a Cartier divisor. 
 The reason  to define the last blow up  is the following. According to  Proposition~\ref{prp:stabuno}, by associating to 
 $[\cE_2]\in\cM$ the sheaf $\cG(\cE_1,\cE_2)$ (notation as in loc.~cit.)   we get a rational map to a moduli space $M_{\ww}$ of semistable sheaves on $S^{[2]}$. The sheaves  $\cG(\cE_1,\cE_2)$ are non stable precisely when $[\cE_2]\in\cD$. The sequence of blow-ups  $f_j$ 
is dictated by Langton's algorithm for semistable reduction. The last blow up $f_0\colon \cM(0)\to \cM(1)$ is an isomorphism but it does change the sheaves parametrized by $\cD(1)$ (making them stable).
\end{rmk}
\begin{prp}\label{prp:incrocinormali}
$\wh{\cM}$ is smooth and the divisor on $\wh{\cM}$ given by $\wh{E}^a+ \ldots +\wh{E}^1$ has simple normal crossings.
\end{prp} 
Before proving Proposition~\ref{prp:incrocinormali} we describe the normal cone of $\cD^k$ in $\cD^h$ (for $0\le h\le k$) away from $\cD^{k+1}$. 
Let
 $[\cF]\in(\cD^k\setminus \cD^{k+1})$, and let $\cC_{\cD^k/\cD^h}([\cF])$ be the fiber at $[\cF]$ of the normal cone of 
$\cD^k$ in $\cD^h$. By the identification in~\eqref{normaleadi} we have an embedding
\begin{equation}\label{fibracono}
\cC_{\cD^k/\cD^h}([\cF])\subset\Hom_S(\cF,\cE_1)^{\vee}\otimes\Ext^1_S(\cF,\cE_1).
\end{equation}
For an integer $l$ we let 
\begin{equation}\label{diellebasso}
D_l(\Hom_S(\cF,\cE_1)^{\vee}\otimes\Ext^1_S(\cF,\cE_1))\subset 
\Hom_S(\cF,\cE_1)^{\vee}\otimes\Ext^1_S(\cF,\cE_1)
\end{equation}
be the determinantal subscheme  parametrizing maps of rank at most $l$ - note the difference between the meaning of $l$ in the above definition and in~\eqref{diellealto}.
\begin{lmm}\label{lmm:conostrato}
Let
 $[\cF]\in(\cD^k\setminus \cD^{k+1})$. Then, referring to  the 
embedding in~\eqref{fibracono}, we have
\begin{equation}\label{cononormale}
\cC_{\cD^k/\cD^h}([\cF])=D_{k-h}(\Hom_S(\cF,\cE_1)^{\vee}\otimes\Ext^1_S(\cF,\cE_1)).
\end{equation}
\end{lmm}
\begin{proof}
We adopt the notation of Remark~\ref{rmk:applin-normale}. An elementary computation gives the following description of the fiber of the normal cone of $D^k=D^k(V_d^{\vee}\otimes W_{d+2(a-t)})$ in  $D^h=D^h(V_d^{\vee}\otimes W_{d+2(a-t)})$ at $\alpha(0)$:
\begin{equation*}
\cC_{D^k/D^h}(\alpha(0))=D_{k-h}(K_0^{\vee}\otimes J_0).
\end{equation*}
The lemma follows from the above equality, the identifications in~\eqref{kappajei}, and the equality in~\eqref{normaleadi}.
\end{proof}
\begin{proof}[Proof of Proposition~\ref{prp:incrocinormali}]
Starting from $j=a$ we prove that for $j\in\{a,\ldots,0\}$ the following hold:
\begin{enumerate}
\item[($R_j$)]
$\cM(j)$ is smooth.
\item[($S_j$)]
$\cD(j)^{j}$ is smooth. 
\item[($T_j$)]
The divisor $E(j)^a+\ldots+E^{j+1}(j)$ on $\cM(j)$ has simple normal crossings. 
\end{enumerate}
We have $\cM(a)=\cM$, $\cD(a)^{a}=\cD^a$. Hence  $\cM(a)$ is smooth by Item~(D) of Lemma~\ref{lmm:stabmuku}, and $\cD(a)^{a}$  is smooth by Items~(1) and~(2) of Proposition~\ref{prp:divuti}.
Since the divisor in Item~($T_a$) is $0$, we have shown  that Items~($R_a$), ($S_a$) and~($T_a$) hold. Assume that  Items~($R_j$), ($S_j$) and~($T_j$) hold. Since $f_{j-1}\colon\cM(j-1)\to \cM(j)$ is the blow up of 
$\cD(j)^j$, and $\cM(j),\cD(j)^j$ are smooth by the inductive hypothesis, it follows that Item~($R_j$) holds.  
Let us prove that Item~($S_{j-1}$) holds.  The composition 
$f_a\circ f_{a-1}\circ\ldots\circ f_{j-1}\colon\cM(j-1)\to \cM$ restricts to an isomorphism 
$(\cM(j-1)\setminus E(j-1)^a\setminus\ldots\setminus E(j-1)^j)\xrightarrow{\sim} (\cM\setminus \cD^j)$, and it maps 
$\cD(j-1)^{j-1}\setminus \cD(j-1)^{j}$ to 
$\cD^{j-1}\setminus \cD^{j}$. By Proposition~\ref{prp:divuti} it follows that $\cD(j-1)^{j-1}$ is smooth away from 
the intersection with $E(j-1)^a\cup\ldots\cup E(j-1)^j$. Let 
$x\in \cD(j-1)^{j-1}\cap(E(j-1)^a\cup\ldots\cup E(j-1)^j)$.   There exists a maximum $k$ with $a\ge k\ge j$  such that $x\in E(j-1)^k$. The restriction of $f_{k-1}\circ \ldots\circ f_{j-1}$ to $E(j-1)^k\setminus E(j-1)^{k+1}$ defines a map
\begin{equation*}
\lambda\colon E(j-1)^k\setminus E(j-1)^{k+1}\lra E(k-1)^k\setminus E(k-1)^{k+1}. 
\end{equation*}
By~\eqref{normaleadi}
 the restriction of $f_{a}\circ \ldots\circ f_{k-1}$ to $E(k-1)^k\setminus E(k-1)^{k+1}$ is a fibration 
\begin{equation}
\begin{matrix}
\PP(\Hom_S(\cE_2,\cE_1)^{\vee}\otimes\Ext^1_S(\cE_2,\cE_1))   & \subset &  E(k-1)^k\setminus E(k-1)^{k+1} \\ 
\downarrow & & \downarrow\theta \\
 [\cE_2]  & \in & \cD^k\setminus\cD^{k+1} 
\end{matrix}
\end{equation}
Composing with $\lambda$ we get the map 
\begin{equation}\label{stangona}
\theta\circ\lambda\colon  E(j-1)^k\setminus E(j-1)^{k+1}\lra \cD^k\setminus\cD^{k+1}. 
\end{equation}
The fiber of $\theta\circ\lambda$ over $ [\cE_2]$ is mapped by $f_{a}\circ \ldots\circ f_{k-1}$ to 
$\PP(\Hom_S(\cE_2,\cE_1)^{\vee}\otimes\Ext^1_S(\cE_2,\cE_1))$, and by Lemma~\ref{lmm:conostrato}     this map is  the blow-up of the projectivized determinantal scheme 
 $\PP(D_1(\Hom_S(\cE_2,\cE_1)^{\vee}\otimes\Ext^1_S(\cE_2,\cE_1)))$, followed by the blow-up of the strict transform of  
 $\PP(D_2(\Hom_S(\cE_2,\cE_1)^{\vee}\otimes\Ext^1_S(\cE_2,\cE_1)))$, etc.
Moreover by Lemma~\ref{lmm:conostrato}  $\cD(j-1)^{j-1}\cap(E(j-1)^k\setminus E(j-1)^{k+1})$ intersects such fiber   in the strict transform of 
$\PP(D_{k-j}(\Hom_S(\cE_2,\cE_1)^{\vee}\otimes\Ext^1_S(\cE_2,\cE_1)))$. Said strict transform is smooth. 
This proves that $\cD(j-1)^{j-1}$ is smooth at $x$.
\end{proof}
\section{Families of slope stable vector bundles on $S^{[2]}$}\label{sec:fibratistab}
\subsection{Road map}\label{subsec:guidasezquattro}
\setcounter{equation}{0}  
In Subsection~\ref{subsec:stabsuaperto} we prove that $\cG(\cE_1,\cE_2)$ is slope stable for all 
$[\cE_2]\in(\cM_{v_2}(S,h_S)\setminus \cD_{v_2}(S,h_S)) $, provided  the polarization of $S^{[2]}$ is close enough to $\bm{\mu}(h_S)$ 
(here $h_S$ is an ${\mathsf a}(v_2)$-generic polarization of $S$), see Proposition~\ref{prp:stabuno}. 
In Subsection~\ref{subsec:semirepl} we prove that if $[\cE_2]\in \cD_{v_2}(S,h_S) $ then $\cG(\cE_1,\cE_2)$ is slope unstable for any choice of  polarization of $S^{[2]}$, see Proposition~\ref{prp:liberoinstab}. 
The remaining subsections are dedicated to the construction of the semistable replacements of the sheaves $\cG(\cE_1,\cE_2)$ for  $[\cE_2]\in \cD_{v_2}(S,h_S) $ - but note that the proof that they are the semistable replacements is given later, in  Section~\ref{sec:estendo}. 
Specifically, in Subsection~\ref{subsec:cohomcompandext} we define, for 
$[\cE_2]\in (\cD^k_{v_2}(S,h_S)\setminus \cD^{k+1}_{v_2}(S,h_S)) $ with $k<a$, the semistable replacement of $\cE_2$   as a double extension of sheaves 
$\cE_1[2]^{+}$, $\cE_1[2]^{-}$ and $\cG(\cE_1,\cH)$, where $\cE_1[2]^{\pm}$ are the rigid slope stable vector bundles studied in~\cite{ogfascimod}, and $[\cH]\in(\cM_{v(a-k)}(S,h_S)\setminus \cD_{v(a-k)}(S,h_S) )$, and we give a similar construction for the semistable replacement of  $\cE_2$ when $[\cE_2]\in \cD^a_{v_2}(S,h_S)$, see Definitions~\ref{dfn:fasciobicaltilde} and~\ref{dfn:replacedia}. Subsection~\ref{subsec:sonostabili} contains the proof that the extensions defined in the previous subsection are slope stable, see Propositions~\ref{prp:stabcalbi} and~\ref{prp:calbisingstab}. We also prove that the extensions are locally free, see Corollary~\ref{crl:localmentelibero} - this is obvious for the semistable replacements of $\cE_2$ if 
$[\cE_2]\notin \cD^a_{v_2}(S,h_S)$, but if $[\cE_2]\in \cD^a_{v_2}(S,h_S)$ it is a pleasant surprise.
\subsection{Stability of $\cG(\cE_1,\cE_2)$ for $[\cE_2]\in(\cM_{v_2}\setminus \cD_{v_2}) $}\label{subsec:stabsuaperto}
\setcounter{equation}{0}  
\begin{hyp-dfn}\label{hyp-dfn:esseaemme}
$S$ is a projective $K3$ surface, $a,k$ are positive integers,  and $D$ is a divisor on $S$ such that
\begin{equation}\label{castiglione}
 D\cdot D=2\cdot(2ak-1).
\end{equation}
The Mukai vectors $v_1,v_2$ on $S$ are given by
\begin{equation}\label{verovudue}
v_1\coloneq\left(2a,D,k\right),\quad v_2\coloneq av_1-(0,0,1)=\left(2a^2,aD,ak-1\right).
\end{equation}
\end{hyp-dfn}
Note that $v_1^2=-2$, $v_2^2=2a^2$ (here squares are with respect to Mukai's pairing). Moreover  $v_2$ is primitive by Remark~\ref{rmk:vutiprim}, 
 and of course also $v_1$.
Let $h_S$ be an ${\mathsf a}(v_2)$-generic polarization of $S$. The moduli space $\cM_{v_1}(S,h_S)$ is a reduced point 
because $h_S$ is also ${\mathsf a}(v_1)$-generic. Throughout the section we let
\begin{equation}
\cM_{v_1}(S,h_S)=\{[\cE_1]\}.
\end{equation}
Thus $\cE_1$ is a spherical vector bundle. The moduli space  $\cM_{v_2}(S,h_S)$ is a HK variety of Type $K3^{[a^2+1]}$. Let  $[\cE_2]\in \cM_{v_2}(S,h_S)$, and let $\cG(\cE_1,\cE_2)$ be the sheaf on $S^{[2]}$ defined in~\cite[Def.~2.1]{og:highdim}. 
By Lemma~3.4 loc.~cit.~$\cG(\cE_1,\cE_2)$ is a torsion free simple sheaf on $S^{[2]}$, and its mock Mukai vector is given by
\begin{equation}\label{sanmichele}
\ww(\cG(\cE_1,\cE_2)) =\ww(D,a)\coloneq 4a^2 \left(2a,   {\bm\mu}(D)-a\delta,   \frac{ a^4}{3} c_2(S^{[2]}) \right),
\end{equation}
where  ${\bm\mu}(D)$  is represented by the divisor in $S^{[2]}$ parametrizing subschemes  
intersecting $D $ non trivially (assume that  $D$ is reduced), see also~\eqref{eccomugrasso}, and 
\begin{equation}\label{duedelta}
\Delta\coloneq\{[Z]\in S^{[2]} \mid \text{$Z$ is non reduced}\},\qquad 
2\delta=\cl(\Delta).
\end{equation}
Below is the main result of the present subsection.
\begin{prp}\label{prp:stabuno}
Keep hypotheses and notation  as above, in particular $h_S$ is ${\mathsf a}(v_2)$-generic.  Let  $ h_{S^{[2]}}$ be a polarization of $S^{[2]}$ which is ${\mathsf a}(\ww(D,a))$-suitable for $\bm{\mu}(h_S)$. 
Let $[\cE_2]\in(\cM_{v_2}(S,h_S)\setminus \cD_{v_2}(S,h_S)) $.  Then 
$\cG(\cE_1,\cE_2)$ is an $h_{S^{[2]}}$ slope stable vector bundle. 
\end{prp}
Below is a result which follows from Proposition~\ref{prp:stabuno}.
\begin{crl}\label{crl:stabuno}
With hypotheses as in Proposition~\ref{prp:stabuno}, and   $\ww\coloneq \ww(D,a)$, the map
\begin{equation}\label{isomaperti}
\begin{matrix}
\cM_{v_2}(S,h_S)\setminus \cD_{v_2}(S,h_S) & \lra & M_{\ww}(S^{[2]},h_{S^{[2]}}) \\
[\cE_2] & \mapsto & [\cG(\cE_1,\cE_2)]
\end{matrix}
\end{equation}
defines an isomorphism 
\begin{equation}\label{mappapsi}
\begin{matrix}
\cM_{v_2}(S,h_S)\setminus \cD_{v_2}(S,h_S) & \overset{\psi}{\overset{\sim}{\lra}} & M_{\ww}(S^{[2]},h_{S^{[2]}})^{\bullet} \\
[\cE_2] & \mapsto & [\cG(\cE_1,\cE_2)]
\end{matrix}
\end{equation}
where $M_{\ww}(S^{[2]},h_{S^{[2]}})^{\bullet}$ is an open dense subset of an 
irreducible component  of  $M_{\ww}(S^{[2]},h_{S^{[2]}})$.
\end{crl}
\begin{proof}
Let $[\cE_2]\in(\cM_{v_2}(S,h_S)\setminus \cD_{v_2}(S,h_S))$. The map in~\eqref{isomaperti}  identifies the deformation space of $\cE_2$ with the deformation space of $\cG(\cE_1,\cE_2)$, see Item~(4) of~\cite[Prop.~2.6]{og:highdim}. This proves that $M_{\ww}(S^{[2]},h_{S^{[2]}})$ is smooth at each point of the image and that the image   is 
an open dense subset of an 
irreducible component  of  $M_{\ww}(S^{[2]},h_{S^{[2]}})$.  Hence in order to finish the proof it suffices to show that the map $\psi$ in~\eqref{mappapsi} is injective. This holds because if 
$[\cE_2]\not=[\cE'_2]\in(\cM_{v_2}(S,h_S)\setminus \cD_{v_2}(S,h_S))$ then $\Hom_{S^{[2]}}(\cG(\cE_1,\cE_2),\cG(\cE_1,\cE'_2))=0$ by the BKR correspondence.
\end{proof}

We prove Proposition~\ref{prp:stabuno} after a series of preliminary results. 
\begin{prp}\label{prp:quadstab}
Let $X_1,X_2$   be   irreducible smooth projective varieties. For $i\in\{1,2\}$  let $h_i$ be an ample class on $X_i$ and $\cV_i$ be an $h_i$ slope-stable vector bundle on $X_i$. 
 Let $\pr_i\colon X_1\times X_2\to X_i$ be the projection and  let $h\coloneq \pr_1^{*} h_1+ \pr_2^{*} h_2$.   
Then $\cV_1\boxtimes \cV_2$ is $h$ slope-stable. 
\end{prp}
\begin{proof}
If $\cV_1\boxtimes \cV_2$ is not $h$ slope-stable there exists an exact sequence 
\begin{equation}
0\lra \cA \lra \cV_1\boxtimes \cV_2 \lra \cB\lra 0
\end{equation}
of  torsion free sheaves such that $0<r(\cA)<r(\cV_1\boxtimes \cV_2)$ and 
\begin{equation}\label{mexico}
\mu_h(\cA)\ge \mu_{h}(\cV_1\boxtimes \cV_2).
\end{equation}
For $i\in\{1,2\}$  let $n_i\coloneq \dim X_i$, and let $d_i\coloneq \int_{X_i} h_i^{n_i}$ be the degree of $X_i$. We have
\begin{equation}\label{pendab}
\mu_{h}(\cV_1\boxtimes \cV_2)= d_2 {m\choose n_2}\mu_{h_1}(\cV_1)+d_1 {m\choose n_1}\mu_{h_2}(\cV_2).  
\end{equation}
where $m\coloneq n_1+n_2-1$.
Let  $p_i\in X_i$  be general points.
Since  
$\cA$ is locally free in codimension $1$ we have
\begin{equation}\label{seansean}
\mu_{h}(\cA)=d_2 {m\choose n_2}\mu_{h_1}(\cA_{|{X_1}\times \{p_2\}})+ 
d_1 {m\choose n_1}\mu_{h_2}(\cA_{|\{p_1\}\times X_2}).
\end{equation}
Let $r_i\coloneq r(\cV_i)$. Since the restrictions of $\cV_1\boxtimes \cV_2$  to $X_1\times \{p_2\}$ and to $\{p_1\}\times X_2$ are isomorphic to the polystable vector bundles   $\cV_1\otimes_{\CC}\CC^{r_2}$ and $\CC^{r_1}\otimes_{\CC}\cV_2$ respectively,  it follows from~\eqref{mexico}, \eqref{pendab}, \eqref{seansean}   that $\mu_{h_1}(\cA_{|X_1\times \{p_2\}})=\mu_{h_1}(\cV_1)$ and 
$\mu_{h_2}(\cA_{|\{p_1\}\times X_2})=\mu_{h_2}(\cV_2)$. These equalities give that there exist vector subspaces 
$0\not=U\subset \CC^{r_2}$ and $0\not=W\subset \CC^{r_1}$ such that in codimension $1$ we have
$\cA_{|X_1\times \{p_2\}}=\cV_1\otimes_{\CC} U$ and $\cA_{|\{p_1\}\times X_2}=W\otimes_{\CC}\cV_2$ respectively. It follows that 
$r(\cA)=r(\cV_1\boxtimes \cV_2)$. This is a contradiction.
\end{proof}
Let $\tau\colon X(S)\to S^2$ be the the blow up of the diagonal $D(S)$. We have a commutative diagram
\begin{equation}\label{commiso}
\xymatrix{ X(S)\ar[d]_{\rho}\ar[rr]^{\tau}    &  &  S^2 \ar[d]^{\pi}\\ 
  S^{[2]}  \ar[rr]^{\gamma} & & S^{(2)} }
\end{equation}
where $\rho,\pi$ are the quotient maps for the natural $\ZZ/(2)$ actions, and $\gamma$ is the cycle (or Hilbert-to-Chow) map. 
For $i\in\{1,2\}$ let $\pr_i\colon S^2\to S$ be the $i$-th projection, and 
let $\tau_i\colon X(S)\to S$ be the composition $\tau_i\coloneq\pr_i\circ \tau$. 
Recall that $v(t)\coloneq tv_1-(0,0,1)$.
\begin{lmm}\label{lmm:americanpie}
Let $0<t\le a$. Let $h_S$ be an ${\mathsf a}(v_2)$-generic polarization of $S$.  
Let
\begin{equation}
[\cH]\in (\cM_{v(t)}(S,h_S)\setminus \cD_{v(t)}(S,h_S)).
\end{equation}
 Then 
$\tau_1^{*}\cE_1\otimes\tau_2^{*}\cH$ and $\tau_1^{*}\cH\otimes\tau_2^{*}\cE_1$
are $\rho^{*}\bm{\mu}(h_S)$ slope stable vector bundles with equal $\rho^{*}\bm{\mu}(h_S)$-slopes.
\end{lmm}
\begin{proof} 
Since $\cH$ is not in $ \cD_{v(t)}(S,h_S)$ it is locally free by Item~(3) of Proposition~\ref{prp:divuti}. The covering involution  of 
$\rho$  maps
 $\rho^{*}\bm{\mu}(h_S)$ to itself and it lifts to an isomorphism between
$\tau_1^{*}\cE_1\otimes\tau_2^{*}\cH$ and  
 $\tau_1^{*}\cH\otimes\tau_2^{*}\cE_1$. Hence the latter have equal $\rho^{*}\bm{\mu}(h_S)$-slopes. 
   
 It remains to prove that $\tau_1^{*}\cE_1\otimes\tau_2^{*}\cH$ and $\tau_1^{*}\cH\otimes\tau_2^{*}\cE_1$
 are $\rho^{*}\bm{\mu}(h_S)$ slope stable.  By symmetry it suffices to prove it for $\tau_1^{*}\cE_1\otimes\tau_2^{*}\cH=\tau^{*}(\cE_1\boxtimes\cH)$.  
Let $P\coloneq \bm{\mu}(h_S)$ and  $h_{S^2}\coloneq \pr_1^{*}h_S+\pr_2^{*}h_S$.  Let 
 $\cF\subset\tau^{*}(\cE_1\boxtimes\cH)$ be a subsheaf with $0<r(\cF)<r_1 r_2=r(\tau^{*}(\cE_1\boxtimes\cH))$. 
 We claim that
\begin{equation}\label{2ugu1ine}
\mu_{\rho^{*}P}(\cF)=\mu_{h_{S^2}}(\tau_{*}\cF)<\mu_{h_{S^2}}(\cE_1\boxtimes\cH)=
\mu_{\rho^{*}P}(\tau^{*}(\cE_1\boxtimes\cH)).
\end{equation}
In fact the intersection of three general divisors in $|N h_{S^2}|$ (for $N>0$  such that  $N h_{S^2}$ is very ample) is contained in the open 
$S^2\setminus D(S)$ over which $\tau$ is an isomorphism. Since $\rho^{*}P=\tau^{*}h_{S^2}$ this implies 
the first equality. 
 
The vector bundle $\cH$ is $h_S$ slope stable (because $c_1(\cH)=D$ is primitive and  $h_S$ is ${\mathsf a}(v(1))$-generic if $t=1$, and  by Proposition~\ref{prp:pendbrill}  if $t\ge 2$) and hence 
$\cE_1\boxtimes\cH$ is $h_{S^2}$ slope stable by Proposition~\ref{prp:quadstab}. Since $\tau_{*}\cF\subset \cE_1\boxtimes\cH$ we get that  the  inequality in~\eqref{2ugu1ine} holds. 

The second equality in~\eqref{2ugu1ine} holds because  
$\rho^{*}P=\tau^{*}h_{S^2}$.  
\end{proof}
\begin{prp}\label{prp:gistabperbig}
Let $0<t\le a$.  Let $h_S$ and $\cH$ be as in Lemma~\ref{lmm:americanpie}.
Then 
$\cG(\cE_1,\cH)$ is $\bm{\mu}(h_S)$ slope stable.
\end{prp}
\begin{proof}
Let $\cG=\cG(\cE_1,\cH)$ and $P\coloneq \bm{\mu}(h_S)$. Let $\cA\subset\cG$  with $0<r(\cA)<r(\cG)$. 
By~\cite[Subsect.~2.4]{og:highdim} we have an exact sequence
\begin{equation}\label{chiassosi}
0\lra \rho^{*}\cG\lra \tau^{*}(\cE_1\boxtimes\cH\oplus \cH\boxtimes\cE_1)\lra \cR\lra 0
\end{equation}
where  $\cR$ is zero away from $E\coloneq\tau^{-1}(D(S))$ is the exceptional divisor of  $\tau$. Let us prove that
\begin{multline}\label{scalzato}
\mu_{P}(\cA)=\mu_{\rho^{*}P}(\rho^{*}\cA)\le \mu_{\rho^{*}P}(\tau^{*}(\cE_1\boxtimes\cH\oplus \cH\boxtimes\cE_1))=\\
= \mu_{\rho^{*}P}(\rho^{*}\cG)=\mu_{P}(\cG).
\end{multline}
In fact the inequality above  holds because $\rho^{*}\cA\subset \tau^{*}(\cE_1\boxtimes\cH\oplus \cH\boxtimes\cE_1)$
by the exact sequence in~\eqref{chiassosi}, 
 and because $\tau^{*}\cE_1\boxtimes\cH$, $\tau^{*}\cH\boxtimes\cE_1$ are $\rho^{*}\bm{\mu}(h_S)$ slope stable (of the same slope)
   by Lemma~\ref{lmm:americanpie}.
In addition the second equality in~\eqref{scalzato} holds because $E\cdot \rho^{*}P^3=0$. This finishes the proof of the validity of~\eqref{scalzato}. Since the vector bundle $\cE_1\boxtimes\cH$ is not isomorphic to  the vector bundle 
$\cH\boxtimes\cE_1$,  the inequality above is an equality only if in codimension $1$ we have
$\rho^{*}\cA=\cE_1\boxtimes\cH$ or $\rho^{*}\cA=\cH\boxtimes\cE_1$.  This is absurd because $\rho^{*}\cA$ is a subsheaf invariant under the action of the covering involution.
\end{proof}
\begin{rmk}
Let $a=1$. Then all sheaves parametrized by $\cM_{v_2}$ are slope stable, but if $[\cH]\in\cB_{v_2}$ then  $\cG(\cE_1,\cH)$ is 
not $\bm{\mu}(h_S)$ slope stable, see Subsection~\ref{subsec:semirepl}. For the validity of the proof of Proposition~\ref{prp:gistabperbig}  it is crucial that $\cH$ is locally free.  
\end{rmk}
\begin{rmk}\label{rmk:rigidistab}
Let $\cE_1[2]^{\pm}$ be the rank-$4a^2$ vector bundles defined in~\cite[Def.~5.1]{ogfascimod}, 
i.e.~$\cE_1[2]^{+}=\rho_{*}(\cE_1\boxtimes\cE_1)^{\Sigma_2}$ is the sheaf of invariants for the  \lq\lq natural\rq\rq\  lift to 
$\tau_1^{*}\cE_1\otimes \tau_2^{*}\cE_1$ of the action  of the symmetric group on two elements $\Sigma_2$ on $X(S)$, and $\cE_1[2]^{-}$ is defined multiplying the natural action by the sign function. Arguing as above one proves that $\cE_1[2]^{\pm}$ is  $\bm{\mu}(h_S)$ slope stable.
\end{rmk}
\subsubsection{Proof of Proposition~\ref{prp:stabuno}}
The sheaf $\cE_1$ is locally free (it is a spherical vector bundle), and $\cE_2$ is locally free by Item~(3) of Proposition~\ref{prp:divuti}. It follows that 
 $\cG(\cE_1,\cE_2)$ is locally free, see~\cite[Prop.~2.4]{og:highdim}. 
 
 Lastly we prove that $\cG(\cE_1,\cE_2)$ is $ h_{S^{[2]}}$ slope stable. Since 
 $ h_{S^{[2]}}$ and $\bm{\mu}(h_S)$ have positive BBF square, 
 $ h_{S^{[2]}}$ and $\bm{\mu}(h_S)$ slope stability are equivalent respectively to
   $ h_{S^{[2]}}$ and $\bm{\mu}(h_S)$ HK-slope stability,  see Remark~\ref{rmk:hkstabpos}. 
 By Proposition~\ref{prp:gistabperbig} 
 $\cG(\cE_1,\cE_2)$ is 
 $\bm{\mu}(h_S)$ slope stable. Since $\cG(\cE_1,\cE_2)$ is modular and $ h_{S^{[2]}}$  is ${\mathsf a}(\ww(D,a))$-suitable for $\bm{\mu}(h_S)$,
 it follows from Corollary~\ref{crl:paragonestab} that
 $\cG(\cE_1,\cE_2)$ is $ h_{S^{[2]}}$ slope stable.

\subsection{Instability of $\cG(\cE_1,\cE_2)$ for $[\cE_2]\in \cD_{v_2} $}\label{subsec:semirepl}
\setcounter{equation}{0}  
Throughout the present subsection Hypothesis-Definition~\ref{hyp-dfn:esseaemme} holds, and  the polarization $h_S$ is 
${\mathsf a}(v_2)$-generic. Let $[\cE_2]\in \cD_{v_2} $, i.e.~$k\coloneq \hom_S(\cF,\cE_1)>0$. By Proposition~\ref{prp:divuti} we have $k\le a$, and $k=a$ if and only if 
$\cE_2$ is not locally free. We treat separately the cases $k<a$ and $k=a$.
\setcounter{equation}{0}  
\subsubsection{$\cG(\cE_1,\cE_2)$ is not slope semistable if $[\cE_2]\in (\cD_{v_2}\setminus  \cD_{v_2}^a)$}
\label{subsubsec:liberoinstab}
Let  $[\cE_2]\in (\cD_{v_2}^k\setminus \cD_{v_2}^{k+1})$ where  $0<k<a$. In particular 
 $\cE_2$ is locally free.
Let
\begin{equation}\label{eccovukappa}
V_k\coloneq \Hom_S(\cE_2,\cE_1)^{\vee}.
\end{equation}
By Proposition~\ref{prp:pendbrill}
 there is an exact sequence
\begin{equation}\label{sinner}
0\lra \cH\lra \cE_2\overset{\psi}{\lra} \cE_1\otimes V_k\lra 0 
\end{equation}
where  $\psi$ is the tautological map, and
$\cH$ is  a slope stable vector bundle   with $v(\cH)=v(a-k)=(a-k)v_1-(0,0,1)$. Pulling back the exact sequence in~\eqref{sinner} to $S^2$ via the projection $\pr_i$ for $i\in\{1,2\}$, tensorizing  by $\pr_{3-i}^{*}\cE_1$  and taking the direct sum, we get the
exact sequence
\begin{equation}\label{musetti}
0\lra \cG(\cE_1,\cH)\lra \cG(\cE_1,\cE_2)\overset{\Psi_{\cE_2}}{\lra} \cG(\cE_1,\cE_1)\otimes V_k\lra 0. 
\end{equation}
Let $\cE_1[2]^{\pm}$ be as in Remark~\ref{rmk:rigidistab}. In~\cite[Prop.~2.3]{og:highdim} we have defined the direct sum decomposition
\begin{equation}\label{aaspezza}
\cG(\cE_1,\cE_1)= \cE_1[2]^{+}\oplus \cE_1[2]^{-}
\end{equation}
as follows. 
We have injections 
\begin{equation*}
\begin{matrix}
\cE_1\boxtimes\cE_1 & \overset{\iota^{+}}{\hra} & \cE_1\otimes\cE_1\oplus \cE_1\otimes\cE_1\\
\xi & \mapsto & (\xi,\xi)
\end{matrix}
\qquad
\begin{matrix}
\cE_1\boxtimes\cE_1 & \overset{\iota^{-}}{\hra} & \cE_1\otimes\cE_1\oplus \cE_1\otimes\cE_1\\
\xi & \mapsto & (\xi,-\xi)
\end{matrix}
\end{equation*}
The morphism $(\iota^{+},\iota^{-})$ is an isomorphism:
\begin{equation}\label{flautotraverso} 
(\iota^{+},\iota^{-})\colon \cE_1\boxtimes\cE_1 \oplus \cE_1\boxtimes\cE_1 \overset{\sim}{\lra} \cE_1\boxtimes\cE_1 \oplus \cE_1\boxtimes\cE_1 .
\end{equation}
Moreover $(\iota^{+},\iota^{-})$  is $\Sigma_2$-equivariant if the action on the first $ \cE_1\boxtimes\cE_1$ addend (respectively the second one) is the first one 
(respectively the second one)   in Remark~\ref{rmk:rigidistab}. Thus~\eqref{flautotraverso} gives a direct sum decomposition in the category of $\Sigma_2$-equivariant bounded complexes on $S^2$. Applying the BKR correspondence one gets the decomposition in~\eqref{aaspezza} - more precisely an isomorphism between the right-hand side and the left-hand side.
 Let $\Psi_{\cE_2}^{\pm}$ be the composition 
\begin{equation}
\cG(\cE_1,\cE_2)\xrightarrow{\Psi_{\cE_2}}\cG(\cE_1,\cE_1)\otimes V_k\to \cE_1[2]^{\pm}\otimes V_k,
\end{equation}
where the second map is the projection associated to the decomposition in~\eqref{aaspezza}. Let 
\begin{equation}\label{bordeaux}
\cA(\cE_2)^{\pm}\coloneq\ker \Psi_{\cE_2}^{\pm}.
\end{equation}
Note that $\cA(\cE_2)^{\pm}$ fits into the exact sequence
\begin{equation}\label{elvis}
0\lra \cG(\cE_1,\cH)\lra \cA(\cE_2)^{\pm}\lra  \cE_1[2]^{\pm}\otimes V_k\lra 0. 
\end{equation}
The exact sequence in~\eqref{musetti} gives the exact sequences
\begin{equation}\label{presley}
0\lra \cA(\cE_2)^{-}\lra \cG(\cE_1,\cE_2)\overset{\Psi_{\cE_2}^{+}}{\lra} \cE_1[2]^{+}\otimes V_k\lra 0 
\end{equation}
and
\begin{equation}\label{tuttifrutti}
0\lra \cA(\cE_2)^{+}\lra \cG(\cE_1,\cE_2)\overset{\Psi_{\cE_2}^{-}}{\lra} \cE_1[2]^{-}\otimes V_k\lra 0.
\end{equation}
\begin{prp}\label{prp:liberoinstab}
If  $[\cE_2]\in (\cD_{v_2}^k\setminus \cD_{v_2}^{k+1})$ where  $0<k<a$ then $\cG(\cE_1,\cE_2)$ is unstable (i.e.~non semistable) for all polarizations 
of $S^{[2]}$. 
\end{prp}
\begin{proof}
By~\cite[Prop.~5.8]{ogfascimod} we have (recall~\eqref{duedelta})
\begin{equation}\label{piumenopend}
\frac{c_1(\cE_1[2]^{\pm})}{r(\cE_1[2]^{\pm})}=\frac{\bm{\mu}(D)}{2a}-\frac{\delta}{2}\pm\frac{\delta}{4a},
\end{equation}
and by~\cite[Prop.~2.7]{og:highdim} we have
\begin{equation}\label{otamendi}
\frac{c_1(\cG(\cE_1,\cE_2))}{r(\cG(\cE_1,\cE_2))}=\frac{c_1(\cG(\cE_1,\cH))}{r(\cG(\cE_1,\cH))}=\frac{c_1(\cG(\cE_1,\cE_1))}{r(\cG(\cE_1,\cE_1))}=\frac{\bm{\mu}(D)}{2a}-\frac{\delta}{2}.  
\end{equation}
Hence the proposition follows from the exact sequence in~\eqref{tuttifrutti}. 
\end{proof}
\subsubsection{$\cG(\cE_1,\cE_2)$ is not slope semistable if $[\cE_2]\in   \cD_{v_2}^a$}\label{subsubsec:notsemidia}
Suppose  that $[\cE_2]\in\cD^a_{v_2}$, i.e.~that $\cE_2$ is not locally free. 
By Proposition~\ref{prp:bordodiemme} there exist a point $p\in S$ and an exact sequence 
\begin{equation}\label{nelduale}
0\lra \cE_2\lra\cE_1\otimes V_a\overset{f}{\lra} \CC_{p}\lra 0,
\end{equation}
where $f$ is identified with a linear map $f(p)\colon \cE_1(p)\otimes V_a\to \CC$, which viewed as a map $\cE_1(p)\to V^{\vee}_a$  is surjective by Lemma~\ref{lmm:stabse}. Note that $p$ is the unique singular point of $\cE_2$ (the localization of $\cE_2$ at $p$ is not free).
From~\eqref{nelduale} we get the exact sequence
\begin{equation}\label{massacarrara}
0\lra \cG(\cE_1,\cE_2) \lra \cG(\cE_1,\cE_1)\otimes V_a\overset{\Phi_{\cE_2}}{\lra}\cG(\cE_1,\CC_{p})\lra 0.
\end{equation}
By the decomposition in~\eqref{aaspezza} the above exact sequence reads
\begin{equation}\label{piumenopsi}
0\lra \cG(\cE_1,\cE_2)\lra(\cE_1[2]^{+}\oplus \cE_1[2]^{-})\otimes V_a\overset{\Phi_{\cE_2}}{\lra} \cG(\cE_1,\CC_{p})\lra 0.
\end{equation}
Let $\Phi_{\cE_2} ^{\pm}$ be the restriction of $\Phi_{\cE_2}$ to $\cE_1[2]^{\pm}\otimes V_a$. 
\begin{prp}\label{prp:fipiusur}
Let $[\cE_2]\in\cD_{v_2}^a$, and let $p$ be the singular point of $\cE_2$. The map $\Phi_{\cE_2}^{+}\colon \cE_1[2]^{+}\otimes V_a\to  \cG(\cE_1,\CC_{p})$ is surjective.
\end{prp}
Granting  the validity of Proposition~\ref{prp:fipiusur}, we proceed to show that $\cG(\cE_1,\cE_2)$ is unstable if $[\cE_2]\in\cD_{v_2}^a$. 
Let 
$\cA(\cE_2)^{+}$ be the kernel of $\Phi^{+}_{\cE_2}$. We have the exact sequence
\begin{equation}\label{psipiu}
0\lra \cA(\cE_2)^{+} \lra \cE_1[2]^{+}\otimes V_a\overset{\Phi^{+}_{\cE_2}}{\lra} 
\cG(\cE_1,\CC_{p})\lra 0.
\end{equation}
Note that $\cA(\cE_2)^{+}$ is a subsheaf of $\cG(\cE_1,\cE_2)$, see~\eqref{piumenopsi}. Since $\Phi^{+}_{\cE_2}$ is surjective the 	quotient 
$\cG(\cE_1,\cE_2)/\cA(\cE_2)^{+}$ is isomorphic to $\cE_1[2]^{-}\otimes V_a$ (see the exact sequence in~\eqref{piumenopsi}). Hence we have 
an exact sequence
\begin{equation}\label{harnargi}
0\lra \cA(\cE_2)^{+}\lra \cG(\cE_1,\cE_2) \overset{\Psi^{-}_{\cE_2}}{\lra} \cE_1[2]^{-}\otimes V_a\lra 0.
\end{equation}
\begin{prp}\label{prp:singolareinstab}
If $[\cE_2]\in\cD_{v_2}^a$ then $\cG(\cE_1,\cE_2)$ is unstable  for all polarizations 
of $S^{[2]}$. 
\end{prp}
\begin{proof}
Since $r(\cA(\cE_2)^{+})=r(\cE_1[2]^{+})$ and $c_1(\cA(\cE_2)^{+})=c_1(\cE_1[2]^{+})$ the proposition follows from the exact sequence in~\eqref{harnargi} and the equalities in~\eqref{piumenopend} and~\eqref{otamendi}. 
\end{proof}
It remains to prove Proposition~\ref{prp:fipiusur}. 
\begin{dfn}
If $g\colon\cE_1\to \CC_p$  the morphism $\Phi_g^{+}\colon \cE_1[2]^{+}\to \cG(\cE_1,\CC_P)$ is the one corresponding via the BKR correspondence to 
\begin{equation}
(\Id\boxtimes g,g\boxtimes\Id)\in \Hom_{S^2}(\cE_1\boxtimes\cE_1,\cE_1\boxtimes\CC_p\oplus \CC_p\boxtimes\cE_1)^{\Sigma_2}.
\end{equation}
\end{dfn}
By the BKR correspondence and simplicity of $\cE_1$ we have an isomorphism 
\begin{equation}\label{ariacondiz}
\begin{matrix}
\cE_1(p)^{\vee} & \overset{\sim}{\lra} & \Hom_{S^{[2]}}(\cE_1[2]^{+},\cG(\cE_1,\CC_P)) \\
g & \mapsto & \Phi_g^{+}
\end{matrix}
\end{equation}
\begin{prp}\label{prp:pioggia}
If $g\in \cE_1(p)^{\vee}$ is non zero then $\Phi_g^{+}$ is a surjection.
\end{prp}
\begin{proof}
Let $b\colon S_p\to S$ be the blow up of $S$ with center $p$, and let $i_p\colon R_p\hra S_p$ be the inclusion of the exceptional divisor. We  view $S_p$ as the closed subset of $S^{[2]}$ given by
\begin{equation}
S_p=\{[Z]\in S^{[2]} \mid p\in\supp Z\},
\end{equation}
and we let $j\colon S_p\hra S^{[2]}$ be the inclusion map.
Note that $R_p$ parametrizes non reduced subschemes of $S$ parametrized by $S_p$, i.e.~those which are supported at $p$. 
We have
\begin{equation}\label{supportato}
\cG(\cE_1,\CC_{p})\cong j_{*}(b^{*}\cE_1).
\end{equation}
We also have an exact sequence of sheaves on $S_p$.
\begin{equation}\label{illusione}
0\lra j^{*}(\cE_1[2]^{+})\lra \cE_1(p)\otimes b^{*}\cE_1 \overset{g}{\lra} \bigwedge^2\cE_1(p)\otimes i_{p,*}\cO_{R_p}\lra 0,
\end{equation}
where $g$ is defined via the projection of $\cE_1(p)\otimes \cE_1(p)=\Sym^2\cE_1(p)\oplus\bigwedge^2\cE_1(p)$ onto the second direct summand. 
Tensorizing the exact sequence in~\eqref{illusione} with $\cO_{R_p}$  we get an exact sequence 
\begin{equation}\label{yakushima}
0\lra \bigwedge^2\cE_1(p)\otimes\cO_{R_p}(1)\lra \cE_1[2]^{+}_{|R_p}
\lra \Sym^2\cE_1(p)\otimes \cO_{R_p}\lra 0.
\end{equation} 
We are ready to prove the proposition. The sheaf $\cG(\cE_1,\CC_{p})$ is supported on $S_p$. 
On $S_p\setminus R_p$ the map  $\Phi^{+}_g$ is given by 
\begin{equation*}
\begin{matrix}
\cE_1[2]^{+}_{|(\wt{S}_p\setminus R_p)}\cong \cE_1(p)\otimes b^{*}\cE_{1|(S\setminus \{p\})} & 
\xrightarrow{\Phi^{+}_{g|(S_p\setminus R_p)}} & b^{*}\cE_{1|(S\setminus \{p\})} \\
s_1\otimes s_2 & \mapsto & g(s_1)s_2
\end{matrix}
\end{equation*}
and hence is surjective. Next we examine $\Phi^{+}_g$ over $R_p$. The exact sequence in~\eqref{yakushima} gives that the restriction of $\Phi^{+}_g$ to $R_p$ factors through the quotient    $\cE_1[2]^{+}_{|R_P}\to \Sym^2\cE_1(p)\otimes  \cO_{R_p}$. Abusing notation we may write
\begin{equation}
\begin{matrix}
\Sym^2\cE_1(p)\otimes \cO_{R_p} & \xrightarrow{\Phi^{+}_{g|R_p}} & \cE_1(p)\otimes \cO_{R_p} \\
s_1\otimes s_2+s_2\otimes s_1 & \mapsto & 2(g(s_1(p))s_2(p)+g(s_2(p))s_1(p))
\end{matrix}
\end{equation}
It follows that $\Phi_g^{+}$ is surjective over $R_p$.
\end{proof}
By the BKR correspondence and simplicity of $\cE_1$ we also have an isomorphism 
\begin{equation}\label{lovenicol}
\begin{matrix}
\cE_1(p)^{\vee} & \overset{\sim}{\lra} & \Hom_{S^{[2]}}(\cG(\cE_1,\cE_1),\cG(\cE_1,\CC_P)) \\
g & \mapsto & \Phi_g
\end{matrix}
\end{equation}
A straightforward argument gives the following.
\begin{clm}\label{clm:raddoppia}
Let $g\in\cE_1(p)^{\vee}$. Let $i\colon \cE_1[2]^{+}\to \cG(\cE_1,\cE_1)$ be the inclusion given by the direct-sum decomposition
 in~\eqref{aaspezza}.
The composition 
\begin{equation}
\cE_1[2]^{+}\xrightarrow{i} \cG(\cE_1,\cE_1) \xrightarrow{\Phi_g} \cG(\cE_1,\CC_P)
\end{equation}
is equal to $2\Phi^{+}_g$. 
\end{clm}
\begin{proof}[Proof of Proposition~\ref{prp:fipiusur}]
By Lemma~\ref{lmm:stabse} the morphism $f$ in~\eqref{nelduale} defines an injection $V_a\hra \cE_1(p)^{\vee}$. Since $V_a$ is non zero the proposition follows from Claim~\ref{clm:raddoppia} and Proposition~\ref{prp:pioggia}.
\end{proof}
\subsection{Cohomology computations, and extensions}\label{subsec:cohomcompandext}
\setcounter{equation}{0}  
Throughout the present subsection Hypothesis-Definition~\ref{hyp-dfn:esseaemme} holds, and  the polarization $h_S$ is 
${\mathsf a}(v_2)$-generic. 

\subsubsection{Cohomology computations and  extensions, I}\label{subsubsec:cohomcompuno}
The BKR equivalence gives the following isomorphisms. First
\begin{equation}\label{dunant}
\Ext^p_{S^{[2]}}(\cE_1[2]^{\pm},\cE_1[2]^{\pm})\overset{\BKR}{\overset{\sim}{\lra}}\Ext^p_S(\cE_1,\cE_1),
\end{equation}
where $\pm$ means that either both are $+$ or both are $-$, and secondly
\begin{equation}\label{mantegazza}
\Ext^p_{S^{[2]}}(\cE_1[2]^{+},\cE_1[2]^{-})\overset{\BKR}{\overset{\sim}{\lra}}
\begin{cases}
\Ext^2_S(\cE_1,\cE_1) & \text{if $p=2$,}\\
0 & \text{otherwise,} 
\end{cases}
\end{equation}
\begin{lmm}\label{lmm:mattiatorre}
Let $0\le k<a$ and let $[\cH]\in(\cM_{v(a-k)}\setminus \cD_{v(a-k)})$. Then the BKR equivalence defines isomorphisms
\begin{equation}\label{extgirigido}
\Ext^p_{S^{[2]}}(\cG(\cE_1,\cH),\cE_1[2]^{\pm})\overset{\BKR}{\overset{\sim}{\lra}}
\begin{cases}
\Ext^1_S(\cH,\cE_1) & \text{if $p\in\{1,3\}$},\\
0 & \text{otherwise.} 
\end{cases}
\end{equation}
and
\begin{equation}\label{extrigidogi}
\Ext^p_{S^{[2]}}(\cE_1[2]^{\pm}, \cG(\cE_1,\cH))\overset{\BKR}{\overset{\sim}{\lra}}
\begin{cases}
\Ext^1_S(\cE_1,\cH) & \text{if $p\in\{1,3\}$},\\
0 & \text{otherwise.} 
\end{cases}
\end{equation}
\end{lmm}
\begin{proof}
The isomorphisms in~\eqref{extgirigido} follow from  the vanishings 
$\Hom_S(\cH,\cE_1)=\Ext^2_S(\cH,\cE_1)=0$. The first group vanishes because $[\cH]\notin \cD_{v(a-k)}$.
The second group vanishes because 
$\Ext^2_S(\cH,\cE_1)\cong\Hom_S(\cE_1,\cH)^{\vee}$ (Serre duality) and because 
 $\Hom_S(\cE_1,\cH)=0$  by semistability of $\cH$.  The isomorphisms in~\eqref{extrigidogi} are obtained by a similar argument, or 
from the isomorphisms in~\eqref{extgirigido} and Serre duality.
\end{proof}
\begin{dfn}\label{dfn:fasciobical}
Assume that $1\le k<a$ and $[\cH]\in(\cM_{v(a-k)}\setminus \cD_{v(a-k)})$ (hence $\ext^1_S(\cE_1,\cH)=2k$ by Remark~\ref{rmk:stratchiusa}). 
If $U_k\subset \Ext^1_S(\cE_1,\cH)$ is a $k$ dimensional subspace, we let $\cB(\cH,U_k)^{+}$ be the (locally free) sheaf on $S^{[2]}$ fitting into the exact  sequence
\begin{equation}\label{bicalsucc}
0\lra  \cG(\cE_1,\cH)\lra \cB(\cH,U_k)^{+}\overset{\lambda^{+}}{\lra} \cE_1[2]^{+}\otimes U_k\lra 0
\end{equation}
with extension class in $\Ext^1_{S^{[2]}}(\cE_1[2]^{+}, \cG(\cE_1,\cH))\otimes U_k^{\vee}=
\Ext^1_S(\cE_1,\cH )\otimes U_k^{\vee}$ given by the inclusion $U_k\hra \Ext^1_S(\cE_1,\cH)$.
\end{dfn}
\begin{prp}\label{prp:extrigidibical}
Let $\cH$, $U_k$ and $\cB(\cH,U_k)^{+}$ be as in Definition~\ref{dfn:fasciobical}.  Let $*$ be either $+$ or $-$. Then the BKR equivalence defines isomorphisms
\begin{equation*}
\Ext^p_{S^{[2]}}(\cB(\cH,U_k)^{+},\cE_1[2]^{*})\overset{\BKR}{\overset{\sim}{\lra}}
\begin{cases}
U_k^{\vee}  & \text{if $p=0$ and $*=+$},\\
\Ann(U_k)  & \text{if $p=1$, or $p=3$ and $*=+$},\\
\Ext^1_S(\cH,\cE_1)  & \text{if $p=3$ and $*=-$},\\
0 & \text{otherwise.} 
\end{cases}
\end{equation*}
where $ \Ann(U_k)\subset \Ext^1_S(\cH,\cE_1)=\Ext^1_S(\cE_1,\cH)^{\vee}$ 
 is the annihilator of $U_k$.
\end{prp}
\begin{proof}
To simplify notation we omit the subscript $S^{[2]}$ from the $\Ext$-groups, and we let $\cG=\cG(\cE_1,\cH)$, $\cB^{+}=\cB(\cH,U_k)^{+}$. Applying the functor $\Hom(-,\cE_1[2]^{*})$ to the exact sequence in~\eqref{bicalsucc} we get the long exact sequence
\begin{multline}\label{serpentone}
\ldots\overset{\partial_{p-1}^{*}}{\lra}\Ext^p(\cE_1[2]^{+},\cE_1[2]^{*})\otimes U^{\vee}_k\lra \Ext^p(\cB^{+},\cE_1[2]^{*})\lra \\
\lra\Ext^p(\cG,\cE_1[2]^{*})\overset{\partial_p^{*}}{\lra} 
\Ext^{p+1}(\cE_1[2]^{+},\cE_1[2]^{*})\otimes U_k^{\vee}\lra\ldots
\end{multline}
We claim that the coboundary maps 
\begin{equation*}
\Ext^1_S(\cH,\cE_1)=\Ext^1(\cG,\cE_1[2]^{*})\overset{\partial_1^{*}}{\lra} 
\Ext^{2}(\cE_1[2]^{+},\cE_1[2]^{*})\otimes U_k^{\vee}=U_k^{\vee}
\end{equation*}
and
\begin{equation*}
\Ext^1_S(\cH,\cE_1)=\Ext^3(\cG,\cE_1[2]^{+})\overset{\partial_3^{+}}{\lra} 
\Ext^{4}(\cE_1[2]^{+},\cE_1[2]^{+})\otimes U_k^{\vee}=U_k^{\vee}
\end{equation*}
are identified with the transpose of the inclusion $U_k\hra \Ext^1_S(\cE_1,\cH)$. In fact the coboundary maps are given by Yoneda product with the extension class of~\eqref{bicalsucc}, and  our assertion follows because Yoneda products match under the BKR equivalence. 
The Proposition is a straightforward consequence of the long exact sequence in~\eqref{serpentone}, equations~\eqref{dunant}, 
\eqref{mantegazza}, Lemma~\ref{lmm:mattiatorre}, and the identification of the coboundaries
 $\partial_1^{*}$, $\partial_3^{+}$ given above. 
\end{proof}
\begin{dfn}\label{dfn:fasciobicaltilde}
If $\cH$, $U_k$ and $\cB(\cH,U_k)^{+}$ are as in Definition~\ref{dfn:fasciobical}, 
 $\cB(\cH,U_k)$ is the (locally free) sheaf on $S^{[2]}$ fitting into the exact  sequence
\begin{equation}\label{bicaltildesucc}
0\lra  \cE_1[2]^{-}\otimes \Ann(U_k)^{\vee}\lra \cB(\cH,U_k)\overset{\lambda}{\lra} \cB(\cH,U_k)^{+}\lra  0
\end{equation}
with extension class (by Proposition~\ref{prp:extrigidibical}) in $\Ann(U_k)\otimes \Ann(U_k)^{\vee}$ 
given by 
the identity map.
\end{dfn}
\begin{clm}\label{clm:stessofalsomukai}
Let $1\le k<a$,   
$[\cH]\in(\cM_{v(a-k)}\setminus \cD_{v(a-k)})$, and $[\cE_2]\in\cM_{v_2}$.
Then
\begin{equation}\label{falsomukai}
\ch(\cB(\cH,U_k))=\ch(\cG(\cE_1,\cE_2)).
\end{equation}
\end{clm}
\begin{proof}
We start by noting that $\ch(\cG(\cE_1,\cE_2))$ is independent of which  $[\cE_2]\in\cM_{v_2}$ we choose.
By Proposition~\ref{prp:interbarca}
there exists $[\cE_2]\in\cD^k_{v_2}(S,h_S)$ such that $\cE_2$ fits into the exact sequence~\eqref{sinner} ($\cH$  
is as in the claim). Then both
$\cG(\cE_1,\cE_2)$ and $\cB(\cH,U_k)$  have filtrations  with associate graded pieces $\cG(\cE_1,\cH)$, $\cE_1[2]^{\pm}\otimes\CC^k$. In fact for $\cB(\cH,U_k)$ this holds by construction, while for $\cG(\cE_1,\cE_2)$ it was shown in Subsection~\ref{subsec:semirepl}. 
The equality in~\eqref{falsomukai} follows. 
\end{proof}
The vector bundle $\cB(\cH,U_k)^{+}$ is a quotient of $\cB(\cH,U_k)$. An analogous subsheaf  of $\cB(\cH,U_k)$ is also relevant for what follows.
\begin{dfn}\label{dfn:fasciobimeno}
Assume that $1\le k<a$, $[\cH]\in(\cM_{v(a-k)}\setminus \cD_{v(a-k)})$ and  $U_k\subset \Ext^1_S(\cE_1,\cH)$ is a $k$ dimensional subspace. We let  
$\cB(\cH,U_k)^{-}\subset\cB(\cH,U_k)$ be defined by 
\begin{equation*}
\cB(\cH,U_k)^{-}\coloneq \ker(\lambda^{+}\circ\lambda),
\end{equation*}
 where $\lambda^{+},\lambda$ are the morphisms in~\eqref{bicalsucc} and~\eqref{bicaltildesucc} respectively.
\end{dfn}
From the exact sequences in~\eqref{bicalsucc} and~\eqref{bicaltildesucc} we get that $\cB(\cH,U_k)^{-}$ sits in the exact sequence
\begin{equation}\label{alemarc}
0\lra  \cE_1[2]^{-}\otimes \Ann(U_k)^{\vee}\lra \cB(\cH,U_k)^{-}\overset{\lambda^{-}}{\lra} \cG(\cE_1,\cH)\lra  0
\end{equation}
By~\eqref{extgirigido} the extension class of the above exact sequence  is an element of 
\begin{equation}\label{alienor}
\Ext^1_S(\cH,\cE_1)\otimes \Ann(U_k)^{\vee}.
\end{equation}
\begin{clm}\label{clm:manetti}
The extension class   of~\eqref{alemarc}, which is an element of the group in~\eqref{alienor},  is given by the inclusion 
$\Ann(U_k)\hra \Ext^1_S(\cH,\cE_1)$. 
\end{clm}
\begin{proof}
Let $e$ be the extension class of the exact sequence in~\eqref{bicaltildesucc}, i.e.~the identity 
$\Ann(U_k)\xrightarrow{\Id} \Ann(U_k)$. The extension class   of~\eqref{alemarc} is the image of $e$ via the map
\begin{equation*}
\Ext^1_{S^{[2]}}(\cB^{+},\cE_1[2]^{-})\otimes \Ann(U_k)^{\vee}\lra \Ext^1_{S^{[2]}}(\cG,\cE_1[2]^{-})\otimes\Ann(U_k)^{\vee}.
\end{equation*}
(We let $\cB^{+}=\cB(\cH,U_k)^{+}$ and $\cG=\cG(\cE_1,\cH)$.)
The map 
\begin{equation*}
\Ext^1_{S^{[2]}}(\cB^{+},\cE_1[2]^{-})\lra \Ext^1_{S^{[2]}}(\cG,\cE_1[2]^{-})\overset{\BKR}{\overset{\sim}{\lra}}\Ext^1_S(\cE_1,\cH)
\end{equation*}
 is injective with image 
$\Ann(U_k)$ (see the proof of Proposition~\ref{prp:extrigidibical}). The claim follows.
\end{proof}
The remainder of the present subsubsection contains different 
 constructions of $\cB(\cH,U_k)^{\pm}$.
Let $\cH$ and $U_k$ be  as in Definition~\ref{dfn:fasciobical}.
 Let
\begin{equation}\label{borges}
0\lra \cH\lra \cF\overset{\psi}{\lra} \cE_1\otimes U_k\lra 0 
\end{equation}
be the exact sequence with extension class in $U_k^{\vee}\otimes\Ext^1_S(\cE_1,\cH)$ given by the inclusion map. Proceeding as in Subsubsection~\ref{subsubsec:liberoinstab} we associate to the exact sequence in~\eqref{borges} an exact sequence
\begin{equation}\label{laver}
0\lra \cG(\cE_1,\cH)\lra \cG(\cE_1,\cF)\overset{\Psi_{\cF}}{\lra} \cG(\cE_1,\cE_1)\otimes U_k\lra 0. 
\end{equation}
The decomposition in~\eqref{aaspezza} gives the decomposition $ \Psi_{\cF}=\Psi^{+}_{\cF}\oplus \Psi^{-}_{\cF}$, where 
$\Psi^{\pm}_{\cF}\colon \cG(\cE_1,\cF)\to \cE_1[2]^{\pm}\otimes U_k$. 
\begin{dfn}\label{dfn:apiumeno}
Assume that $\cH$ and $U_k$ are  as in Definition~\ref{dfn:fasciobical}, and $\cF$ is the sheaf  in~\eqref{borges} defined above. We let  
\begin{equation}\label{corelli}
\cA(\cH,U_k)^{\pm}\coloneq\ker \Psi_{\cF}^{\mp}.
\end{equation}
\end{dfn}
The sheaf  
$\cA(\cH,U_k)^{\pm}$ fits into the exact sequence
\begin{equation}\label{bizet}
0\lra \cG(\cE_1,\cH)\lra \cA(\cH,U_k)^{\pm}\lra  \cE_1[2]^{\pm}\otimes U_k\lra 0. 
\end{equation}
\begin{prp}\label{prp:freddo}
The extension class of the exact sequence in~\eqref{bizet}, which by  Lemma~\ref{lmm:mattiatorre} is an element of $U_k^{\vee}\otimes\Ext^1_S(\cE_1,\cH)$, is given by  the inclusion. In particular $\cA(\cH,U_k)^{+}\cong \cB(\cH,U_k)^{+}$. 
\end{prp}
\begin{proof}
We have a direct sum decomposition
\begin{multline}\label{fabian}
U_k^{\vee}\otimes\Ext^1_{S^2}(\cE_1\boxtimes\cE_1\oplus\cE_1\boxtimes\cE_1 ,\cE_1\boxtimes\cH\oplus \cH\boxtimes\cE_1)= \\
=\bigoplus_{i,j\in\{1,2\}} U_k^{\vee}\otimes  \Ext^1(\cE_1,\cH),
\end{multline}
where the indices $i,j$ correspond to first/second addends of the decompositions of the two sheaves. The sheaves above are $\Sigma_2$ equivariant sheaves. 
The non trivial element of $\Sigma_2$ permutes the $(1,1)$, $(2,2)$ addends and the $(1,2)$, $(2,1)$ ones. 
We also have a direct sum decomposition
\begin{equation}
U_k^{\vee}\otimes\Ext^1_{S^2}(\cE_1\boxtimes\cE_1,\cE_1\boxtimes\cH\oplus \cH\boxtimes\cE_1)=
\bigoplus_{l=1}^2 U_k^{\vee}\otimes  \Ext^1(\cE_1,\cH),
\end{equation}
where the index $l$ corresponds to first/second addend of the decomposition of the second sheaf. 
Let
\begin{equation*}
\begin{matrix}
\bigoplus_{i,j\in\{1,2\}} U_k^{\vee}\otimes  \Ext^1(\cE_1,\cH) & \xrightarrow{\Lambda^{\pm}} & 
\bigoplus_{l=1}^2 U_k^{\vee}\otimes  \Ext^1(\cE_1,\cH)   \\
(e_{1,1},e_{1,2},e_{2,1},e_{2,2}) & \mapsto & (e_{1,1}\pm e_{2,1},e_{1,2}\pm e_{2,2}) 
\end{matrix}
\end{equation*}
By the definition of the decomposition in~\eqref{aaspezza} the map
\begin{equation*}
\Ext^1_{S^{[2]}}( \cG(\cE_1,\cE_1),\cG(\cE_1,\cH))\xrightarrow{\Psi^{\pm}_{\cF}}\Ext^1_{S^{[2]}}( \cE_1[2]^{\pm},\cG(\cE_1,\cH))
\end{equation*}
corresponds, via the BKR equivalence, to the restriction of $\Lambda^{\pm}$ to the $\Sigma_2$-invariant subspaces.
Let $\ov{e}\in U_k^{\vee}\otimes  \Ext^1_{S^{[2]}}( \cG(\cE_1,\cE_1),\cG(\cE_1,\cH))$ be the extension class of~\eqref{laver}. By the BKR equivalence 
$\ov{e}$ corresponds to a $\Sigma_2$-invariant  element $e$ of the Ext-group in~\eqref{fabian}. By definition  of~\eqref{laver} we have $e=(f,0,0,f)$ where $f$ is the extension class of the exact sequence in~\eqref{borges}, i.e.~given by the inclusion  $U_k\hra\Ext^1_S(\cE_1,\cH)$. 
Since
\begin{equation}
\Lambda^{+}(f,0,0,f)=(f,f),\qquad \Lambda^{-}(f,0,0,f)=(f,-f)
\end{equation}
the proposition follows.
\end{proof}
Next we give an analogue of Proposition~\ref{prp:freddo} with $\cB(\cH,U_k)^{+}$ replaced by $\cB(\cH,U_k)^{-}$. Of course here $\cH$ and $U_k$ are  as in Definition~\ref{dfn:fasciobical}. Let 
\begin{equation}\label{giovanni}
0\lra \cE_1\otimes  \Ann(U_k)^{\vee}\lra \cF\lra \cH \lra 0
\end{equation}
be the exact sequence with extension class in 
$\Ext^1_S(\cH, \cE_1)\otimes  \Ann(U_k)^{\vee}$ given by the inclusion  $ \Ann(U_k)\hra \Ext^1_S(\cH,\cE_1)$. 
The exact sequence in~\eqref{giovanni} gives rise to the exact sequence
\begin{equation}\label{referendum}
0\lra \cG(\cE_1,\cE_1)\otimes  \Ann(U_k)^{\vee}\lra \cG(\cE_1, \cF)\lra \cG(\cE_1,\cH) \lra 0.
\end{equation}
By the decomposition 
$\cG(\cE_1,\cE_1)= \cE_1[2]^{+}\oplus  \cE_1[2]^{-}$ and the above exact sequence we may view $\cE_1[2]^{\pm}\otimes  \Ann(U_k)^{\vee}$ as a subsheaf of $\cG(\cE_1, \cF)$. 
\begin{dfn}\label{dfn:ciaccau}
Assume that $\cH$, $U_k$ are  as in Definition~\ref{dfn:fasciobical}, and $\cF$ is the vector bundle  in~\eqref{giovanni} defined above. We let 
\begin{equation*}
\cC(\cH,U_k)^{\pm}\coloneq\cG(\cE_1, \cF)/\cE_1[2]^{\mp}\otimes  \Ann(U_k)^{\vee}. 
\end{equation*}
\end{dfn}
The inclusion $ \cE_1[2]^{\mp}\hra \cG(\cE_1,\cE_1)$ gives the exact sequence
\begin{equation}\label{privacy}
0\lra \cE_1[2]^{\pm}\otimes \Ann(U_k)^{\vee}\lra \cC(\cH,U_k)^{\pm} \lra \cG(\cE_1,\cH)\lra 0.
\end{equation}
\begin{prp}\label{prp:sole}
The extension class of the exact sequence in~\eqref{privacy}, which   by 
Lemma~\ref{lmm:mattiatorre} is an element of $\Ann(U_k)^{\vee}\otimes\Ext^1_S(\cH,\cE_1)$, is given by  the inclusion map. In particular 
$\cC(\cH,U_k)^{-}\cong \cB(\cH,U_k)^{-}$. 
\end{prp}
\begin{proof}
The first statement  is proved by an argument analogous to the proof of the first statement of Proposition~\ref{prp:freddo}.
The isomorphism $\cC(\cH,U_k)^{-}\cong \cB(\cH,U_k)^{-}$ follows from the first statement and Claim~\ref{clm:manetti}.
\end{proof}
\subsubsection{Cohomology computations  and  extensions, II}\label{subsubsec:cohomcompdue}
\begin{prp}\label{prp:extgiedue}
Let $p\in S$. Then we have
\begin{equation*}
\Ext^p_{S^{[2]}}(\cG(\cE_1,\CC_{p}),\cE_1[2]^{\pm})\cong
\begin{cases}
\cE_1(p) & \text{if $p\in\{2,4\}$,}\\
0 & \text{otherwise.}
\end{cases}
\end{equation*}
\end{prp}
\begin{proof}
By the BKR correspondence, $\Ext^p_{S^{[2]}}(\cG(\cE_1,\CC_{p}),\cE_1[2]^{\pm})$ is the subspace of $\Sigma_2$ invariants of
\begin{equation}\label{tantecoom}
\bigoplus_{q+r=p}\left(\Ext^q_S(\cE_1,\cE_1)\otimes\Ext^r_S(\CC_{p},\cE_1)
\oplus \Ext^q_S(\CC_{p},\cE_1)\otimes\Ext^r_S(\cE_1,\cE_1)\right).
\end{equation}
The proposition follows because $\cE_1$ is a spherical vector bundle, $\Ext^r_S(\CC_{p_0},\cE_1)$ vanishes except for $r=2$, and
$\Ext^2_S(\CC_{p},\cE_1)\cong\Hom(\cE_1,\CC_p)^{\vee}\cong\cE_1(p)$. 
\end{proof}
\begin{dfn}
Let $p\in S$ and $W_a\subset \cE_1(p)^{\vee}$ be an $a$-dimensional vector subspace. By the isomorphism in~\eqref{ariacondiz} we get 
a morphism 
\begin{equation*}
\cE_1[2]^{+}\otimes W_a \xrightarrow{\Phi_{W_a}^{+}} \cG(\cE_1,\CC_{p}).
\end{equation*}
\end{dfn}
Note that $\Phi_{W_a}^{+}$ is a surjection by Proposition~\ref{prp:pioggia}.
\begin{dfn}\label{dfn:bipipiu}
Given $p\in S$ and $W_a\subset \cE_1(p)^{\vee}$ an $a$-dimensional vector subspace,   $\cB(p,W_a)^{+}$  is  the sheaf fitting into the exact sequence
\begin{equation}\label{bimenosing}
0\lra \cB(p,W_a)^{+} \overset{\theta}{\lra} \cE_1[2]^{+}\otimes W_a\xrightarrow{\Phi_{W_a}^{+}}  \cG(\cE_1,\CC_{p}) \lra 0.
\end{equation}
\end{dfn}
\begin{prp}\label{prp:alglin}
Let $p\in S$ and $W_a\subset \cE_1(p)^{\vee}$ be an $a$-dimensional  subspace. 
\begin{equation*}
\Ext^p_{S^{[2]}}(\cB(p,W_a)^{+},\cE_1[2]^{-})\cong
\begin{cases} 
\Ann W_a  & \text{if $p=1$,}\\
\cE_1(p)  & \text{if $p=3$,}\\
0 & \text{otherwise,}
\end{cases}
\end{equation*}
where $\Ann W_a\subset \cE_1(p)$ is the annihilator of $W_a\subset \cE_1(p)^{\vee}$.
\end{prp}
\begin{proof}
Applying the functor $\Hom(-,\cE_1[2]^{+})$ to the exact sequence in~\eqref{bimenosing} we get the long exact sequence
\begin{multline}\label{rosagovona}
\ldots\lra \Ext^p_{S^{[2]}}(\cB(p,W_a)^{+},\cE_1[2]^{-})\lra \Ext^{p+1}_{S^{[2]}}(\cG(\cE_1,\CC_{p}),\cE_1[2]^{-})\lra \\
\lra
 \Ext^{p+1}_{S^{[2]}}(\cE_1[2]^{+},\cE_1[2]^{-})\otimes W_a^{\vee}\lra \ldots
\end{multline}
By~\eqref{mantegazza} the group $\Ext^q_{S^{[2]}}(\cE_1[2]^{+},\cE_1[2]^{-})$ vanishes unless $q=2$, and in that case it is canonically identified with $\CC$. 
We claim that the homomorphism 
\begin{equation}\label{tarquiniovipera}
\cE_1(p)=\Ext^2(\cG(\cE_1,\CC_{p}),\cE_1[2]^{+})\lra
\Ext^2(\cE_1[2]^{+},\cE_1[2]^{+})\otimes W_a^{\vee}=W_a^{\vee}
\end{equation}
is identified  with the transpose of the inclusion map  $W_a\hra \cE_1(p)^{\vee}$. Before proving the claim we note that  the proposition follows from the claim and the exact sequence in~\eqref{rosagovona}. The BKR correspondence  identifies the morphism $\Phi_{W_a}^{+}$ with the $\Sigma_2$-equivariant morphism
\begin{equation}\label{conteduca}
 \begin{matrix}
 \cE_1\boxtimes\cE_1\otimes W_a & \lra & \cE_1\boxtimes\CC_p\oplus\CC_p\boxtimes\cE_1 \\
 \sigma\boxtimes\tau\otimes g & \mapsto & (\sigma \boxtimes g(\tau),g(\sigma) \boxtimes \tau)
\end{matrix}
\end{equation}
and the homomorphism  in~\eqref{tarquiniovipera}  with the homomorphism
\begin{equation*}
\Ext^2_{S^2}(\cE_1\boxtimes\CC_p\oplus\CC_p\boxtimes\cE_1,(\cE_1\boxtimes\cE_1)^{\rm tw})^{\Sigma_2}\lra
\Ext^2_{S^2}(\cE_1\boxtimes\cE_1\otimes W_a,(\cE_1\boxtimes\cE_1)^{\rm tw})^{\Sigma_2}
\end{equation*}
 induced by the morphism in~\eqref{conteduca} ($(\cE_1\boxtimes\cE_1)^{\rm tw}$ is $\cE_1\boxtimes\cE_1$ with the $\Sigma_2$-action given by the \lq\lq natural\rq\rq\ one multiplied by the sign character). Hence the claim reduces to the statement that, given $g\colon \cE_1\to\CC_p$ (i.e.~$g\in\cE_1(p)^{\vee}$) the induced homomorphism
\begin{equation}
\cE_1(p)=\Ext^2_S(\cE_1,\CC_p)\lra \Ext^2_S(\cE_1,\cE_1)=\CC
\end{equation}
is identified  with $g$. One way to see this is to go through Serre duality.
\end{proof}
\begin{dfn}\label{dfn:replacedia}
Given $p\in S$ and an $a$-dimensional subspace $W_a\subset \cE_1(p)^{\vee}$, $\cB(p,W_a)$  is  the sheaf fitting into the exact sequence
\begin{equation}\label{bicalsing}
0\lra  \cE_1[2]^{-}\otimes (\Ann W_a)^{\vee} \lra \cB(p,W_a) \overset{\lambda}{\lra}  \cB(p,W_a)^{+}  \lra 0
\end{equation}
with extension class the identity map in (see Proposition~\ref{prp:alglin}) 
\begin{equation*}
\Ext^1_{S^{[2]}}(\cB(p,W_a)^{+} ,\cE_1[2]^{-})\otimes (\Ann W_a)^{\vee}=(\Ann W_a)\otimes (\Ann W_a)^{\vee}.
\end{equation*}
\end{dfn}
Let $[\cE_2]\in \cD_{v_2}^a$. Thus we have the exact sequence in~\eqref{nelduale}. The morphism $f$ appearing  in~\eqref{nelduale} is identified with a linear map $f(p)\colon \cE_1(p)\otimes V_a\to \CC$, which viewed as a map $\cE_1(p)\to V^{\vee}_a$  is surjective by Lemma~\ref{lmm:stabse}. It follows that the transpose  $f(p)^t\colon V_a\to \cE_1(p)^{\vee} $  is injective. 
Note that $p$ is the unique singular point of $\cE_2$.
\begin{prp}\label{prp:nicolciacco}
Let $[\cE_2]\in \cD_{v_2}^a$, and let $p$ be the unique singular point of $\cE_2$.
Abusing notation let $V_a\subset  \cE_1(p)^{\vee} $ be the image of $f(p)^t$. Then  $\cB(p,V_a)^{+}$ is isomorphic to the sheaf $\cA(\cE_2)^{+}$  defined  in Subsubsection~\ref{subsubsec:notsemidia}. More precisely we have a commutative diagram
\begin{equation}
\xymatrix{ 
0 \ar[r] & \cB(p,V_a)^{+} \ar[d]\ar[r]^{\theta} & \cE_1[2]^{+}\otimes V_a \ar[d]^{\Id}\ar[r]^{\Phi^{+}_{V_a}} & 
\cG(\cE_1,\CC_{p})\ar[d]^{2\Id} \ar[r] & 0 \\
0 \ar[r] &\cA(\cE_2)^{+}  \ar[r] & \cE_1[2]^{+}\otimes V_a \ar[r]^{\Phi^{+}_{\cE_2}} & 
\cG(\cE_1,\CC_{p}) \ar[r] & 0,
 }
\end{equation}
where the left vertical arrow is an isomorphism.
\end{prp}
\begin{proof}
This follows from   Claim~\ref{clm:raddoppia}.
\end{proof}
\begin{clm}\label{clm:falsodimitri}
Let $p\in S$, $W_a\subset \cE_1(p)^{\vee}$ be an $a$-dimensional vector subspace, and $[\cE_2]\in\cM_{v_2}$.
Then
\begin{equation}\label{falsodimitri}
\ch(\cB(p,W_a))=\ch(\cG(\cE_1,\cE_2)).
\end{equation}
\end{clm}
\begin{proof}
The right hand side of~\eqref{falsodimitri} does not depend on which $[\cE_2]\in\cM_{v_2}$ we choose. Let 
$[\cE_2]\in\cM_{v_2}^a$, with unique singular point $p$. By~\eqref{bicalsing}, \eqref{bimenosing} and~\eqref{massacarrara} we have
\begin{equation*}
\ch(\cB(p,W_a))=a\ch(\cE_1[2]^{-})+a\ch(\cE_1[2]^{+})-\ch(\cG(\cE_1,\CC_p))=\ch(\cG(\cE_1,\cE_2)).
\end{equation*}
\end{proof}
\subsection{Slope stability of extensions}\label{subsec:sonostabili}
\setcounter{equation}{0}  
Throughout the present subsection Hypothesis-Definition~\ref{hyp-dfn:esseaemme} holds, and  the polarization $h_S$ is 
${\mathsf a}(v_2)$-generic. We prove that the  sheafs defined in Subsubsections~\ref{subsubsec:cohomcompuno} and~\ref{subsubsec:cohomcompdue} are slope stable for certain polarizations of $S^{[2]}$.
\subsubsection{Slope stability of $ \cB(\cH,U_k)$}
The main result of the present subsubsection is the following.
\begin{prp}\label{prp:stabcalbi}
Assume Hypothesis-Definition~\ref{hyp-dfn:esseaemme}. 
Let  $ h_{S^{[2]}}$ be a polarization of $S^{[2]}$ which is ${\mathsf a}(\ww(D,a))$-suitable for $\bm{\mu}(h_S)$. Let $\cH$, $U_k$  be as in Definition~\ref{dfn:fasciobical}. 
The sheaf $\cB(\cH,U_k)$ of Definition~\ref{dfn:fasciobicaltilde} is $ h_{S^{[2]}}$ slope stable. 
\end{prp}
Proposition~\ref{prp:stabcalbi} is proved by applying Proposition~\ref{prp:paragonestab}. Hence we start 
by proving  results on slope semistability for 
\begin{equation}\label{ampiopi}
P\coloneq\bm{\mu}(h_S).
\end{equation}
\begin{prp}\label{prp:calbisemist}
Assume Hypothesis-Definition~\ref{hyp-dfn:esseaemme}, and  
let $\cH$, $U_k$  be as in Definition~\ref{dfn:fasciobical}. 
The sheaf $\cB(\cH,U_k)$ of Definition~\ref{dfn:fasciobicaltilde} is $P$ slope (and  HK slope) semistable. 
\end{prp}
\begin{proof}
Since the BBF square of $P$ is positive, $P$  slope (semi)semistability is the same as 
$P$ HK slope (semi)stability, see Remark~\ref{rmk:hkstabpos}. We prove that  $\cB(\cH,U_k)$ is $P$ HK slope (semi)semistable.
Let $\cB=\cB(\cH,U_k)$ and $\cB^{-}=\cB(\cH,U_k)^{-}$.
Let $0=\cB_0\subset\cB_1\subset\cB_2\subset\cB_3=\cB$ be the filtration with 
$\cB_2\coloneq \cB^{-}$, 
$\cB_1\coloneq\cE_1[2]^{-}\otimes \Ann(U_k)^{\vee}$ 
(see~\eqref{alemarc}).
Let $\cS\subset\cB$ be a subsheaf with $0<r(\cS)<r(\cB)$. 
Let  $0\eqcolon\cS_0\subset\cS_1\subset\cS_2\subset\cS_3\coloneq\cS$ be the  filtration induced by the filtration of $\cB$. 
Thus $\cS_2$ is the kernel of the restriction of $\lambda$ to $\cS$, where $\lambda$ is as in~\eqref{bicaltildesucc}, and   
$\cS_1\subset\cB_2=\cB^{-}$ is the kernel of the restriction of $\lambda^{-}$ to $\cS_2$, where $\lambda^{-}$ is as in~\eqref{alemarc}. 
  For $i\in\{1,2,3\}$ let $\cA_i\coloneq \cS_i/\cS_{i-1}$. We have 
\begin{equation}\label{incquozuno}
\cA_1\subset \cE_1[2]^{-}\otimes \Ann(U_k)^{\vee}, \quad \cA_2\subset  \cG(\cE_1,\cH), 
\quad  \cA_3\subset \cE_1[2]^{+}\otimes U_k.
\end{equation}
Recalling~\eqref{piumenopend} and~\eqref{otamendi} we get that
\begin{equation*}
\mu^{HK}_P(\cE_1[2]^{\pm})=\mu^{HK}_P(\cG(\cE_1,\cH))=\mu^{HK}_P(\cB)=(D\cdot h_S)/2a.
\end{equation*}
The vector bundle $\cE_1[2]^{\pm}$ is $P$ slope stable (see Remark~\ref{rmk:rigidistab}), and by Proposition~\ref{prp:gistabperbig} so is $\cG(\cE_1,\cH)$. 
It follows that if $r(\cA_i)\not=0$ then $\mu^{HK}_P(\cA_i)\le (D\cdot h_S)/2a$. Thus
\begin{equation}\label{catenadiseg}
\mu^{HK}_P(\cS)=\sum_{i=1}^3\frac{r(\cA_i)}{r(\cS)}\mu^{HK}_P(\cA_i)\le 
\sum_{i=1}^3\frac{r(\cA_i)}{r(\cS)}\frac{D\cdot h_S}{2a}=\frac{D\cdot h_S}{2a}=\mu^{HK}_P(\cB).
\end{equation}
(If $r(\cA_i)=0$ we set $r(\cA_i)\mu_P(\cA_i)/r(\cS)=0$.)
\end{proof}
Since  $\cB(\cH,U_k)$ is $P$ slope  (and  HK slope) semistable, in order to apply Proposition~\ref{prp:paragonestab} one needs to control 
subsheaves $\cS\subset\cB(\cH,U_k)$  such that 
$\mu_P(\cS)=\mu_P(\cB(\cH,U_k))$.  Referring to the commutative diagram in~\eqref{commiso} (and recalling~\eqref{duedelta}), we have $\pi\circ\tau(\Delta)=D(S)$, where $D(S)\subset S^2$ is the diagonal.
\begin{prp}\label{prp:nonsollevo}
Assume Hypothesis-Definition~\ref{hyp-dfn:esseaemme}, and  
let $\cH$, $U_k$  be as in Definition~\ref{dfn:fasciobical}. 
Let $T^{+}\subset U_k$ be a non zero subspace, and let $\cS\subset \cE_1[2]^{+}\otimes T^{+}$ be a subsheaf such that 
$ \cE_1[2]^{+}\otimes T^{+}/\cS$   is torsion  with support of codimension at least $2$ away from $\Delta$. Then the inclusion 
 $\iota\colon\cA\hra \cE_1[2]^{+}\otimes U_k$ does not lift to a homomorphism $\cS\lra \cB(\cH,U_k)^{+}$ (see the exact sequence in~\eqref{bicalsucc}).
\end{prp}
\begin{proof}
 Let $\rho\colon X(S)\to S^{[2]}$ and $\tau\colon S^{[2]}\to S^2$ be as in~\eqref{commiso}. Let $\cF$ be the  
 vector bundle appearing in~\eqref{borges} (see Definition~\ref{dfn:apiumeno}).
We claim that we have a commutative diagram
\begin{equation*}
\xymatrix{ 
   \rho^{*}\cG(\cE_1,\cH)\ar[d]\ar[r] & \rho^{*}\cB(\cH,U_k)^{+} \ar[d]^{f}\ar[r] & 
\rho^{*}\cE_1[2]^{+}\otimes  U_k\ar[d]  \\
 \tau^{*}(\cE_1\boxtimes\cH\oplus \cH\boxtimes\cE_1)\ar[r] & \tau^{*}(\cE_1\boxtimes\cF\oplus \cF\boxtimes\cE_1)\ar[r] &   \tau^{*}(\cE_1\boxtimes\cE_1)\otimes U_k.
 }
\end{equation*} 
In fact the first and third vertical arrows exist by definition of $\cG(\cE_1,\cH)$ and $\cE_1[2]^{+}$ respectively, and existence of a vertical arrow $f$ making the diagram commute follows from Proposition~\ref{prp:freddo} (recall that $\cA(\cH,U_k)^{+}\subset \cG(\cE_1,\cF)$). 
If there exists  a lift  $\wt{\iota}\colon\cS\lra \cB(\cH,U_k)^{+}$ of $\iota$ then 
$\rho^{*}(\iota)\colon\rho^{*}\cS\lra \rho^{*}\cE_1[2]^{+}\otimes  U_k$
 lifts to  
$\rho^{*}(\wt{\iota})\colon\rho^{*}\cS\lra \rho^{*}\cB(\cH,U_k)^{+}$. 
Pushing forward by $\tau$ we get a commutative diagram
\begin{equation*}
\xymatrix{ 
\cE_1\boxtimes\cH\oplus \cH\boxtimes\cE_1\ar[r] & \cE_1\boxtimes\cF\oplus \cF\boxtimes\cE_1\ar[r] &   \cE_1\boxtimes\cE_1\otimes U_k \\
& \tau_{*}(\rho^{*}\cS)\ar[u]^{\wt{g}}\ar[ru]^{g} & 
 }
\end{equation*} 
By  hypothesis the map $g$ factors as
\begin{equation}
\tau_{*}(\rho^{*}\cS)\hra \cE_1\boxtimes\cE_1\otimes T^{+}\overset{G}{\hra} \cE_1\boxtimes\cE_1\otimes U_k,
\end{equation}
and the quotient $\cE_1\boxtimes\cE_1\otimes T^{+}/\tau_{*}(\rho^{*}\cS)$ has support of codimension at least $2$ because $\tau$ contracts $\rho^{-1}(\Delta)$. The lift $\wt{g}$ of $g$ defines a lift of $G$ 
away from the support of $\cE_1\boxtimes\cE_1\otimes T^{+}/\tau_{*}(\rho^{*}\cS)$, which has codimension at least $2$. It follows that this lift extends to a lift $\wt{G}\colon \cE_1\boxtimes\cE_1\otimes T^{+}\to \cE_1\boxtimes\cF\oplus \cF\boxtimes\cE_1$ of the inclusion $G$. The extension class of the exact sequence
\begin{equation*}
0\lra \cE_1\boxtimes\cH\oplus \cH\boxtimes\cE_1\lra \cE_1\boxtimes\cF\oplus \cF\boxtimes\cE_1\lra   \cE_1\boxtimes\cE_1\otimes U_k\lra 0
\end{equation*}
is given by 
\begin{equation*}
(\ov{e},\ov{e})\in \Ext^1_{S^2}(\cE_1\boxtimes\cE_1,\cE_1\boxtimes\cH\oplus \cH\boxtimes\cE_1)\otimes  U_k^{\vee}
=\bigoplus_{l=1}^2\Ext^1(\cE_1,\cH)\otimes  U_k^{\vee},
\end{equation*}
where $\ov{e}$ is the extension class of~\eqref{borges} - see the proof of Proposition~\ref{prp:freddo}. Since $\ov{e}$ is given by the inclusion  $ \Ann(U_k)\hra \Ext^1_S(\cE_1,\cH)$ there is no such lift. This contradiction shows that $\iota$ does not lift.

\end{proof}
\begin{prp}\label{prp:garfagnana}
Assume Hypothesis-Definition~\ref{hyp-dfn:esseaemme}, and  
let $\cH$, $U_k$  be as in Definition~\ref{dfn:fasciobical}. 
 Let $\cS_2\subset \cB(\cH,U_k)^{-}$ be a subsheaf, and let $\cS_1\coloneq \ker(\lambda^{-}_{|\cS_2})$, where $\lambda^{-}$ is as in~\eqref{alemarc}. If 
\begin{enumerate}
\item
 $\cG(\cE_1,\cH)/\lambda^{-}(\cS_2)$  is torsion  with support of codimension at least $2$ away from $\Delta$, and
\item
there exists a subspace $T^{-}\subset \Ann(U_k)^{\vee}$ such that $\cS_1\subset  \cE_1[2]^{-}\otimes T^{-}$ and $\cE_1[2]^{-}\otimes T^{-}/\cS_1 $   is torsion  with support of codimension at least $2$ away from $\Delta$, 
\end{enumerate}
then $T^{-}=\Ann(U_k)^{\vee}$.
\end{prp}
\begin{proof}
The proof is similar to that of Proposition~\ref{prp:nonsollevo}. Let $\cF$ be the vector bundle appearing in~\eqref{giovanni} (see definition~\ref{dfn:ciaccau}).  Let $\cD\subset \cE_1\boxtimes\cF\oplus \cF\boxtimes\cE_1$ be the image of the \lq\lq diagonal\rq\rq\ embedding 
$\cE_1\boxtimes\cE_1\otimes \Ann(U_k)^{\vee}\hra \cE_1\boxtimes\cF\oplus \cF\boxtimes\cE_1$ induced by the embedding 
$ \cE_1\otimes  \Ann(U_k)^{\vee}\lra \cF$. 
By Proposition~\ref{prp:sole}  there exists a commutative diagram 
 of sheaves on $X(S)$:
\begin{equation*}
\xymatrix{ 
   \rho^{*}\cE_1[2]^{-}\otimes \Ann(U_k)^{\vee} \ar[d]\ar[r] & \rho^{*}\cB(\cH,U_k)^{-} \ar[d]^{f}\ar[r] & 
\rho^{*}\cG(\cE_1,\cH)\ar[d]  \\
 \tau^{*}(\cE_1\boxtimes\cE_1)\otimes \Ann(U_k)^{\vee}\ar[r] & \tau^{*}((\cE_1\boxtimes\cF\oplus \cF
 \boxtimes\cE_1)/\cD)\ar[r] &   \tau^{*}(\cE_1\boxtimes\cH\oplus \cH
 \boxtimes\cE_1),
 }
\end{equation*} 
One finishes the proof arguing as in the proof  of Proposition~\ref{prp:nonsollevo}. 
\end{proof}
Recall that $P\coloneq\bm{\mu}(h_S)$.
\begin{prp}\label{prp:destabdicalbi}
Assume Hypothesis-Definition~\ref{hyp-dfn:esseaemme}, and  
let $\cH$, $U_k$  be as in Definition~\ref{dfn:fasciobical}. 
 If $\cS\subset\cB(\cH,U_k)$ is a sheaf with $0<r(\cS)<r(\cB(\cH,U_k))$ such that 
$\mu_P^{HK}(\cS)=\mu_P^{HK}(\cB(\cH,U_k))$ then one of the following holds:
\begin{enumerate}
\item[(a)]
There exists a non zero subspace $ T^{-}\subset \Ann(U_k)^{\vee}$ such that $\cS\subset \cE_1[2]^{-}\otimes T^{-}$ and 
$(\cE_1[2]^{-}\otimes T^{-})/\cS$ is torsion  with support of codimension at least $2$ away from $\Delta$. 
\item[(b)]
$\cS\subset \cB(\cH,U_k)^{-}$ and the  quotient $\cB(\cH,U_k)^{-}/\cS$ is torsion with support of codimension at least $2$ away from 
$\Delta$. 
\item[(c)]
There exists a  non zero subspace $T^{+}\subsetneq U_k$ such that $\cS\subset \lambda^{-1}(\cE_1[2]^{+}\otimes T^{+})$ 
 and  $ \lambda^{-1}(\cE_1[2]^{+}\otimes T^{+})/\cS$  is torsion  with support of codimension at least $2$ away from $\Delta$. 
\end{enumerate}
\end{prp}
\begin{proof}
We adopt the notation introduced in the proof of Proposition~\ref{prp:calbisemist}. 
We claim that the following hold:
\begin{enumerate}
\item
There exists a subspace $T^{-}\subset \Ann(U_k)^{\vee}$ such that 
$\cS_1\subset \cE_1[2]^{-}\otimes T^{-}$ and the quotient
$(\cE_1[2]^{-}\otimes T^{-})/\cS_1$ is torsion with support of codimension at least $2$ away from $\Delta$. 
\item
If $\cS_2\not=0$ then $\cG(\cE_1,\cH)/\cS_2$ is torsion with support of codimension at least $2$ away from $\Delta$.
\item
There exists a subspace $T^{+}\subset U_k$ such that 
$\cS_3\subset \cE_1[2]^{+}\otimes T^{+}$ and 
$(\cE_1[2]^{+}\otimes T^{+})/\cS_3$ is torsion with support of codimension at least $2$ away from $\Delta$. 
\end{enumerate}
In fact by hypothesis the inequality in~\eqref{catenadiseg} is an equality. 
It follows that for each $i\in\{1,2,3\}$ such that $r(\cS_i)>0$ (i.e.~$\cS_i\not=0$ because $\cS_i$ is a subsheaf of a locally-free sheaf) we have 
$\mu_P^{HK}(\cS_i)=\mu_P^{HK}(\cB_i/\cB_{i-1})$. 
By $P$  HK slope polystability of the sheaves $\cB_i/\cB_{i-1}$   
(recall the inclusions in~\eqref{incquozuno}) one gets Items~(1), (2), (3) above except for the statement on the support of the quotient sheaf. The latter holds because $P^3\cdot D\ge 0$ for every non zero effective divisor $D$ on $S^{[2]}$, with  equality only if $D=\Delta$. 

What is left to prove is the following: if $i\in\{2,3\}$ and $\cS_i\not=0$ then  the cokernel of the inclusion 
$\cS_{i-1}\hra (\cB_i/\cB_{i-1})$  is torsion with support of codimension at least $2$ away from $\Delta$. Suppose that $\cS_3\not=0$. 
If $\cS_2=0$ then the inclusion 
 $\iota\colon\cS_2\hra \cE_1[2]^{+}\otimes U_k$  lifts to a homomorphism $\cS_2\lra \cB(\cH,U_k)^{+}$. This contradicts
 Proposition~\ref{prp:nonsollevo}. Thus $\cS_2\not=0$ and hence $\cG(\cE_1,\cH)/\cS_2$ is torsion with support of codimension at least $2$ away from $\Delta$ by Item~(2) above. If $\cS_2\not=0$ one gets that $\cS_1\not=0$  and  $T^{-}= \Ann(U_k)^{\vee}$ by
  Proposition~\ref{prp:garfagnana}.
\end{proof}
\begin{proof}[Proof of Proposition~\ref{prp:stabcalbi}]
First note that since the BBF squares of $P=\bm{\mu}(h_S)$ and $h_{S^{[2]}}$ are positive, $P$ ($h_{S^{[2]}}$) slope (semi)semistability is the same as 
$P$ ($h_{S^{[2]}}$) HK slope (semi)stability.
Suppose that $\cB=\cB(\cH,U_k)$ is not $h_{S^{[2]}}$ HK slope stable.  Since $\cB$ is modular (see Claim~\ref{clm:stessofalsomukai}),  
$P$ slope semistable
(see Proposition~\ref{prp:calbisemist}), and $ h_{S^{[2]}}$  is ${\mathsf a}(\ww(D,a))$-suitable for $\bm{\mu}(h_S)$, it follows from Proposition~\ref{prp:paragonestab} that there exists a subsheaf  $\cS\subset \cB$, with $0<r(\cS)<r(\cB)$, such that 
\begin{equation}\label{uguaglianza}
\mu_P(\cS)=\mu_P(\cB),\qquad \mu_{h_{S^{[2]}}}(\cS)\ge \mu_{h_{S^{[2]}}}P(\cB).
\end{equation}
By the equality in~\eqref{uguaglianza} one of Items~(a), (b), (c) of Proposition~\ref{prp:destabdicalbi} holds.
One checks that  in each case the inequality in~\eqref{uguaglianza} does not hold. If Item~(a) or~(b) holds, this follows at once from the equalities in~\eqref{piumenopend} and~\eqref{otamendi}. Lastly suppose that  Item~(c) of Proposition~\ref{prp:destabdicalbi} holds. Then we have
\begin{equation}
c_1(\cS)=c_1(\cE_1[2]^{-})\cdot k+c_1(\cG(\cE_1,\cH))+c_1(\cE_1[2]^{+})\cdot \dim T^{+}-m\delta,
\end{equation}
for some $m\ge 0$. Computing one gets that
\begin{equation}\label{caciocavallo}
\frac{c_1(\cS)}{r(\cS)}=\frac{c_1(\cB)}{r(\cB)}-\left(\frac{a(k-\dim T^{+})+m}{r(\cS)}\right)\delta.
\end{equation}
We have $\dim\Ann(U_k)=\dim U_k=k$ (see Definition~\ref{dfn:fasciobical}) and thus $\dim T^{+}<k$. Hence the equality in~\eqref{caciocavallo} contradicts the inequality in~\eqref{uguaglianza}.
\end{proof}
\subsubsection{Stability of certain extensions, II}
The main result of the present subsubsection is the following.
\begin{prp}\label{prp:calbisingstab}
Assume Hypothesis-Definition~\ref{hyp-dfn:esseaemme}. 
Let  $ h_{S^{[2]}}$ be a polarization of $S^{[2]}$ which is ${\mathsf a}(\ww(D,a))$-suitable for $\bm{\mu}(h_S)$. 
Let $p\in S$ and $W_a\subset \cE_1(p)^{\vee}$ be an $a$-dimensional  subspace. 
The sheaf $\cB(p,W_a)$ of Definition~\ref{dfn:replacedia} is $ h_{S^{[2]}}$ slope stable. 
\end{prp}
Before proving  the above result we go through some preliminary results. 
\begin{prp}\label{prp:pisemistabile}
Assume Hypothesis-Definition~\ref{hyp-dfn:esseaemme}, and  let $P=\bm{\mu}(h_S)$. Let $p\in S$ and $W_a\subset \cE_1(p)^{\vee}$ be an $a$-dimensional  subspace. The sheaf $\cB(p,W_a)$  is  $P$ slope (and HK slope) semistable. 
\end{prp}
\begin{proof}
The proof is analogous to the proof of Proposition~\ref{prp:calbisemist}. 
Since the BBF square of $P$ is positive, $P$  slope (semi)semistability is the same as 
$P$ HK slope (semi)semistability. We prove that  $\cB(p,W_a)$ is $P$ HK slope (semi)semistable.
Let $\cS\subset \cB(p,W_a)$ be a non zero subsheaf (i.e.~with non zero rank) such that $\mu_P^{HK}(\cS)=
\mu_P^{HK}(\cB(p,W_a))$.
Let $\cS_1\subset\cS$ be the kernel of the restriction of $\lambda$ to $\cS$, and let $\cS_2=\cS$, $\cS_0=0$. 
  For $i\in\{1,2\}$ let $\cA_i\coloneq \cS_i/\cS_{i-1}$. We have 
\begin{equation}\label{migranti}
\cA_1\subset \cE_1[2]^{-}\otimes (\Ann W_a)^{\vee}, \quad \cA_2\subset   \cB(p,W_a)^{+}.
\end{equation}
The sheaves $\cE_1[2]^{-}\otimes (\Ann W_a)^{\vee}$ and $\cB(p,W_a)^{+}$ are $P$ HK slope polystable , see Remark~\ref{rmk:rigidistab} and, for $\cB(p,W_a)^{+}$, the exact sequence in~\eqref{bimenosing}. 
It follows that if $r(\cA_i)\not=0$ then 
\begin{equation*}
\mu^{HK}_P(\cA_i)\le \mu^{HK}_P(\cE_1[2]^{\pm})=(D\cdot h_S)/2a.
\end{equation*}
Setting $r(\cA_i)\mu_P(\cA_i)/r(\cS)=0$ if 
$r(\cA_i)=0$, we get that 
\begin{equation}\label{gazametro}
\mu^{HK}_P(\cS)=\sum_{i=1}^2\frac{r(\cA_i)}{r(\cS)}\mu^{HK}_P(\cA_i)\le 
\sum_{i=1}^2\frac{r(\cA_i)}{r(\cS)}\frac{D\cdot h_S}{2a}=\frac{D\cdot h_S}{2a}.
\end{equation}
Since $\mu^{HK}_P(\cB(p,W_a))=(D\cdot h_S)/2a$, we are done.
\end{proof}

Let $X$ be a smooth variety and let $j\colon Y\hra X$ be the inclusion of a smooth (pure) codimension $2$ subvariety. Let $\cU,\cV$ be  locally free sheaves on $X$. Let $\cI_Y\subset\cO_X$ be the ideal sheaf of $Y$. We have the map 
\begin{multline}\label{globloc}
\Ext^1_X(\cI_Y\otimes \cU,\cV)\lra H^0(X,\ul{Ext}^1(\cI_Y\otimes \cU,\cV))=H^0(X,\ul{Ext}^2(j_{*}\cO_Y\otimes \cU,\cV))= \\
=j_{*}\det\cN_{Y/X}\otimes\cU^{\vee}\otimes\cV=H^0(X,\ul{Hom}(\cU_{|Y},\cV_{|Y}\otimes j_{*}\det\cN_{Y/X})).
\end{multline}
The following result is an exercise left to the reader.
\begin{lmm}\label{lmm:cosicomesei}
Let 
\begin{equation}\label{nastassia}
0\lra \cV\lra \cF\lra \cI_Y\otimes \cU\lra 0
\end{equation}
be an exact sequence, and let $\varphi\in H^0(X,\ul{Hom}(\cU_{|Y},\cV_{|Y}\otimes j_{*}\det\cN_{Y/X}))$ be the image  of the extension class of~\eqref{nastassia} by
 the map in~\eqref{globloc}. If $\varphi$ is an injection of vector bundles then $\cF$ is locally free.
\end{lmm}
Recall the exact sequences (see Definitions~\ref{dfn:bipipiu} and~\ref{dfn:replacedia})
\begin{equation}\label{corederoma}
0\lra \cB(p,W_a)^{+} \overset{\theta}{\lra} \cE_1[2]^{+}\otimes W_a\xrightarrow{\Phi_{W_a}^{+}}  \cG(\cE_1,\CC_{p}) \lra 0,
\end{equation}
\begin{equation}\label{triathlon}
0\lra  \cE_1[2]^{-}\otimes (\Ann W_a)^{\vee} \lra \cB(p,W_a) \overset{\lambda}{\lra}  \cB(p,W_a)^{+}  \lra 0.
\end{equation}
Apply the functor $\ul{Hom}(\cdot,\cE_1[2]^{-})$ to the exact sequence in~\eqref{corederoma} to get the isomorphism
\begin{equation}
\ul{Ext}^1_{S^{[2]}}(\cB^{+}(p,W_a),\cE_1[2]^{-})\overset{\sim}{\lra} \ul{Ext}^2_{S^{[2]}}(\cG(\cE_1,\CC_{p}),\cE_1[2]^{-}).
\end{equation}
Recall that $\cG(\cE_1,\CC_{p})\cong j_{*}(b^{*}\cE_1)$, where $j\colon S_p\hra S^{[2]}$ is the inclusion map, see~\eqref{supportato}. Since $S_p$ is smooth of codimension $2$ in $S^{[2]}$, and $b^{*}\cE_1$ is locally free, we have an isomorphism
\begin{equation}
 \ul{Ext}^2_{S^{[2]}}(\cG(\cE_1,\CC_{p}),\cE_1[2]^{-})\cong j_{*}(b^{*}\cE_1^{\vee}\otimes\omega_{S_p}\otimes j^{*}\cE_1[2]^{-}).
\end{equation}
Hence the local-to-global spectral sequence abutting to $\Ext^{\bu}_{S^{[2]}}(\cB^{+}(p,W_a),\cE_1[2]^{-})$ gives the exact sequence
\begin{multline}\label{venditti}
0\lra \Ext^1_{S^{[2]}}(\cB^{+}(p,W_a),\cE_1[2]^{-})\overset{\epsilon}{\lra} H^0(S_p, b^{*}\cE_1^{\vee}\otimes\omega_{S_p}\otimes j^{*}\cE_1[2]^{-})
\lra \\
\lra H^2(S^{[2]}, \ul{Hom}(\cE_1[2]^{+},\cE_1[2]^{-}))\otimes W_a^{\vee}.
\end{multline}
(Note that $\ul{Hom}(\cB^{+}(p,W_a),\cE_1[2]^{-})=\ul{Hom}(\cE_1[2]^{+},\cE_1[2]^{-})$.)
\begin{rmk}
We have the isomorphism
\begin{multline*}
\Ext^2_{S^{[2]}}(\cG(\cE_1,\CC_{p}),\cE_1[2]^{-})\xrightarrow{\sim}H^0(S^{[2]},\ul{Ext}^2_{S^{[2]}}(\cG(\cE_1,\CC_{p}),\cE_1[2]^{-}))=\\
=H^0(S_p, b^{*}\cE_1^{\vee}\otimes\omega_{S_p}\otimes j^{*}\cE_1[2]^{-}).
\end{multline*}
because $\ul{Ext}^p_{S^{[2]}}(\cG(\cE_1,\CC_{p}),\cE_1[2]^{-})=0$ for $p\in\{0,1\}$ ($\cG(\cE_1,\CC_{p})$ is supported on $S_p$ which has codimension $2$ in $S^{[2]}$). Via the above identification, the exact sequence in~\eqref{venditti} is identified with the exact sequence in~\eqref{rosagovona} with $p=1$.
\end{rmk}
Similarly to the exact sequence in~\eqref{illusione} we have the exact sequence
\begin{equation}\label{tarasbulba}
0\lra j^{*}\cE_1[2]^{-}\lra \cE_1(p)\otimes b^{*}\cE_1\overset{h}{\lra} \Sym^2\cE_1(p)\otimes i_{p,*}\cO_{R_p}\lra 0,
\end{equation}
where $h$ is defined via the projection of $\cE_1(p)\otimes \cE_1(p)=\Sym^2\cE_1(p)\oplus\bigwedge^2\cE_1(p)$ onto the first direct summand.  
We have the isomorphism
\begin{equation}
\Psi\colon\cE_1(p)\overset{\sim}{\lra} H^0(S_p, b^{*}\cE_1^{\vee}\otimes\omega_{S_p}\otimes j^{*}\cE_1[2]^{-})
\end{equation}
given by the inverse of the composition
\begin{multline*}
H^0(S_p, b^{*}\cE_1^{\vee}\otimes\omega_{S_p}\otimes j^{*}\cE_1[2]^{-})
\xrightarrow{\sim} H^0(S_p, \cE_1(p)\otimes b^{*}(\cE_1^{\vee}\otimes \cE_1)\otimes\cO_{S_p}(R_p)) \xrightarrow{\sim}  \\
\xrightarrow{\sim} \cE_1(p)\otimes H^0(S, \cE_1^{\vee}\otimes \cE_1)=\cE_1(p)\otimes \CC\Id_{\cE_1}=\cE_1(p),
\end{multline*}
where the first isomorphism is obtained from the long exact sequence associated to~\eqref{tarasbulba} and the isomorphism
 $\omega_{S_p}\cong\cO_{S_p}(R_p)$.
\begin{prp}\label{prp:estloclib}
Let $p\in S$ and $W_a\subset \cE_1(p)^{\vee}$ be an $a$-dimensional  subspace. 
Let $q\colon (\Ann W_a)^{\vee}\twoheadrightarrow U$ be a non zero quotient of $(\Ann W_a)^{\vee}$, and let $\cF(p,W_a)_U$ be the sheaf fitting into the exact sequence
\begin{equation}\label{armenia}
0\lra  \cE_1[2]^{-}\otimes U \lra \cF(p,W_a)_U \overset{\lambda_U}{\lra}  \cB(p,W_a)^{+}  \lra 0
\end{equation}
with extension class  in 
(see Proposition~\ref{prp:alglin}) 
\begin{equation}
\Ext^1_{S^{[2]}}(\cB(p,W_a)^{+} ,\cE_1[2]^{-})\otimes U=(\Ann W_a)\otimes U^{\vee}
\end{equation}
given by the transpose $q^t \colon U^{\vee}\hra \Ann W_a$. Then $\cF(p,W_a)_U$ is locally free. 
\end{prp}
\begin{proof}
 Since $\cE_1[2]^{-}$ is locally free, and 
 $\cB(p,W_a)^{+}$ is locally free away from $S_p$, $ \cF(p,W_a)_U$ is locally free away from $S_p$. We must show that 
 $ \cF(p,W_a)_U$ is locally free at each point $x\in S_p$. 
On a sufficiently small affine open subset  $A\subset S^{[2]}$ containing $x$ we have 
\begin{equation}\label{altamarea}
\cB(p,W_a)^{+}\cong \cI_{S_p}\otimes\cU\oplus \cV
\end{equation}
where 
$\cU,\cV$ are locally free sheaves on $A$, and  the restrictions of $\cU$ and $b^{*}\cE_1$ to $A\cap S_p$ are naturally identified.
Because of the isomorphism in~\eqref{altamarea} we may apply Lemma~\ref{lmm:cosicomesei}. Let 
\begin{equation*}
\varphi_U\in H^0(S_p, b^{*}\cE_1^{\vee}\otimes\omega_{S_p}\otimes j^{*}\cE_1[2]^{-}\otimes U)
\end{equation*}
be the image of the extension class of~\eqref{armenia} via the map $\epsilon$ in~\eqref{venditti}  (note that $\det\cN_{S_p/S^{[2]}}\cong \omega_{S_p}$). Then $\varphi_U$ gives a map 
$b^{*}\cE_1\lra j^{*}\cE_1[2]^{-}\otimes\omega_{S_p}\otimes U$: 
by  Lemma~\ref{lmm:cosicomesei} it suffices to prove that  it
is injective (as map of vector bundles) at every point of $A\cap S_p$. Since $U$ is any non zero quotient of $(\Ann W_a)^{\vee}$ we need to show that for any non zero $v\in\cE_1(p)$ the map  
\begin{equation}\label{palais}
\Psi(v)\colon b^{*}\cE_1\lra j^{*}\cE_1[2]^{-}\otimes\omega_{S_p}
\end{equation}
 is an injection of vector bundles.

First we prove that the map in~\eqref{palais} is an injection of vector bundles away from $R_p$. Over $S_p\setminus R_p$ we have 
$j^{*}\cE_1[2]^{-}\cong \cE_1(p)\otimes\cE_{1|(S_p\setminus R_p)}$, and the canonical bundle 
$\omega_{S_p}$ is trivial.  Hence the map in~\eqref{palais} is given by
\begin{equation}
\begin{matrix}
b^{*}\cE_{1|(S_p\setminus R_p)} & \lra & 
\cE_1(p)\otimes b^{*}\cE_{1|(S_p\setminus R_p)}  \\
s & \mapsto & v\otimes s
\end{matrix}
\end{equation}
This is an injection of vector bundles.

Lastly we prove that the map in~\eqref{palais} is an injection of vector bundles on $R_p$. Tensoring the exact sequence in~\eqref{tarasbulba} with 
$\omega_{S_p}\cong \cO_{S_p}(R_p)$ we get the exact sequence
\begin{equation*}
0\to j^{*}\cE_1[2]^{-}\otimes\cO_{S_p}(R_p)\to \cE_1(p)\otimes b^{*}\cE_1\otimes\cO_{S_p}(R_p)\to \Sym^2\cE_1(p)\otimes i_{p,*}\cO_{R_p}(-1)\to 0.
\end{equation*}
Restricting to $R_p$   one gets the exact sequence
\begin{equation*}
0\lra \Sym^2\cE_1(p)\otimes \cO_{R_p}\lra j^{*}\cE_1[2]^{-}\otimes\omega_{S_p}\otimes \cO_{R_p} \lra \bigwedge^2\cE_1(p)\otimes \cO_{R_p}(-1)\lra 0.
\end{equation*}
It follows that the restriction  $\Psi(v)_{\mid R_p}$  factors through a morphism $\ov{\Psi(v)}_{|R_p}$ to 
$\Sym^2\cE_1(p)\otimes \cO_{R_p}$. We claim that
\begin{equation}\label{pensione}
\begin{matrix}
b^{*}\cE_{1|R_p}\cong\cE_1(p)\otimes\cO_{R_p} & \xrightarrow{\ov{\Psi(v)}_{|R_p}} & \Sym^2\cE_1(p)\otimes \cO_{R_p} \\
s & \mapsto & v\cdot s
\end{matrix}
\end{equation}
Before giving the proof of this we note that it shows that the map in~\eqref{palais} is an injection of vector bundles on $R_p$ also at points of $R_p$. Let us prove that~\eqref{pensione} holds. 
Let $\wt{\Psi}(v)\colon b^{*}\cE_1\lra \cE_1(p)\otimes b^{*}\cE_1\otimes\cO_{S_p}(R_p)$ be the morphism obtained by composing $\Psi(v)$ with the inclusion $j^{*}\cE_1[2]^{-}\otimes\cO_{S_p}(R_p)\hra \cE_1(p)\otimes b^{*}\cE_1\otimes\cO_{S_p}(R_p)$. Let $s$ be a local section of $b^{*}\cE_1$, and let $z$ be a local generator of the ideal of $R_p$ in $S_p$. Then
\begin{equation}\label{icm2026}
\wt{\Psi}(v)(s)=v\otimes s=z(v\otimes s+s\otimes v)z^{-1}/2+z(v\otimes s-s\otimes v)z^{-1}/2.
\end{equation}
Viewing $j^{*}\cE_1[2]^{-}\otimes\cO_{S_p}(R_p)$ as a subsheaf of $\cE_1(p)\otimes b^{*}\cE_1\otimes\cO_{S_p}(R_p)$, we have the following:
 $\frac{1}{2}(v\otimes s-s\otimes v)z^{-1}$ is a local section of $j^{*}\cE_1[2]^{-}\otimes\cO_{S_p}(R_p)$, and since in~\eqref{icm2026} it is multiplied by $z$ 
\begin{equation*}
\wt{\Psi}(v)(s)_{|R_p}=\left(z(v\otimes s+s\otimes v)z^{-1}/2\right)_{\mid R_p}.
\end{equation*}
Now $z(v\otimes s+s\otimes v)z^{-1}/2$ is a non zero local section of $j^{*}\cE_1[2]^{-}\otimes\cO_{S_p}(R_p)$ along $R_p$, in fact its restriction to a local section of $R_p$  is equal to $v\cdot s$ in the notation of~\eqref{pensione}. 
\end{proof}
Letting $q=\Id_{(\Ann W_a)^{\vee}}$ one gets the following result.
\begin{crl}\label{crl:localmentelibero}
Let $p\in S$ and $W_a\subset \cE_1(p)^{\vee}$ be an $a$-dimensional  subspace. 
 Then $\cB(p,W_a)$ is locally free. 
\end{crl}
\begin{prp}\label{prp:qualidestab}
Let hypotheses be as in Proposition~\ref{prp:pisemistabile}. Let $\cS\subset \cB(p,W_a)$ be a non zero subsheaf  such that 
$\mu_P^{HK}(\cS)=\mu_P^{HK}(\cB(p,W_a))$. Then one of the following holds.
\begin{enumerate}
\item[(a)]
There exists a subspace $T^{-}\subset (\Ann W_a)^{\vee}$ such that $\cS\subset \cE_1[2]^{-}\otimes T^{-}$ (this makes sense by~\eqref{triathlon}) and $(\cE_1[2]^{-}\otimes T^{-})/\cS$ is torsion with support of codimension at least $2$ away from 
$\Delta$. 
\item[(b)]
There exists a subspace $T^{+}\subset W_a$ such that $\theta(\lambda(\cS))\subset \cE_1[2]^{+}\otimes T^{+}$ and  $(\theta\circ \lambda)^{-1}(\cE_1[2]^{+}\otimes T^{+})/\cS$ is torsion with support of codimension at least $2$ away from 
$\Delta$. 
\end{enumerate}
\end{prp}
\begin{proof}
We adopt the notation of the proof of Proposition~\ref{prp:pisemistabile}.
We claim that the following hold:
\begin{enumerate}
\item
There exists a non zero subspace $T^{-}\subset  (\Ann W_a)^{\vee}$ such that 
$\cA_1\subset \cE_1[2]^{-}\otimes T^{-}$ and 
$(\cE_1[2]^{-}\otimes T^{-})/\cA_1$ is torsion with support of codimension at least $2$ away from $\Delta$. 
\item
There  exists a subspace $T^{+}\subset W_a$ such that 
$\theta(\cA_2)\subset \cE_1[2]^{+}\otimes T^{+}$ and 
$(\cE_1[2]^{+}\otimes T^{+})/\theta(\cA_2)$ is torsion with support of codimension at least $2$ away from $\Delta$. 
\end{enumerate}
The key point is that the inequality in~\eqref{gazametro} is an equality. The rest of the argument is the same as that which was given to prove Items~(1), (2) and~(3)  in the proof of Proposition~\ref{prp:destabdicalbi}. 

It remains to prove that if $T^{+}\not=0$ then $(\theta\circ \lambda)^{-1}(\cE_1[2]^{+}\otimes T^{+})/\cS$ is torsion with support of codimension at least $2$ away from $\Delta$. By Item~(1) above it suffices to show that $T^{-}= (\Ann W_a)^{\vee}$. Suppose the contrary. Let $U\coloneq (\Ann W_a)^{\vee}/T^{-}$ and let $(\Ann W_a)^{\vee}\twoheadrightarrow U$ be the quotient map. Note that $U$ is non zero because by
hypothesis $T^{+}\subsetneq W_a$. The morphism 
$\lambda\colon \cB(p,W_a) \to \cB(p,W_a)^{+}$ defines an inclusion $\ov{\iota}\colon\cA_2\hra \cB(p,W_a)^{+}$ which lifts to an inclusion
 $\iota\colon\cA_2\hra \cF(p,W_a)_U$:
\begin{equation*}
\xymatrix{ 
0 \ar[r] & \cE_1[2]^{-}\otimes U \ar[r] & \cF(p,W_a)_U \ar[r]^{\lambda_U} & \cB(p,W_a)^{+} \ar[r] & 0   \\
&  & \cA_2\ar@{^{(}->}[u]^{\iota}\ar@{^{(}->}[ru]^{\ov{\iota}} & &
 }
\end{equation*} 
Let $T^{+}\subset W_a$ be as in Item~(2) above. \emph{Away from $\Delta$} the inclusion $\cA_2\hra \cE_1[2]^{+}\otimes T^{+}$ is an isomorphism in codimension $1$. By Proposition~\ref{prp:estloclib} the sheaf $\cF(p,W_a)_U$ is locally free and hence by Hartog's Theorem  \emph{away from $\Delta$}  there exists an inclusion
$\bm{\iota}\colon\cE_1[2]^{+}\otimes T^{+}\hra \cF(p,W_a)_U$ which restricts to $\iota$ on $\cA_2$. This implies that 
\emph{away from $\Delta$} we have
\begin{equation}\label{tchou}
\Phi_{W_a}^{+}(\cE_1[2]^{+}\otimes T^{+})=0.
\end{equation}
The map $\cE_1[2]^{+}\otimes T^{+}\to\cG(\cE_1,p)$ given by the restriction of $\Phi_{W_a}^{+}$ is surjective by 
Proposition~\ref{prp:pioggia}. Since the support of $\cG(\cE_1,p)$ is $S_p$ and $S_p$ is not contained in $\Delta$ (in fact $\Delta\cap S_p=R_p$), the equality in~\eqref{tchou} is absurd. This proves that $T^{-}= (\Ann W_a)^{\vee}$.
\end{proof}
\begin{proof}[Proof of Proposition~\ref{prp:calbisingstab}]
It is analogous to the proof of Proposition~\ref{prp:stabcalbi}.
Since the BBF squares of $P=\bm{\mu}(h_S)$ and $h_{S^{[2]}}$ are positive, $P$ ($h_{S^{[2]}}$) slope (semi)stability is the same as 
$P$ ($h_{S^{[2]}}$) HK slope (semi)stability.
Suppose that $\cB=\cB(p,W_a)$ is not $h_{S^{[2]}}$ HK slope stable.  Since $\cB$ is modular (see Claim~\ref{clm:falsodimitri}),  
$P$ HK slope semistable
(by Proposition~\ref{prp:pisemistabile}), and $ h_{S^{[2]}}$  is ${\mathsf a}(\ww(D,a))$-suitable for $\bm{\mu}(h_S)$, Proposition~\ref{prp:paragonestab} gives that there exists a subsheaf  $\cS\subset \cB$, with $0<r(\cS)<r(\cB)$, such that 
\begin{equation}\label{erodegrande}
\mu_P^{HK}(\cS)=\mu_P^{HK}(\cB(p,W_a)),\qquad \mu_{h_{S^{[2]}}}^{HK}(\cS)\ge\mu_{h_{S^{[2]}}}^{HK}(\cB(p,W_a)).
\end{equation}
 By the equality in~\eqref{uguaglianza} one of Items~(a), (b),  of Proposition~\ref{prp:qualidestab} holds.
In each case the inequality in~\eqref{erodegrande} does not hold. 
\end{proof}
\section{Extension of the  map $\cM_{v_2}\dra M_{\ww(D,a)}$.}\label{sec:estendo}
\subsection{Main result}\label{subsec:spalmo}
\setcounter{equation}{0}
Throughout the section we assume  Hypothesis-Definition~\ref{hyp-dfn:esseaemme}, and we suppose that    $S,h_S,h_{S^{[2]}}$ are as in 
the statement of Proposition~\ref{prp:stabuno}.
We often set 
$\cM=\cM_{v_2}(S,h_S)$, $\ww=\ww(D,a)$,  $M_{\ww}^{\bu} =M_{\ww}(S^{[2]},h_{S^{[2]}})^{\bu} $ (see Corollary~\ref{crl:stabuno} for the definition of the latter), $\cD^k=\cD^k_{v_2}(S,h_S)$, $\cD=\cD^1$. 
We have defined  (see  Corollary~\ref{crl:stabuno}) the  regular and birational map 
\begin{equation}\label{viscontioccupato}
\begin{matrix}
\cM\setminus \cD & \overset{\psi}{\lra} & M_{\ww}^{\bu} \\
[\cE_2] & \mapsto & [\cG(\cE_1,\cE_2)]
\end{matrix}
\end{equation}
If $[\cE_2]\in\cD$  the sheaf $\cG(\cE_1,\cE_2)$ is unstable (see Propositions~\ref{prp:liberoinstab} and~\ref{prp:singolareinstab}), and in addition not locally free if 
$[\cE_2]\in\cD^a$. Nonetheless the map $\psi$ extends to a regular map $\ov{\psi}$ which is an isomorphism.  
  In the present section we prove this result, and we show that $M_{\ww}^{\bu}$ is a connected (projective!) component   of 
  $M_{\ww} =M_{\ww}(S^{[2]},h_{S^{[2]}})$.

In order to describe the restriction of $\ov{\psi}$ to $\cD$ we recall a few results.  Let $[\cE_2]\in (\cD^k\setminus\cD^{k+1})$ where $a>k\ge 1$. By Proposition~\ref{prp:pendbrill}  $\cE_2$ fits into the exact sequence 
\begin{equation}\label{ravel}
0\lra\cH\lra \cE_2\lra\cE_1\otimes \Hom_S(\cE_2,\cE_1)^{\vee} \lra 0,
\end{equation}
where $\cH$ is an $h_S$ slope stable vector with 
$[\cH]\in(\cM_{v(a-k)}(S,h_S)\setminus \cD_{v(a-k)}(S,h_S))$. Applying the functor $\Hom_S(\cE_1,-)$ to the exact sequence in~\eqref{ravel} we get the inclusion (see Remark~\ref{rmk:arancia}) $j_{\cE_2}\colon  \Hom_S(\cE_2,\cE_1)^{\vee} \hra \Ext^1_S(\cE_1,\cH)$. Let 
\begin{equation}\label{geiedue}
J_{\cE_2}\coloneq j_{\cE_2}(\Hom_S(\cE_2,\cE_1)^{\vee})\subset \Ext^1_S(\cE_1,\cH).
\end{equation}
Next let $[\cE_2]\in \cD^a$. Then 
 $\cD^a=\cB_{v_2}(S,h_S)$
(see Proposition~\ref{prp:divuti}), and by  
 Proposition~\ref{prp:bordodiemme} $\cE_2^{\vee}\cong\cE_1\otimes V_a$ where $V_a$ is an $a$-dimensional (complex) vector space. Thus we have a unique exact sequence 
\begin{equation}\label{fabrizio}
0\lra\cE_2\lra \cE_1\otimes V_a\overset{f}{\lra}\CC_p \lra 0.
\end{equation}
The map $f$  defines a linear map 
\begin{equation}\label{mappaeffepi}
f(p)\colon \cE_1(p)\to V_a^{\vee}
\end{equation}
By 
  Lemma~\ref{lmm:stabse}  $f(p)$  is surjective, and hence $\im(f(p)^t)\subset \cE_1(p)^{\vee}$ has dimension $a$. Thus
\begin{equation}\label{deandre}
\im(f(p)^t)\in\Gr(a,\cE_1(p)^{\vee}).
\end{equation}
\begin{thm}\label{thm:dadueauno}
Assume Hypothesis-Definition~\ref{hyp-dfn:esseaemme}. 
Let  $\ww\coloneq\ww(D,a)$, and let $ h_{S^{[2]}}$ be a polarization of $S^{[2]}$ which is ${\mathsf a}(\ww)$-suitable for $\bm{\mu}(h_S)$. 
\begin{enumerate}
\item[(a)]
The  birational map in~\eqref{viscontioccupato} extends to a regular map
\begin{equation}\label{razreg}
\cM_{v_2}(S,h_S)\overset{\ov{\psi}}{\lra} M_{\ww}(S^{[2]},h_{S^{[2]}})^{\bullet}.
\end{equation} 
\item[(b)]
Let $[\cE_2]\in \cD^k(S,h_S)\setminus\cD^{k+1}(S,h_S)$ where $a\ge k\ge 1$, and let 
$[\cF]=\ov{\psi}([\cE_2])$. If $a>k\ge 1$ then $\cF\cong\cB(\cH,J_{\cE_2})$ where $\cH$ 
is as 
in~\eqref{ravel}, and
$J_{\cE_2}$ is as in~\eqref{geiedue}.  If $k=a$ then  $\cF\cong\cB(p,\im(f(p)^t))$ where $p$, $\im(f(p)^t)$ are as 
in~\eqref{fabrizio} and~\eqref{deandre}
 respectively.  

\item[(c)]
The map $\ov{\psi}$ is bijective.
\end{enumerate}
\end{thm}
Theorem~\ref{thm:dadueauno} is proved in Subsection~\ref{subsec:alfinlaprova}. Let
$\wt{M}_{\ww}(S^{[2]},h_{S^{[2]}})^{\bullet}\to M_{\ww}(S^{[2]},h_{S^{[2]}})^{\bullet}$ be the normalization map.
Below is a consequence of Theorem~\ref{thm:dadueauno}.
\begin{crl}\label{crl:isomnorm}
The map $\ov{\psi}$ in~\eqref{razreg} lifts to an isomorphism 
\begin{equation}
\wt{\psi}\colon\cM_{v_2}(S,h_S)\xrightarrow{\sim} 
\wt{M}_{\ww}(S^{[2]},h_{S^{[2]}})^{\bullet}.
\end{equation}
\end{crl}
\begin{rmk}
Computations suggest that  $M_{\ww}(S^{[2]},h_{S^{[2]}})^{\bullet}$ is smooth only if $a= 1$. 
\end{rmk}
\subsection{Road map}\label{subsec:stazfori}
\setcounter{equation}{0}  
 Let $[\cE_2]\in\cD^k\setminus \cD^{k+1}$ for $1\le k\le a$, let
 $f\colon (T,0)\to (\cM,[\cE_2])$ be a map of pointed varieties, where $T$ is a smooth curve, and let $f^{*}\GG$ be the pull-back of the  $\cM$-flat) sheaf on $S^{[2]}\times\cM$ with fiber $\cG(\cE_1,\cE'_2)$ on $[\cE'_2]\in\cM$ (such a sheaf always exists locally $\cM$ in the classical topology). 
The desemitabilizing sequence for $\cG(\cE_1,\cE'_2)$ given in~\eqref{subsec:semirepl} dictates an elementary modification $\wt{f^{*}\GG}$ of $f^{*}\GG$ which is isomorphic to $f^{*}\GG$ away from the central fiber $S^{[2]}\times\{0\}$.
 The main result of Subsection~\ref{subsec:ridsucurva} is the following: if the image of the differential $df(0)$ does not belong to the tangent cone to $\cD$ at $[\cE_2]$ then the restriction of $\wt{f^{*}\GG}$  to  $S^{[2]}\times\{0\}$ is slope stable, and its isomorphism class depends on $[\cE_2]$ but \emph{not} on $f$.

As was explained in Subsection~\ref{subsec:panorama}, in order to prove Theorem~\ref{thm:dadueauno} one needs to consider the blow-up $\wh{\cM}$ of  $\cM$ defined in Subsection~\ref{subsec:scoppiascoppia} and perform elementary modifications. The picture one gets is iterative. In order to carry it out we need to discuss certain extensions of sheaves and relate them to spaces of complete collineations: this is the subject of  Subsection~\ref{subsec:primomaggio}. 

In Subsection~\ref{subsec:estdipsi} we prove (Proposition~\ref{prp:estendopsi}) that $\psi$ extends to a regular map $\wh{\psi}\colon \wh{\cM}\to M_{\ww}^{\bullet}$ and we identify the values of $\wh{\psi}$ on points of the exceptional set $\wh{E}^a\cup \ldots\cup\wh{E}^1$. The proof  is a monstruous semistable reduction. Theorem~\ref{thm:dadueauno} follows easily from Proposition~\ref{prp:estendopsi}.

\subsection{First step of semistable reduction along a curve}\label{subsec:ridsucurva}
\setcounter{equation}{0}  
\subsubsection{Preliminaries: elementary modifications and extension classes}\label{subsubsec:elemod}
We collect some results that are useful in proving  Items~(a) and~(b) of Theorem~\ref{thm:dadueauno}. Let $X,T$ be varieties, and $D\subset T$ be an effective divisor (a closed subscheme with locally principal ideal sheaf). 
Let $\iota\colon X\times D\hra X\times T$ be the inclusion morphism. Assume that $\cF$ is a sheaf on $X\times T$, and that we are given an exact sequence of sheaves on $X\times D$:
\begin{equation}\label{hasbagliato}
0\lra \cA\overset{\lambda}{\lra} \cF_{|X\times D}\overset{\phi}{\lra} \cB\lra 0.
\end{equation}
Then $\phi$ defines a morphism $\wt{\phi}\colon \cF\to\iota_{*}\cB$. Let $\cH\coloneq\ker\wt{\phi}$. Hence we have the following exact sequence  of sheaves on $X\times T$:
\begin{equation}\label{accaeffebi}
0\lra \cH\overset{\alpha}{\lra} \cF\overset{\wt{\phi}}{\lra} \iota_{*}\cB\lra 0.
\end{equation}
\begin{lmm}\label{lmm:modificapiatta}
Keep notation as above. Assume that $\cF$ is  $T$-flat and that $\cB$ is locally free. 
Then $\cH$ is  $T$-flat.
\end{lmm}
\begin{proof}
We must show that $Tor_1^{\cO_T}(\cH,\cS)=0$
 for every sheaf of $\cO_{T}$-modules $\cS$. By  the exact sequence in~\eqref{accaeffebi} and $T$-flatness of $\cF$ it suffices to show that 
$Tor_2^{\cO_T}(\iota_{*}\cB,\cS)=0$ for every $\cS$ as above. 
Since $\cB$ is  locally free it suffices to prove that $Tor_2^{\cO_T}(\iota_{*}\cO_{X\times D},\cS)=0$. The latter holds because, since $D$ is Cartier,  
$\iota_{*}\cO_{X\times D}$ has a two-step locally-free resolution.
\end{proof}
Now assume that $T$ is a smooth curve and $D=\{0\}\subset T$. Let $X_0\coloneq X\times\{0\}$, 
$\cF_0\coloneq \cF_{|X_0}$, $\cH_0\coloneq \cH_{|X_0}$. 
Tensoring the exact sequence in~\eqref{accaeffebi} with $\cO_{X_0}$ we get the exact sequence
\begin{equation}\label{biaccaa}
0\lra \cB\lra \cH_0\lra\cA\lra 0.
\end{equation}
One describes the extension class of~\eqref{biaccaa} as follows.  Let 
\begin{equation}
\Theta_T(0)\overset{\kappa}{\lra} \Ext^1_X(\cF_0,\cF_0)
\end{equation}
be the Kodaira-Spencer map at $0$, and let
\begin{equation}
\begin{matrix}
\Ext^1_X(\cF_0,\cF_0) & \overset{\mu}{\lra} & \Ext^1_X(\cA_0,\cB_0) \\
e  & \mapsto & \lambda\cup e\cup \phi
\end{matrix}
\end{equation}
where $\lambda,\phi$ are the maps in~\eqref{hasbagliato}.
\begin{lmm}\label{lmm:estmod}
Assume that $T$ is a smooth curve and $D=\{0\}\subset T$. The extension class 
of~\eqref{biaccaa} is (up to $\CC^{*}$) equal to $\mu\circ \kappa(v)$ where $v$ is a generator of 
$ \Theta_T(0) $. 
\end{lmm}
\begin{proof}
Let  $Y\coloneq\Spec\CC[t]/(t^2)$. The non zero $v\in  \Theta_T(0)$ corresponds to an embedding $j\colon Y\hra T$ mapping the closed point to $0$, and 
$\kappa(v)$ is the extension class of the exact sequence 
\begin{equation}\label{effeeffe}
0\lra \cF_0\overset{\cdot \ov{t}}{\lra} j^{*}\cF\lra \cF_0\lra 0
\end{equation}
(ve view it as a sequence of sheaves on $X_0$). 
Let $\alpha$ be the morphism in~\eqref{accaeffebi}.
The map $j^{*}(\alpha)\colon j^{*}\cH\to j^{*}\cF$ has image equal to 
$\lambda(\cA)$, the subsheaf $\ov{t}\cH\subset j^{*}\cH$ is mapped to $\ov{t}\lambda(\cA)$. It follows that we have an  isomorphism between the extension in~\eqref{biaccaa} with the extension obtained from~\eqref{effeeffe} after pull-back via 
$\lambda\colon\cA\to\cF_0$ and push-out via $\phi\colon\cF_0\to\cB$. This proves the lemma.
\end{proof}
\subsubsection{First step of semistable reduction along a curve, I}\label{subsubsec:ridsemuno}
Let $a>k\ge 1$ and let $[\cE_2]\in(\cD^k\setminus\cD^{k+1})$. Let $(T,0)$ be a smooth pointed curve and let 
\begin{equation}\label{curvainemme}
f\colon (T,0)\to(\cM,[\cE_2])
\end{equation}
 be a (regular) map of pointed varieties. We assume that $f^{-1}(\cD)=\{0\}$ set-theoretically. Let $f^{*}\GG$ be the sheaf on $S^{[2]}\times T$ given by $(\Id_{S^{[2]}},f)^{*}\GG$. Then  $f^{*}\GG$  is $T$-flat because $\GG$ is $\cM$-flat. For $t\in T$ let $(f^{*}\GG)_t$ be the restriction of 
$f^{*}\GG$  to $S^{[2]}\times \{t\}$: it is isomorphic to $\cG(\cE_1,\cF)$ where $[\cF]=f(t)$.  If $t\not=0$ then $(f^{*}\GG)_t$ is $h_{S^{[2]}}$ slope stable by Proposition~\ref{prp:stabuno}. 
By~\eqref{tuttifrutti} we have an exact sequence
\begin{equation}\label{nicolpatente}
0\lra \cA(\cE_2)^{+}\lra \cG(\cE_1,\cE_2)\xrightarrow{\Psi_{\cE_2}^{-}} \cE_1[2]^{-}\otimes \Hom_S(\cE_2,\cE_1)^{\vee} \lra 0. 
\end{equation}
The above exact sequence is slope destabilizing (for any polarization) by Proposition~\ref{prp:liberoinstab}. 
 Moreover (see~\eqref{elvis}) we have the exact sequence
\begin{equation}\label{senzascuola}
0\lra \cG(\cE_1,\cH)\lra \cA(\cE_2)^{+}\lra  \cE_1[2]^{+}\otimes \Hom_S(\cE_2,\cE_1)^{\vee} \lra 0, 
\end{equation}
where $\cH$ is the vector bundle on $S$  appearing in the exact sequence in~\eqref{ravel}. Let $J_{\cE_2}$ be given by~\eqref{geiedue}. 
Then  by definition of $\cB(\cH,J_{\cE_2})^{+}$ 
\begin{equation}\label{midnightcowboy}
\cA(\cE_2)^{+}\cong \cB(\cH,J_{\cE_2})^{+}.
\end{equation}
  Hence Proposition~\ref{prp:freddo} gives the next result.
\begin{clm}\label{clm:ornella}
The extension class of~\eqref{senzascuola} belongs to 
\begin{equation*}
\Ext^1_{S^{[2]}}(\cE_1[2]^{+},\cG(\cE_1,\cH))\otimes \Hom_S(\cE_2,\cE_1)\overset{\BKR}{=\joinrel=\joinrel=}\Ext^1_S(\cE_1,\cH)\otimes \Hom_S(\cE_2,\cE_1).
\end{equation*}
\end{clm}
Let $\wt{f^{*}\GG}$ be the elementary modification of $f^{*}\GG$ determined by the  exact sequence in~\eqref{nicolpatente}, i.e.~the sheaf on $S^{[2]}\times T$ fitting into the exact sequence
\begin{equation}
0\lra \wt{f^{*}\GG} \lra f^{*}\GG\overset{\wt{\Psi}_{\cE_2}^{-}}{\lra}  
\iota_{*}\cE_1[2]^{-}\otimes \Hom_S(\cE_2,\cE_1)^{\vee}\lra 0.
\end{equation}
For $t\in T$ let $(\wt{f^{*}\GG})_t\coloneq \wt{f^{*}\GG}_{|S^{[2]}\times \{t\}}$. Then $(\wt{f^{*}\GG})_t\cong (f^{*}\GG)_t$ if $t\not=0$. On the other hand  
$(\wt{f^{*}\GG})_0$ fits into the exact sequence (see Subsubsection~\ref{subsubsec:elemod} and Claim~\ref{clm:ornella})
\begin{equation}\label{potenzstab}
0\lra \cE_1[2]^{-}\otimes \Hom_S(\cE_2,\cE_1)^{\vee} \lra (\wt{f^{*}\GG})_0  \lra  \cB(\cH,J_{\cE_2})^{+}\lra 0. 
\end{equation}
Here we determine the extension class of the above exact sequence. First we note that by Proposition~\ref{prp:extrigidibical} we have the BKR isomorphism
\begin{equation}\label{vanoni}
 \Hom_S(\cE_2,\cE_1)^{\vee}\otimes \Ann J_{\cE_2}\overset{\beta_{\cE_2}}{\overset{\sim}{\lra}}\Ext^1_{S^{[2]}}(\cB(\cH,J_{\cE_2})^{+},\cE_1[2]^{-}\otimes \Hom_S(\cE_2,\cE_1)^{\vee}).
\end{equation}
Let $\ov{df}_0$ be the composition of the linear maps
\begin{equation}\label{proiettonorm}
\Theta_{T}(0)\overset{df_0}{\lra}  \Theta_{\cM}([\cE_2])\overset{p}{\lra} 
\cN_{\cD^k/\cM}([\cE_2]),
\end{equation}
where $p$ is the projection. By Remark~\ref{rmk:eccoilnormale} we have the identification
\begin{equation}\label{delellis}
\cN_{\cD^k/\cM}([\cE_2])=\Hom_S(\cE_2,\cE_1)^{\vee}\otimes\Ext^1_S(\cE_2,\cE_1).
\end{equation}
 Applying the functor $\Hom_S(-,\cE_1)$ to the exact sequence in~\eqref{ravel} we get the inclusion 
\begin{equation}\label{cheever}
\Ext^1_S(\cE_2,\cE_1)\hra \Ext^1_S(\cH,\cE_1).
\end{equation}
\begin{lmm}\label{lmm:swimmer}
The inclusion in~\eqref{cheever} has image equal to $\Ann J_{\cE_2}$.
\end{lmm}
\begin{proof}
 Applying the functor $\Hom_S(-,\cE_1)$ to the exact sequence in~\eqref{ravel} we get the exact sequence
\begin{equation*}
0\lra \Ext^1_S(\cE_2,\cE_1)\lra \Ext^1_S(\cH,\cE_1)\xrightarrow{\partial_{\cE_2}} \Ext^2_S(\cE_1,\cE_1)
\otimes \Hom_S(\cE_2,\cE_1)\lra 0.
\end{equation*}
Let $e\in \Ext^1_S(\cE_1,\cH)\otimes \Hom_S(\cE_2,\cE_1)$ be  the extension class of~\eqref{ravel}. Then $e=\sum_{i=1}^k j_{\cE_2}(v_i)\otimes v_i$ where $\{v_1,\ldots,v_k\}$ is a basis of 
$\Hom_S(\cE_2,\cE_1)$. 
 If $s\in \Ext^1_S(\cH,\cE_1)$ then $\partial_{\cE_2} (s)=s\cup e=\sum_{i=1}^k (s\cup  j_{\cE_2}(v_i))\otimes v_i$. The result follows because, since 
 $\Ext^2_S(\cE_1,\cE_1)=\CC\Id_{\cE_2}$, the cup product 
 $(s\cup  j_{\cE_2}(v_i))$ is identified with Serre duality.
\end{proof}
The identification in~\eqref{delellis} together with Lemma~\ref{lmm:swimmer} gives the isomomorphism
\begin{equation}\label{inclusivo}
\alpha_{\cE_2}\colon\cN_{\cD^k/\cM}([\cE_2])\overset{\sim}{\lra}
\Hom_S(\cE_2,\cE_1)^{\vee}\otimes\Ann J_{\cE_2}.
\end{equation}
Composing  $\alpha_{\cE_2}$ and the isomorphism in~\eqref{vanoni}
we get the  isomorphism
\begin{equation}\label{vivamarcello}
\cN_{\cD^k/\cM}([\cE_2])\overset{\beta_{\cE_2}\circ\alpha_{\cE_2}}{\xrightarrow{\ \ \ \sim\ \ \ }}\Ext^1_{S^{[2]}}(\cB(\cH,J_{\cE_2})^{+},\cE_1[2]^{-}\otimes \Hom_S(\cE_2,\cE_1)^{\vee}).
\end{equation}
\begin{prp}\label{prp:lepiogge}
Let $v$ be a generator of $\im\ov{df}_0$ (see~\eqref{proiettonorm}). The extension class of the exact sequence in~\eqref{potenzstab} is equal to $\beta_{\cE_2}\circ\alpha_{\cE_2}(v)$.
\end{prp}
\begin{proof}
 Let $\cG\coloneq \cG(\cE_2,\cE_1)$ and $V_k\coloneq \Hom_S(\cE_2,\cE_1)^{\vee}$. By the BKR equivalence we have the commutative diagram
\begin{equation*}
\xymatrix{ \Ext^1_S(\cE_2,\cE_2)\ar[r]^{\zeta}\ar[d]_{\BKR} &  \Ext^1_S(\cE_2,\cE_1\otimes V_k)\ar[r]^{\xi}\ar[d]_{\BKR} & V_k\otimes\Ann J_{\cE_2}\ar[d]_{=} \\
\Ext^1_{S^{[2]}}(\cG,\cG) \ar[r]^{\ov{\zeta}\qquad\ }     &  \Ext^1_{S^{[2]}}(\cG,\cE_1[2]^{-}\otimes V_k) \ar[r]^{\quad\quad\ov{\xi}}      & V_k\otimes\Ann J_{\cE_2} \\ }
\end{equation*}
Here $\zeta$, $\xi$  are obtained by applying the functors $\Hom_S(\cE_2,-)$, $\Hom_S(-,\cE_1\otimes V_k)$ to the exact sequence in~\eqref{ravel} respectively, and the domain of $\xi$  is as indicated by Lemma~\ref{lmm:swimmer}.
Moreover the maps $\ov{\zeta}=\BKR(\zeta)$, $\ov{\xi}=\BKR(\xi)$ are obtained by applying the functors 
$\Hom_{S^{[2]}}(\cG,-)$, $\Hom_S(-,\cE_1[2]^{-}\otimes V_k)$ to the exact sequence in~\eqref{nicolpatente} respectively, given the 
isomorphisms $\cA(\cE_2)^{+}\cong \cB(\cH,J_{\cE_2})^{+}$ and $\beta_{\cE_2}$ in~\eqref{vanoni}.

Let $\wt{v}\in \Ext^1_S(\cE_2,\cE_2)$ be such that $v=p(\wt{v})$, where $p$ is the projection to the normal space (see~\eqref{proiettonorm}). By Lemma~\ref{lmm:estmod} the extension class of~\eqref{potenzstab} is given (up to $\CC^{*}$)
by  $\beta_{\cE_2}(\ov{\xi}\circ \ov{\zeta}(\BKR(\wt{v})))$. 
On the other hand $\alpha_{\cE_2}(v)=\xi\circ \zeta(\wt{v})$. By commutativity of the  diagram, $\alpha_{\cE_2}(v)=\ov{\xi}\circ \ov{\zeta}(\BKR(\wt{v}))$.
\end{proof}
\begin{crl}\label{crl:gendirstab}
Let $v$  generate $\im\ov{df(0)}$. If $\alpha_{\cE_2}(v)\colon \Hom_S(\cE_2,\cE_1)\to\Ann J_{\cE_2}$ 
is an isomorphism  then $(\wt{f^{*}\GG})_0$ is isomorphic to the vector bundle  $\cB(\cH,J_{\cE_2})$, and it is $h_{S^{[2]}}$ slope stable.
\end{crl}
\begin{proof}
By Definition~\ref{dfn:fasciobicaltilde} the extension class of the exact sequence
\begin{equation}\label{novembre}
0\lra  \cE_1[2]^{-}\otimes (\Ann J_{\cE_2})^{\vee}\lra \cB(\cH,J_{\cE_2})\overset{\lambda}{\lra} 
\cB(\cH,J_{\cE_2})^{+}\lra  0
\end{equation}
 is given by the identity $\Id\colon \Ann J_{\cE_2}\to \Ann J_{\cE_2}$. 
By Proposition~\ref{prp:lepiogge} the extension class of the exact sequence in~\eqref{potenzstab} is equal to 
$\beta_{\cE_2}\circ\alpha_{\cE_2}(v)$. Since $\alpha_{\cE_2}(v)\colon \Hom_S(\cE_2,\cE_1)\to \Ann J_{\cE_2}$ is an isomorphism, the two extensions are isomorphic. More precisely: the transpose  
$\alpha^t_{\cE_2}(v)\colon \Ann J_{\cE_2}^{\vee}\to  \Hom_S(\cE_2,\cE_1)^{\vee}$ is an isomorphism defining an isomorphism 
$\cE_1[2]^{-}\otimes (\Ann J_{\cE_2})^{\vee}\xrightarrow{\sim}\cE_1[2]^{-}\otimes \Hom_S(\cE_2,\cE_1)^{\vee}$ which extends to an isomorphism between  the exact sequence in~\eqref{novembre} and the exact sequence in~\eqref{potenzstab}.
\end{proof}
\subsubsection{First step of semistable reduction along a curve, II}\label{subsubsec:annalisa}
This is the analogue of Subsusbsection~\ref{subsubsec:ridsemuno} with $\cD^k\setminus\cD^{k+1}$ replaced by $ \cD^a$. We adopt notation introduced in that subsusbsection, and we avoid repeating statements that carry over without modifications.
 Let $[\cE_2]\in \cD^a$, and let  $f\colon (T,0)\to(\cM,[\cE_2])$ be a (regular) map of pointed varieties, with $T$  a smooth  curve. We assume that  that  (set-theoretically)
 $f^{-1}(\cD)=\{0\}$. By~\eqref{harnargi} we have an exact sequence
\begin{equation}\label{cavatina}
0\lra \cA(\cE_2)^{+}\lra \cG(\cE_1,\cE_2) \overset{\Psi^{-}_{\cE_2}}{\lra} \cE_1[2]^{-}\otimes V_a\lra 0.
\end{equation}
The above exact sequence is slope destabilizing (for any polarization) by Proposition~\ref{prp:singolareinstab}. 
  We have the exact sequence (see~\eqref{psipiu})
\begin{equation}\label{attesanicol}
0\lra \cA(\cE_2)^{+} \lra \cE_1[2]^{+}\otimes V_a\overset{\Phi^{+}_{\cE_2}}{\lra} 
\cG(\cE_1,\CC_{p})\lra 0.
\end{equation}
Let $p\in S$ be the unique singular point of $\cE_2$. Recall that  $f(p)^t\colon V_a\hra \cE_1(p)^{\vee}$ is injective. Abusing notation we denote 
$\im f(p)^t$ by $V_a$. Proposition~\ref{prp:nicolciacco} gives an isomorphism
\begin{equation}\label{maduro}
\cA(\cE_2)^{+} \overset{\sim}{\lra} \cB(p,V_a)^{+}.
\end{equation}
Let $\wt{f^{*}\GG}$ be the sheaf on $S^{[2]}\times T$ fitting into the exact sequence
\begin{equation}
0\lra \wt{f^{*}\GG} \lra f^{*}\GG\overset{\wt{\Psi}_{\cE_2}^{-}}{\lra}  
\iota_{*}\cE_1[2]^{-}\otimes V_a\lra 0.
\end{equation}
Then $(\wt{f^{*}\GG})_t\cong (f^{*}\GG)_t$ if $t\not=0$, and
$(\wt{f^{*}\GG})_0$ fits into the exact sequence (see Subsubsection~\ref{subsubsec:elemod} and~\eqref{maduro})
\begin{equation}\label{masha}
0\lra \cE_1[2]^{-}\otimes V_a \lra (\wt{f^{*}\GG})_0  \lra  \cB(p,V_a)^{+}\lra 0. 
\end{equation}
We describe the extension class of the above exact sequence. We have (see Proposition~\ref{prp:alglin}) the  BKR isomorphism
\begin{equation}\label{passacaglia}
\beta_{\cE_2}\colon V_a\otimes \Ann V_a\xrightarrow{\sim}\Ext^1_{S^{[2]}}(\cB(p,V_a)^{+},\cE_1[2]^{-}\otimes V_a).
\end{equation}
Let $\ov{df}_0$ be the composition of the linear maps
\begin{equation}\label{gigiproietti}
\Theta_{T}(0)\overset{df_0}{\lra}  \Theta_{\cM}([\cE_2])\overset{p}{\lra} 
\cN_{\cD^k/\cM}([\cE_2]),
\end{equation}
where $p$ is the projection. By  Remark~\ref{rmk:eccoilnormale} we have the isomorphism
\begin{equation}\label{ammirati}
\cN_{\cD^a/\cM}([\cE_2])\xrightarrow{\sim}\Hom_S(\cE_2,\cE_1)^{\vee}\otimes\Ext^1_S(\cE_2,\cE_1).
\end{equation}
Applying the functor $\Hom_S(-,\cE_1)$ to~\eqref{fabrizio} we get the isomorphism
\begin{equation}
\CC\Id_{\cE_1}\otimes V_a^{\vee}\xrightarrow{\sim} \Hom_S(\cE_2,\cE_1),
\end{equation}
and the exact sequence
\begin{equation}\label{killmesoftly}
0\lra \Ext^1_S(\cE_2,\cE_1)\overset{\partial}{\lra}  \Ext^2_S(\CC_p,\cE_1) \overset{f(p)}{\lra} 
\Ext^2_S(\cE_1,\cE_1)\otimes V_a^{\vee}\lra 0.
\end{equation}
In the above exact sequence $f(p)$ is the map in~\eqref{mappaeffepi}, and one identifies it with the map $\Ext^2_S(\CC_p,\cE_1) \lra
\Ext^2_S(\cE_1,\cE_1)$ via Serre duality. It follows that the coboundary map in~\eqref{killmesoftly} defines an isomorphism
\begin{equation}\label{winner}
\begin{matrix}
\Ext^1_S(\cE_2,\cE_1) & \xrightarrow{\sim} & \Ann V_a\subset \Ext^2_S(\CC_p,\cE_1)\\
x & \mapsto & \partial(x)\qquad\qquad\qquad\quad
\end{matrix}
\end{equation}
Hence  the isomorphism in~\eqref{ammirati} reads
\begin{equation}\label{sciarpabella}
\alpha_{\cE_2}\colon \cN_{\cD^a/\cM}([\cE_2])\xrightarrow{\sim}V_a\otimes\Ann V_a.
\end{equation}
\begin{prp}
Let $v$ be a generator of $\im\ov{df}_0$ (see~\eqref{proiettonorm}). The extension class of the exact sequence in~\eqref{masha} is equal to $\beta_{\cE_2}\circ\alpha_{\cE_2}(v)$.
\end{prp}
\begin{proof}
 Let $\cG\coloneq \cG(\cE_2,\cE_1)$. The BKR equivalence gives the commutative diagram
\begin{equation*}
\xymatrix{ \Ext^1_S(\cE_2,\cE_2)\ar[r]^{\zeta}\ar[d]_{\BKR} &  \Ext^1_S(\cE_2,\cE_1\otimes V_a)\ar[r]^{\xi}\ar[d]_{\BKR} & V_a\otimes\Ann V_a\ar[d]_{=} \\
\Ext^1_{S^{[2]}}(\cG,\cG) \ar[r]^{\ov{\zeta}\qquad\ }     &  \Ext^1_{S^{[2]}}(\cG,\cE_1[2]^{-}\otimes V_a) \ar[r]^{\quad\quad\ov{\xi}}      & V_k\otimes\Ann V_a \\ }
\end{equation*}
Here $\zeta$, $\xi$  are obtained by applying the functors $\Hom_S(\cE_2,-)$, $\Hom_S(-,\cE_1\otimes V_k)$ to the exact sequence in~\eqref{fabrizio} respectively, and the domain of $\xi$  is as indicated because of the isomorphism in~\eqref{winner}.
Moreover the maps $\ov{\zeta}=\BKR(\zeta)$, $\ov{\xi}=\BKR(\xi)$ are obtained by applying the functors 
$\Hom_{S^{[2]}}(\cG,-)$, $\Hom_S(-,\cE_1[2]^{-}\otimes V_k)$ to the exact sequence in~\eqref{cavatina} respectively, given the 
isomorphisms $\cA(\cE_2)^{+}\cong \cB(p,V_a)^{+}$ (see~\eqref{maduro})  and $\beta_{\cE_2}$ (see~\eqref{passacaglia}).

Let $\wt{v}\in \Ext^1_S(\cE_2,\cE_2)$ be such that $v=p(\wt{v})$, where $p$ is the projection to the normal space (see~\eqref{gigiproietti}). By Lemma~\ref{lmm:estmod} the extension class of~\eqref{masha} is given (up to $\CC^{*}$)
by  $\beta_{\cE_2}(\ov{\xi}\circ \ov{\zeta}(\BKR(\wt{v})))$. 
On the other hand $\alpha_{\cE_2}(v)=\xi\circ \zeta(\wt{v})$. By commutativity of the  diagram, $\alpha_{\cE_2}(v)=\ov{\xi}\circ \ov{\zeta}(\BKR(\wt{v}))$.
\end{proof}
Arguing as in the proof of Corollary~\ref{crl:gendirstab} one proves the following result (recall that $V_a$ is naturally isomorphic to $\Hom(\cE_2,\cE_1)^{\vee}$).
\begin{crl}\label{crl:annasciarpa}
Let $v$ generate  $\im\ov{df(0)}$. If $\alpha_{\cE_2}(v)\colon \Hom(\cE_2,\cE_1)\to \Ann V_a$ is an isomorphism   then $(\wt{f^{*}\GG})_0$ is isomorphic to the vector bundle  $\cB(p,V_a)$, and it is $h_{S^{[2]}}$ slope stable.
\end{crl}
\begin{rmk}
The  results  in Corollaries~\ref{crl:gendirstab}, \ref{crl:annasciarpa}
are quite close to a proof Items~(a) and~(b) of Theorem~\ref{thm:dadueauno}. This is clear if one looks at the arguments in Subsection~\ref{subsec:alfinlaprova}. 
\end{rmk}
\subsection{Extensions and complete collineations}\label{subsec:primomaggio}
\setcounter{equation}{0}  
\subsubsection{A family of extensions and its parameter space}
Throughout the present subsubsection $X$ is a projective variety, and  $\cF_1,\cF_3$ are sheaves on $X$ such that 
\begin{equation}\label{zerozerozerouno}
\Hom_X(\cF_3,\cF_1)=\Ext^1_X(\cF_1,\cF_1)=0,\quad \Hom_X(\cF_1,\cF_1)=\CC\Id_{\cF_1}.
\end{equation}
\begin{dfn}
Given  a subspace $U\subset \Ext^1_X(\cF_3,\cF_1)$, let
\begin{equation}\label{essesotto}
0\lra  \cF_1\otimes U^{\vee}\lra \cS_U\lra\cF_3\lra  0
\end{equation}
be the  exact sequence with extension class (in $U^{\vee}\otimes \Ext^1_X(\cF_3,\cF_1)$) given by  the inclusion map 
$U\hra \Ext^1_X(\cF_3,\cF_1)$.
\end{dfn}
Applying the functor 
$\Hom_X(-,\cF_1)$ to the exact sequence in~\eqref{essesotto}
we get the exact sequence
\begin{equation}
0\lra \CC\Id_{\cF_1}\otimes U\overset{\partial_U}{\lra} \Ext^1_X(\cF_3,\cF_1)\lra \Ext^1_X(\cS_U,\cF_1)\lra 0.
\end{equation}
Since $\partial_U$  is  the inclusion map $U\hra  \Ext^1_X(\cF_3,\cF_1)$, we get an isomorphism
\begin{equation}\label{quozextuno}
\eta_U\colon\Ext^1_X(\cF_3,\cF_1)/U\overset{\sim}{\lra} \Ext^1_X(\cS_U,\cF_1).
\end{equation}
Let $W$ be a vector space such that
\begin{equation}\label{kappadimw}
k\coloneq \dim W\le \dim  \left(\Ext^1_X(\cF_3,\cF_1)/U\right). 
\end{equation}
\begin{dfn}
If $\varphi \colon W\to \Ext^1_X(\cF_3,\cF_1)/U$ is a linear map,  let
\begin{equation}\label{essephi}
0\lra  \cF_1\otimes W^{\vee}\lra \cS_{\varphi } \lra\cS_U\lra  0
\end{equation}
be the exact sequence with extension class (in $W^{\vee}\otimes \Ext^1_X(\cS_U,\cF_1)$) given by $\eta_U\circ\varphi$. 
\end{dfn}
Since the isomorphism class of $\cS_{\varphi }$ depends only on the homothety class $[\varphi]$ of 
$\varphi$ we also denote $\cS_{\varphi }$ by $\cS_{[\varphi] }$. 
Fix $U,W$ as above,   and let
\begin{equation}\label{defpielle}
L\coloneq(U,W),\qquad \bP_L\coloneq \PP(\Hom(W,  \Ext^1_X(\cF_3,\cF_1)/U)).
\end{equation}
Let $\cF_1,\cF_3$ and $U,W$ be as above. The K\"unneth decomposition  for sheaves on $X\times \bP_L$ defines an isomorphism (note that $\Hom_X(\cS_U,\cF_1)=0$) 
\begin{multline}\label{perelisa}
\Ext^1(\cS_U\boxtimes\cO_{\bP_L}(-1),(\cF_1\otimes W^{\vee})\boxtimes\cO_{\bP_L}) \cong \\
\cong
\Hom_{\bP_L}(\cO_{\bP_L}(-1),W^{\vee}\otimes  \Ext^1_X(\cS_U,\cF_1)\otimes \cO_{\bP_L}). 
\end{multline}
\begin{dfn}
Let $\cS_{\bP_L}$ be the sheaf on $X\times\bP_L$ fitting  into the exact sequence
\begin{equation}\label{scaccopanzetti}
0\lra  (\cF_1\otimes W^{\vee})\boxtimes\cO_{\bP_L}\lra \cS_{\bP_L} 
\lra\cS_U\boxtimes\cO_{\bP_L}(-1)\lra  0
\end{equation}
with extension class the  tautological section of the right-hand side of~\eqref{perelisa}.   
\end{dfn}
If $[\varphi]\in\bP_L$ then
\begin{equation}
\cS_{\bP_L|X\times [\varphi]}\cong \cS_{[\varphi]}.
\end{equation}
\begin{clm}\label{clm:splitexseq}
Let $[\varphi]\in \bP_L$, and let $U({\varphi })\subset \Ext^1_X(\cF_3,\cF_1)$ be the subspace containing $U$ such that $U(\varphi )/U=\im\varphi$.
There is a well-defined split exact sequence 
\begin{equation}\label{ialuronico}
0\lra \cS_{U(\varphi )} \lra \cS_{\varphi } \lra \cF_1\otimes(\ker\varphi)^{\vee}\lra  0.
\end{equation}
\end{clm}
\begin{proof}
Applying the functor $\Hom_X(-,\cF_1\otimes (\ker\varphi)^{\vee})$ to the exact sequence in~\eqref{essephi} we get the exact sequence
\begin{multline*}
 \Hom_X(\cS_U,\cF_1)\otimes (\ker\varphi)^{\vee}\lra \Hom_X(\cS_{\varphi},\cF_1\otimes (\ker\varphi)^{\vee}) \lra \\
\lra \CC\Id_{\cF_1}\otimes W\otimes (\ker\varphi)^{\vee}\overset{\partial}{\lra} 
\Ext^1_X(\cS_U,\cF_1)\otimes (\ker\varphi)^{\vee}
\end{multline*}
The coboundary $\partial$  is given by composition with 
$\varphi$, and hence vanishes. It follows that there is a surjection $\cS_{\varphi}\twoheadrightarrow \cF_1\otimes(\ker\varphi)^{\vee}$, and since $\Hom_X(\cF_3,\cF_1)=0$ it is unique. The kernel of the above surjection is isomorphic to $\cS_{U(\varphi )}$, and the exact sequence in~\eqref{ialuronico} is split.
\end{proof}
Let $\bD^s_L\subset \bP_L$ be the determinantal subscheme defined (set theoretically) by
\begin{equation*}
\bD^s_L\coloneq\{[\varphi]\in\bP_L \mid \dim\ker\varphi\ge s\}.
\end{equation*}
We have the chain of closed subsets
\begin{equation}\label{catenadiw}
\es= \bD^{k}_L\subsetneq\bD^{k-1}_L\subsetneq \bD^{k-2}_L\subsetneq \ldots \subsetneq \bD^{1}_L\subsetneq \bD^{0}_L=\bP_L.
\end{equation}
\begin{rmk}\label{rmk:normrango}
Let $[\varphi]\in(\bD^{s}_L\setminus \bD^{s+1}_L)$. We  have an isomorphism
\begin{equation}\label{normdeter}
w_{\varphi}\colon\cN_{\bD^{s}_L/\bP_L}([\varphi])\overset{\sim}{\lra} \Hom(\ker\varphi,\coker\varphi),
\end{equation}
defined as follows.
Let 
$\wh{\bD}^{s}_L\subset\Hom(W,  \Ext^1_X(\cF_3,\cF_1)/U)\setminus\{0\}$ be the punctured cone over $\bD^{s}_L$, i.e.~the determinantal variety parametrizing non zero maps $W\to  \Ext^1_X(\cF_3,\cF_1)/U$ with kernel of dimension at least $s$. The differential at $\varphi$ of the projection map $\Hom(W,  \Ext^1_X(\cF_3,\cF_1)/U)\setminus\{0\}\lra \bP_L$ induces an isomorphism 
\begin{equation}\label{disturbatore}
\cN_{\wh{\bD}^{s}_L/\Hom(W,  \Ext^1_X(\cF_3,\cF_1)/U)}(\varphi)\overset{\sim}{\lra} \cN_{\bD^{s}_L/\bP_L}([\varphi]).
\end{equation}
A straightfoward computation gives that 
\begin{equation*}
\Theta_{\wh{\bD}^{s}_L}(\varphi)=\{\alpha\in \Hom(W,  \Ext^1_X(\cF_3,\cF_1)/U)
 \mid \alpha(\ker\varphi)\subset \im\varphi\}.
\end{equation*}
Hence, letting $\pi\colon   \Ext^1_X(\cF_3,\cF_1)/U\lra  \coker\alpha$ be   the quotient map, we have a well-defined isomorphism
\begin{equation*}
\begin{matrix}
\cN_{\wh{\bD}^{s}_W/W^{\vee}\otimes  \Ext^1_X(\cF_3,\cF_1)}(\varphi) & \overset{\sim}{\lra} & \Hom(\ker\varphi,\coker\varphi) \\
\ov{\alpha} & \mapsto & \pi\circ\left(\alpha_{|\ker\varphi}\right)
\end{matrix}
\end{equation*}
This, in conjunction with~\eqref{disturbatore}, defines the isomorphism in~\eqref{normdeter}.
\end{rmk}
\subsubsection{From  $\bP_L$ to a space of complete collineations}
We replace $\bP_L$ by  a space of complete collineations (see~\cite{vainsencher-collineations}) and we modify $\cS_{\bP_L}$ via elementary modifications dictated by the exact sequences in~\eqref{ialuronico}. The end result is a family of sheaves on $X$, each of which is isomorphic to 
$\cS_{\varphi}$ where $\varphi$ is an injection 
$W\hookrightarrow \Ext^1_X(\cF_3,\cF_1)/U$. We start by introducing the blow-ups which define the space of complete collineations: they are  analogous to those appearing in Definition~\ref{dfn:emmeconjei}. 
\begin{dfn}\label{dfn:compcoll}
Let 
varieties $\bP_L(k-1),\ldots,\bP_L(0)$,  birational maps
\begin{equation*}
\bP_L(0)\xrightarrow{g^L_0}\bP_L(1)\xrightarrow{g^L_1}\ldots \bP_L(k-1)\xrightarrow{g^L_{k-1}}\bP_L,
\end{equation*}
  and   closed  subsets $\bE_L(j)^s,\bD_L(j)^t\subset \bP_L(j)$ for $j\in\{k-1,\ldots,0\}$, $s> j\ge  t\ge 0$
 be as follows. 
\begin{enumerate}
\item
$\bP_L(k-1)=\bP_L$,  $g_{k-1}\colon\bP_L(k-1)\to\bP_L$ is  the identity, $\bE_L(k-1)^s=\es$ and 
$\bD_L(k-1)^t=\bD_L^t$.
\item
 $g^L_{j-1}$    and $\bE_L(j-1)^s,\bD_L(j-1)^t$ are defined by iteration starting from $j=k-1$ by the following prescriptions:
\begin{enumerate}
\item[(2a)]
$g^L_{j-1}\colon \bP_L(j-1)\to  \bP_L(j)$ is the blow up of $\bD_L(j)^{j}$
\item[(2b)]
 If $ s> j$ then $\bE_L(j-1)^{s}$ is the  strict transform of $\bE_L(j)^{s}$, while  
 $\bE_L(j-1)^{j}$ is the exceptional divisor of $g^L_{j-1}$. 
\item[(2c)]
If $j-1\ge t\ge 0$ then 
$\bD_L(j-1)^t$ is the  strict transform of $\bD_L(j)^t$. 
\end{enumerate}
\end{enumerate}
\end{dfn}
  We have the  chain of closed subsets 
\begin{equation}\label{catenadielle}
\es\subsetneq\bD_L(j)^{j}\subsetneq  \ldots \subsetneq \bD_L(j)^{1}\subsetneq \bD_L(j)^{0}=\bP_L(j).
\end{equation}
analogous to that in~\eqref{catenadiw}. For $k-1\ge s> j$ we let
\begin{equation}\label{eccemaggugu}
\bE_L(j)^{\ge s}\coloneq \bE_L(j)^{k-1}\cup\bE_L(j)^{k-2}\cup\ldots\bE_L(j)^{s}
\end{equation}
\begin{rmk}\label{rmk:prerospiglio}
Let  $ s\in\{k-1,\ldots,1\}$.
The restriction of $g^L_{k-1}\circ g^L_{k-2}\circ\ldots\circ g^L_{s-1}$ to 
$\bE_L(s-1)^{s}\setminus\bE_L(s-1)^{\ge s+1}$ 
maps to $ \bP_L\setminus \bD_L^{s+1}$, and is identified with the blow up of $\bD_L^{s}\setminus \bD_L^{s+1}$.  Of course the exceptional divisor is 
$\bE_L(s-1)^{s}\setminus\bE_L(s-1)^{\ge s+1}$. Let  $[\varphi]\in\left(\bD_L^{s}\setminus \bD_L^{s+1}\right)$: the isomorphism 
in~\eqref{normdeter} defines an isomorphism
\begin{equation}\label{goldvartns}
(g^L_{k-1}\circ g^L_{k-2}\circ\ldots\circ g^L_{s-1})^{-1}([\varphi])\overset{\sim}{\lra} \PP(\Hom(\ker\varphi,\coker\varphi)).
\end{equation}
\end{rmk}
Let  $k-1\ge s\ge j\ge 1$.
Then $g^L_{k-1}\circ g^L_{k-2}\circ\ldots\circ g^L_{j-1}$ maps  
$\bE_L(j-1)^{\ge s}\setminus\bE_L(j-1)^{\ge s+1}$  to $\bE_L(s-1)^{s}\setminus\bE_L(s-1)^{\ge s+1}$. Let 
\begin{equation}
h_{j-1,\ge s}^L\colon  (\bE_L(j-1)^{\ge s}\setminus\bE_L(j-1)^{\ge s+1})\lra (\bD_L^{s}\setminus \bD_L^{s+1})
\end{equation}
 be the composition
\begin{equation*}
\bE_L(j-1)^{\ge s}\setminus\bE_L(j-1)^{\ge s+1}\lra (\bE_L(s-1)^{s}\setminus\bE_L(s-1)^{\ge s+1})\lra (\bD_L^{s}\setminus \bD_L^{s+1}).
\end{equation*}
\begin{key-rmk}\label{key-rmk:rospiglione}
Let  $k-1\ge s\ge j\ge 1$, and let  $[\varphi]\in\left(\bD_L^{s}\setminus \bD_L^{s+1}\right)$. Let $U(\varphi )\subset \Ext^1_X(\cF_3,\cF_1)$ be as in Claim~\ref{clm:splitexseq}. Let 
\begin{equation}
W(\varphi )\coloneq\ker\varphi,\qquad L(\varphi )\coloneq(U(\varphi ),W(\varphi )).
\end{equation}
 We have a tautological identification $\PP(\Hom(\ker\varphi,\coker\varphi))=\bP_{L(\varphi )}$, and hence the isomorphism in~\eqref{goldvartns} reads
\begin{equation}
u_{\varphi}\colon (g^L_{k-1}\circ g^L_{k-2}\circ\ldots\circ g^L_{s-1})^{-1}([\varphi])\overset{\sim}{\lra} \bP_{L(\varphi )}.
\end{equation}
A local computation (as in the proof of Lemma~\ref{lmm:conostrato}) gives the schematic equality
\begin{equation}
u_{\varphi}^{-1}(\bD_{L(\varphi )}^r)=(g^L_{k-1}\circ g^L_{k-2}\circ\ldots\circ g^L_{s-1})^{-1}([\varphi])\cap \bD_L(s-1)^r.
\end{equation}
 It follows that we have a commutative diagram
\begin{equation*}
\xymatrix{(h_{j-1,\ge s}^L)^{-1}([\varphi]) \ar[rr]^{{g^L_s\circ\ldots\circ g^L_{j-1}}_{|\ldots} \qquad\quad\ } \ar[d]^{\wr}   &  &  
(g^L_{k-1}\circ g^L_{k-2}\circ\ldots\circ g^L_{s-1})^{-1}([\varphi])\ar[d]^{u_{\varphi}} \\ 
\bP_{L(\varphi )}(j-1)\ar[rr]^{g^{\bar{L}}_{s}\circ\ldots\circ g^{\bar{L}}_{j-1}}   & &  \bP_{L(\varphi )} }
\end{equation*}
\end{key-rmk}
\subsubsection{Analogue of Subsubsections~\ref{subsubsec:ridsemuno} and~\ref{subsubsec:annalisa}}
Let $[\varphi]\in\left(\bD_L^s\setminus \bD_L^{s+1}\right)$,  let $T$ be a  a smooth curve, and let
$f\colon T\to\bP_L$ be a map such that $f(0)=[\varphi]$. Let   $\cS_T$ be the sheaf on $X\times T$ given by  $\cS_T\coloneq(\Id_X\times f)^{*}\cS_{\bP_L}$. We modify    $\cS_T$  according to the exact sequence
 in~\eqref{ialuronico}, i.e.
\begin{equation}\label{acido}
0\lra \cS_{U(\varphi )} \xrightarrow{\zeta} \cS_{\varphi } \xrightarrow{\xi} \cF_1\otimes(\ker\varphi)^{\vee}\lra  0.
\end{equation}
 More precisely let $\iota\colon X=X\times\{0\}\hra X\times T$ be the inclusion map, and let $\wt{\cS}_T$ be the sheaf on 
 $X\times T$ fitting into the exact sequence
\begin{equation}
0\lra \wt{\cS}_T\lra   \cS_T \lra \iota_{*} \cF_1\otimes(\ker\varphi)^{\vee}\lra 0.
\end{equation}
 The restriction of $ \wt{\cS}_T$ to $X\times\{0\}$ fits into the exact sequence (see Claim~\ref{clm:splitexseq})
\begin{equation}\label{aeronautica}
0\lra \cF_1\otimes(\ker\varphi)^{\vee} \lra \wt{\cS}_{T|X\times\{0\}}\lra \cS_{U(\varphi )} \lra  0,
\end{equation}
where $U(\varphi )\subset\Ext^1_X(\cF_3,\cF_1)$ is as in Claim~\ref{clm:splitexseq}.
Our next goal is to describe the extension class of the exact sequence in~\eqref{aeronautica}. Applying the functor 
$\Hom_X(-,\cF_1)$ to the exact sequence 
\begin{equation}\label{buscaglione}
0\lra  \cF_1\otimes U(\varphi )^{\vee}\lra \cS_{U(\varphi) } \lra\cF_3\lra  0
\end{equation}
we get the exact sequence
\begin{equation}
0\lra \CC\Id_{\cF_1}\otimes U(\varphi )\overset{\partial}{\lra} \Ext^1_X(\cF_3,\cF_1)\lra \Ext^1_X(\cS_{U(\varphi )},\cF_1)\lra 0.
\end{equation}
Since the  coboundary map is identified with the inclusion  $U(\varphi )\hra  \Ext^1_X(\cF_3,\cF_1)$, we get an isomorphism 
$\coker\varphi\xrightarrow{\sim} \Ext^1_X(\cS_{U(\varphi )},\cF_1)$. Tensorizing by $(\ker\varphi)^{\vee}$  we get an isomorphism
\begin{equation}\label{fisioterapia}
\iota_{\varphi}\colon\Hom(\ker\varphi,\coker\varphi)\xrightarrow{\sim} \Ext^1_X(\cS_{U(\varphi )},\cF_1\otimes (\ker\varphi)^{\vee})
\end{equation}
Let $\ov{df}_0\coloneq p\circ df_0$ be the composition 
\begin{equation}
\Theta_{T}(0)\overset{df_0}{\lra}  \Theta_{\bP_L}([\varphi])\overset{p}{\lra} 
\cN_{\bD_L^s/\bP_L}([\varphi]),
\end{equation}
where $p$ is the quotient map. 
\begin{lmm}
Suppose  that the equalities in~\eqref{zerozerozerouno} hold.  The extension class of the exact sequence in~\eqref{aeronautica} is equal to 
$\iota_{\varphi}\circ w_{\varphi}\circ \ov{df}_0(v)$, where  $v\in\Theta_{T}(0)$ is non zero,  $w_{\varphi}$ is as in~\eqref{normdeter} and 
$\iota_{\varphi}$ is as in~\eqref{fisioterapia}. 
\end{lmm}
\begin{proof}
The exact sequences~\eqref{essesotto}, \eqref{essephi} and~\eqref{acido} give the diagram
\begin{equation}\label{semprevisa}
\xymatrix{
   & \cF_1\otimes(\ker\varphi)^{\vee} & \\
\cF_1\otimes W^{\vee} \ar[ur]^{\xi\circ\alpha} \ar[r]^{\quad\ \alpha}  & \cS_{\varphi } \ar[r]^{\beta} \ar[u]^{\xi}& \cS_U \ar[d]^{\gamma_{U}}\\
 & \cS_{U(\varphi) } \ar[u]^{\zeta} \ar[ur]^{\beta\circ\zeta}\ar[r]^{\quad\gamma_{U(\varphi)}} & \cF_3& 
  }
\end{equation}
We have the maps
\begin{equation}
\Theta_T(0) \xrightarrow{\kappa} \Ext^1_X(\cS_{\varphi},\cS_{\varphi}) \xrightarrow{\mu}   \Ext^1_X(\cS_{U(\varphi )},\cF_1\otimes(\ker\varphi)^{\vee}),
\end{equation}
where $\kappa$  
is the Kodaira-Spencer map at $0$, and $\mu(\gamma)\coloneq \xi\cup\gamma\cup \zeta$.
By Lemma~\ref{lmm:estmod} the extension class is equal to $\mu\circ \kappa(v)$. 
After shrinking $T$ around $0$ there is 
a lift  $\wt{f}\colon T\to W^{\vee}\otimes\left(\Ext^1_X(\cF_3,\cF_1)/U\right)$ of $f\colon T\to\bP_L$. 
Let   
$d\wt{f}_0\colon\Theta_T(0)\to W^{\vee}\otimes\left(\Ext^1_X(\cF_3,\cF_1)/U\right)$ be the differential  of $\wt{f}$ at $0$. Tensorizing   both sides of~\eqref{quozextuno} by $W^{\vee}$ we get the isomorphism
\begin{equation}
\eta_L\colon W^{\vee}\otimes(\Ext^1_X(\cF_3,\cF_1)/U)\xrightarrow{\sim} \Ext^1_X(\cS_U,\cF_1\otimes W^{\vee}).
\end{equation}
Then $\mu\circ \kappa(v) = (\xi\circ\alpha)\cup \eta_L(d\wt{f}_0(v))\cup (\beta\circ\zeta)$.
Hence it suffices to prove that
\begin{equation}\label{romyschneider}
(\xi\circ\alpha)\cup  \eta_L(d\wt{f}_0(v))\cup (\beta\circ\zeta)=\iota_{\varphi}\circ w_{\varphi}\circ \ov{df}_0(v).
\end{equation}
Consider the  diagram
\begin{equation*}
\xymatrix{
 U\otimes W^{\vee}\ar[r]^{\partial_U\qquad}\ar[d] &  \Ext^1_X(\cF_3,\cF_1\otimes W^{\vee})\ar[r]^{\cup\gamma_U  } \ar[d]^{\xi\circ\alpha\cup} & 
\Ext^1_X(\cS_U,\cF_1\otimes W^{\vee}) \ar[d]^{(\xi\circ\alpha)\cup-\cup (\beta\circ\zeta)} \\
 U(\varphi)\otimes (\ker\varphi)^{\vee}\ar[r]^{\partial_{U(\varphi)}\qquad} &  \Ext^1_X(\cF_3,\cF_1\otimes (\ker\varphi)^{\vee})
 \ar[r]^{\cup\gamma_{U(\varphi)} \  \ }  & 
\Ext^1_X(\cS_{U(\varphi)},\cF_1\otimes (\ker\varphi)^{\vee})  
  }
\end{equation*}
We claim that the above diagram is commutative. This is clear for the left-most square. To prove that  the right-most square is commutative it suffices to show that $\gamma_U\circ\beta\circ\zeta=\gamma_{U(\varphi)}$. This holds because $\beta\circ\zeta$ fits into the commutative diagram
\begin{equation}
\xymatrix{
\cF_1\otimes U^{\vee} \ar[r] & \cS_U\ar[r]^{\gamma_U } & \cF_3\ar[d]^{\Id} \\
\cF_1\otimes U(\varphi)^{\vee} \ar[r] \ar[u]& \cS_{U(\varphi)}\ar[r]^{\gamma_{U(\varphi)} } \ar[u]^{\beta\circ\zeta}& \cF_3
  }
\end{equation}
 where the left-most vertical arrow is obtained from the transpose of the inclusion $U\hra U(\varphi)$. In other words the extension defining $\cS_U$ is the push-out of the extension defining $\cS_{U(\varphi)}$ via the map 
 $\cF_1\otimes U(\varphi)^{\vee} \to \cF_1\otimes U^{\vee}$. 
Let   
$d\wt{\wt{f}}_0(v)\colon W\to\Ext^1_X(\cF_3,\cF_1)$ be a lift of $d\wt{f}_0(v)\colon W\to(\Ext^1_X(\cF_3,\cF_1)/U)$. By commutativity of the diagram above we get that
\begin{multline*}
(\xi\circ\alpha)\cup  \eta_L(d\wt{f}_0(v))\cup (\beta\circ\zeta) =(\xi\circ\alpha)\cup \gamma_U\cup d\wt{\wt{f}}_0(v)\cup (\beta\circ\zeta)=\\
=\gamma_{U(\varphi)}\cup (\xi\circ\alpha)\cup d\wt{\wt{f}}_0(v).
\end{multline*}
Since  $\gamma_{U(\varphi)}\cup (\xi\circ\alpha)\cup d\wt{\wt{f}}_0(v)=\iota_{\varphi}\circ w_{\varphi}\circ \ov{df}_0(v)$, this proves the validity of~\eqref{romyschneider}.
\end{proof}
\subsubsection{Sheaves $\cS_{\bP_L}(j)$ on $X\times \bP_L(j)$}\label{subsubsec:garroni}
Suppose that~\eqref{zerozerozerouno} holds. Let  $U\subset \Ext^1_X(\cF_3,\cF_1)$ be a subspace, let $W$ be a vector space such that~\eqref{kappadimw} holds, and let $L=(U,W)$. Let $j\in\{k-1,\ldots,0\}$. In the present subsubsection we construct 
 coherent sheaves 
$\cS_{\bP_L}(j)$ on $X\times \bP_L(j)$, for $j\in\{k-1,\ldots,0\}$,   such that
\begin{enumerate}
\item[($1_j$)]
$\cS_{\bP_L}(j)$ is flat over $\bP_L(j)$.
\item[($2_j$)]
 If $z\in \left(\bD_L(j)^s\setminus \bD_L(j)^{s+1}\right)$ there exist a vector space $V(z)$ of dimension $s$, a  subspace $U(z)\subset\Ext^1_X(\cF_3,\cF_1)$  of dimension $k-s$ containing $U$, and
 an exact sequence
\begin{equation}\label{cringe}
0\lra \cS_{U(z)}\lra {\cS_{\bP_L}(j)}_{|X\times\{z\}}\lra \cF_1\otimes V(z)\lra 0.
\end{equation}
\end{enumerate}
We have $\bP_W(k-1)=\bP_W$, and  $g_{k-1}\colon\bP_W(k-1)\to\bP_W$ is  the identity. We let
$\cS_{\bP_W}(k-1)=\cS_{\bP_W}$. Then 
Item~($1_{k-1}$) holds because  $\cS_{\bP_W}$ is an extension of sheaves which are flat over $\bP_W(j)$,  and 
Item~($2_{k-1}$) holds by Claim~\ref{clm:splitexseq}. 
If $\dim W=1$ only Items~($1_{0}$)  and 
Item~($2_{0}$) need to be checked, and  we have proved that they hold. The rest of the proof is 
by induction on  $k$ (the dimension of $W$) and,  once $U,W$ are fixed,   we prove that ($1_{k-1}$), ($2_{k-1}$),...,($1_{0}$), ($2_{0}$) hold  working our way downwards.  Hence we fix a subspace $U\subset \Ext^1_X(\cF_3,\cF_1)$ and a vector space $W$ with 
$2\le \dim W\le \dim \Ext^1_X(\cF_3,\cF_1)/U$, and we assume that the statements above  hold whenever $k<\dim W$, and that Items~($1_{j}$), ($2_{j}$)  hold for the given $U,W$. 

Our first task is to define $\cS_{\bP_W}(j-1)$.  Let $z\in \left(\bD_L(j)^s\setminus \bD_L(j)^{s+1}\right)$, and let $U(z)$, $V(z)$ be as in  
Item~($1_{j}$). 
Applying the functor 
$\Hom_X(-,\cF_1)$ to the exact sequence in~\eqref{cringe} one gets an isomorphism
\begin{equation}
\CC\Id_{\cF_1}\otimes V(z)^{\vee}\xrightarrow{\sim}\Hom_X({\cS_W(j)}_{|X\times\{z\}},\cF_1).
\end{equation}
(We have $\Hom_X(\cS_{U(z) } ,\cF_1)=0$ because  $\Hom_X(\cF_3,\cF_1)=0$.)  Since $\bD_L(j)^s=\es$ if $s>j$ it follows that there exists an exact sequence
\begin{equation}
0\lra \cK\lra {\cS_{\bP_L}(j)}_{|X\times \bD_L(j)^j}\lra \cQ\lra 0
\end{equation}
which gives the exact sequence in~\eqref{cringe} when restricted to $X\times \{z\}$ for every $z\in \bD_L(j)^j$.
Let  $h_{j-1}\colon X\times \bE_L(j-1)^j\to X\times  \bD_L(j)^j$ be the restriction of $\Id_X\times g_{j-1}$. 
Pulling back by $h_{j-1}$ we get the exact sequence
\begin{equation}
0\lra h_{j-1}^{*}\cK\lra h_{j-1}^{*}{\cS_{\bP_L}(j)}_{|X\times \bE_L(j-1)^j}\lra h_{j-1}^{*}\cQ\lra 0.
\end{equation}
Note that $h_{j-1}^{*}\cQ$ is flat over $\bE_L(j-1)^j$. 
Let $\nu\colon X\times \bE_L(j-1)^j\hra X\times \bP_L(j-1)^j$ be the inclusion map. We let  $\cS_{\bP_L}(j-1)$ be the sheaf fitting into the exact sequence
\begin{equation}\label{dageiageimeno}
0\lra \cS_{\bP_L}(j-1)\lra (\Id_X\times g_{j-1})^{*}\cS_{\bP_L(j)}\overset{\tau_{j-1}}{\lra} \nu_{*}\left(h_{j-1}^{*}\cQ\right)\lra 0,
\end{equation}
Let us prove that Item~($1_{j-1}$) holds. Let $\cA$ be a sheaf of 
$\cO_{X\times\bP_L(j)}$-modules
 (not necessarily coherent). Since  $h_{j-1}^{*}\cQ$ is flat over $\bE_L(j-1)^j$, and $\bE_L(j-1)^j$ is a Cartier divisor on $\bP_L(j)$, the spectral sequence for $\Tor$ in~\cite[Example~15.62.2]{stacks-project} gives that  the sheaf 
$Tor_2^{\cO_{\bP_L(j)}}(\nu_{*}\left(h_{j-1}^{*}\cQ\right),\cA)$ vanishes. By the long exact sequence of $Tor$'s 
that one gets by  tensorizing the exact sequence in~\eqref{dageiageimeno}, we get that 
$Tor_1^{\cO_{\bP_L(j)}}(\cS_{\bP_L}(j-1),\cA)$ vanishes. This proves that
$\cS_{\bP_L}(j-1)$ is flat over $\bP_L(j-1)$, i.e.~that  Item~($1_{j-1}$) holds.

Now we prove that Item~($2_{j-1}$) holds. Away from $\bE_L(j-1)^{\ge j}$ (see~\eqref{eccemaggugu}) the map $g_{j-1}$ defines an isomorphism 
$(\bP_L(j-1)\setminus \bE_L(j-1)^{\ge j})\xrightarrow{\sim} (\bP_L(j)\setminus \bD_L(j)^j)$ mapping 
$(\bD_L(j-1)^s\setminus \bE_L(j-1)^j)$ to $(\bD_L(j)^s\setminus \bD_L(j)^j)$. If 
$z\notin \bE_L(j-1)^{\ge j}$ then 
\begin{equation*}
 \cS_{\bP_L}(j-1)\cong {\cS_{\bP_L}(j)}_{|X\times\{g_{j-1}(z)\}}. 
\end{equation*}
Hence the statement in ($2_{j-1}$) holds for all $z\notin \bE_L(j-1)^{\ge j}$ because  ($2_{j}$) holds. It remains to prove that the statement in ($2_{j-1}$) holds for all $z\in \bE_L(j-1)^{\ge j}$. Let $s\in\{k-1,\ldots,j\}$ be such that 
\begin{equation}
z\in  (\bE_L(j-1)^{\ge s}\setminus\bE_L(j-1)^{\ge s+1}).
\end{equation}
(Notation as in~\eqref{eccemaggugu}.) 
Adopting the notation of Remark~\ref{rmk:prerospiglio} we get that
\begin{equation}
h_{j-1,\ge s}^L(z)=[\varphi]\in (\bD_L^{s}\setminus \bD_L^{s+1}).
\end{equation}
Key-Remark~\ref{key-rmk:rospiglione} gives an isomorphism
\begin{equation}
v_{\varphi}\colon (h_{j-1,\ge s}^L)^{-1}([\varphi]) \xrightarrow{\sim} \bP_{\bar{L}(\varphi )}(j-1),
\end{equation}
where $\bar{L}(\varphi )=(\ov{U}(\varphi ),\ov{W}(\varphi ))$ with $\ov{U}(\varphi )\subset\Ext^1_X(\cF_3,\cF_1)$  the subspace containing $U$ such that $\im\varphi=\ov{U}(\varphi )/U$, and 
$\ov{W}(\varphi )=\ker\varphi$.  Thus $\dim\ov{W}(\varphi )<\dim W$ and hence, by the inductive hypothesis, there exists a sheaf $\cS_{\bP_{\bar{L}(\varphi )}}(j-1)$ on $X\times \bP_{\bar{L}(\varphi )}(j-1)$ for which Items~($1_{j-1}$), ($2_{j-1}$) hold. By construction we have an isomorphism
\begin{equation}
\cS_{\bP_L}(j-1)_{|X\times (h_{j-1,\ge s}^L)^{-1}([\varphi])}\cong (\Id_X\times v_{\varphi})^{*}\cS_{\bP_{\bar{L}(\varphi )}}(j-1).
\end{equation}
 Item~($2_{j-1}$)  holds for $z\in (h_{j-1,\ge s}^L)^{-1}([\varphi])$ because it holds for 
$\cS_{\bP_{\bar{L}(\varphi )}}(j-1)$ and because
\begin{equation}
\bD_L(j-1)^s\cap (h_{j-1,\ge s}^L)^{-1}([\varphi])=v_{\varphi}^{-1}(\bD_{\bar{L}(\varphi )}(j-1)^s).
\end{equation}
\subsection{Extension of $\psi$ to $\wh{\cM}$}\label{subsec:estdipsi}
\setcounter{equation}{0}  
Recall that $\cM=\cM_{v_2}(S,h_S)$. 
We assume that there is a universal family $\EE_2$ on $S\times\cM$. If a universal family does not exist then it exist locally on $\cM$ (\'etale topology), and the proof that we give goes through by replacing $\cM$ 
with an \'etale  covering $\{\cU_i\}_{i\in I}$   such that  there exists a universal family  on $S\times\cU_i$ for each $i\in I$.
Let $\GG$ be the sheaf on $S^{[2]}\times \cM$ given by (see~\cite[Subsect.~2.1]{og:highdim})
\begin{equation}
 \GG\coloneq\cG(\cE_1\boxtimes\cO_{\cM},\EE_2).
\end{equation}
Then $\GG$ is flat over $\cM$ by Prop.~2.5 in loc.~cit., and for  $[\cE_2]\in\cM$ the restriction of $\GG$ to $S^{[2]}\times\{[\cE_2]\}$ is isomorphic to $\cG(\cE_1,\cE_2)$. For $x\in\cM$ let $\GG_x\coloneq\GG_{|S^{[2]}\times\{x\}}$. By Proposition~\ref{prp:stabuno} and Subsection~\ref{subsec:semirepl} the sheaf $\GG_x$ is stable if $x\notin\cD$.  We let 
\begin{equation}\label{miavisconti}
\begin{matrix}
\cM\setminus \cD & \overset{\psi}{\lra} & M_{\ww}^{\bu} \\
[\cE_2] & \mapsto & [\cG(\cE_1,\cE_2)]
\end{matrix}
\end{equation}
be the map  associated to the restriction of $\GG$ to $S^{[2]}\times(\cM\setminus\cD)$, see~\eqref{viscontioccupato}. 

In Subsection~\ref{subsec:scoppiascoppia} we have defined a chain of birational maps
\begin{equation*}
\wh{\cM}=\cM(0)\xrightarrow{f_0}\cM(1)\xrightarrow{f_1}\ldots \cM(a)\xrightarrow{f_a}\cM.
\end{equation*}
\begin{rmk}\label{rmk:pennica}
In Subsection~\ref{subsec:scoppiascoppia} we have defined closed prime divisors $E(j)^s\subset\cM(j)$ for $ s>j$ (note: $E(j)^s=\es$ if $s>a$) with the following properties:
\begin{enumerate}
\item
$f_a\circ f_{a-1}\circ\ldots\circ f_j$ maps $\cM(j)_{*}\coloneq\cM(j)\setminus (E(j)^a\cup E(j)^{a-1}\cup\ldots\cup E(j)^{j+1})$ isomorphically to 
$\cM\setminus \cD^{j+1}$. 
\item
Let  $k>j$. Then $f_a\circ f_{a-1}\circ\ldots\circ f_j$ maps $ E(j)^k\setminus E(j)^{k+1}$ to $\cD^k\setminus\cD^{k+1}$. 
\end{enumerate}
In particular we have the disjoint union
\begin{equation}\label{stratiemmegei}
\cM(j)=\cM(j)_{*}\sqcup \left(E(j)^{j+1}\setminus E(j)^{j+2}\right)\sqcup
\ldots\sqcup \left(E(j)^{a-1}\setminus E(j)^{a}\right)\sqcup E(j)^{a}.
\end{equation}
For $j=0$, i.e.~for $\wh{\cM}$, we get the disjoint union
\begin{equation}\label{stratiemmecappello}
\wh{\cM}=\left(\cM\setminus\cD\right)\sqcup \left(\wh{E}^1\setminus \wh{E}^2\right)\sqcup
\ldots\sqcup \left(\wh{E}^{a-1}\setminus \wh{E}^{a}\right)\sqcup \wh{E}^{a}.
\end{equation}
\end{rmk}
In the present subsection we prove the following key result.
\begin{prp}\label{prp:estendopsi}
The map $\psi$ extends to a regular map  $\wh{\psi}\colon \wh{\cM}\to M_{\ww}^{\bullet}$. Let  $x\in \wh{E}^1\cup\ldots \cup \wh{E}^a$. The
 $h_{S^{[2]}}$ slope stable vector bundle  $\cF$ on $S^{[2]}$ such that $[\cF]=\wh{\psi}(x)$ is described as follows.
Let $k\in\{1,\ldots,a\}$ be such that $x\in \wh{E}^k\setminus \wh{E}^{k+1}$. 
Let $f_a\circ\ldots\circ f_0(x)=[\cE_2]$ (hence $[\cE_2]\in\cD^k\setminus\cD^{k+1}$). 
 If $k<a$ then $\cF\cong\cB(\cH,J_{\cE_2})$ where $\cH$ 
is as 
in~\eqref{ravel}, and
$J_{\cE_2}$ is as in~\eqref{geiedue}.  If $k=a$ then  $\cF\cong\cB(p,\im(f(p)^t))$ where $p$, $\im(f(p)^t)$ are as 
in~\eqref{fabrizio} and~\eqref{deandre}
 respectively.  
\end{prp}
The proof of the above proposition is at the end of the subsection.
\begin{dfn}\label{dfn:fibramenoa}
Let  $[\cE_2]\in(\cD^k\setminus\cD^{k+1})$. If $a>k>0$ let $\cH$ be the vector bundle on $S$ fitting into the exact 
sequence~\eqref{ravel}, let $J_{\cE_2}$ be as in~\eqref{geiedue}, and set
\begin{equation*}
\cF_1\coloneq \cE_1[2]^{-},\qquad \cF_3\coloneq \cB(\cH,J_{\cE_2})^{+}.
\end{equation*}
If $k=a$ let $p$ be the singular point of $\cE_2$,  let $V_a$ be as in~\eqref{fabrizio} (note that $V_a$ is canonically identified with
 $\Hom_S(\cE_2,\cE_1)^{\vee}$), and set
\begin{equation*}
\cF_1\coloneq \cE_1[2]^{-},\qquad \cF_3\coloneq \cB(p,V_a)^{+}.
\end{equation*}
Then the equalities in~\eqref{zerozerozerouno} hold (see~\eqref{dunant}, Proposition~\ref{prp:extrigidibical}, and 
Proposition~\ref{prp:alglin}). 
Let 
\begin{equation}
\bP_{\cE_2}\coloneq \PP(\Hom_S(\cE_2,\cE_1)^{\vee}\otimes  \Ext^1_X(\cF_3,\cF_1)).
\end{equation}
In other words $\bP_{\cE_2}= \bP_{L}$, where $L\coloneq(\Hom_S(\cE_2,\cE_1),\{0\})$ (see~\eqref{defpielle}). This makes sense 
because the inequality in~\eqref{kappadimw} holds, in fact it is an equality.  Let $\bD_{\cE_2}(j)\coloneq \bD_{L}(j)$ and 
$\bP_{\cE_2}(j)\coloneq \bP_{L}(j)$. 
\end{dfn}
Let  $[\cE_2]\in(\cD^k\setminus\cD^{k+1})$, and let 
 $k>j\ge 0$. 
We define an  isomorphism
\begin{equation}\label{strangelove}
\theta_{\cE_2}(j)\colon \bP_{\cE_2}(j)\xrightarrow{\sim}  (f_a\circ f_{a-1}\circ\ldots\circ f_j)^{-1}([\cE_2])
\end{equation}
as follows. Away from $\cD^{k+1}$ the map $f_a\circ f_{a-1}\circ\ldots\circ f_{k-1}\colon \cM(k-1) \to \cM$   is the blow-up of 
 $\cD^k\setminus\cD^{k+1}$. Hence, referring to~\eqref{vivamarcello}, we have the isomorphism
\begin{equation*}
\theta_{\cE_2}(k-1)\coloneq \PP((\beta_{\cE_2}\circ\alpha_{\cE_2})^{-1})\colon \bP_{\cE_2}(k-1)\xrightarrow{\sim} 
(f_a\circ f_{a-1}\circ\ldots\circ f_j)^{-1}([\cE_2]).
%
%
\end{equation*}
 By
 Lemma~\ref{lmm:conostrato} $\theta_{\cE_2}(k-1)$ defines an isomorphism
\begin{equation}\label{moodforlove}
\bD_{\cE_2}^{h}\xrightarrow{\theta_{\cE_2}(k-1)_{|\ldots}} \cD(k-1)^{h}\cap (f_a\circ f_{a-1}\circ\ldots\circ f_{k-1})^{-1}([\cE_2]).
\end{equation}
 If $j<k-1$  we have series of birational maps
\begin{equation*}
(f_a\circ f_{a-1}\circ\ldots\circ f_j)^{-1}([\cE_2])\xrightarrow{f_{j|\ldots}}\ldots\xrightarrow{f_{k-2|\ldots}} (f_a\circ f_{a-1}\circ\ldots\circ f_{k-1})^{-1}([\cE_2])
\end{equation*}
which are identified with the composition $\bP_{\cE_2}(j)\xrightarrow{g_j}\ldots \xrightarrow{g_{k-2}}\bP_{\cE_2}(k-1)$ of the blow ups in 
  Definition~\ref{dfn:compcoll}  because of the isomorphism in~\eqref{moodforlove}. This identifcation defines the isomorphism 
  in~\eqref{strangelove}.
  
  Note that  if $h\le j$ then
  $\theta_{\cE_2}(j)$ defines an isomorphism 
\begin{equation}\label{chansondhelene}
{\theta_{\cE_2}(j)}_{|\ldots}\colon 
\bD_{\cE_2}(j)^h\xrightarrow{\sim} \cD(j)^h\cap (f_a\circ f_{a-1}\circ\ldots\circ f_j)^{-1}([\cE_2]).
\end{equation}
\begin{prp}\label{prp:grandinduz}
For $j\in\{a,a-1,\ldots,0\}$ there exists a sheaf $\GG(j)$ on $S^{[2]}\times \cM(j)$ which is flat over  $\cM(j)$ and such that the following hold:
\begin{enumerate}
\item[($A_j$)]
Let $x\in\cM(j)\setminus (E(j)^a\cup\ldots\cup E(j)^{j+1})$. Then 
$\GG(j)_x\cong \GG_{f_a\circ\ldots f_j(x)}$.
\item[($B_j$)]
Let  $[\cE_2]\in(\cD^k\setminus\cD^{k+1})$, and let 
 $k>j\ge 0$. Then we have an isomorphism
\begin{equation}\label{varazze}
(\Id_{S^{[2]}}\times \theta_{\cE_2}(j))^{*}\GG(j)\cong \cS_{\bP_{\cE_2}}(j),
\end{equation}
where $\cS_{\bP_{\cE_2}(j)}$ is the sheaf on $S^{[2]}\times\bP_{\cE_2}(j)$ of Subsubsection~\ref{subsubsec:garroni}. 
%
\end{enumerate}
\end{prp}
\begin{proof}
By definition $\cM(a)=\cM$. We set $\GG(a)\coloneq \GG$. Assuming that  there exists a sheaf $\GG(j)$ such that Items~($A_j$), ($B_j$) hold for $j\in \{a,a-1,\ldots,1\}$, we prove that   there exists a sheaf $\GG(j-1)$ such that Items~($A_{j-1}$), ($B_{j-1}$)
hold.    Let $x\in \cD(j)^j$, and let $\GG(j)_x\coloneq \GG(j)_{S^{[2]}\times\{x\}}$. 
We claim that 
\begin{equation}\label{marcoguardo}
\hom_{S^{[2]}}(\GG(j)_x,\cE_1[2]^{-})=j,
\end{equation}
and that the tautological map
\begin{equation}\label{mottana}
\GG(j)_x\lra \Hom_{S^{[2]}}(\GG(j)_x,\cE_1[2]^{-})^{\vee}\otimes \cE_1[2]^{-}
\end{equation}
is surjective. To see why,  suppose first that $x\in \cD(j)^j\setminus (E(j)^a\cup\ldots\cup E(j)^{j+1})$. Let $[\cE_2]\coloneq f_a\circ\ldots f_j(x)$. Note that 
$[\cE_2]\in\cD^j\setminus\cD^{j+1}$ (see Remark~\ref{rmk:pennica}). If $a>j$ then~\eqref{marcoguardo} and surjectivity of~\eqref{mottana} follow from the exact sequence in~\eqref{nicolpatente}, the isomorphism in~\eqref{midnightcowboy}, and Proposition~\ref{prp:extrigidibical}. If $a=j$ then~\eqref{marcoguardo} and surjectivity of~\eqref{mottana} follow from~\eqref{cavatina}, \eqref{maduro}, and Proposition~\ref{prp:alglin}. Now suppose that $x\in \cD(j)^j\cap (E(j)^a\cup\ldots\cup E(j)^{j+1})$. Then 
(see~\eqref{stratiemmegei}) there exists $a\ge k>j$ such that 
$x\in \cD(j)^j\cap E(j)^k\setminus  E(j)^{k+1}$. By the isomorphism in~\eqref{chansondhelene} we have $z\in\bD_{\cE_2}(j)^j$. 
By Item~($B_j$)  and Item~($2_j$) of Subsubsection~\ref{subsubsec:garroni} there exists an exact sequence
\begin{equation}
0\lra \cS_{U(z)}\lra \GG(j)_x\lra \cE_1[2]^{-}\otimes V(z)\lra 0,
\end{equation}
where $\dim V(z)=j$. Since $\Hom_{S^{[2]}}(\cS_{U(z)},\cE_1[2]^{-})=\{0\}$ this proves that~\eqref{marcoguardo} holds and also the surjectivity of the map in~\eqref{mottana}.

Recall that $E(j-1)^j=f_{j-1}^{-1}(\cD(j)^j)$. By the result on $\GG(j)_x$ for $x\in \cD(j)^j$
that we have just proved there exist a vector bundle $\cU^j$ on $E(j-1)^j$ (pulled back from a vector bundle on $\cD(j)^j$) and a surjection
\begin{equation}\label{gala}
(\Id_{S^{[2]}}\times f_{j-1})^{*}\GG(j)_{|S^{[2]}\times E(j-1)^j}\lra \cE_1[2]^{-}\boxtimes\cU^j\lra 0.
\end{equation}
Let $\GG(j-1)$ be  the sheaf on $S^{[2]}\times \cM(j-1)$  fitting into the exact sequence
\begin{equation}\label{valtellina}
0\lra \GG(j-1)\lra (\Id_{S^{[2]}}\times f_{j-1})^{*}\GG(j)\overset{\tau}{\lra} \iota_{*}\left(\cE_1[2]^{-}\boxtimes\cU^j\right)\lra 0,
\end{equation}
where $\iota\colon S^{[2]}\times E(j-1)^j\to S^{[2]}\times \cM(j-1)$ is the inclusion map, and the map $\tau$ is defined by the surjection in~\eqref{gala}. 
The sheaf $\GG(j-1)$ is flat over $\cM(j-1)$ by Lemma~\ref{lmm:modificapiatta}. It remains to prove that Items~($A_{j-1}$), ($B_{j-1}$)
hold. Let $x\in\cM(j)\setminus (E(j-1)^a\cup\ldots\cup E(j-1)^{j})$. We have an isomorphism $\GG(j-1)_x\cong \GG(j)_{f_{j-1}(x)}$ because $x\notin E(j-1)^j$. Since $f_{j-1}(x)\notin E(j)^a\cup\ldots\cup E(j)^{j+1}$ we get that  Item~($A_{j-1}$) holds. Lastly we prove that Item~($B_{j-1}$) 
holds.  First we note that we have a commutative diagram
\begin{equation*}
\xymatrix{
\quad\bP_{\cE_2}(j-1)\quad \ar[r]^{\theta_{\cE_2}(j-1)\qquad\  }\ar[d]^{g^{\cE_2}_{j-1}} &  \quad(f_a\circ\ldots\circ f_{j-1})^{-1}([\cE_2])\ar[d]^{f_{j-1|\ldots}} & \\
\bP_{\cE_2}(j) \ar[r]^{\theta_{\cE_2}(j)\qquad} &   \ \ \, (f_a\circ\ldots\circ f_{j})^{-1}([\cE_2])
  }
\end{equation*}
Pulling back the  exact sequence in~\eqref{valtellina} by the map $\Id_{S^{[2]}}\times \theta_{\cE_2}(j-1)$, and recalling the
 isomorphism in~\eqref{varazze}, we get the exact sequence
\begin{equation*}
0\to (\Id_{S^{[2]}}\times \theta_{\cE_2}(j-1))^{*}\GG(j-1)\lra (\Id_{S^{[2]}}\times g^{\cE_2}_{j-1})^{*}\cS_{\bP_{\cE_2}}(j)
\overset{\pi_{j-1}}{\lra} 
 \nu_{*}\left(\cR\right)\to 0,
\end{equation*}
where $\nu\colon S^{[2]}\times \bE_{\cE_2}(j-1)^j\hra S^{[2]}\times \bP_{\cE_2}(j-1)$ is the inclusion map, the sheaf $\cR$ and the map 
$\pi_{j-1}$
are identified with the sheaf $h_{j-1}^{*}\cQ$ and the map 
$\tau_{j-1}$ in~\eqref{dageiageimeno} respectively. By the exact sequence  in~\eqref{dageiageimeno} it follows that  Item~($B_{j-1}$) 
holds. 
\end{proof}
\begin{proof}[Proof of Proposition~\ref{prp:estendopsi}]
Let $\wh{\GG}\coloneq \GG(0)$. Thus $\wh{\GG}$ is a sheaf on $S^{[2]}\times\wh{\cM}$ which is flat over $\wh{\cM}$ (recall that $\wh{\cM}=\cM(0)$). Let $x\in\wh{\cM}$, and let $[\cE_2]\coloneq f_a\circ\ldots\circ f_0(x)$.  

If 
$x\notin \wh{E}^a\cup\ldots\cup\wh{E}^1$ then 
\begin{equation}\label{davidbyrne}
\wh{\GG}_x\cong\cG(\cE_1,\cE_2)
\end{equation}
 by Item~($A_0$) of Proposition~\ref{prp:grandinduz}. 
Moreover  $[\cE_2]\in\cM\setminus\cD$ (because $x\notin \wh{E}^a\cup\ldots\cup\wh{E}^1$) and hence $\cG(\cE_1,\cE_2)$ is 
$h_{S^{[2]}}$ slope stable by Proposition~\ref{prp:stabuno}. 

If $x\in \wh{E}^a\cup\ldots\cup\wh{E}^1$ then by~\eqref{stratiemmecappello} there exists $k\in\{1,\ldots,a\}$ such that 
$x\in\wh{E}^k\setminus\wh{E}^{k+1}$  (if $k=a$ this means that $x\in\wh{E}^a$).  If $k<a$ let $\cH$ and $J_{\cE_2}$ be as in 
Definition~\ref{dfn:fibramenoa}, 
if $k=a$ let $p$ and $V_a$ be as in loc.~cit.
By the isomorphism in~\eqref{strangelove} there exists $z\in \bP_{\cE_2}(0)$ such that $\theta_{\cE_2}(0)(z)=x$, and by Item~($B_0$) of Proposition~\ref{prp:grandinduz} we have 
\begin{equation}\label{CdD}
\wh{\GG}_x\cong {\cS_{\bP_{\cE_2}(0)}}_{|S^{[2]}\times\{z\}}.
\end{equation}
By~\eqref{catenadielle} (see also Definition~\ref{dfn:compcoll}) $z\in\bD_{\cE_2}(0)\setminus \bD_{\cE_2}(1)$. By Item~($2_0$) of Subsubsection~\ref{subsubsec:garroni} the right-hand side of~\eqref{CdD} is isomorphic to $\cS_{U(z)}$ where 
\begin{equation}\label{manuela}
U(z)\subset 
\begin{cases}
\Ext^1_{S^{[2]}}(\cB(\cH,J_{\cE_2})^{+}\cong\Ann J_{\cE_2} & \text{if $k<a$,}\\
\Ext^1_{S^{[2]}}(\cB(p,V_a)^{+}\cong\Ann V_a & \text{if $k=a$,}
\end{cases}
\end{equation}
is a subspace of dimension $k$. (The isomorphisms above are given by Propositions~\ref{prp:extrigidibical} and~\ref{prp:alglin}.) Since 
$\dim\Ann J_{\cE_2}=k$ and $\dim\Ann V_a=a$ ($=k$),  the inclusion in~\eqref{manuela} is an equality. It follows that 
\begin{equation}\label{prontosoccorso}
\wh{\GG}_x\cong 
\begin{cases}
\cB(\cH,J_{\cE_2}) & \text{if $k<a$,}\\
\cB(p,V_a) & \text{if $k=a$.}
\end{cases}
\end{equation}
Hence $\wh{\GG}_x$ is  locally-free and $h_{S^{[2]}}$ slope stable. In fact local-freeness is clear if $k<a$ because $\cB(\cH,J_{\cE_2})$ is an extension of locally-free sheaves, and if $k=a$ it is Corollary~\ref{crl:localmentelibero}, while stability of $\cB(\cH,J_{\cE_2})$ is the content of Proposition~\ref{prp:stabcalbi}, and stability of $\cB(p,V_a)$ is the content of Proposition~\ref{prp:calbisingstab}. Since  $\wh{\GG}$ is a  flat 
family of locally-free $h_{S^{[2]}}$ slope stable sheaves on $S^{[2]}$ parametrized by $\wh{\cM}$ (and hence locally-free itself), and~\eqref{davidbyrne} holds for a general $x\in\wh{\cM}$, it induces a regular classifying morphism $\wh{\psi}\colon\wh{\cM}\to  M_{\ww}^{\bullet}$. The 
isomorphism in~\eqref{prontosoccorso} gives the description of the vector bundle $\cF$ such that $[\cF]=\wh{\psi}(x)$ for
 $x\in\wh{E}^k\setminus\wh{E}^{k+1}$.
\end{proof}
\subsection{Proof of  Theorem~\ref{thm:dadueauno}}\label{subsec:alfinlaprova}
\setcounter{equation}{0}  
Let $\Gamma\subset\cM\times   M_{\ww}^{\bullet}$ be the image of $\wh{\cM}\xrightarrow{\wh{f}\times\wh{\psi}}\cM\times   M_{\ww}^{\bullet}$.
Then $\Gamma$ is closed, and by  Proposition~\ref{prp:estendopsi} the projection $p_{\gamma}\colon \Gamma\to \cM$ is bijective. Hence $p_{\gamma}$ is an isomorphism because $\cM$ is smooth. Thus $\wh{\psi}$ descends to a regular map $\ov{\psi}\colon\cM\to  M_{\ww}^{\bullet}$ which extends $\psi$. This proves Item~(a). Item~(b) follows at once from Proposition~\ref{prp:estendopsi}.
 It remains to prove Item~(c), i.e.~that $\ov{\psi}$ is bijective. It is surjective because the image is closed and dense in $  M_{\ww}^{\bullet}$. We proceed to prove that
$\ov{\psi}$ is injective. Let $[\cE_2]\in\cM$. If  $[\cE_2]\in\cD^k\setminus\cD^{k+1}$ with $1\le k<a$ let $\cH$ and $J_{\cE_2}$ be as in 
Definition~\ref{dfn:fibramenoa}, 
if $[\cE_2]\in\cD^a$ let $p$ and $V_a$ be as in loc.~cit. 
Let $\cF(\cE_2)$ be the $h_{S^{[2]}}$  slope stable vector bundle on $S^{[2]}$ such that 
$\ov{\psi}([\cE_2])=[\cF(\cE_2)]$. 
Then 
\begin{equation*}
\hom_{S^{[2]}}(\cF(\cE_2),\cE_1[2]^{+})=\hom_{S^{[2]}}(\cE_1[2]^{-},\cF(\cE_2))=
\begin{cases}
0 & \text{if $[\cE_2]\in\cM^{*}\coloneq\cM\setminus\cD$,} \\
k & \text{if $[\cE_2]\in\cD^k\setminus\cD^{k+1}$.}
\end{cases}
\end{equation*}
In fact the  equality for $[\cE_2]\in\cM^{*}$ follows from Lemma~\ref{lmm:mattiatorre} for $\cH=\cE_2$, because $\cF(\cE_2)\cong\cG(\cE_1,\cE_2)$. 
 The  equality for  $[\cE_2]\in\cD^k\setminus\cD^{k+1}$
follows from the exact sequences~\eqref{bicalsucc}, \eqref{bicaltildesucc} and~\eqref{mantegazza},  Lemma~\ref{lmm:mattiatorre}  if $k<a$, and from  the exact sequences~\eqref{bimenosing}, \eqref{bicalsing}
 and~\eqref{mantegazza}, Proposition~\ref{prp:alglin} if $k=a$. Suppose that $[\cE_2],[\cE'_2]\in\cM$ and that $\ov{\psi}([\cE_2])=\ov{\psi}([\cF(\cE'_2)])$. By Items~(1), (2) above $[\cE_2],[\cE'_2]$ belong to the same stratum of the stratification
\begin{equation*}
\cM=\cM^{*}\sqcup\left(\cD^1\setminus\cD^{2}\right)\sqcup\ldots\sqcup \left(\cD^k\setminus\cD^{k+1}\right)\sqcup\ldots\sqcup\cD^a.
\end{equation*}
If $[\cE_2],[\cE'_2]\in\cM^{*}$ then $\cF(\cE_2)\cong\cG(\cE_1,\cE_2)$, $\cF(\cE'_2)\cong\cG(\cE_1,\cE'_2)$. Thus 
$\cG(\cE_1,\cE_2)\cong\cG(\cE_1,\cE'_2)$. It follows at once from the BKR correspondence that $\cE_2\cong\cE'_2$ 
and hence $[\cE_2]=[\cE'_2]$.

Suppose that $[\cE_2],[\cE'_2]\in\cD^k\setminus\cD^{k+1}$ with $1\le k<a$. 
Then $\cF(\cE_2)\cong\cB(\cH,J_{\cE_2})$ and $\cF(\cE'_2)\cong\cB(\cH',J_{\cE'_2})$. Thus 
$\cB(\cH,J_{\cE_2})\cong\cB(\cH',J_{\cE'_2})$. 
It follows from  the exact sequences~\eqref{bicalsucc}, \eqref{bicaltildesucc} and Item~(2) above that 
$\cB(\cH,J_{\cE_2})^{+}\cong\cB(\cH',J_{\cE'_2})^{+}$ and 
$\cG(\cE_1,\cH)\cong \cG(\cE_1,\cH')$.  By the BKR correspondence we get that  $\cH\cong\cH'$. Thus (see~\eqref{estensione}) we have exact sequences
\begin{equation}\label{baldassarre}
0\lra \cH\lra \cE_2\overset{\varphi}{\lra}\cE_1 \otimes \Hom_S(\cE_2,\cE_1)^{\vee}\lra 0 
\end{equation}
and 
\begin{equation}\label{peruzzi}
0\lra \cH\lra \cE'_2\overset{\varphi}{\lra}\cE_1 \otimes \Hom_S(\cE'_2,\cE_1)^{\vee}\lra 0,
\end{equation}
with extension classes given by inclusions $\iota\colon \Hom_S(\cE'_2,\cE_1)^{\vee}\hra \Ext^1_S(\cE_1,\cH)$ and  
$\iota'\colon \Hom_S(\cE'_2,\cE_1)^{\vee}\hra \Ext^1_S(\cE_1,\cH)$. 
Since  $\cB(\cH,J_{\cE_2})^{+}\cong\cB(\cH,J_{\cE'_2})^{+}$, the extension classes of the exact sequences (see~\eqref{bicalsucc})
\begin{equation*}
0\lra  \cG(\cE_1,\cH)\lra \cB(\cH,J_{\cE_2})^{+}\overset{\lambda^{+}}{\lra} \cE_1[2]^{+}\otimes J_{\cE_2}\lra 0
\end{equation*}
and
\begin{equation*}
0\lra  \cG(\cE_1,\cH)\lra \cB(\cH,J_{\cE'_2})^{+}\overset{\lambda^{+}}{\lra} \cE_1[2]^{+}\otimes J_{\cE'_2}\lra 0
\end{equation*}
are the same (up to $\CC^{*}$). By Proposition~\ref{prp:extrigidibical} we get that $J_{\cE_2}=J_{\cE'_2}$, i.e.~$\im(\iota)=\im(\iota')$. This proves that the extension classes of~\eqref{baldassarre} and~\eqref{peruzzi} are the same. Thus $\cE_2\cong\cE'_2$ 
and hence $[\cE_2]=[\cE'_2]$.

If  $[\cE_2],[\cE'_2]\in\cD^a$ the proof that $[\cE_2]=[\cE'_2]$ is analogous.
\qed
\section{The Donaldson-Mukai map}\label{sec:conticonti}
\subsection{The Mukai map: recap}
\setcounter{equation}{0}  
Let $S$ be a  $K3$ surface. 
Let $\cF$ be a sheaf on $S\times T$, flat over $T$.
 Set 
\begin{equation}
\begin{matrix}
H(S) & \overset{\theta_{\cF}}{\lra} & H^2(T) \\
x & \mapsto & q_{T,*}\left[\ch(\cF)\cdot q_S^{*}(x^{\vee}\cdot \Td_S^{1/2})\right]_6,
\end{matrix}
\end{equation}
where  $q_S,q_T$ are the projections of $S\times T$ to $S$ and $T$,
$(x_0+x_2+x_4)^{\vee}\coloneq x_0-x_2+x_4$ if 
 $x_{2i}\in H^{2i}(S)$, and the subscript $6$ denotes the component in $H^6(S\times S^{[2]})$. 
 
Now assume that the Mukai vector of $\cF_{|S\times\{t\}}$ is equal to a fixed $v$:
\begin{equation}
v(\cF_{|S\times\{t\}})=v=(r,l,s)\in H^0(S;\ZZ)_{>0}\oplus\NS(S)\oplus H^4(S;\ZZ)\qquad \forall t\in T.
\end{equation}
Let $\cL$ be a line-bundle on $T$. The restrictions of $\theta_{\cF}$ and  $\theta_{\cF\otimes q_T^{*}\cL}$ 
to the  orthogonal of $v$ for Mukai's symmetric bilinear pairing (defined as $\la x,y\ra\coloneq \int_S(-x^{\vee}\cdot y)$) are the same (they are not equal on all of $H(S)$).  

Suppose that $h_S$ is a polarization of $S$, and let $\cM_v\coloneq \cM_v(S,h_S)$. Assume   that there exists a universal sheaf $\cF_v$ on $S\times\cM_v$ (in particular $v$ is primitive).  One sets
\begin{equation}
\begin{matrix}
v^{\bot} & \overset{\theta_{v}}{\lra} & H^2(\cM_{v}) \\
x & \mapsto & q_{\cM_{v},*}\left[\ch(\cF_v)\cdot q_S^{*}(x^{\vee}\cdot \Td_S^{1/2})\right]_6.
\end{matrix}
\end{equation}
Suppose that  $v^2>0$, i.e.~$\dim\cM_v>2$. Then $\theta_{v}$ is an isomorphism of integral Hodge structures, if $F^2 H(S)\coloneq F^2 H^2(S)$ and 
$F^1 H(S)\coloneq (F^1 H^2(S)\oplus H^0(S)\oplus H^4(S))$, and it matches Mukai's  pairing and the BBF quadratic form on $H^2(\cM_v)$, see~\cite{og:accaduespmod,yoshioka:modfvsupab}. 

Here and in the sequel the following consequence of the Grothendieck-Riemann-Roch Theorem will be handy.
\begin{lmm}\label{lmm:grrinclusione}
Let $i\colon X\hra Y$ be the inclusion of a smooth closed subvariety $X$ of pure codimension $d$ in a smooth variety $Y$. Let $\cF$ be a coherent sheaf on $X$. Then, \emph{modulo $H^{>2d+2}(X;\QQ)$}, we have
\begin{equation}
\ch(i_{*}\cF)\equiv r(\cF)\cl(Y)+i_{*}c_1(\cF)-\frac{r(\cF)}{2}\cl(Y)\cdot c_1(Y)+\frac{r(\cF)}{2}i_{*}c_1(X).
\end{equation}
\end{lmm}
Let $v_0\coloneq(1,0,-1)$.  Associating  to $[Z]\in S^{[2]}$ the isomorphism class of the ideal sheaf 
$\cI_Z$ we get an isomorphism $S^{[2]}\xrightarrow{\sim}\cM_{v_0}$.  We identify $S^{[2]}$  with  $\cM_{v_0}$ via this map. 
 If $\cZ\subset S\times S^{[2]}$ is the universal  subscheme then 
$\cI_{\cZ}$ is a universal sheaf on $S\times \cM_{v_0}$. Hence we have the isomorphism $\theta_{v_0}\colon v_0^{\bot}\xrightarrow{\sim} H^2(S^{[2]})$. 
One checks that 
\begin{equation}\label{eccodelta}
\delta\coloneq  \theta_{v_0}(1,0,1)=\cl(\{[Z]\in S^{[2]}\mid \text{$Z$ is non-reduced}).
\end{equation}
 (Use Lemma~\ref{lmm:grrinclusione}.) For $\beta\in H^2(S)$ we set
\begin{equation}\label{eccomugrasso}
{\bm\mu}(\beta)\coloneq \theta_{v_0}(0,\beta,0).
\end{equation}
The notation agrees with that in~\cite[(2.2.6)-(2.2.7)]{og:highdim} (except that here $\mu$ morphed into ${\bm\mu}$).
\begin{rmk}\label{rmk:rinogaetano}
Let $v_2$ be as in~\eqref{verovudue}. If $m_0\equiv 3\pmod{4}$ then there exists a universal sheaf on $S\times\cM_{v_2}$, see ~\cite[Rmk.~A.7]{mukvb}. If $m_0\equiv 1\pmod{4}$ there exists a universal sheaf if there exists a divisor $C$ on $S$ such that $C\cdot D$ is odd, see loc.~cit. We have (independently of the existence of a universal sheaf)
\begin{equation}\label{classedivdet}
\theta_{v_2}(v_1)=\cl(\cD_{v_2}).
\end{equation}
(This makes sense because $v_1\bot v_2$.)
\end{rmk}
\subsection{Families of $\PP^n$-bundles and the Donaldson-Mukai map}\label{subsec:bladerunner}
\setcounter{equation}{0}  
Let $X$ be a HK manifold, let $T$ be a complex space,  and let $\cP\to X\times T$ a $\PP^n$-bundle. One may define a 
Donaldson-Mukai map $\lambda_{\cP}\colon H^2(X)\to H^2(T)$ proceeding as follows. 
Let $2n$ be the dimension of $X$. Let 
$\bD\colon H^2(X)\to H^{4n-2}(X)$ be the composition of the  isomorphism $H^2(X)\overset{\sim}{\lra} H^2(X)^{\vee}$ defined by the BBF bilinear symmetric form and  the  isomorphism $H^2(X)^{\vee}\overset{\sim}{\lra}  H^{4n-2}(X)$ given by Poincar\`e duality.  Thus $\bD$ is characterized by the equality
\begin{equation}
\int_X\bD(\alpha)\wedge\beta=\la\alpha,\beta\ra_X,\qquad \forall \alpha,\beta\in H^2(X),
\end{equation}
where $\la-,-\ra_X$ is the BBF bilinear symmetric form of $X$. 
Let $p_X,p_T$ be the projections of $X\times T$ to the two factors. The \emph{Donaldson-Mukai map} associated to $\cP$ is given by
\begin{equation}\label{donmukai}
\begin{matrix}
H^2(X) & \overset{\lambda_{\cP}}{\lra} & H^2(T) \\
\alpha & \mapsto & p_{T,*}\left[c_2(\gotg(\cP))\cdot p_X^{*}(\bD(\alpha))\right]
\end{matrix}
\end{equation}
\begin{rmk}\label{rmk:casoproj}
Suppose that $F$ is a vector bundle on $X\times T$. Then $\gotg(\PP(F))\cong End^0(F)$ (see~\eqref{senzatraccia}),
 and hence
\begin{equation}
\lambda_{\PP(F)}(\alpha)=p_{T,*}\left[\Delta(F)\cdot p_X^{*}(\bD(\alpha))\right].
\end{equation}
\end{rmk}
\begin{rmk}\label{rmk:tantiacca}
Let $X$ be projective of dimension $2n$, and $H$ be a very ample divisor on $X$. Let $D_1,\ldots,D_{2n-1}\in|H|$, with transverse intersection $C=D_1\cap\ldots\cap D_{2n-1}$.  Let $h\coloneq \cl(H)\in H^2(X)$, and $c\coloneq \cl(C)\in H_2(X)$. Then
\begin{equation}\label{riducoacurva}
\lambda_{\cP}(h)=\frac{1}{c_X\cdot(2n-1)!!}c_2(\gotg(\cP))/c,
\end{equation}
where $/$ denotes slant product. In fact the above formula follows from the equality
\begin{equation}
\bD(h)=\frac{1}{c_X\cdot(2n-1)!!} h^{2n-1}.
\end{equation}
\end{rmk}
Let $\ov{\ww}$ be a mock Mukai vector for $X$, see~\eqref{doppiavubarra}.  
Let 
$\cP_{\ov{\ww}}\to X\times M_{\ov{\ww}}(X,\omega)$
be the universal bundle of projective spaces, i.e. 
\begin{equation}
{\cP_{\ov{\ww}}}_{|X\times \{[\cP]\}}\cong \cP.
\end{equation}
for any  $[\cP]\in M_{\ov{\ww}}(X,\omega)$. 
Set 
\begin{equation}\label{donmukmoduli}
\lambda_{\ov{\ww}}\coloneq \lambda_{\cP_{\ov{\ww}}}.
\end{equation}
Below is a result valid in the case that $\omega=c_1(L)$ where $L$ is an ample line bundle.
\begin{prp}\label{prp:ampiolambda}
Let $(X,L)$ be a polarized HK manifold. Let $Y\subset M_{\ov{\ww}}(X,\omega)$ be a projective subscheme. The restriction of  
$\lambda_{\ov{\ww}}(c_1(L))$ to  $Y$ is ample.
\end{prp}
\begin{proof}
Let $2n=\dim X$. There exists $m$ such that $L^{\otimes m}$ is very ample and the following hold.
\begin{enumerate}
\item
If $D_1,\ldots,D_{2n-1}\in |L^{\otimes m}|$ have transverse intersection $C=D_1\cap\ldots\cap D_{2n-1}$, then the restriction $\cP_{|C}$ is slope stable for every $[\cP]\in Y$.
\item
If  $[\cP]\in Y$ then the map $H^1(X,\gotg(\cP))\to H^1(C,\gotg(\cP)_{|C})$ is injective.
\end{enumerate}
In fact Item~(1) holds by Theorem~\ref{thm:princadj} and the analogous Bogomolov Effective Restriction Theorem for slope-stable vector bundles, see~\cite{bogomolov:restriction} or~\cite[Thm.~7.3.5]{huy-lehn:librofasci} if $\dim X=2$, and~\cite[Thm.~5.2]{langer:semistposchar} in general.
Item~(2) holds by Serre's classical vanishing results. Hence we have a regular restriction map
\begin{equation}
\begin{matrix}
Y & \xrightarrow{\rho} & \ov{\cM}_C \\
[\cP] & \mapsto & [\cP_{|C}]
\end{matrix}
\end{equation}
where $\ov{\cM}_C$ is a moduli space of semi-stable projective bundles on the curve $C$ (a moduli space of semistable vector bundles because 
the corresponding Brauer class is necessarily trivial), and $\rho$  
 is finite onto its image. By the equality in~\eqref{riducoacurva} we get that, up to a positive multiple, $\lambda_{\ov{\ww}}(c_1(L))$ is equal to 
 $\rho^{*}(\alpha)$ where $\alpha$ is the ample determinant line-bundle on $\ov{\cM}_C$, and hence is ample.
\end{proof}
\subsection{Stable  vector bundles as stable projective bundles}\label{subsec:stabletwisted}
\setcounter{equation}{0}  
Let $X$ be a compact  K\"ahler manifold, with K\"ahler class $\omega$, and let $\cP\to X$ be a $\PP^{r-1}$-bundle. There is a notion of 
$\omega$ slope-stability for $\cP$, defined in terms of the associated principal $\PGL_r(\CC)$-bundle on $X$, see Remark~\ref{rmk:stabequiv}. If 
$V\to X$ is an $\omega$ slope-stable rank-$r$ vector bundle then  $\PP(V)\to X$ is $\omega$ slope-stable as $\PP^{r-1}$-bundle. 
It follows that by mapping $[\cF]\in M_{\ww}(S^{[2]},h_{S^{[2]}})^{\bullet}$ (notation of Theorem~\ref{thm:dadueauno}) to $\PP(\cF)$ we get a map from $M_{\ww}(S^{[2]},h_{S^{[2]}})^{\bullet}$ to a suitable moduli space of  $\PP^{8a^3-1}$-bundles. This map is an embedding because $S^{[2]}$ has no  non-trivial line-bundles. The numerical invariants (see Subsection~\ref{subsec:risulprinc}) of the relevant 
$\PP^{8a^3-1}$-bundles are the following.
 Set $\ww_a\coloneq\ww(D,a)$.
Let
\begin{equation}\label{anonimoveneziano}
\wh{\gamma}_{a}\coloneq [-({\bm\mu}(D)-a\delta)]\in H^2(S^{[2]};\ZZ)/2a H^2(S^{[2]};\ZZ).
\end{equation}
Let $[\cF]\in M_{\ww_a}$. By~\cite[Lemma~2.5]{huybrechts-schroer:brauergrp} we have
\begin{equation}\label{etaottoacubo}
\eta_{8a^3}(\PP(\cF))=\ov{\gamma}_{a}\coloneq [4a^2\wh{\gamma}_{a}]\in H^2(S^{[2]};\ZZ)/8a^3 H^2(S^{[2]};\ZZ).
\end{equation}
Moreover $\Delta(\PP(\cF))=\Delta(\cF)=4a^6 c_2(S^{[2]})/3 $. Let (see~\eqref{eccowbarra})
\begin{equation}\label{doppiavuahilb}
\ov{\ww}_a\coloneq(8a^3,\ov{\gamma}_a,  \frac{ 4a^6}{3} c_2(S^{[2]}) ),
\end{equation}
and let $M_{\ov{\ww}_a}(S^{[2]},h_{S^{[2]}})$ be the moduli space of slope-stable $\PP^{8a^3-1}$-bundles $\cP$ on $S^{[2]}$ 
with $\eta_{8a^3}(\PP(\cF))=\ov{\gamma}_{a}$ and $\Delta(\cP)=4a^6 c_2(S^{[2]})/3 $.
By the preceding discussion  we have an isomorphism
\begin{equation}\label{italiabosnia}
\begin{matrix}
M_{\ww_a}(S^{[2]},h_{S^{[2]}})^{\bullet} & \xrightarrow{\sim} & M_{\ov{\ww}_a}(S^{[2]},h_{S^{[2]}})^{\bullet} \\
[\cF] & \mapsto & [\PP(\cF)]
\end{matrix}
\end{equation}
where $M_{\ov{\ww}_a}(S^{[2]},h_{S^{[2]}})^{\bullet} \subset M_{\ov{\ww}_a}(S^{[2]},h_{S^{[2]}})$  is  closed and irreducible. 
 Let 
\begin{equation}\label{modnorm}
\wt{M}_{\ov{\ww}_a}(S^{[2]},h_{S^{[2]}})^{\bullet}\xrightarrow{\nu_a}M_{\ov{\ww}_a}(S^{[2]},h_{S^{[2]}})^{\bullet}
\end{equation}
 be the normalization. 
 Let $v(a)$ be as in~\eqref{centocelle}. By the isomorphism in~\eqref{italiabosnia} and Corollary~\ref{crl:isomnorm} we have an isomorphism
\begin{equation}\label{torbellamonaca}
\wt{\psi}\colon\cM_{v(a)}(S,h_{S})  \xrightarrow{\sim}  \wt{M}_{\ov{\ww}_a}(S^{[2]},h_{S^{[2]}})^{\bullet} \\
\end{equation}
\begin{rmk}\label{rmk:quadratogamma}
Let $\gamma_{a}\coloneq -({\bm\mu}(D)-a\delta)$, so that  $\wh{\gamma}_{a}=[\gamma_{a}]$. Then
\begin{equation}\label{sergioleone}
\divisore (\gamma_a)=1,
\qquad
q_{S^{[2]}}(\gamma_a)\equiv 2a^2-2 \pmod{4a}.
\end{equation}
In fact let $\cl(D)=l \beta$ where $\beta\in H^{1,1}_{\ZZ}(S^{[2]};\ZZ)$ is primitive. Then $\gcd\{l,2a\}=1$ by~\eqref{castiglione}. It follows that 
$\divisore (\gamma_a)=1$. The congruence is straightforward.
Note:
\begin{equation}
q_{S^{[2]}}(\gamma_a)\equiv
\begin{cases}
-2 \pmod{4a} & \text{if $a$ is even,} \\
2a-2  \pmod{4a}  & \text{if $a$ is odd.}
\end{cases}
\end{equation}
\end{rmk}
\subsection{The Donaldson-Mukai map for $\wt{M}_{\ov{\ww}_a}$}\label{subsec:donaldson-mukai}
\setcounter{equation}{0}  
\subsubsection{Statement of the result}\label{subsubsec:donmukmap}
Hypotheses and notation are as in  Subsections~\ref{subsec:bladerunner} and~\ref{subsec:stabletwisted}. We often set 
$M_{\ov{\ww}_a}^{\bullet}=M_{\ov{\ww}_a}(S^{[2]},h_{S^{[2]}})^{\bullet}$ etc.
Let  
$\wt{\lambda}_{\ov{\ww}_a}\coloneq \nu_a^{*}\circ\lambda_{\ov{\ww}_a}$ (see~\eqref{donmukmoduli} and~\eqref{modnorm}). Thus we have maps
\begin{equation}
 H^2(S^{[2]})\overset{\wt{\lambda}_{\ov{\ww}_a}}{\lra} H^2(\wt{M}_{\ov{\ww}_a}^{\bullet})\overset{\wt{\psi}^{*}}{\overset{\sim}{\lra}} H^2(\cM_{v(a)}).
\end{equation}
Let $\wt{\varphi}_{\ov{\ww}_a}\colon H^2(S^{[2]})\to H^2(\wt{M}_{\ov{\ww}_a}^{\bullet})$ and 
$\varphi_{\ov{\ww}_a}\colon H^2(S^{[2]})\to H^2(\cM_{v(a)})$ be given by
\begin{equation}\label{fitildefi}
\wt{\varphi}_{\ov{\ww}_a}\coloneq 32^{-1}\cdot a^{-4}\cdot\wt{\lambda}_{\ov{\ww}_a},\qquad
\varphi_{\ov{\ww}_a}\coloneq \wt{\psi}^{*}\circ \wt{\varphi}_{\ov{\ww}_a}.
\end{equation}
\begin{prp}\label{prp:isogenia}
Let hypotheses  be as in Subsection~\ref{subsec:spalmo}.  Suppose that  there exists a universal sheaf  on $S\times\cM_{v(a)}$. 
 Then
\begin{equation*}
\varphi_{\ov{\ww}_a}(\gamma)=
\begin{cases}
 \theta_{v(a)}(\beta+\frac{(\beta,D)}{2a}\eta_S), & \text{if $\gamma=\bm{\mu}(\beta)$}. \\
\theta_{v(a)}(v_1) & \text{if $\gamma=\delta$},
\end{cases}
\end{equation*}
where $(-,-)$ is the intersection form, and $\eta_S\in H^4(S)$ is the orientation form.
\end{prp}
\begin{crl}\label{crl:isomraz}
The map $\wt{\varphi}_{\ov{\ww}_a}\colon H^2(S^{[2]})\to H^2(\wt{M}_{\ov{\ww}_a}^{\bullet})$ is a rational Hodge isometry 
(we are not assuming that  there exists a universal sheaf  on $S\times\cM_{v(a)}$). 
\end{crl}
\begin{proof}
It is  clear that $\wt{\varphi}_{\ov{\ww}_a}$ is a morphism of Hodge  structures. It remains to prove that it is a rational isometry.  Suppose first that there exists a universal sheaf  on $S\times\cM_{v(a)}$. The map $\bm{\mu}\colon H^2(S)\to H^2(S^{[2]})$ matches the intersection form on $H^2(S)$ with the BBF form on $H^2(S^{[2]})$, and the image is the orthogonal of $\delta$, which has square $-2$. The map $H^2(S)\to H(S)$ 
defined by $\beta\mapsto \beta+\frac{(\beta,D)}{2a}\eta_S$ matches the intersection form on $H^2(S)$ with the Mukai form, and the image is Mukai-orthogonal to $v_1$. Moreover $v_1^2=-2=\delta^2$ (squares are for  the Mukai form).  Mukai's map
 $\theta_{v(a)}$ matches the Mukai form on $v(a)^{\bot}$ with the BBF form on $H^2(\cM_{v(a)})$. It follows
 that $\varphi_{\ov{\ww}_a}$  is a rational Hodge isometry, and hence so is $\wt{\varphi}_{\ov{\ww}_a}$. 
 
 If no universal sheaf  on $S\times\cM_{v(a)}$ exists, we may specialize $S$ so that a universal sheaf exists, see Remark~\ref{rmk:rinogaetano}. Since  $\wt{\varphi}_{\ov{\ww}_a}$ is locally constant it follows that it is a rational isometry. 
\end{proof}
The result below will be useful later on, when considering twistor families. 
\begin{crl}\label{crl:kahlcomp}
There exists a non empty subcone $\cU\subset\cK(S^{[2]})$ of the K\"ahler cone containing $\bm{\mu}(h_S)$ in its closure and such that 
$\wt{\varphi}_{\ov{\ww}_a}(\omega)\in\cK(\wt{M}_{\ov{\ww}_a}^{\bullet})$ for all $\omega\in\cU$.
\end{crl}
\begin{proof}
 For $\epsilon>0$ and small enough, $\bm{\mu}(h_S)-\epsilon\delta$ is in the  ample cone of   $S^{[2]}$.  By Proposition~\ref{prp:isogenia} and~\eqref{classedivdet} we have
\begin{equation}\label{catturavuoto}
\varphi_{\ov{\ww}_a}(\bm{\mu}(h_S)-\epsilon\delta)=
\theta_{v(a)}\left(h_S+\frac{(h_S,D)}{2a}\eta_S\right)-\epsilon\cl(\cD_{v_2}).
\end{equation}
We claim that, as is well-known, $\theta_{v(a)}(h_S+(h_s,D)\eta_S/2a)$ is big and nef. We sketch the argument for the reader's convenience. We are free to replace $h_S$ by a high positive multiple. Let $C$ be a smooth curve in the complete linear system $|h_S|$. Then there is an open $V_C\subset (\cM_{v(a)}\setminus \cB_{v(a)})$ with complement of codimension at least $2$  which has the following properties.
\begin{enumerate}
\item
If $[\cF]\in V_C$ the restriction $\cF_{|C}$ is slope semistable. 
\item
The map $m_C\colon V_C\to U_C(2a^2,\cO_C(D))$ to the moduli space $U_C(2a^2,\cO_C(D))$ (of slope semistable rank-$2a^2$ vector bundles on 
$C$ with determinant isomorphic to $\cO_C(D)$) defined by $[\cF]\mapsto [\cF_{|C}]$ is injective. 
\end{enumerate}
We claim that  
\begin{equation}\label{sollevocanonico}
-m_C^{*}c_1(K_{U_C(2a^2,\cO_C(D))})=4a^2\theta_{v(a)}(\beta+(\beta,D)\eta_S/2a)_{V_C}.
\end{equation}
To prove this, let $\cH\coloneq \cF_{|C\times V_C}$, where $\cF$ is a universal sheaf on $S\times \cM_{v(a)}$. Let $p_C,p_{V_C}$ and 
$q_S,q_{\cM_{v(a)}}$ be the projections of $C\times V_C$ and $S\times \cM_{v(a)}$ to first and second factors respectively. The equality in~\eqref{sollevocanonico} follows from the equation below (given by the Grothendieck-Riemann-Roch Theorem):
\begin{multline*}
c_1\left(R^1p_{V_C,*} End^0(\cE)\right)=p_{V_C,*}\left(\Delta(\cH)\right)= \\
=q_{\cM_{v(a)},*}\left(\Delta(\cF)\cdot q_S^{*}\beta\right)_{|V_C}
=4a^2\theta_{v(a)}\left(\beta+(\beta,D)\eta_S/2a\right)_{|V_C}.
\end{multline*}
(Note that $\cF_{|C}$ is simple for every $[\cF]\in U_C$ because every sheaf parametrized by $\cM_{v(a)}$ is simple and $C\in |N h_S|$ for $N$ very large.)
Since  $U_C(2a^2,\cO_C(D))$ is a Fano variety,  it follows from Items(1), (2) and~\eqref{sollevocanonico} that $\theta_{v(a)}(\beta+(\beta,D)\eta_S/2a)$ is big. 
Nefness follows from a more refined  analysis.

Since $\theta_{v(a)}(h_S+(h_s,D)\eta_S/2a)$ is big and nef, the equation in~\eqref{catturavuoto} shows that $\varphi_{\ov{\ww}_a}(\bm{\mu}(h_S)-\epsilon\delta)$ is ample for $\epsilon>0$ small enough. The corollary follows by openness of the K\"ahler cone.
\end{proof}
The proof of Proposition~\ref{prp:isogenia} is in Subsection~\ref{subsec:natascia}.
\subsubsection{First step}\label{subsubsec:primopasso}
Hypotheses  as in Proposition~\ref{prp:isogenia}. Set $\cM\coloneq\cM_{v(a)}$. Let $\EE_1$ be the pull-back of $\cE_1$ to $S\times \cM$, and 
let $\EE_2$ be a universal sheaf  on $S\times\cM_{v(a)}$.  
Then we have the sheaf $\cG(\cM)\coloneq\cG(\EE_1,\EE_2)$ on $S^{[2]}\times \cM$  (see~\cite[Subsect.~2.1]{og:highdim}), 
such that for $[\cE_2]\in \cM$ we have 
\begin{equation*}
\cG(\cM)_{|S^{[2]}\times\{[\cE_2]\}}\cong\cG(\cE_1,\cE_2). 
\end{equation*}
In the present Subsubsection we prove the  following result.
\begin{prp}\label{prp:libano}
Let hypotheses be as in Proposition~\ref{prp:isogenia}. Let  $\gamma\in H^2(S^{[2]})$. Then
\begin{equation}\label{fognini}
\lambda_{\ww}(\gamma)=q_{\cM,*}\left[\Delta(\cG(\cM))\cdot q_{S^{[2]}}^{*}(\bD(\gamma))\right]-16 a^4\la\gamma,\delta\ra\theta_{v_2}(v_1),
\end{equation}
where $q_{S^{[2]}},q_{\cM}$ are the projections of $S^{[2]}\times \cM$ to $S^{[2]}$, $\cM$, and $\la-,-\ra$ is the BBF bilinear symmetric form of $S^{[2]}$. 
\end{prp}
We prove the  result above after a series of preliminary computations.
 
Let $\cM^{\dag}\coloneq\cM\setminus\cD^2$ (Note: if $a=1$ then $\cM^{\dag}=\cM$). The restriction of $\wh{f}$ to 
$\wh{\cM}\setminus(\wh{E}^2\cup\ldots\cup\wh{E}^a)$ (if $a=1$ the latter is $\wh{\cM}$) is an isomorphism 
\begin{equation}\label{michelpiccoli}
 \wh{\cM}\setminus(\wh{E}^2\cup\ldots\cup\wh{E}^a)\xrightarrow{\sim}\cM^{\dag}
\end{equation}
mapping $\wh{E}^1\setminus \wh{E}^2$ isomorphically to $\cD^1\setminus\cD^2$. In the following we identify $\cM^{\dag}$ with the left-hand side of~\eqref{michelpiccoli}, and $\cD^1\setminus\cD^2$ with $\wh{E}^1\setminus \wh{E}^2$. Let 
$\GG^{\dag}\coloneq\GG(0)_{|S^{[2]}\times\cM^{\dag}}$, and let $\cG(\cM)^{\dag}\coloneq\cG(\cM)_{|S^{[2]}\times\cM^{\dag}}$. The restriction 
to $S^{[2]}\times\cM^{\dag}$ of the exact sequence in~\eqref{valtellina} reads
\begin{equation}
0\lra \GG^{\dag}\lra \cG(\cM)^{\dag}\lra \iota^{\dag}_{*}\left(\cE_1[2]^{-}\boxtimes\cL\right)\lra 0,
\end{equation} 
where $\iota^{\dag}\colon S^{[2]}\times\cD^{\dag}\hra S^{[2]}\times\cM^{\dag}$ is the inclusion, and $\cL$ is a line-bundle on 
$\cD^{\dag}$. 
\begin{lmm}\label{lmm:dagaromana}
Let hypotheses be as in Proposition~\ref{prp:isogenia}. Let  $\gamma\in H^2(S^{[2]})$. Then
\begin{equation}
\lambda_{\ww}(\gamma)_{|\cM^{\dag}}=p_{\cM^{\dag},*}\left[\Delta(\cG(\cM)^{\dag})\cdot p_{S^{[2]}}^{*}\bD(\gamma)\right]
-16 a^4\la \gamma,\delta\ra\cl(\cD^{\dag}),
\end{equation}
where $p_{S^{[2]}},p_{\cM^{\dag}}$ are the projections of $S^{[2]}\times \cM^{\dag}$ to $S^{[2]}$, $\cM^{\dag}$ respectively. 
\end{lmm}
\begin{proof}
Let $r_{\cD^{\dag}}$   be the projection of $S^{[2]}\times \cD^{\dag}$ to the  second factor. 
Then (use Lemma~\ref{lmm:grrinclusione} and~\cite[Etns~(5.4.2)-(5.4.3)]{ogfascimod}) \emph{modulo $H^{>4}(S^{[2]}\times\cM^{\dag})$} we have
\begin{multline*}
\ch(\iota^{\dag}_{*}(\cE_1[2]^{-}\boxtimes\cL))=4a^2p_{\cM^{\dag}}^{*}\cl(\cD^{\dag})+\\
+p_{S^{[2]}}^{*}(2a {\bm\mu}(D)-a(2a+1)\delta)\cdot p_{\cM^{\dag}}^{*}\cl(\cD^{\dag})+4a^2\iota^{\dag}_{*}\left( r^{*}_{\cD^{\dag}} c_1(\cL)\right)
-2a^2 p_{\cM^{\dag}}^{*}\cl(\cD^{\dag})^2.
\end{multline*}
It follows that
\begin{eqnarray*}
p_{\cM^{\dag},*}\left[\ch_2(\iota^{\dag}_{*}(\cE_1[2]^{-}\boxtimes\cL))\cdot p_{S^{[2]}}^{*} \bD(\gamma)\right]& = & 
\la\gamma,2a{\bm\mu}(D)-a(2a+1)\delta\ra\cl(\cD^{\dag}) \\
p_{\cM^{\dag},*}\left[\ch_1(\cG(\cM)^{\dag})\cdot \ch_1(\iota^{\dag}_{*}(\cE_1[2]^{-}\boxtimes\cL)) \cdot p_{S^{[2]}}^{*} \bD(\gamma)\right]& = & 
\la\gamma,16a^4{\bm\mu}(D)-16a^5\delta\ra\cl(\cD^{\dag}) \\ 
p_{\cM^{\dag},*}\left[\ch_1(\iota^{\dag}_{*}(\cE_1[2]^{-}\boxtimes\cL))^2 \cdot p_{S^{[2]}}^{*} \bD(\gamma)\right]& = & 0.
\end{eqnarray*}
 The lemma follows from the formulae above and the equality
\begin{multline*}
\Delta(\GG^{\dag})=\Delta(\cG(\cM)^{\dag})+16a^3\ch_2(\iota^{\dag}_{*}(\cE_1[2]^{-}\boxtimes\cL))-\\
-2\ch_1(\cG(\cM)^{\dag})\cdot \ch_1(\iota^{\dag}_{*}(\cE_1[2]^{-}\boxtimes\cL))+\ch_1(\iota^{\dag}_{*}(\cE_1[2]^{-}\boxtimes\cL))^2.
\end{multline*}
\end{proof}
\begin{proof}[Proof of Proposition~\ref{prp:libano}]
By~\eqref{classedivdet} we have $\cl(\cD^{\dag})=\theta_{v_2}(v_1)$. Thus by Lemma~\ref{lmm:dagaromana}  the restrictions to 
$\cM^{\dag}$  of the two sides of~\eqref{fognini} are equal.  Since $\cM\setminus\cM^{\dag}$ has codimension $4$ in $\cM$ (or is empty if $a=1$) the restriction map  $H^2(\cM)\to H^2(\cM^{\dag})$ is an isomorphism. It follows that the two sides of~\eqref{fognini} are equal.
\end{proof}
\subsubsection{Passing to $X(S)$}
Hypotheses and notation are as in Subsubsection~\ref{subsubsec:primopasso}. The goal of the present subsubsection is to prove the following result.
\begin{prp}\label{prp:lontanolontano}
Let hypotheses be as in Proposition~\ref{prp:isogenia}. Let  $\gamma\in H^2(S^{[2]})$. Then
\begin{equation*}
q_{\cM,*}\left[\Delta(\cG(\cM))\cdot q_{S^{[2]}}^{*}\bD(\gamma)\right]=
\begin{cases}
16a^3\theta_{v_2}(2a\beta+(\beta,D)\eta_S), & \text{if $\gamma=\bm{\mu}(\beta)$}, \\
0 & \text{if $\gamma=\delta$},
\end{cases}
\end{equation*}
where $q_{S^{[2]}},q_{\cM}$ are the projections of $S^{[2]}\times \cM$ to $S^{[2]}$, $\cM$ respectively,  $\beta\in H^2(S)$, $\delta$ is as in~\eqref{eccodelta},  $(-,-)$ is the intersection product on $H^2(S)$.
\end{prp}
The proof is at the end of the subsubsection. 

We  pass to $X(S)$, see the commutative diagram in~\eqref{commiso}. We  adopt the notation introduced therein. 
Let $\ov{q}_{X(S)},\ov{q}_{\cM}$ be the projections of $X(S)\times \cM$ to first and second factor respectively.
Recall that $\rho\colon X(S)\to S^{[2]}$ is a double cover. Let $\brho\colon X(S)\times \cM\to S^{[2]}\times \cM$ be given by 
$\brho\coloneq\rho\times \Id$.
\begin{clm}\label{clm:dividoperdue}
Let hypotheses be as in Proposition~\ref{prp:isogenia}. Let  $\gamma\in H^2(S^{[2]})$. Then
\begin{equation}\label{favignana}
q_{\cM,*}\left[\Delta(\cG(\cM))\cdot q_{S^{[2]}}^{*}\bD(\gamma)\right]=
\frac{1}{2} \ov{q}_{\cM,*}\left[\Delta(\brho^{*}\cG(\cM))\cdot \ov{q}_{X(S)}^{*}\rho^{*}\bD(\gamma)\right].
\end{equation}
\end{clm}
\begin{proof}
This holds because $\rho$ has  degree $2$.
\end{proof}

The pull-back $\brho^{*}\cG(\cM)$ is described as follows. For $i\in\{1,2\}$ let $\btau_i\colon X(S)\times \cM\to S\times \cM$ be given by $\btau_i=\tau_i\times \Id$.  
Let
\begin{equation}\label{giulatesta}
\cF(\EE_1,\EE_2)\coloneq\btau_1^{*}\EE_1\otimes \btau_2^{*}\EE_2\oplus \btau_1^{*}\EE_2\otimes \btau_2^{*}\EE_1.
\end{equation}
Let $E\subset X$ be the exceptional divisor of the blow-up (of the diagonal) $\tau\colon X(S)\to S^2$. There is a 
morphism  
$\brho^{*}\cG(\cM)\lra\cF(\EE_1,\EE_2)$ which is an isomorphism 
away from $E$.  Let $\epsilon\colon E\to S$ be the restriction of $\tau$ to $E$ (the image is the diagonal of $S^2$, which we identify with $S$), and let
$\bepsilon \colon E\times \cM\to S\times \cM$ be given by $\bepsilon=\epsilon\times \Id$.  We let $\cR\coloneq \bepsilon^{*}(\EE_1\otimes\EE_2)$. 
Then  $\brho^{*}\cG(\cM)$  fits into the exact sequence
\begin{equation}
0\lra \brho^{*}\cG(\cM)\lra \cF(\EE_1,\EE_2)\overset{\bvarphi}{\lra}
\biota_{*}\cR\lra 0,
\end{equation}
where  
$\biota\colon E\times \cM\hra X\times \cM$ is the inclusion, and $\bvarphi(s_1,s_2)\coloneq (s_{1|E\times \cM}-s_{2|E\times \cM})$ for  the components $s_1,s_2$ are of a (local) section with respect to the direct sum decomposition in~\eqref{giulatesta}.
 Then (use Lemma~\ref{lmm:grrinclusione}) \emph{modulo $H^{>4}(X(S)\times\cM)$} 
\begin{equation*}
\ch(\biota_{*}\cR)=4a^3 \ov{q}_{X(S)}^{*}(e)+
 \ov{q}_{X(S)}^{*}(e)\cdot \left(2a^2\btau_1^{*}c_1(\EE_1)+ 2a\btau_2^{*} c_1(\EE_2)\right)-2a^3 \ov{q}_{X(S)}^{*}(e^2).
\end{equation*}
Since $\ch(\brho^{*}\cG(\cM))=\ch(\cF(\EE_1,\EE_2))-\ch(\biota_{*}\cR)$, the equation above gives that
\begin{multline}\label{lanaveva}
\Delta(\brho^{*}\cG(\cM))=
\Delta(\cF(\EE_1,\EE_2))+\ov{q}_{X(S)}^{*}(e)\cdot \left(32a^5\btau_1^{*}c_1(\EE_1)+ 32a^4\btau_2^{*} c_1(\EE_2)\right)- \\
-16a^6 \ov{q}_{X(S)}^{*}(e^2)-8a^3 \ov{q}_{X(S)}^{*}(e)\cdot c_1(\cF(\EE_1,\EE_2)).
\end{multline}
\begin{clm}\label{clm:sergioendrigo}
Let hypotheses be as in Proposition~\ref{prp:isogenia}. Let  $\gamma\in H^2(S^{[2]})$. Then
\begin{equation*}
 \ov{q}_{\cM,*}\left[\Delta(\brho^{*}\cG(\cM))\cdot \ov{q}_{X(S)}^{*}\rho^{*}\bD(\gamma)\right]=
\ov{q}_{\cM,*}\left[\Delta(\cF(\EE_1,\EE_2))\cdot \ov{q}_{X(S)}^{*}\rho^{*}\bD(\gamma)\right].
\end{equation*}
\end{clm}
\begin{proof}
This follows from~\eqref{lanaveva} and the formulae (note that $\rho^{*}(\delta)=e$)
\begin{eqnarray*}
\ov{q}_{\cM,*}\left[\ov{q}_{X(S)}^{*}(e)\cdot \btau_1^{*}c_1(\EE_1)\cdot \ov{q}_{X(S)}^{*}(\rho^{*}\bD(\gamma))\right] & = & 0, \\
\ov{q}_{\cM,*}\left[\ov{q}_{X(S)}^{*}(e)\cdot \btau_2^{*}c_1(\EE_2)\cdot \ov{q}_{X(S)}^{*}(\rho^{*}\bD(\gamma))\right] & = & 2\la \gamma,\delta\ra\theta_{\EE_2}(\eta_S), \\
\ov{q}_{\cM,*}\left[\ov{q}_{X(S)}^{*}(e^2)\cdot \ov{q}_{X(S)}^{*}(\rho^{*}\bD(\gamma))\right] & = & 0,\\
\ov{q}_{\cM,*}\left[\ov{q}_{X(S)}^{*}(e)\cdot c_1(\cF(\EE_1,\EE_2))\cdot \ov{q}_{X(S)}^{*}(\rho^{*}\bD(\gamma))\right] & = & 8a\la \gamma,\delta\ra\theta_{\EE_2}(\eta_S), \\
\end{eqnarray*}
where $\eta_S\in H^4(S;\ZZ)$ is the orientation class.
\end{proof}
\begin{prp}\label{prp:immugamma}
Let $\gamma\in H^2(S^{[2]})$. Then
\begin{equation*}
q_{\cM,*}\left[\Delta(\cF(\EE_1,\EE_2))\cdot q_{S^{[2]}}^{*}\bD(\gamma)\right]
=
\begin{cases}
32a^3\theta_{v_2}(2a\beta+(\beta,D)\eta_S), & \text{if $\gamma=\bm{\mu}(\beta)$}. \\
0 & \text{if $\gamma=\delta$}.
\end{cases}
\end{equation*}
\end{prp}
\begin{proof}
Let $\ov{p}_{S},\ov{p}_{\cM}$  be the projections of $S\times \cM$ to first and second factor.
We have
\begin{multline*}
\ch_1(\cF(\EE_1,\EE_2))^2  =  4a^4\left(\btau_1^{*}\ov{p}_S^{*}\cl(D)^2+2\btau_1^{*}\ov{p}_S^{*}\cl(D)\cdot \btau_2^{*}\ov{p}_S^{*}\cl(D)+ \btau_2^{*}\ov{p}_S^{*}\cl(D)^2 \right)+\\
+8a^3 \left(\btau_1^{*}\ov{p}_S^{*}\cl(D)+\btau_2^{*}\ov{p}_S^{*}\cl(D) \right)\cdot \left(\btau_1^{*}\ch_1(\EE_2)+ \btau_2^{*}\ch_1(\EE_2)\right)+\\
+4a^2\left(\btau_1^{*}\ch_1(\EE_2)^2+2\btau_1^{*}\ch_1(\EE_2)\cdot \btau_2^{*}\ch_1(\EE_2)+ \btau_2^{*}\ch_1(\EE_2)^2\right)
\end{multline*}
\begin{multline*}
\ch_2(\cF(\EE_1,\EE_2))  =  2a^2\left(\btau_1^{*}\ov{p}^{*}_S\ch_2(\cE_1)+\btau_2^{*}\ov{p}^{*}_S\ch_2(\cE_1)\right)+ \\
+\btau_1^{*}\ov{p}_S^{*}\cl(D)\cdot\btau_2^{*}\ch_1(\EE_2)+\btau_1^{*}\ch_1(\EE_2)\cdot\btau_2^{*}\ov{p}_S^{*}\cl(D)+ \\
+ 2a\left(\btau_1^{*}\ch_2(\EE_2)+ \btau_2^{*}\ch_2(\EE_2)\right).
\end{multline*}
From the above equalities one gets that
\begin{equation*}
\ov{q}_{\cM,*}\left[\ch_1(\cF(\EE_1,\EE_2))^2\cdot  \ov{q}_{X(S)}^{*}\rho^{*}\bD(\gamma)\right]  =  
\begin{cases}
64a^3(\beta,D)\theta_{\EE_2}(\eta_S) & \text{if $\gamma=\bm{\mu}(\beta)$}, \\
0 & \text{if $\gamma=\delta$}.
\end{cases}
\end{equation*}
\begin{equation*}
\ov{q}_{\cM,*}\left[\ch_2(\cF(\EE_1,\EE_2))\cdot \ov{q}_{X(S)}^{*}\rho^{*}\bD(\gamma)\right]  =
\begin{cases}
 2 (\beta,D)\theta_{\EE_2}(\eta_S)- 4a  \theta_{\EE_2}(\beta), & \text{if $\gamma=\bm{\mu}(\beta)$}, \\
0 & \text{if $\gamma=\delta$}.
\end{cases}
\end{equation*}
In deriving the above equalities  one uses the equalities
\begin{equation*}
\rho^{*}\bD({\bm\mu}(\beta))=\tau_1^{*}\beta\cdot \tau_2^{*}\eta_S+\tau_1^{*}\eta_S\cdot \tau_2^{*}\beta,\qquad 
\rho^{*}\bD(\delta)=2e\cdot(\tau_1^{*}\eta_S+\tau_2^{*}\eta_S),
\end{equation*}
\begin{equation*}
\ov{p}_{\cM,*}\left[\ch_1(\EE_2)^2\cdot  \ov{p}_S^{*}\beta\right]  =  2a(\beta,D)\theta_{\EE_2}(\eta_S)=\ov{q}_{\cM,*}\left[\btau_1^{*}\ch_1(\EE_2)\cdot \btau_2^{*}\ch_1(\EE_2) \cdot \rho^{*}\bD(\bm{\mu}(\beta)\right].
\end{equation*}
The proposition follows from the above equalities.
\end{proof}
\begin{proof}[Proof of Proposition~\ref{prp:lontanolontano}]
Put together Claims~\ref{clm:dividoperdue}, \ref{clm:sergioendrigo}, and Proposition~\ref{prp:immugamma}.
\end{proof}
\subsubsection{Proof of Proposition~\ref{prp:isogenia}}\label{subsec:natascia}
Put together Propositions~\ref{prp:libano} and~\ref{prp:lontanolontano}.
\subsubsection{Relation between the  integral lattices}
%
%
 By Corollary~\ref{crl:isomraz} we have the  isometry
\begin{equation}
\wt{\varphi}_{\ov{\ww}_a}\colon H^2(S^{[2]};\QQ)\lra H^2(\wt{M}_{\ov{\ww}_a}^{\bullet}).
\end{equation} 
Here we are concerned with the relation  between the cohomology groups with integral coefficients, i.e.~$H^2(S^{[2]};\ZZ)$ and 
$H^2(\wt{M}_{\ov{\ww}_a}^{\bullet};\ZZ)$.
Let
\begin{equation}\label{divduea}
H^2(S^{[2]};\ZZ)_{\wh{\gamma}_{a}}\coloneq\{\alpha\in H^2(S^{[2]};\ZZ)\mid (\alpha,\gamma_a)_{S^{[2]}}\equiv 0\pmod{2a}\},
\end{equation}
where   $\wh{\gamma}_{a},\gamma_a$ are as in~\eqref{anonimoveneziano} and Remark~\ref{rmk:quadratogamma} respectively. The notation is coherent because the right-hand side of~\eqref{divduea} depends only on $\wh{\gamma}_{a}$.
Then $\wt{\varphi}_{\ov{\ww}_a}$ maps $H^2(S^{[2]};\ZZ)_{\ov{\gamma}_{a}}$ to 
$H^2(\wt{M}_{\ov{\ww}_a}^{\bullet};\ZZ)$, see Proposition~\ref{prp:isogenia}. 
 The restriction to $H^2(S^{[2]};\ZZ)_{\wh{\gamma}_{a}}$  of the BBF quadratic form on $H^2(S^{[2]};\ZZ)$ has discriminant $8a^2$. 
Since the BBF quadratic form on $H^2(\wt{M}_{\ov{\ww}_a}^{\bullet};\ZZ)$ has discriminant  $2a^2$, it follows that 
\begin{equation}\label{indicedue}
[H^2(\wt{M}_{\ov{\ww}_a}^{\bullet};\ZZ):\wt{\varphi}_{\ov{\ww}_a}(H^2(S^{[2]};\ZZ)_{\wh{\gamma}_{a}})]=2.
\end{equation}
The \lq\lq missing element\rq\rq\ is as follows. There exists $\beta_0\in H^2(S;\ZZ)$ such that (see the sentence following~\eqref{sergioleone})
\begin{equation}
(\beta_0,D)=a(k+2t)-1.
\end{equation}
Then 
\begin{equation}
\frac{1}{2}\varphi_{\ov{\ww}_a}((\bm{\mu}(2\beta_0 - D)+\delta)=\theta_{v(a)}(a+\beta_0+t\eta_S)\in H^2(\cM_{v(a)};\ZZ).
\end{equation}
Hence
\begin{equation}
H^2(\wt{M}_{\ov{\ww}_a}^{\bullet};\ZZ)=\Span\left(\wt{\varphi}_{\ov{\ww}_a}\left(H^2(S^{[2]};\ZZ)_{\wh{\gamma}_{a}}\right),\, 
\frac{1}{2}\wt{\varphi}_{\ov{\ww}_a}(\bm{\mu}(2\beta_0 - D)+\delta)\right).
\end{equation}
\section{Proofs of the main results}\label{sec:bottofinale} 
\subsection{Twistor families and projectively hyperholomorphic vector bundles}
\setcounter{equation}{0}
Let $X$ be a HK manifold,  $\omega\in\cK(X)$,   $\sigma$ a holomorphic symplectic form, and 
$F(\alpha)_{\CC}\coloneq \la \sigma,\ov{\sigma},\omega\ra_{\CC}$. Let 
\begin{equation}\label{basetwistor}
T(\omega)\coloneq \{[\alpha]\in \PP(F(\alpha)_{\CC}) \mid q_X([\alpha])=0\}
\end{equation}
be the twistor conic  associated to 
$(X,\omega)$ and  
$f\colon\cX(\omega)\to T(\omega)$ be the corresponding twistor family. Thus $f$ is a holomorphic map of complex manifolds. The  smooth manifold underlying $\cX(\omega)$ is uniquely identified with the product of 
the smooth manifolds underlying $X$ and $T(\omega)$, and the smooth map underlying $f$ is the projection. 
In particular the space of differential forms of any fiber    $X_t\coloneq f^{-1}(t)$ (an \emph{$\omega$-twistor deformation of $X$}) is identified with that of $X$.
For 
\begin{equation}\label{parametroti}
t=[x\omega+y\sigma+z\ov{\sigma}]\in  T(\omega)
\end{equation}
  let  $\sigma_t\coloneq x\omega+y\sigma+z\ov{\sigma}$ (determined up to rescaling by $\CC^{*}$), and let 
$0\not=\omega_t\in F(\alpha)_{\CC}$ be real, with cohomology class orthogonal (for the BBF quadratic form) to 
$\sigma_t$ (and hence also to $\ov{\sigma}_t$), and such that $\{\omega_t,\sigma_t+\ov{\sigma}_t,i\sigma_t-i\ov{\sigma}_t\}$ defines the same orientation of $F(\alpha)_{\RR}\coloneq \la \omega,\sigma+\ov{\sigma},i\sigma-i\ov{\sigma}\ra_{\RR}$ as 
$\{\omega,\sigma+\ov{\sigma},i\sigma-i\ov{\sigma}\}$ (thus $\omega_t$ is determined up to rescaling by $\RR_{+}$). 
Then $\omega_t\in\cK(X_t)$ is the \emph{twistor K\"ahler class of $X_t$}, and
\begin{equation}\label{formasimplfibratwist}
H^0(\Omega^2_{X_t})=\la x\sigma+y\ov{\sigma}+z\omega\ra.
\end{equation}
\begin{rmk}\label{rmk:perioditorto}
For $t\in  T(\omega)$ we have an integral isometry $\Pi_t\colon H^2(X)\xrightarrow{\sim} H^2(X_t)$ because $\cX(\omega)$, as smooth manifold, is  the product of 
 $X$ and $C(\omega)$. Thus we may view the period map of the twistor family as
\begin{equation}
\begin{matrix}
T(\omega) & \overset{\wp_{\cX(\omega)}}{\lra} & \PP(H^2(X)) \\
t & \mapsto & [\Pi_t^{*}(\sigma_t)]
\end{matrix}
\end{equation}
By~\eqref{formasimplfibratwist} the image is equal to $T(\omega)$ (here we make no difference between a closed $2$-form and its cohomolgy class) and $\wp_{\cX(\omega)}(t)=t$ for all $t$.

\end{rmk}
\begin{rmk}\label{rmk:tortogenerico}
The twistor conic $T(\omega)$   (or  the twistor family) is \emph{generic}  if 
\begin{equation}
F(\omega)_{\CC}^{\bot}\cap H^{1,1}_{\ZZ}(X)=\{0\}.
\end{equation}
Equivalently, $\omega_t^{\bot}\cap H^{1,1}_{\ZZ}(X_t)=\{0\}$  for every $t\in  T(\omega)$. 
If $\cX(\omega)\to T(\omega)$ is generic then $H^{1,1}_{\ZZ}(X_t)=\{0\}$ for $t$ away from a countable subset of $T(\omega)$. 
\end{rmk}
Let $\eta\in H^2(X,\mu_r)$, and let $\alpha=i(\eta)$ where $i\colon H^2(X,\mu_r)\to  H^2(X_t,\cO_{X}^{*})$ is the natural map. Let $t\in T(\omega)$. Since the twistor family is differentiably a product  $\eta$ defines a class $\eta_t\in H^2(X_t,\mu_r)$.
Let $\alpha_t=i_t(\eta_t)$ where $i_t\colon H^2(X_t,\mu_r)\to  H^2(X_t,\cO_{X_t}^{*})$  is the natural map.
Below is a fundamental result of Verbitsky,  see~\cite[Thm~11.1. Cor.~10.1]{verb:iperolom} and~\cite{perego:twistkobhitch} for the twisted version.
\begin{thm}[Verbitsky]\label{thm:iperolom}
Let $X$ be a HK manifold, with K\"ahler class $\omega$, and let $\eta\in H^2(X,\mu_r)$. 
\begin{enumerate}
\item
Let $\cE$ be an $\alpha$-twisted vector bundle of rank $r$ on $X$. If  $\cE$ is
$\omega$-slope-stable and the discriminant $\Delta(F)$ remains a  Hodge class on  all $\omega$-twistor deformations of $X$ then there exists a $\PP^{r-1}$-bundle (locally trivial in the classical topology) 
$\cP\to\cX(\omega)$ whose restriction to $X_0$ is isomorphic to $\PP(\cE)$, or equivalently there exists  an $\alpha_t$-twisted vector bundle $\cE_t$  on $X_t$ extending $\cE$.
\item
Let  $t\in T(\omega)$.  Then $\cE_t$ is  $\omega_t$-slope-stable.
\item
Let  $t\in T(\omega)$ be as in Item~(2). The universal deformation space (see~\cite{meazzini-onorati:twistdeftns}) of $\cE_t$ is isomorphic to the 
universal deformation space of $\cE$. 
\end{enumerate}
\end{thm}
\subsection{Families  of projective bundles and  twistor families}
\setcounter{equation}{0}
Let $X$ be a HK manifold, and $\omega\in\cK(X)$. Let $\cX(\omega)\to T(\omega)$ be the twistor family, with $X_{t_0}=X$. 
Suppose that $\cY\to T(\omega)$ is a proper holomorphic submersive map, and that $\cQ\to \cX(\omega)\times_{T(\omega)}\cY$ is a (holomorphic) $\PP^{r-1}$-bundle. For $t\in T(\omega)$ we have the Donaldson-Mukai map
\begin{equation}
H^2(X_t)\overset{\lambda_t}{\lra} H^2(Y_t)
\end{equation}
given by the formula in~\eqref{donmukai} with $\cP$ the restriction of  $\cQ$ to $X_t\times Y_t$. 
\begin{ass}\label{ass:abicidi}
With notation as above, suppose the following.
\begin{enumerate}
\item[(a)]
$Y_{t_0}$ is a HK manifold.
\item[(b)]
There exists $c\in\QQ_{+}$ such that $c\cdot \lambda_{t_0}$ is a (rational) Hodge isometry. 
\item[(c)]
$\wh{\omega}\coloneq \lambda_{t_0}(\omega)$ is a K\"ahler class of $Y_{t_0}$.
\end{enumerate}
\end{ass}
Under  Assumption~\ref{ass:abicidi} we have the twistor conic for 
$Y_{t_0}$ given by $T(\wh{\omega})$.  
Since $c\cdot \lambda_{t_0}$ is a (rational) Hodge isometry we have
  the isomorphism
\begin{equation}
\begin{matrix}
T(\omega) & \overset{u}{\overset{\sim}{\lra}} & T(\wh{\omega})\\
[\alpha] & \mapsto & c\cdot \lambda_{t_0}([\alpha]).
\end{matrix}
\end{equation}
 Let  $\cZ=\cZ(\wh{\omega})\to T(\wh{\omega})$ be the twistor family of $(Y_{t_0},\wh{\omega})$.
Since $u(t_0)=[c\cdot \wh{\omega}]$  we have $Z_{u(t_0)}\cong Y_{t_0}$. For $t\in T(\omega)$ let $\Pi_t\colon  H^2(X_{t_0}) \xrightarrow{\sim} H^2(X_t)$,  
$\Upsilon_t\colon  H^2(Z_{u(t_0)}) \xrightarrow{\sim} H^2(Z_{u(t)}) $ be the twistor isomorphisms  
 (see Remark~\ref{rmk:perioditorto}), and let
\begin{equation}
\psi_t\colon H^2(Z_{u(t)})\xrightarrow{\sim} H^2(Y_t)
\end{equation}
 be 
 the composition
\begin{equation*}
H^2(Z_{u(t)})\xrightarrow{\Upsilon_t^{-1}}H^2(Z_{u(t_0)})=H^2(Y_{t_0})\xrightarrow{(c\cdot \lambda)^{-1}_{t_0}}H^2(X_{t_0})
\xrightarrow{\Pi_t} H^2(X_t)\xrightarrow{c\cdot \lambda_{t}} H^2(Y_t).
\end{equation*}
\begin{clm}\label{clm:stessiperiodi}
Keep Assumption~\ref{ass:abicidi}. Let  $t\in T(\omega)$, and suppose that $Y_t$ is K\"ahler (and hence hyperk\"ahler). Then $\psi_t$ is an isomorphism of integral Hodge structures matching the BBF quadratic forms. 
\end{clm}
\begin{proof}
The map $\lambda_t\circ\Pi_t\circ\lambda^{-1}_{t_0}\colon H^2(Y_{t_0})\lra H^2(Y_t)$ is an isomorphism over $\ZZ$ matching the BBF quadratic forms because $\lambda_t,\lambda_{t_0}$ are defined via the  $\PP^{r-1}$-bundle  
$\cQ\to \cX(\omega)\times_{T(\omega)}\cY$. Since $\Upsilon_t$ is also an isomorphism over $\ZZ$ matching the BBF quadratic forms, it follows that $\psi_t$ is an isomorphism over $\ZZ$ matching the BBF quadratic forms. It remains to prove that  $\psi_t$ is an isomorphism of  Hodge structures. Since it is an integral isometry it suffices to show that 
\begin{equation}\label{morfismohodge}
\psi_t(H^{2,0}(Z_{u(t)}))=  H^{2,0}(Y_t).
\end{equation}
Let $t=[x\omega+y\sigma+z\ov{\sigma}]$ (notation as in~\eqref{parametroti}). Then 
\begin{eqnarray}
H^{2,0}(X_t) & = & \la x\omega+y\sigma+z\ov{\sigma}\ra,\\
H^{2,0}(Z_{u(t)}) &= &\la\Upsilon_t(x u(\omega)+y u(\sigma)+z u(\ov{\sigma}))\ra.
\end{eqnarray}
Since  $c\cdot\lambda_t$ is an isomorphism  of (rational) Hodge structures,
\begin{equation*}
H^{2,0}(Y_t)=\la c\cdot\lambda_{t}(x\omega+y\sigma+z\ov{\sigma})\ra.
\end{equation*}
The equality in~\eqref{morfismohodge} follows.
\end{proof}
Let
\begin{equation}
T(\omega)_0\coloneq\{t\in T(\omega) \mid \text{$Y_t$ is HK and $\exists f_t\colon Y_t\xrightarrow{\sim} Z_{u(t)}$ s.t.~$H^2(f_t)=\psi_t$}\}.
\end{equation}
Note
 that by construction $t_0\in T(\omega)_0$. By openness of K\"ahlerianity  and 
 local Torelli  $T(\omega)_0$  is open in 
$ T(\omega)$. 
\begin{prp}\label{prp:frontiera}
Keep Assumption~\ref{ass:abicidi}. Let $t_1\in \partial T(\omega)_0=\ov{T(\omega)}_0\setminus T(\omega)_0$. 
Then  
\begin{enumerate}
\item
$H^{1,1}_{\ZZ}(X_{t_1})\not=\{0\}$, and
\item
$\lambda_{t_1}(\omega_{t_1})$ is not a K\"ahler class of $Y_{t_1}$.
\end{enumerate}
\end{prp}
\begin{proof}
The dimension of the space of global holomorphic $2$-forms of $Y_t$ is equal to $1$ for $t\in T(\omega)_0$. By upper-semicontinuity of direct images of analytic coherent sheaves it follows that  there exists $0\not=\sigma(t_1)\in H^0(Y_{t_1},\Omega^2_{Y_{t_1}})$ extending to a global holomorphic symplectic form on $Y_t$ for $t\in T(\omega)_0$. It follows that 
$\sigma(t_1)$ is also a symplectic form. In particular $Y_{t_1}$ has the volume-form $(\sigma(t_1)\wedge\ov{\sigma}(t_1))^{2n}$ where 
$2n=\dim Y_t$.
The argument for the next step is similar to those which appear  (for example) in~\cite[Sect.~4]{huybrechts:basichk}.
 The graph of the isomorphism $Y_{t_0}\xrightarrow{\sim} Z_{t_0}$ extends to the graph $\Gamma(f_t)$ of $f_t$ for all $t\in T(\omega)_0$. The volume of $\Gamma(f_t)$ is uniformly bounded with respect to the product of the volumes forms on $Y_t$ and $Z_{u(t)}$ given by  
 $(\sigma(t)\wedge\ov{\sigma}(t))^{2n}$, $(\tau(t)\wedge\ov{\tau}(t))^{2n}$ respectively where $\sigma(t)$ is as above, and $\tau(t)$  is a symplectic form on $Z_{u(t)}$ such that   $q_{Z_{u(t)}}(\tau(t),\ov{\tau}(t))=1$.
Hence $\Gamma(f_{t})$ specializes, for $t\mapsto t_1$  to an (effective) analytic cycle $\Gamma$ on $Y_{t_1}\times Z_{u(t_1)}$, of pure dimension equal to $\dim Y_t=\dim Z_{u(t)}$, given by 
\begin{equation}
\Gamma=\Gamma_0+\sum_{i} m_i W_i, 
\end{equation}
where $\Gamma_0$ is the graph of  a bimeromorphic  map 
 $g\colon Y_{t_1}\dra Z_{u(t_1)}$, and each $W_i$ is mapped by the two projections to proper closed subsets of 
 $Y_{t_1}$, $Z_{u(t_1)}$.
 
(1): Suppose that $H^{1,1}_{\ZZ}(X_{t_1})=\{0\}$. Then $H^{1,1}_{\ZZ}(Z_{u(t_1)})=\{0\}$ because the composition 
 $\Pi_{t_1}\circ (c\cdot\lambda_{t_0})^{-1}\circ\Upsilon_{t_1}\colon H^2(Z_{u(t_1)})\to H^2(X_{t_1})$ is an isomorphism of rational Hodge structures. Hence 
 $Z_{u(t_1)}$ contains no prime divisors, and since $g$ is an isomorphism away from closed subsets of codimension at least $2$, the analogous statement holds for $Y_{t_1}$.
 Hence   each $W_i$ is mapped by the projections to  $Z_{u(t_1)}$, $Y_{t_1}$ to  (closed) subsets of codimensions at least $2$.  It follows that $g^{*}\colon H^2(Z_{u(t_1)})\to H^2(Y_{t_1})$ is equal to 
 $\psi_{t_1}$. Moreover $g$ is a regular (holomorphic) map because $Z_{u(t_1)}$ contains no curves. Since $g$ is an isomorphism and 
 $g^{*}=\psi_{t_1}$, we get that $t_1\in  T(\omega)_0$, contradiction. 
 
 (2): Assume that  $\lambda_{t_1}(\omega_{t_1})$ is a K\"ahler class of $Y_{t_1}$, in particular the latter is K\"ahler and hence hyperk\"ahler.
The proof of~\cite[Thm.~2.5]{huybrechts:kaehlercone} gives that  each $W_i$ is mapped by the two projections to proper closed subsets of 
 $Y_{t_1}$, $Z_{u(t_1)}$ of codimension at least $2$. It follows that $g^{*}\colon H^2(Z_{u(t_1)})\to H^2(Y_{t_1})$ is equal to 
 $\psi_{t_1}$. Since $\psi_{t_1}(\wh{\omega}_{u(t_1)})=c\cdot\lambda_{t_1}(\omega_{t_1})$ with $c>0$ we get that 
 $g^{*}(\wh{\omega}_{u(t_1)})$ (here $\wh{\omega}_{u(t_1)}$ is the twistor K\"ahler class of $Z_{u(t_1)}$) is a K\"ahler class on $Y_t$, and hence $g$ is an isomorphism by Prop.~2.1 loc.~cit. 
Thus $t_1\in T(\omega)_0$ because $g^{*}=\psi_{t_1}$, contradiction.
\end{proof}
\begin{crl}\label{crl:frontiera}
Keep Assumption~\ref{ass:abicidi}. Let $t\in T(\omega)$. If $H^{1,1}_{\ZZ}(X_t)=\{0\}$, or $\lambda_{t}(\omega_{t})\in\cK(Y_{t})$, 
then $t\in T(\omega)$, and in particular $Y_t$ is isomorphic to $Z_{u(t)}$.
\end{crl}
\begin{rmk}\label{rmk:esorciccio}
In the proof of Proposition~\ref{prp:frontiera} we have shown that,   for every $t\in T(\omega)$, 
$Y_t$ is  a holomorphic symplectic manifold bimeromorphic to $Z_{u(t)}$.
\end{rmk}
\subsection{The proof of Theorem~\ref{thm:isogmodvb}}
\setcounter{equation}{0}
 Let $\ov{\ww}_a$ be as in~\eqref{doppiavuahilb}. If $[\cF]\in M_{\ov{\ww}_a}^{\bu}$ then by~\eqref{atorto} and~\cite[Ex.~2.12]{og:highdim} we have
\begin{equation}
{\mathsf a}(\cF^{\vee}\otimes\cF)={\mathsf a}(a)\coloneq 2^{16}\cdot 5\cdot a^{18}.
\end{equation}
By hypothesis the K\"ahler class $\omega$ of $X$ belongs to an open ${\mathsf a}(a)$-chamber of $\cK(X)$. Replacing $\omega$ by a K\"ahler class belonging to the same open ${\mathsf a}(a)$-chamber (see Proposition~\ref{prp:campol}) we may assume  that 
\begin{equation}\label{omegagenerico}
\omega^{\bot}\cap H^{1,1}_{\ZZ}(X)=\{0\}. 
\end{equation}
Let $(S,h_S)$  be a polarized $K3$ surface as in Theorem~\ref{thm:dadueauno}, and let $\wt{\omega}\in\cK(S^{[2]})$ 
which belongs to an open ${\mathsf a}(a)$-chamber 
 containing $\bm\mu(h_S)$ in its closure, and   such that 
\begin{equation}\label{moonlight}
\wt{\omega}^{\bot}\cap H^{1,1}_{\ZZ}(S^{[2]})=\{0\}. 
\end{equation}
By Theorem~\ref{thm:dadueauno} (and Proposition~\ref{prp:campol})
$M_{\ov{\ww}_a}(S^{[2]},\wt{\omega})$ contains an irreducible component $M_{\ov{\ww}_a}(S^{[2]},\wt{\omega})^{\bu}$ whose normalization $\wt{M}_{\ov{\ww}_a}(S^{[2]},\wt{\omega})^{\bu}$ is isomorphic to the HK manifold $\cM_{v(a)}(S,h_S)$.

By~\cite[Props.~3.7, 5.4]{huybrechts:bourbtorelli} there exist a finite sequence of generic twistor families $\cX(\omega_i)\to T(\omega_i)$ for $i\in\{1,\ldots,n\}$, points $t_i\in T(\omega_i)$ for $i\in\{1,\ldots,n-1\}$ and $s_j\in T(\omega_j)$ $j\in\{2,\ldots,n\}$ such that the following hold.
\begin{enumerate}
\item[(A)]
$\cX(\omega_1)\to T(\omega_1)$ is the twistor family associated to $(S^{[2]},\wt{\omega})$,
\item[(B)]
$\cX(\omega_n)\to T(\omega_n)$ is the twistor family associated to $(X,\omega)$, 
\item[(C)]
For $i\in\{1,\ldots,n-1\}$ the twistor fibers $X(\omega_i)_{t_i}$ and $X(\omega_{i+1})_{s_{i+1}}$ are isomorphic, and they have trivial Picard group.
\end{enumerate}
 Parallel transport of $\ov{\ww}_a$ via the chain of twistor families above gives a mock Mukai vector $\ov{\ww}'_a$ for $X$. The first (and main) step is to prove that the statement of Theorem~\ref{thm:isogmodvb} holds for $M_{\ov{\ww}'_a}(X,\omega)$. 

Applying Theorem~\ref{thm:iperolom} to $M_{\ov{\ww}_a}(S^{[2]},\wt{\omega})^{\bu}$ and the first twistor family we get a holomorphic map  
\begin{equation}\label{chefatica}
\wt{M}_{\ov{\ww}_a}(\cX(\omega_1))^{\bu}\to T(\omega_1)
\end{equation}
 with fiber over $t$ isomorphic to the normalization
$\wt{M}_{\ov{\ww}_a(t)}(\cX(\omega_1)_t)^{\bu}$ of the  component of the moduli space $M_{\ov{\ww}_a(t)}(\cX(\omega_1)_t)$ 
corresponding to $M_{\ov{\ww}_a}(S^{[2]},\wt{\omega})^{\bu}$ via Theorem~\ref{thm:iperolom}. 
We wish to apply Corollary~\ref{crl:frontiera}  to   $\cY\to T(\omega_1)$ given by~\eqref{chefatica}, and $\cQ$ the universal bundle of projective spaces. We need to check that Assumption~\ref{ass:abicidi} holds. Item~(a) holds because 
$\wt{M}_{\ov{\ww}_a}(S^{[2]},\wt{\omega})^{\bu}$ is isomorphic to  $\cM_{v(a)}(S,h_S)$, which is a HK manifold of Type $K3^{[a^2+1]}$.
Item~(b) holds  with $\lambda_{t_0}=\wt{\lambda}_{\ov{\ww}_a}$ and $c=32^{-1}\cdot a^{-4}$ by Corollary~\ref{crl:isomraz}. We claim that for suitable choices of $\wt{\omega}$ also Item~(c) holds.
Let $L$ be an ample class on $S^{[2]}$  
which belongs to an open ${\mathsf a}(a)$-chamber 
 containing $\bm\mu(h_S)$ in its closure. Then $\lambda_{\ov{\ww}_a}(c_1(L))$ is ample by Proposition~\ref{prp:ampiolambda}.
  This is not enough because~\eqref{moonlight} cannot hold with 
$\wt{\omega}=c_1(L)$. However, since $\lambda_{\ov{\ww}_a}$ is a Hodge isometry $\lambda_{\ov{\ww}_a}(\wt{\omega})$ is a K\"ahler class for $\wt{\omega}$ close enough to $c_1(L)$. Thus we may choose $\wt{\omega}$ as above and such that 
$\lambda_{\ov{\ww}_a}(\wt{\omega})$ is a K\"ahler class. Then  Assumption~\ref{ass:abicidi} holds. By Item~(C) above it follows that 
$X(\wt{\omega})_{t_1}\cong X(\omega_{2})_{s_{2}}$ is a HK manifold of Type $K3^{[a^2+1]}$ with trivial Picard group. By Corollary~\ref{crl:frontiera} we may iterate the above argument up to the last twistor family $\cX(\omega_n)\to T(\omega_n)$,
 i.e.~$\cX(\omega)\to T(\omega)$. From this Item~(1) of Theorem~\ref{thm:isogmodvb} follows at once, see Remark~\ref{rmk:esorciccio}. Item~(2) if $H^{1,1}_{\ZZ}(X)=\{0\}$ follows from Corollary~\ref{crl:frontiera}. If $X$ is projective with ample line bundle $L$, arguing as above we get that if 
 $\omega$ is close enough to $c_1(L)$ then $\lambda_{\ov{\ww}'_a}(\omega)$ is a K\"ahler class. Thus Item~(2)  follows from Corollary~\ref{crl:frontiera} also in this case.
 
We have proved that  Items~(1) and~(2) of  Theorem~\ref{thm:isogmodvb} hold for a certain choice of 
mock Mukai vector $(8a^3,\ov{\gamma}^{*}_a,4a^6 c_2(X)/3)$.
 Recall that 
$\ov{\gamma}^{*}_a=[4a^2 \wh{\gamma}^{*}_a]$ where $\wh{\gamma}^{*}_a=[\gamma_a]\in H^2(X,\ZZ)/2a H^2(X;\ZZ)$ and~\eqref{divequad} holds.
By the divisibility condition in~\eqref{divequad} it follows that there exists $\gamma_a\in H^2(X,\ZZ)$ such that 
$\ov{\gamma}^{*}_a=[4a^2 \gamma_a]$ and
\begin{equation}\label{marioadorf}
\divisore (\gamma_a)=1,
\qquad
q_X(\gamma_a)= 2a^2-2. 
\end{equation}
The  monodromy group $\Mon^2(X)<O(H^2(X;\ZZ))$ is the index-$2$ subgroup of positive spinor norm orthogonal transfomations, and hence it acts transitively on the set of $\gamma_a\in H^2(X,\ZZ)$ such that~\eqref{marioadorf} holds. From this one gets the validity 
of Items~(1) and~(2) of  Theorem~\ref{thm:isogmodvb} for any $\ov{\gamma}_a$ by repeating the argument above with twistor families whose twistor conics (in the period space) connect two markings $\phi,\phi'$ of $H^2(X;\ZZ)$ such that $\phi(\gamma^{*}_a)=\phi'(\gamma_a)$.

 Item~(3) of  Theorem~\ref{thm:isogmodvb} follows from Item~(1) and Corollary~\ref{crl:isomraz}.

\qed
\subsection{The proof of Theorem~\ref{thm:ognik3a2+1}}
\setcounter{equation}{0}  
Let hypotheses and notation be as in Subsection~\ref{subsec:spalmo}. By Corollary~\ref{crl:kahlcomp} there exists a K\"ahler class 
$\omega\in\cK(S^{[2]})$   such that the following hold:
\begin{enumerate}
\item[(a)]
$\omega$ belongs to an open ${\mathsf a}(a)$-chamber which contains $\bm\mu(h_S)$ in its closure, 
\item[(b)]
$\omega^{\bot}\cap H^{1,1}_{\ZZ}(X)=\{0\}$, and
\item[(c)]
$\wt{\omega}\coloneq \wt{\varphi}_{\ov{\ww}_a}(\omega)\in\cK(\wt{M}_{\ov{\ww}_a}(S^{[2]},\wh{\omega})^{\bullet})$.
\end{enumerate}
By~\cite[Props.~3.7, 5.4]{huybrechts:bourbtorelli} there exist generic Twistor families $\cZ(\wt{\omega}(i))\to T(\wt{\omega}(i))$ of HK manifolds of Type $K3^{[a^2+1]}$, for $i\in\{0,\ldots,m\}$,  with the following properties:
\begin{enumerate}
\item
There exists $x_0\in T(\wt{\omega}(0))$  such that $Z(\wt{\omega}(0))_{x_0}\cong \wt{M}_{\ov{\ww}_a}(S^{[2]},\wh{\omega})^{\bullet}$,  
 $\wt{\omega}(0)=\wt{\omega}$.  
\item
There exists  $x_m\in T(\wt{\omega}(m))$ such that $Z(\wt{\omega}(m))_{x_m}\cong W$. 
\item
For all $i\in\{0,\ldots,m-1\}$ there exist $s_i\in T(\wt{\omega}(i))$,  $\ov{s}_{i+1}\in T(\wt{\omega}(i+1))$ with
\begin{equation}\label{irandisfatta}
Z(\wt{\omega}(i))_{s_i}\cong Z(\wt{\omega}(i+1))_{\ov{s}_{i+1}},\qquad H^{1,1}_{\ZZ}(Z(\wt{\omega}(i))_{s_i})=\{0\}.
\end{equation}
\end{enumerate}
Next we define twistor families of HK manifolds of Type $K3^{[2]}$ via $\wt{\varphi}_{\ov{\ww}_a}$ and the isometries obtained from $\wt{\varphi}_{\ov{\ww}_a}$ via parallel transport.
 Let $\cX(\omega)\to T(\omega)$ be the twistor family of HK manifolds of Type $K3^{[2]}$ associated to $\omega$. 
Then we have the isomorphism $u\colon T(\omega)\xrightarrow{\sim} T(\wt{\omega}(0))$ defined by $\wt{\varphi}_{\ov{\ww}_a}$: let
 $t_0\coloneq u^{-1}(s_0)$. 
By Corollary~\ref{crl:frontiera} (and the second equality in~\eqref{irandisfatta}) applied to the family $\cY\to T(\omega)$ with fiber the moduli space $\wt{M}(X(\omega)_t,\omega_t)^{\bu}$ over $t$ 
we have $t_0\in T(\omega)_0$, in particular there exists an isomorphism 
\begin{equation}
f_{t_0}\colon  \wt{M}_{\ov{\ww}_{a,t_0}}(X(\omega)_{t_0},\omega_{t_0})^{\bullet}  \xrightarrow{\sim} Z(\wt{\omega}(0))_{s_0}
\end{equation}
By~\eqref{irandisfatta} the K\"ahler cone of $Z(\wt{\omega}(0))_{s_0}$ (and of $\wt{M}_{\ov{\ww}_{a,t_0}}(X(\wt{\omega}(0))_{t_0},\wt{\omega}(0)_{t_0})^{\bullet}$) is one connected component of the cone of positive (BBF) square elements of $H^{1,1}_{\RR}$. 
Since $\wt{\varphi}_{\ov{\ww}_{a,t_0}}$ is an isometry over $\QQ$ it follows that also  the K\"ahler cone of $X(\omega)_{t_0}$ 
 is one connected component of the cone of positive (BBF) square elements of $H^{1,1}_{\RR}$ and also that
\begin{equation}\label{gelato}
\wt{\varphi}_{\ov{\ww}_{a,t_0}}(\cK(X(\omega)_{t_0}))=\cK(\wt{M}_{\ov{\ww}_a(t_0)}(X(\omega)_{t_0}).
\end{equation}
Now $Z(\wt{\omega}(1))_{u_1}\cong Z(\wt{\omega}(0))_{s_0}$, hence the twistor family $\cZ(\omega(1))$ is given by a K\"ahler class in 
$Z(\wt{\omega}(0))_{s_0}$, which corresponds to a K\"ahler class $\omega(1)$  on $X(\omega)_{t_0}$ by~\eqref{gelato}. 
Iterating   we get generic Twistor families 
$\cX(\omega(i))\to T(\omega(i))$ of HK manifolds of Type $K3^{[2]}$, for $i\in\{0,\ldots,m\}$, and isomorphisms $u(i)\colon T(\omega(i)) 
\xrightarrow{\sim} T(\wt{\omega}(i)) $ of conics with $\cX(\omega(0))\to T(\omega(0))$ equal to $\cX(\omega)\to T(\omega)$ and $u(0)=u$. 
Let $\ov{\ww}(i)_{t}$ be the mock Mukai vector for $X(\omega(i))_t$ obtained by parallel transport from $\ov{\ww}_a$ for  $i\in\{0,\ldots,m\}$, and
$t_i\coloneq u(i)^{-1}(s_i)$  for  $i\in\{0,\ldots,m-1\}$.
By Corollary~\ref{crl:frontiera}
 we have  $t_i\in T(\omega(i))_0$ for $i\in\{0,\ldots,m-1\}$, and hence 
\begin{equation}
\wt{M}_{\ov{\ww}(i)_{t_i}}(X(\omega(i))_{t_i},\omega(i)_{t_i})^{\bu}\cong Z(\omega(i))_{\ov{s}_i},\qquad i\in\{0,\ldots,m\}.
\end{equation}
Now apply  Corollary~\ref{crl:frontiera} and Remark~\ref{rmk:esorciccio} to the  twistor family  $\cX(\omega(m))\to T(\omega(m))$ 
and the corresponding family $\cY\to T(\omega(m))$ with fiber $\wt{M}_{\ov{\ww}(m)_{t}}(X(\omega(m))_t,\omega(m)_t)^{\bu}$ over $t$. 
Let $t_{*}\coloneq u(m)^{-1}(x_m)$ (recall that $Z(\wt{\omega}(m))_{x_m}\cong W$).
We get that $W$ is bimeromorphic to  
$\wt{M}_{\ov{\ww}(m)_{t_{*}}}(X(\omega(m))_{t_{*}},\omega(m)_{t_{*}})^{\bullet}$, and that they are isomorphic if 
$H^{1,1}_{\ZZ}(X(\omega(m))_{t_{*}})=\{0\}$, i.e.~if $H^{1,1}_{\ZZ}(W)=\{0\}$.  
 
 If $W$ is projective we  argue differently. Suppose first that $H^{1,1}_{\ZZ}(W)=\ZZ\alpha$ with $q_W(\alpha)>0$. We have proved that there exist a HK manifold $X$ of Type $K3^{[2]}$ and a K\"ahler class $\omega$  on $X$ for which $\wt{M}_{\ov{\ww}_a}(X,\omega)^{\bu}$ is holomorphic symplectic manifold and there is a bimeromorphic map
\begin{equation}\label{stremato}
g\colon W\dra \wt{M}_{\ov{\ww}_a}(X,\omega)^{\bu}.
\end{equation}
The composition of the rational isogeny $H^2(X)\xrightarrow{\sim} H^2(\wt{M}_{\ov{\ww}_a}(X,\omega)^{\bu})$ 
 in~\eqref{stessisuqu} and the  rational isogeny $H^2(\wt{M}_{\ov{\ww}_a}(X,\omega)^{\bu})\xrightarrow{\sim} H^2(W)$ is 
  rational isogeny $H^2(X)\xrightarrow{\sim} H^2(W)$. It follows that  $H^{1,1}_{\ZZ}(X)=\ZZ h$ where $q_X(h)>0$ and $h,\omega$ belong to the same connected component of the cone of classes in $H^{1,1}_{\RR}(X)$ of positive square. Thus $h$ is ample, and moreover there is a single 
open ${\mathsf a}(a)$-chamber in $\cK(X)$. Thus   $\wt{M}_{\ov{\ww}_a}(X,\omega)^{\bu}\cong \wt{M}(X,h)^{\bu}$, in particular it is projective (and hence hyperk\"ahler). Since it has Picard number $1$, the map $g$ in~\eqref{stremato} is an isomorphism. 
This proves the last statement in  Theorem~\ref{thm:ognik3a2+1} if $W$ is projective with cyclic Picard group.

Now we prove  the last statement in  Theorem~\ref{thm:ognik3a2+1} for an arbitrary  projective $W$. There exist a projective  family 
$f\colon \cW\to T$ over a smooth connected curve $T$ and $t_0,t_1\in T$  such that $W_{t_0}\coloneq f^{-1}(t_0)$ has cyclic Picard group, $W\cong W_{t_1}\coloneq f^{-1}(t_1)$. We have proved that  $W_{t_0}\cong \wt{M}_{\ov{\ww}_a}(X_0,h_0)^{\bu}$ for $(X_0,h_0)$ a suitable polarized HK manifold of Type $K3^{[2]}$.
Starting from the rational isogeny  
 in~\eqref{stessisuqu} we associate to $f\colon \cW\to T$ a  projective  family 
$g\colon \cX\to T$ of HK manifolds of Type $K3^{[2]}$, with relative polarization $\cH$, such that $ (X_{0},h_0)\cong (g^{-1}(t_0),H_{|g^{-1}(t_0)})$ and such that the following holds.  Let 
$p\colon \wt{M}_{\ov{\ww}_a}(\cX,\cH)^{\bu}\to T$ be the relative moduli space with fiber $\wt{M}_{\ov{\ww}_a}(X_t,h_t)^{\bu}$ over $t$ (to be precise we might need to first do a base change because a priori $\wt{M}_{\ov{\ww}_a}(X_t,h_t)^{\bu}$ is \lq\lq not unique\rq\rq). Then there exist an open subset $U\subset T$ containing $t_0$ and an isomorphism
\begin{equation}
\xymatrix{
f^{-1}(U) \ar[rr]^{\sim} \ar[dr]_{f_{|\ldots}} & & p^{-1}(U) \ar[dl]^{g_{|\ldots}} \\
& U & }
\end{equation}
Thus we have an effective limit cycle $\Gamma$ on $W_{t_1}\times \wt{M}_{\ov{\ww}_a}(X_{t_1},h_{t_1})^{\bu}$ of pure dimension $2a^2+2$. 
Write $\Gamma=\Gamma_0+\sum_i m_i Z_i$ where $\Gamma_0$ is the graph of a birational map $W_{t_1}\dra \wt{M}_{\ov{\ww}_a}(X_{t_1},h_{t_1})^{\bu}$ and each of the $Z_i$'s maps to a proper closed subset of the two factors $W_{t_1}\times \wt{M}_{\ov{\ww}_a}(X_{t_1},h_{t_1})^{\bu}$. We have that $\Gamma^{*}(\wt{\lambda}_{\ov{\ww}_a}(h_{t_1})$ is an ample class on $W_{t_1}$, an hence arguing as in the proof of Proposition~\ref{prp:frontiera} (i.e.~quoting the proof of~\cite[Thm.~2.5]{huybrechts:kaehlercone}) we get that 
$\Gamma_0$ is the graph of an isomorphism.
\qed

\subsection{The proof of Theorem~\ref{thm:isogenie}}
\setcounter{equation}{0}  
\begin{lmm}\label{lmm:ciclialg}
Let hypotheses be as in Theorem~\ref{thm:isogmodvb} with $X$  projective and $\omega=h$  an ample class, so that  $\wt{M}_a\coloneq \wt{M}_{\ov{\ww}_a}(X,h)^{\bu}$ is projective. Then the following hold:
\begin{enumerate}
\item
There exists $Z\in\CH^4_{\QQ}(X\times \wt{M}_a)$ such that $\cl(Z)\in H^6(X)\otimes H^2(\wt{M}_a)\cong H^2(X)^{\vee}\otimes H^2(M_a)$ equals the isogeny in~\eqref{stessisuqu}. 
\item
There exists $P\in\CH^{2a^2+2}_{\QQ}(\wt{M}_a\times X)$ such that $\cl(P)\in H^{4a^2+2}(\wt{M}_a)\otimes H^2(X)\cong H^{2}(\wt{M}_a)^{\vee}\otimes H^2(X)$ equals the inverse of the isogeny in~\eqref{stessisuqu}. 
\end{enumerate}
\end{lmm}
\begin{proof}
Since $X$ and $\wt{M}_a$ are projective the rational class $32^{-1}\cdot a^{-4}\cdot c_2(\gotg(\wt{\cP}^{\bullet}_{\ov{\ww}_a}))$ lifts to a class 
$Z_0\in\CH^2_{\QQ}(X\times \wt{M}_a)$ whose cohomology class has K\"unneth components $(4,0)$, $(2,2)$ and $(0,4)$ ( $X$ and $ \wt{M}_a$ are simply connected). Choose $(x_0,m_0)\in X\times \wt{M}_a$, and let  $i_{m_0}\colon X\hra X\times \wt{M}_a$,  $i_{x_0}\colon \wt{M}_a\hra X\times \wt{M}_a$ be given by $i_{m_0}(x)\coloneq (x,m_0)$,  $i_{x_0}(m)\coloneq (x_0,m)$. 
The algebraic cycle  $\ov{Z}_0\coloneq Z_0-i_{m_0,*}(i_{m_0}^{*}Z_0)-i_{x_0,*}(i_{x_0}^{*}Z_0)$ has 
K\"unneth $(2,2)$ component equal to that of $Z_0$ and no other K\"unneth $(2,2)$ component.  Let 
$\wt{Z}_0\coloneq\ov{Z}_0\cdot p_X^{*}(h_X^2)\in\CH^4_{\QQ}(X\times \wt{M}_a)$ where 
$p_X\colon X\times \wt{M}_a\to \wt{M}_a$ is the projection and
$h_X\in\CH^1(X)$ is an ample divisor class. The  identity $\int_X a\cdot b\cdot \cl(h_X)^2=(a,b)_X\cdot q_X(h_X)+2(a,h_X)_X\cdot(b,h_X)_X$, valid for any $a,b\in H^2(X)$, gives that 
the action of $\cl(\wt{Z}_0)$ on 
$c\in H^2(X)$ is given by
\begin{equation}
\cl(\wt{Z}_0)(c)=q_X(h_X)\cdot\cl(Z_0)(c)+2(c,h_X)_X\cdot \cl(Z_0)(h_X).
\end{equation}
%
Let $h_{\wt{M}_a}\in\CH^1(\wt{M}_a)$ be such that  
$\cl(h_{\wt{M}_a})=\cl(Z_0)(\cl(h_X))\in H^{1,1}_{\QQ}(\wt{M}_a)$. Let
\begin{equation}
Z\coloneq q_X(h_X)^{-1}\cdot\left(\wt{Z}_0-\frac{2}{3}p_X^{*}(h_X^3)\cdot p_{\wt{M}_a}^{*}(h_{\wt{M}_a})\right),
\end{equation}
where $p_{\wt{M}_a}\colon X\times \wt{M}_a\to \wt{M}_a$ is the projection. Then Item~(1) holds. The proof of Item~(2) is analogous.

\end{proof}
We are ready to prove Theorem~\ref{thm:isogenie}. By Theorem~\ref{thm:ognik3a2+1} for $i\in\{1,2\}$ there exists a polarized HK manifold 
$(X_i,h_i)$ such that $W_i\cong \wt{M}_{\ov{\ww}_a}(X_i,h_i)^{\bullet}$. By Lemma~\ref{lmm:ciclialg} there exist $Z_2\in\CH^4_{\QQ}(X_2\times W_2)$ such that $\cl(Z_2)\in  H^6(X_2)\otimes H^2(W_2)$ equals the isogeny $\phi_2\colon H^2(X_2)\to H^2(W_2)$ in~\eqref{stessisuqu}, and  
$P_1\in\CH^{2a_1^2+2}_{\QQ}(W_1\times X_1)$ such that $\cl(P_1)\in H^{4a_1^2+2}(W_1)\otimes H^2(X_1)$ equals the inverse $\phi_1^{-1}$ of the isogeny 
$\phi_1\colon H^2(X_1)\to H^2(W_1)$ in~\eqref{stessisuqu}. By~\cite[Thm.~1.1]{markman:isogenies} there exists an algebraic cycle 
$V\in \CH^4_{\QQ}(X_1\times X_2)$ such that the K\"unneth component $\cl(V)_{6,2}\in H^6(X_1)\otimes H^2(X_2)$ equals the Hodge isometry 
$\phi_2\circ f\circ \phi_1^{-1}$. 
The cycle class $Z\coloneq Z_2\circ V\circ P_1$ (composition of correspondences) satisfies the thesis of   Theorem~\ref{thm:isogenie}.
\qed


 \end{document}